\documentclass[11pt, twoside]{report}
\usepackage{fancyhdr}
{}
\usepackage[utf8]{inputenc}
\usepackage[top=1.15in, bottom=1.15in, left=1.17in, right=1.17in]{geometry}
\usepackage{amsthm, amsmath, amssymb, amscd, latexsym, multicol, verbatim, enumerate, graphicx,xy, color}
\usepackage{amstext,amsfonts,amssymb,amscd,amsbsy,amsmath}

\usepackage[nottoc,numbib]{tocbibind}
\usepackage[titletoc]{appendix}
\usepackage{tikz,float}
\usetikzlibrary{patterns}
\usepackage[english]{babel}
\usepackage{cite}
\usepackage{stmaryrd}
\usepackage[bookmarks=false, colorlinks=true, pdfstartview=FitV, linkcolor=blue, citecolor=blue, urlcolor=blue, breaklinks=true]{hyperref}

\usepackage{helvet}         
\usepackage{courier}        
\usepackage{type1cm}  
\usepackage{multicol}        
\usepackage[bottom]{footmisc}
\usepackage{longtable}

\usepackage{hyperref}
\usepackage{makeidx}         
\usepackage{graphicx}        

\makeindex             

\usepackage{amstext,amsfonts,amssymb,amscd,amsbsy,amsmath}
\usepackage[ruled,vlined]{algorithm2e}
\usepackage{algpseudocode}
\usepackage{tikz-cd}	
\usepackage{tikz}

\numberwithin{equation}{section}

\theoremstyle{definition}
\newtheorem{definition}{Definition}
\newtheorem{proposition}{Proposition}
\newtheorem{theorem}{Theorem}
\newtheorem{lemma}{Lemma}
\newtheorem{corollary}{Corollary}
\newtheorem{remark}{Remark}
\newtheorem{example}{Example}
\newtheorem{question}{Question}
\newtheorem{conjecture}{Conjecture}
\newtheorem*{corollary*}{Corollary}
\newtheorem*{lemma*}{Lemma}
\newtheorem*{theorem*}{Theorem}
\newtheorem*{proposition*}{Proposition}
\newtheorem*{problem*}{Problem}

\newcommand{\w}{\underline{w}}
\newcommand{\p}{\mathbf{p}}

\newcommand{\boundellipse}[3]
{(#1) ellipse (#2 and #3)
}

\newcommand{\X}{\mathcal{X}}
\newcommand{\G}{\mathcal{G}}
\newcommand{\A}{\mathcal{A}}

\newcommand{\ff}{\mathcal{F}}

\newcommand{\bff}{\mathbf{f}}
\newcommand{\bm}{\mathbf{m}}
\newcommand{\bn}{\mathbf{n}}
\newcommand{\bfff}{\mathbf{f}}
\newcommand{\bk}{\mathbf{k}}
\newcommand{\bw}{\mathbf{w}}
\newcommand{\bv}{\mathbf{v}}
\newcommand{\bu}{\mathbf{u}}

\newcommand{\area}{\mathsf{A}}

\newcommand{\pa}{\mathsf{pa}}

\newcommand{\PP}{\mathbb{P}}

\newcommand{\ZZ}{\mathbb{Z}}

\DeclareMathOperator{\trop}{trop}
\DeclareMathOperator{\mult}{mult}
\DeclareMathOperator{\Proj}{Proj}

\DeclareMathOperator{\res}{res}
\DeclareMathOperator{\es}{\mathsf{es}}
\DeclareMathOperator{\ord}{ord}
\DeclareMathOperator{\ind}{ind}

\newcommand{\Hom}{\operatorname*{Hom}}
\newcommand{\Lie}{\operatorname*{Lie}}
\newcommand{\Spec}{\operatorname*{Spec}}
\newcommand{\ddeg}{{\operatorname*{deg}}}
\newcommand{\gr}{{\operatorname*{gr}}}
\newcommand{\lie}{\mathfrak}
\newcommand{\Gr}{{\operatorname*{Gr}}}
\newcommand{\mut}{{\operatorname*{mut}}}

\newcommand{\ba}{\mathbf{a}}
\newcommand{\bb}{\mathbf{b}}
\newcommand{\bc}{\mathbf{c}}
\newcommand{\bd}{\mathbf{d}}
\newcommand{\be}{\mathbf{e}}

\newcommand{\val}{\mathfrak{v}}

\DeclareMathOperator{\Flag}{\mathcal{F}\hspace{-1.6pt}\ell}

\DeclareMathOperator{\init}{in}
\DeclareMathOperator{\wei}{wt}
\DeclareMathOperator{\sgn}{sgn}
\DeclareMathOperator{\height}{ht}
\DeclareMathOperator{\conv}{conv}
\DeclareMathOperator{\cone}{cone}
\DeclareMathOperator{\rank}{rank}

\title{Toric degenerations: a bridge between representation theory, tropical geometry and cluster algebras}
\author{Lara Bossinger}
\date{May 2018}

\begin{document}

\maketitle
\thispagestyle{empty}

\newpage

\thispagestyle{empty}

\emph{Kurzzusammenfassung:}
In dieser Arbeit untersuchen wir torische Degenerierungen projektiver Varietäten. Wir interessieren uns für Konstruktionen solcher aus der Darstellungstheorie, der tropischen Geometrie und der Theorie von Cluster Algebren.
Ziel ist es, durch analysieren bestimmter Spezialfälle die Zusammenhänge der verschiedenen Theorien besser zu verstehen.

Im Fokus sind deshalb Varietäten, auf die eine Vielzahl von Methoden angewandt werden können: Grassmannsche, Fahnenvarietäten und Schubertvarietäten.

Wir vergleichen als ersten Schritt die torischen Varietäten, die als Degenrierungen erhalten werden. Vor allem interessiert uns ob isomorphe torische Varietäten von verschiedenen Kontruktionen erhalten werden.
Dies ist häufig der Fall, z.B. für die Grassmannsche von Geraden im $\mathbb C^n$ können alle torischen Varietäten, die man mit Methoden der tropischen Geometrie erhält (bis auf Isomorphie) auch mit Hilfe der Darstellungstheorie konstruiert werden.

Ein erstes allgemeines Resultat (für projektive Varietäten) lässt auf weitere tiefere Zusammenhänge hoffen: torische Degenerierungen, die mit Hilfe einer Bewertung und der Theorie von Newton-Okounkov Körpern erzeugt werden lassen sich (unter gewissen Bedingungen) mit Hilfe der tropischen Geometrie realisieren.

\newpage

\thispagestyle{empty}

\emph{Abstract:} In this thesis we study toric degenerations of projective varieties. 
We compare different constructions to understand how and why they are related. 
In focus are toric degenerations obtained from representation theory, tropical geometry or cluster algebras.
Often those rely on valuations and the theory of Newton-Okounkov bodies.
Toric degenerations can be seen as a combinatorial shadow of the original objects. 
The goal is therefore to understand why the different theories are so closely related, by understanding the toric degenerations they yield first.
We choose Grassmannians, flag varieties and Schubert varieties as starting point as here many different constructions are applicable. 
One of our main results gives a sufficient condition for when toric degenerations obtained using full-rank valuations, independent of how these are constructed, can be realized using tropical geometry.

\tableofcontents

\chapter{Introduction}\label{chap:intro}

\emph{Toric varieties} are popular objects in algebraic geometry due to a dictionary between their geometric properties (e.g. dimension, degree)  and properties of associated combinatorial objects (e.g. fans, polytopes). This dictionary can be extended from toric varieties to varieties admitting a \emph{toric degeneration}. A toric degeneration is a (flat) family of varieties that share many properties with each other.
We mostly consider $1$-parameter toric degenerations of certain projective varieties $X$. 
These are flat families $\varphi:\mathcal F\to \mathbb {A}^{1}$, where the fiber over zero (also called \emph{special} fiber) is a toric variety and all other fibers are isomorphic to $X$. 
Once we have such a degeneration, some of the algebraic invariants of $X$ are the same for all fibers (e.g. dimension, degree, Hilbert-polynomial), hence the computation can be done on the toric fiber. 
In the case of a toric variety such invariants are easier to compute than in the case of a general variety due to a nice combinatorial description. 
\begin{example}\label{exp:Gr(2,4)}
Consider the Grassmannian $\Gr(2,\mathbb C^4)$ of $2$-dimensional subspaces of $\mathbb C^4$. It is given by the vanishing of the ideal $I=(p_{12}p_{34}-p_{13}p_{24}+p_{14}p_{23})$ in the polynomial ring on Pl\"ucker variables. 
A toric degeneration of $\Gr(2,\mathbb C^4)$ is given by the family $I_t=(p_{12}p_{34}-p_{13}p_{24}+tp_{14}p_{23})$ for $t\in \mathbb C$.
Setting $t=1$ we obtain $\Gr(2,\mathbb C^4)$, and setting $t=0$ we get the toric variety defined by the vanishing of $I_0=(p_{12}p_{34}-p_{13}p_{24})$.
\end{example}
The study of toric degenerations has various applications in pure and applied mathematics, for example in mirror symmetry and statistics.
Tailored to the variety of interest, it is a great challenge to decide which toric degeneration has the desired properties. The task is therefore to study and compare all possible constructions. In this context, varieties from representation theory can be thought of as a fertile ground to develop different techniques and test for their fitness. Three main fields intersect here: \emph{representation theory}, \emph{tropical geometry} and the \emph{theory of cluster algebras}. All three can be applied to these varieties and yield toric degenerations with the associated combinatorial data encoding geometric properties. 

One can think about the combinatorics appearing in this setting (e.g. cones, polytopes, semi-groups) as a shadow of a deeper connection between the theories. The aim is to understand this connection and develop a global framework into which all three settings can be embedded. In the process of doing so, this thesis is concerned with understanding first special cases to obtain an intuition for the global picture.

\section{Background and Motivation}

We explain in more detail the constructions of toric degenerations that are in focus. The first is the framework of \emph{birational sequences} by Fang, Fourier and Littelmann introduced in \cite{FFL15}, see \S\ref{subsec:birat}. This work has its origin in representation theory or Lie theory. Second, we consider toric degenerations arising in tropical geometry by \emph{tropicalizing} projective varieties, summarized in \S\ref{subsec:trop}. For background on this topic we refer to the textbook by Maclagan and Sturmfels \cite{M-S}. The third context in which toric degenerations arise that is of great interest to us is the theory of cluster algebras introduced by Fomin and Zelevinsky \cite{FZ02} with constructions of toric degenerations due to Gross, Hacking, Keel and Kontsevich in \cite{GHKK14}. We summarize it briefly in \S\ref{subsec:cluster}.

In all three settings the notion of \emph{Newton-Okounkov body} appears. Let $X$ be a projective variety and $\mathbb{C}[X]=:A$ its homogeneous coordinate ring. 
For a valuation $\val:A\setminus \{0\}\to \mathbb Z^N$ (see Definition~\ref{def: valuation}), by $S(A,\val)$ we denote the value semi-group (the image of the valuation). 
The \emph{Newton-Okounkov cone} is the convex hull of $S(A,\val)\cup\{0\}$. After intersecting this cone with a particular hyperplane one obtains a convex body, called \emph{Newton-Okounkov body}. 
These objects have been introduced in a series of papers by different authors (\cite{Oko98},\cite{LM09},\cite{KK12},\cite{An13}) as a far generalization of Newton polytopes to study the asymptotics of line bundles on $X$. 
If $S(A,\val)$ is finitely generated, the Newton-Okounkov body is a polytope and there exists a flat degeneration of $X$ into a toric variety $Y$. 
The Newton-Okounkov body in this setting is the polytope associated to the normalization of $Y$.

\subsection{Birational sequences}\label{subsec:birat}
In \cite{FFL15} Fang, Fourier and Littelmann introduce the notion of a \emph{birational sequence}. They work with partial flag varieties and spherical varieties associated to a connected complex reductive algebraic group $G$. For simplicity we explain the case of flag varieties here.
Fixing Borel subgroup $B\subset G$ and a maximal torus $T\subset B$, let $R^+$ be the set of positive roots for $G$ and $N$ the cardinality of $R^+$. 
The main idea is to use the representation theory of $G$ to obtain coordinates on $G/B$ such that $\mathbb C(G/B)\cong \mathbb C(x_1,\dots,x_N)$.
On the right hand side by choosing a monomial order (resp. a total order on $\mathbb Z^N$) one can define \emph{lowest-term valuations} in a straight forward way (more details in \S\ref{sec:pre_birat}).
This idea is used frequently, we encounter it again when considering valuations constructed using cluster algebra structures.

To every positive root $\beta\in R^+$ there exists a one-parameter root subgroup $U_{-\beta}\subset U^-$, where $U^-$ is the unipotent radical of the opposite Borel subgroup $B^-\subset G$. 

\begin{definition}\label{def:birat seq}
A \emph{birational sequence} is a sequence of positive roots $S=(\beta_1,\dots,\beta_N)$ such that the multiplication map 
$U_{-\beta_1}\times\dots\times U_{-\beta_N}\to U^-$
is a birational morphism. 
\end{definition}

In particular, as $U_{-\beta_1}\times\dots\times U_{-\beta_N}\cong \mathbb A^N$, a birational sequence yields $\mathbb C(G/B)\cong \mathbb C(U^-)\cong \mathbb C(x_1,\dots,x_N)$.
Fixing a total order $\prec$ on $\mathbb Z^N$, they construct a valuation $\val_S$ on the ring of $U$-invariant functions on $G$ that restricts to the homogeneous coordinate ring of the flag variety $G/B$. Further, they study the associated Newton-Okounkov body. 
Their methods generalize constructions in representation theory (Lie theory) and use ideas from \emph{PBW-filtrations}. For example, consider $S_n$ the Weyl group of $SL_n$ and a reduced expression $\w_0$ of the longest word $w_0\in S_n$.
Choosing the birational sequence consisting of the simple roots associated to $\w_0$ and the reverse lexicographic order on $\mathbb Z^N$, they recover the toric degeneration of $SL_n/B$ by Gonciulea and Lakshmibai \cite{GL96} and the Gelfand-Tsetlin polytope \cite{GT50}. 
This degeneration was further studied by Kogan and Miller in \cite{KM05}. Generalizing to arbitrary flag varieties and arbitrary reduced expressions of $w_0$, the degenerations by Caldero \cite{Cal02} and Alexeev-Brion \cite{AB04} are recovered. 
They give degenerations of flag varieties to toric varieties associated to string polytopes introduced by Littelmann \cite{Lit98} and Berenstein-Zelevinsky \cite{BZ01}. 
The string polytopes parametrize elements of Lusztig's dual canonical basis. 
In \cite{Kav15} Kaveh showed how string polytopes are realized as Newton-Okounkov bodies.
Another well-studied basis for $G$-representations (and so for the homogeneous coordinate rings of flag varieties) was introduced in a series of papers by Feigin, Fourier and Littelmann (\cite{FFL11} and \cite{FeFL16}), generalizing Feigin's work \cite{Fe12}. 
This basis is parametrized by the \emph{Feigin-Fourier-Littelmann-Vinberg polytope}. The polytope exists in types A and C, and its lattice points parametrize the above mentioned basis of $G$-representations. 
Analogously to the case of string polytopes, FFLV-polytopes are realizable as Newton-Okounkov bodies as showed by Kiritschenko \cite{Kir15}. 
The FFLV-degeneration can also be recovered in the framework of \cite{FFL15}, by choosing the birational sequence to consist of all postive roots in a particular \emph{good ordering} (see \cite{FFL15}).

One starting point for this thesis was to answer the following question:
\begin{question}\label{Q:how gen birat}
Does the framework of birational sequences extend beyond known toric degenerations in representation theory?
\end{question}
To make this question more precise let us briefly introduce how toric degenrations arise from tropical geometry and cluster algebras.


\subsection{Tropical Geometry}\label{subsec:trop}
Tropical geometry is a relatively new field at the intersection of algebraic geometry and polyhedral geometry.
We are mostly interested in the \emph{tropicalization} of complex projective varieties, which essentially means studying the algebraic variety over the \emph{tropical semiring} instead of over $\mathbb C$.
The tropical semiring is $\mathbb R\cup\{-\infty\}$ with multiplication in $\mathbb R$ being replaced by addition and addition in $\mathbb R$ being replaced by taking the minimum. 

In this sense, the tropicalization of a projective variety $X=V(I)\subset \mathbb P^{n-1}$, denoted $\trop(X)\subset\mathbb R^n$, is the support of a rational polyhedral complex of dimension $\dim X$ (see \cite[Theorem 3.3.5]{M-S}).
It can be interpreted as a combinatorial shadow of its algebraic counterpart $X$. 
This computational approach to tropical geometry is closely related to commuative algebra and Gr\"obner theory. 
In fact, $\trop(X)$ is contained in the Gr\"obner fan associated to $X=V(I)$, i.e. every cone $C\subset\trop(X)$ has an associated (monomial-free) \emph{initial ideal} $\init_C(I)$, a deformation of the ideal $I$ defining $X$.
Using Gr\"obner theory, every cone $C\subset\trop(X)$ yields a flat degeneration of $X$ into $V(\init_C(I))$. 
In particular, if $\init_C(I)$ is a binomial prime ideal this yields a toric degeneration of $X$. In this case $C$ is called a \emph{maximal prime} cone of $\trop(X)$.

The tropicalization of $\Gr(2,\mathbb C^n)$ has been studied in \cite{SS04} by Speyer and Sturmfels. 
They show that $\trop(\Gr(2,\mathbb C^n))$ is parametrized by trivalent trees with $n$ leaves, i.e. to every maximal cone they associate a trivalent tree that encodes the initial ideal corresponding to the cone. 
Further, they prove that every inital ideal coming from a maximal cone in $\trop(\Gr(2,\mathbb C^n))$ is prime, hence yields a toric degeneration of $\Gr(2,\mathbb C^n)$. This enables us to formulate a more precise version of Question~\ref{Q:how gen birat} in this case:

\begin{question}\label{Q:trop from birat}
Can we find a birational sequence for $\Gr(2,\mathbb C^n)$ corresponding to every maximal cone of $\trop(\Gr(2,\mathbb C^n))$?
\end{question}

Besides $\Gr(2,\mathbb C^n)$ very little is known about the tropicalization of (partial) flag varieties and more generally spherical varieties. The cases of $\Gr(3,\mathbb C^6)$ and $\Gr(3,\mathbb C^7)$ were computed in \cite{HJJS} (see also \cite{SS04}) and more conceptually Mohammadi and Shaw study $\trop(\Gr(3,\mathbb C^n))$ in \cite{MoSh}.

The tropicalization of a variety $X$ is closely related to valuations on its homogeneous coordinate ring $\mathbb C[X]$ as studied in \cite{KM16} by Kaveh and Manon. They associate a full-rank valuation to every maximal prime cone $C$ in $\trop(X)$.
They further introduce the notion of \emph{Khovanskii basis}, a set of algebra generators for $\mathbb C[X]$, whose images under the valuation generate the value semi-group. 
Kaveh and Manon show, that the existance of a maximal prime cone in the tropicalization is equivalent to the existance of a finite Khovanksii basis for the associated valuation.
This further implies that the corresponding Newton-Okounkov body is a polytope and they show how it can be computed from $C$.

From a representation theoretic point of view the full flag variety $\Flag_n$ is generally an object of great interest. Inspired by \cite{KM16} and keen on applying their methods this lead us to the question:

\begin{question}\label{Q:trop flag}
What is the tropicalization of $\Flag_n$? Can it be computed in small cases, can we find finite Khovanskii bases from it, and if so are the associated Newton-Okounkov bodies related to those from birational sequences?
\end{question}


\subsection{Cluster algebras}\label{subsec:cluster}
Cluster algebras were introduced in \cite{FZ02} by Fomin and Zelevisnky and quickly grew to become a research area on their own. They are commutative rings endowed with \emph{seeds} (maximal sets of algebraically independent generators) related by \emph{mutation} (local transformations exchanging one seed by another). At their origin they are closely related to the representation theory of finite dimensional algebras, but also many objects related to algebraic groups have a cluster structure. For example, the homogeneous coordinate ring of Grassmannians (see \cite{FZ02}, \cite{Sco06}), double Bruhat cells (see \cite{BFZ05}), (partial) flag varieties (see \cite{GLS13}) or Richardson varieties (see \cite{Lec16}). 

A geometric appraoch to cluster algebras was introduced by Fock and Goncharov in \cite{FG06}. In this setting they work with \emph{cluster varieties}, schemes glued from algebraic tori (one for every seed) with gluing given by the birational transformations induced by mutation. They come in two flavours, $\mathcal A$- and $\mathcal X$-cluster varieties, one being the \emph{mirror dual} to the other as developed by Gross, Hacking, Keel and Kontsevich in \cite{GHKK14}. Among other things, they define \emph{$\vartheta$-bases} for cluster algebras and toric degenerations of (partial compactifications of) cluster varieties. 
The $\mathcal X$-cluster variety comes endowed with a Laurent polynomial, the \emph{superpotential}, whose tropicalization gives a polyhedral cone and a polytope as a slice of this cone. The superpotential polytope is the polytope associated to the special fibre of the toric degeneration.

A similar approach for Grassmannians can be found in recent work of Rietsch and Williams \cite{RW17}. They consider the $\mathcal A$- and $\mathcal X$-cluster varieties contained in the Grassmannian and combine ideas from Newton-Okounkov bodies with cluster duality and mirror symmetry. 
Using $\X$-cluster coordinates as coordinates for the Grassmannian they construct lowest term valuations on the homogeneous coordinate ring for every seed.
On the $\mathcal A$-cluster variety they consider a potential function that was defined by Marsh and Rietsch in \cite{MR04}. Its tropicalization yields a polytope. 
They show that the Newton-Okounkov body associated with the valuation is the polytope given by the potential. 
In particular, they obtain explicit inequalities describing the Newton-Okounkov polytope. 

\begin{question}\label{Q:cluster}
Can the toric degenerations of $\Flag_n$ (resp. $\Gr(k,\mathbb C^n)$) arising from tropicalizing a superpotential be recovered as toric degenerations from the tropicalization of $\Flag_n$ (resp. $\Gr(k,\mathbb C^n)$) or birational sequences?
\end{question}

A first hint towards a positive answer to this question for $\Flag_n$ was given by Magee in \cite{Mag15}. He recovers the Gelfand-Tsetlin polytope as a superpotential polytope in a particular seed.
Further results in this direction are obtained by Genz-Koshevoy-Schumann in \cite{GKS} and \cite{GKS2}, who generalize Magee's result to flag varieties of simple, simply connected, simply laced algebraic groups. They recover the classical string and Lusztig parametrizations from the superpotential.

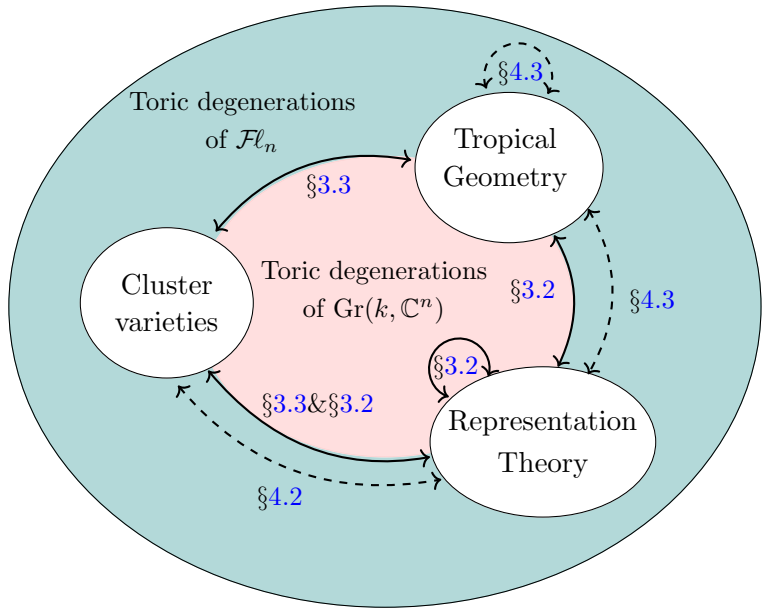
\begin{figure}
\centering
\begin{tikzpicture}

\draw [fill, teal, opacity=.3] \boundellipse{1.9,-.2}{5}{4};
\draw \boundellipse{1.9,-.2}{5}{4};
\draw[white, fill] \boundellipse{1.9,-.2}{2.5}{2};
\draw[pink, fill, opacity=.5] \boundellipse{1.9,-.2}{2.5}{2};
\draw[fill, white] \boundellipse{3.55,1.65}{1.25}{1};
\draw \boundellipse{3.55,1.65}{1.25}{1};
\node at (3.5,2) {Tropical};
\node at (3.5,1.5) {Geometry};

\draw[fill, white] \boundellipse{4,-2}{1.5}{1};
\draw \boundellipse{4,-2}{1.5}{1};
\node at (4, -1.75) {Representation};
\node at (4,-2.3) {Theory};

\draw[fill, white] \boundellipse{-1,-.15}{1.15}{1};
\draw \boundellipse{-1,-.15}{1.15}{1};
\node at (-1,0.1) {Cluster};
\node at (-1,-.4) {varieties};

\node at (1.75,.25) {\small{Toric degenerations}};
\node at (1.75,-.25) {\small{of $\Gr(k,\mathbb C^n)$}};

\node at (0,2.5) {\small{Toric degenerations}};
\node at (0,2) {\small{of $\Flag_n$}};

\draw [<->, thick] (-0.35,.8) to [out=50,in=170] (2.25,1.75);
    \node at (1.15,1.4) {\small{\S\ref{sec:BFFHL}}};
\draw [<->, thick, dashed] (3.25,2.7) arc (200:-30:.45cm);
\node at (3.7,2.9) {\small{\S\ref{sec:BLMM}}}; 

\draw [<->, thick, dashed] (4.6,1.1) to [out=300, in=60] (4.6,-1.1);
    \node at (5.45,-0.125) {\small{\S\ref{sec:BLMM}}}; 
\draw [<->, thick] (4.15,.75) to [out=300, in=65] (4.2,-1);
    \node[left] at (4.3,0) {\small{\S\ref{sec:Bos}}};

\draw [<->, thick] (2.75,-1.4) arc (250:-15:.4cm);
\node[above right] at (2.4,-1.3) {\small{\S\ref{sec:Bos}}}; 

\draw [<->, thick] (-.45,-1.05) to [out=310,in=190] (2.5,-2.2);
    \node[above] at (1,-1.8) {\small{\S\ref{sec:BFFHL}\&\S\ref{sec:Bos}}}; 
\draw [<->, thick, dashed] (-.85,-1.25) to [out=310,in=190] (2.65,-2.5);
    \node at (.5,-2.75) {\small{\S\ref{sec:BF}}}; 


\end{tikzpicture}
\caption{The landscape of toric degenerations subject in this thesis.}
\label{fig:overview}
\end{figure}

\section{Results}

We summarize below the results in this thesis and explain which of the above questions could be answered in which generality.
In Chapter~\ref{chap:prep} we recall the necessary general background on the representation theory of $SL_n(\mathbb C)$, tropical geometry, valuations and cluster algebras.
We explain quasi-valuations with weighting matrices and prove a general result relating arbitrary full-rank valuations with such in \S\ref{sec:val and quasival}.
Chapter~\ref{chap:Grass} studies different constructions of toric degenerations of the Grassmannians. 
More specifically, in \S\ref{sec:Bos} the class of \emph{iterated} birational sequences is defined. As an application one obtains for $\Gr(2,\mathbb C^n)$ a precise relation between birational sequences and the tropical Grassmannian. In \S\ref{sec:BFFHL} the connection between the tropical Grassmannian and the cluster combinatorics given by plabic graphs is studied for $\Gr(2,\mathbb C^n)$ and $\Gr(3,\mathbb C^6)$.
The last part of the thesis is Chapter~\ref{chap:flag}, which focusses on flag and Schubert varieties in type $A$. 
In \S\ref{sec:BF} we show how string polytopes arise from the \cite{GHKK14}-superpotential for flag and Schubert varieties in type $A$.
Then in \S\ref{sec:BLMM} we compute the tropical flag varieties $\trop(\Flag_4)$ and $\trop(\Flag_5)$ together with the Newton-Okounkov bodies obtained from the Kaveh-Manon construction. The resulting toric degenerations are compared with those from string polytopes and the FFLV polytope. 


\subsection{Valuations and their weighting matrices}

In \S\ref{sec:val and quasival} we study (quasi-)valuations with weighting matrices (see Definition~\ref{def: quasi val from wt matrix}) as introduced in \cite{KM16}.
Our leading example is $A=\mathbb C[X]$ the homogeneous coordinate ring of a projective variety.
An embedding of $X\hookrightarrow \mathbb P^{n-1}$ yields a presentation $A=\mathbb C[x_1,\dots,x_n]/I$.
Given a full-rank valuation $\val:A\setminus\{0\}\to \mathbb Z^d$ we define the weighting matrix $M_\val\in \mathbb Z^{n\times d}$ of $\val$ (see Definiton~\ref{def: wt matrix from valuation}).
Given some additional technical assumptions (that are fulfilled when dealing with the homogeneous coordinate ring of a projective variety) we obtain the following key-theorem:

\begin{theorem}\label{thm:blabla}
Under the assumptions described above, if $\init_{M_\val}(I)$ (see Definiton~\ref{def: init wrt M}) is prime, then the value semigroup $S(A,\val)$ is generated by the images $\val(\bar x_i)$ for $\bar x_i\in A$. Moreover, the Newton-Okounkov body is the convex hull of these images and the $\bar x_i$ form a Khovanskii basis.
\end{theorem}

A precise formulation is Theorem~\ref{thm: val and quasi val with wt matrix} in \S\ref{sec:val and quasival}. A direct consequence is that in case the Theorem applies, the toric degeneration of $X$ obtained from the valuation $\val$ can be associated with a prime cone in $\trop(X)$.
It is in fact a very powerful tool: many of the following results are applications or direct consequences of Theorem~\ref{thm:blabla}.
The key idea is to use higher-dimensional Gröbner theory and methods of Kaveh-Manon in \cite{KM16} for arbitrary full-rank valuations.
In particular, this links any toric degeneration induced by a valuation (independent from how the valuation is obtained, e.g. using birational sequences or cluster algebras) to those from tropical geometry.


\subsection{Toric degenerations of Grassmannians}

Consider the Grassmannian $\Gr(k,\mathbb C^n)$ embedded in the projective space $\mathbb{P}(\bigwedge^{k} \mathbb{C}^n)$ via the Pl{\"u}cker embedding. 
We define in Definition~\ref{def:itseq} a new class of birational sequences for Grassmannians called \emph{iterated sequences}.
More specifically, for $\Gr(2,\mathbb C^n)$, in Algorithm~\ref{alg:tree assoc to seq} we reveal their close connection to labelled trivalent trees parametrizing maximal prime cones of $\trop(\Gr(2,\mathbb C^n))$.
Let $S$ be an iterated sequence for $\Gr(2,\mathbb C^n)$ and $T_S$ the trivalent tree that is the output of the algorithm.
We consider a valuation $\val_S$ (see Definition~\ref{def:valseq}) associated with $S$ and the weighting matrix it defines.
In a key-proposition (Proposition~\ref{prop: init M_S= init C_S}) we show that the initial ideal of the Plücker ideal $I_{2,n}$ with respect to the weighting matrix coincides with the initial ideal with respect to the cone defined by $T_S$.
As the latter is prime (see \cite{SS04}) this enables us to apply Theorem~\ref{thm:blabla} to obtain:

\begin{theorem}\label{thm:intro birat}
\begin{enumerate}[(i)]
    \item For every iterated sequence $S$ for $\Gr(2,\mathbb C^n)$ the value semigroup of the associated valuation is generated by the images of Plücker coordinates. That is, the Plücker coordinates form a Khovanskii basis.
    \item For every iterated sequence $S$ for $\Gr(2,\mathbb C^n)$ there exists a maximal prime cone $C_S$ in $\trop(\Gr(2,\mathbb C^n))$ such that the associated toric degenerations of $\Gr(2,\mathbb C^n)$ are isomorphic.
    \item For every maximal prime cone $C$ in $\trop(\Gr(2,\mathbb C^n))$ there exists an iterated sequence $S_C$, such that the induced toric degenerations of $\Gr(2,\mathbb C^n)$ are isomorphic. 
\end{enumerate}
\end{theorem}

A more precise formulation that implies all of the above results can be found in Theorem~\ref{thm:main birat}.
Note that this gives an answer to Question~\ref{Q:trop from birat} for $\Gr(2,\mathbb C^n)$.

\medskip

Having in mind Question~\ref{Q:cluster} we would like to combine techniques from the tropical Grassmannian with the cluster algebra structure on $\mathbb C[\Gr(k,\mathbb C^n)]$. Similar ideas are discussed in \cite{SW05}, where they show how the two settings are related combinatorially. 

In \S\ref{sec:BFFHL} we apply Theorem~\ref{thm:blabla} to the valuations defined by Rietsch-Williams \cite{RW17} for every seed of the cluster algebra $\mathbb C[\Gr(k,\mathbb C^n)]$ mentioned above.
More precisely, we focus on seeds that are encoded by \emph{plabic graphs} (introduced by Postnikov \cite{Pos06}). Here the $\A$-cluster variables consist of only Plücker coordinates.
Let $\val_\G$ be the valuation associated with a plabic graph $\G$ and $M_\G$ the weighting matrix of $\val_\G$.
By $I_{k,n}$ we denote the Plücker ideal describing the Grassmannians with respect to the Plücker embedding (see \S\ref{sec:pre-grass}).
Summarizing Proposition~\ref{prop: plabic lin form}, Theorem~\ref{thm: RW val and wt vect} and Corollary~\ref{cor: integral NO vs prime} we obtain:

\begin{theorem}\label{thm: wt matrix BFFHL intro}
If the initial ideal $\init_{M_\G}(I_{k,n})$ is prime, then the toric degeneration obtained from the valuation $\val_\G$ can be realized as a degeneration from the tropicalization of $\Gr(k,\mathbb C^n)$. In this case, the associated Newton-Okounkov body is the convex hull of the valuation images of Plücker coordinates and the Plücker coordinates form a Khovanskii basis.

Moreover, if the Newton-Okounkov body of $\val_\G$ is not integral (see \cite[\S8]{RW17}), then the initial ideal $\init_{M_\G}(I_{k,n})$ is not prime.

\end{theorem}

We analyze $\Gr(3,\mathbb C^6)$ computationally and study in more detail $\Gr(2,\mathbb C^n)$ in \S\ref{sec:case gr2n}\footnote{Based on joint work with Xin Fang, Ghislain Fourier, Milena Hering and Martina Lanini in \cite{BFFHL}.}.
In the latter case, we define a weight vector $\bw_\G\in \mathbb R^{\binom{n}{2}}$ for every plabic graph $\G$ and show that it lies in the relative interior of a maximal prime cone of $\trop(\Gr(2,\mathbb C^n))$.
We use the bijection of labelled triangulations of a disk with $n$ marked points with plabic graphs for $\Gr(2,\mathbb C^n)$ (see \cite[Algorithm 12.1]{KW14}).
Let $T$ be the labelled trivalent tree that is the dual graph to the triangulation, which is mapped to $\G$ under the bijection. 
Consider a weight vector ${\bf w}_T$ in the relative interior of the maximal cone $C\subset\trop(\Gr(2,\mathbb C^n))$ with associated tree $T$. 
The following theorem shows how the combinatorial bijections in this case have in fact a deeper meaning and lead us to an answer of Question~\ref{Q:cluster} for $\Gr(2,\mathbb C^n)$.

\begin{theorem}\label{thm: BFFHL intro}
Let $\G$ be a plabic graph for $\Gr(2,\mathbb C^n)$ and $T$ the corresponding labelled trivalent tree. Then the associated initial ideals $\init_{{\bw}_\G}(I_{2,n})$ and  $\init_{{\bf w}_T}(I_{2,n})$ are equal.
\end{theorem}

In combination with Theorem~\ref{thm: wt matrix BFFHL intro} this proves the expectation of Kaveh-Manon (see \cite[page 6]{KM16}) that the construction of Rietsch-Williams (based on the cluster structure) agrees essentially with theirs (based on the tropicalization).
Further, in combination with Theorem~\ref{thm:main birat} we now have a complete picture for toric degenerations of $\Gr(2,\mathbb C^n)$: (up to isomorphism) the constructions using birational sequences, the tropical Grassmannian, and the cluster structure yield the same toric varieties as flat degenerations of $\Gr(2,\mathbb C^n)$.



\subsection{Toric degenerations of flag and Schubert varieties}

In \S\ref{sec:BLMM} and \S\ref{sec:BF} we consider the variety $\Flag_n$ of full flags $\{0\}= V_0\subset V_1\subset\cdots\subset V_{n-1}\subset V_n=\mathbb C^n$ of vector subspaces of $\mathbb C^n$ with ${\rm dim}_{\mathbb C}(V_i) = i$. 
In view of Question~\ref{Q:cluster} we want to understand if Newton-Okounkov bodies from birational sequences in \cite{FFL15} are related to the cluster structure on flag varieties in \S\ref{sec:BF}\footnote{Based on joint work with Ghislain Fourier.}. 
As a starting point we decided to study the special case of string polytopes.
In \cite{GP00} Gleizer and Postnikov use \emph{pseudoline arrangements} (see \S\ref{sec:BF}~Defintion~\ref{def:pseusoline arr}) associated to reduced expressions of $w_0\in S_n$ and \emph{rigorous paths} in these to paramatrize the inequalities for string cones $C_{\w_0}\subset \mathbb R^N$. 
We extend their result by adding weight inequalities encoded combinatorially in the pseaudoline arrangement and obtain the \emph{weighted string cones} $\mathcal C_{\w_0}\subset \mathbb R^{N+n-1}$ as defined in \cite{Lit98}. 
Intersecting $\mathcal C_{\w_0}$ with the preimage of a weight $\lambda\in \mathbb R^{n-1}$ of an appropriate projection $\pi:\mathbb R^{N+n-1}\to \mathbb R^{n-1}$ yields the string polytope $\pi^{-1}(\lambda)\cap\mathcal C_{\w_0}$. 
Generalizing to arbitrary $w\in S_n$ and following Caldero \cite{Cal03} we obtain similarly the string cone, weighted string cone and string polytope for the Schubert variety $X(w)\subset \Flag_n$.

We introduce a second polyhedral cone $S_{\w_0}\subset\mathbb R^N$ associated to a pseudoline arrangement in a dual way: the variables are associated to the faces of the diagram as opposed to the vertices in case of the string cone.
From additional weight inequalities and a second projection $\tau:\mathbb R^{N+n-1}\to\mathbb R^{n-1}$ we get a weighted cone $\mathcal S_{\w_0}$ and polytopes $\tau^{-1}(\lambda)\cap\mathcal S_{\w_0}$ for $\lambda\in \mathbb R^{n-1}$.
As in the case of string cones, we obtain these also for arbitrary $w\in S_n$. For simplicity we denote for now the corresponding projection also by $\pi$ and $\tau$. 
The first combinatorial result of our study is the following (see Theorem~\ref{thm:unimod}).

\begin{theorem}\label{thm:dual cones intro}
For every $\w\in S_n$, the two cones $\mathcal{C}_{\w}$ and $\mathcal{S}_{\w}$ are unimodularly equivalent and the lattice-preserving linear map is given by the duality of faces and vertices in the pseudoline arrangement. 
Moreover, this linear map restricts to linear bijections between the polytopes
$\pi^{-1}(\lambda) \cap \mathcal{C}_{\w} \cong \tau^{-1}(\lambda) \cap \mathcal{S}_{\w}$ and the cones $S_{\w} \cong C_{\w}$.
\end{theorem}

The cone $\mathcal S_{\w_0}$ appears in the framework of mirror symmetry for cluster varieties \cite{GHKK14}. 
Recall that $\Flag_n=SL_n/B$ for the Borel subgroup of upper triangular matrices $B$. Denote by $B^-\subset SL_n$ the Borel subgroup of lower triangular matrices and by $U\subset B$ (resp. $U^-\subset B^-$) the unipotent radical with all diagonal entries being $1$. 
The double Bruhat cell $G^{e,w_0}= B^- \cap Bw_0B$ is an $\mathcal A$-cluster variety (see \cite{BFZ05}) and can be identified with an open subset of $Bw_0B/U$.
By \cite{Lit98} the weighted string cone parametrizes a basis of $\mathbb C[Bw_0B/U]$.
Let $\mathcal{X}$ be the mirror dual of the $\mathcal{A}$-cluster variety $G^{e,w_0}$ and let $s_0=s_{\hat\w_0}$ be the seed of the cluster algebra $\mathbb{C}[G^{e,w_0}]$ corresponding to the reduced expression $\hat\w_0 = s_1 \, s_2s_1 \, \cdots s_{n-1}\cdots s_2s_1$. 
Let $W$ be the superpotential defined by the sum of the $\vartheta$-functions for every frozen variable in $s_0$ as introduced in \cite{GHKK14}. 
Then $W^{\trop}$ denotes the tropicalization of the superpotential. Magee has shown in \cite{Mag15} (see also Goncharov-Shen in \cite{GS15}) that 
\[
\mathcal{S}_{\hat\w_0} = \{ x \in \mathbb{R}^{N+n-1} \mid W^{\trop}\vert_{\mathcal{X}_{s_{0}}}(x)\geq 0 \} =: \Xi_{s_0}.
\]
We show that the mutation of the pseudoline arrangement and hence of the cone $\mathcal{S}_{\w_0}$, is compatible with the mutation of the superpotential \cite{GHK15} by introducing mutation of the rigorous paths defining the cone. 
We obtain the following (see also \cite{GKS}, where Genz-Koshevoy-Schumann obtain a similar result in the context of crystal graphs):
\begin{theorem}
Let $\w_0$ be an arbitrary reduced expression of $w_0 \in S_{n}$ and $s_{\w_0}$ be the seed corresponding to the pseudoline arrangement, $\mathcal{X}_{s_{\w_0}}$ the toric chart of the seed $s_{\w_0}$. Then 
\[
\mathcal{S}_{\w_0} =  \{ x \in \mathbb{R}^{N+n-1} \mid W^{\trop}\vert_{\mathcal{X}_{s_{\w_0}}}(x)\geq 0 \} =: \Xi_{s_{\w_0}},
\]
the polyhedral cone defined by the tropicalization of $W$ expressed in the seed $s_{\w_0}$.
\end{theorem}

Consider $w \in S_{n}$ arbitrary and $\w$ a reduced expression of $w$. Let $W$ be as above and consider its restriction $\res_{\w}(W\vert_{\mathcal X_{s_{\w_0}}})$ to the mirror dual of the $\mathcal{A}$-cluster variety $G^{e,w}$.  
Let $s_{\w}$ be the corresponding seed in the cluster algebra (see Definition~\ref{def:quiver pa}). 
Then the tropicalization of the restriction yields again a cone $\Xi_{s_{\w}}$. The last result of this section establishes an answer to Question~\ref{Q:cluster} for Schubert varieties.

\begin{theorem}
Let $\w \in S_{n}$, and fix $\w_0=\w s_{i_{\ell(w)+1}}\dots s_{i_N}$ a reduced expression of $w_0\in S_n$. 
Let $s_{\w}$ resp. $s_{\w_0}$ be the corresponding seeds. Then $\mathcal{S}_{\w}$ is the cone $\Xi_{s_{\w}}$ defined by the tropicalization of the restricted superpotential $\res_{\w}(W\vert_{\mathcal X_{s_{\w_0}}})$. 
\end{theorem}


\medskip

In view of Question~\ref{Q:trop flag} we study in \S\ref{sec:BLMM}\footnote{Based on joint work with Sara Lamboglia, Kalina Mincheva and Fatemeh Mohammdi in \cite{BLMM}.} the tropicalization of the flag variety. Consider therefore the natural embeddeding of $\Flag_n$ in a product of Grassmannians using the Pl\"ucker coordinates. 
We denote by $I_n$ the defining ideal of $\Flag_n$ with respect to this embedding. 
We produce toric degenerations of $\Flag_n$ as Gr\"obner degenerations coming from the initial ideals associated to the maximal cones of $\trop(\Flag_n)$. 
For the case of maximal cones with non-prime associated initial ideal we suggest a procedure (see \S\ref{sec:BLMM}~Procedure~\ref{alg:Kh_Basis}) of how to recover prime cones from reembedding the variety. We successfully apply it to $\Flag_4$.

The following is our main results on the tropicalization of the flag varieties $\Flag_4$ and $\Flag_5$. More detailed formulations can be found in \S\ref{sec:BLMM}
Theorem~\ref{flag4}, Theorem~\ref{flag5}, and Proposition~\ref{prop:output}.

\begin{theorem}\label{thm: intro trop Flag}
The tropical variety $\trop(\Flag_4)\subset \mathbb R^{14}/\mathbb R^3$ is a 6-dimensional fan with $78$ maximal cones. 
From prime cones we obtain four non-isomorphic toric degenerations.
After applying Procedure~\ref{alg:Kh_Basis} we obtain at least two additional non-isomorphic toric degenerations from non-prime cones.

The tropical variety $\trop(\Flag_5)\subset \mathbb R^{30}/\mathbb {R}^4$ is a 10-dimensional fan with $69780$ maximal cones. From prime cones we obtain 180 non-isomorphic toric degenerations.
\end{theorem}

In view of Question~\ref{Q:trop flag} and following \cite{KM16}, we further compute the Newton-Okounkov polytopes associated to maximal prime cones. These are the polytopes associated to the normalizations of the toric varieties we obtain. 
We compare these with certain Newton-Okounkov polytopes arising in the setting of \cite{FFL15}, more precisely to string polytopes and the FFLV polytope.

\begin{theorem}\label{ts-comparison} 
For $\Flag_4$ there is at least one new toric degeneration arising from prime cones of $\trop(\Flag_4)$ in comparison to those obtained from string polytopes and the FFLV polytope.

For $\Flag_5$ there are at least 168 new toric degenerations arising from prime cones of $\trop(\Flag_5)$ in comparison to those obtained from string polytopes and the FFLV polytope.
\end{theorem}

Applying Theorem~\ref{thm:blabla} to valuations for string polytopes (this is a particular case of valuations using birational sequences) we further obtain a surprising connection to the \emph{Minkowski property} of string polytopes (see Defintion~\ref{def:mp}) in Theorem~\ref{thm: quasival for string} and Corollary~\ref{cor: prime implies MP}.

\newpage

\newpage

\textbf{Acknowledgements.} Firstly, I would like to thank my family, Olaf, Britta and Berit, and my friends. 
They stood with me through stressful periods and endured my mood swings and social incompetence that were the side effects of preparing a PhD-thesis.

I would like to thank my collaborators on the projects that are part of this thesis. Without them, obviously, most of the results would not be in this thesis today. They are 

    $\bullet$ Xin Fang, Ghislain Fourier, Milena Hering and Martina Lanini for \S\ref{sec:case gr2n},
    
    $\bullet$ Ghislain Fourier for \S\ref{sec:BF}, and
    
    $\bullet$ Sara Lamboglia, Kalina Mincheva and Fatemeh Mohammadi for \S\ref{sec:BLMM}.

\noindent
I am deeply grateful to Bernd Sturmfels and Ghislain Fourier for their support and advice throughout my PhD.  
Further, I would like to thank Bernd for introducing me to Sara, Kalina and Fatemeh and suggesting the problems solved in \S\ref{sec:BLMM}. 

I was incredibly lucky to be welcomed in the (extended) working group of Peter Littelmann in Cologne, my math-family consisting of (among others): Xin, Bea Schumann, Jacinta Torres, Ghislain, Michael Ehrig, Valentin Rappel, Christian Steinert, Oksana Yakimova and, of course, Peter.

As this thesis covers a range of topics (including some quite far from my mathematical background) I am grateful to my math-friends who introduced me to these topics and broadend my mathematical horizon.
These include Fatemeh and Sara, who explained to me tropical geometry, Gr\"obner theory, and Macaulay2, which are now among my favourite topics.
Further, Alfredo N\'ajera Ch\'avez and Timothy Magee: 
without their patience in explaining the theory of cluster algebras to me over and over again from different points of view I would have probably given up on the topic, which I enjoy working on a lot by now.

I would like to thank the organizers of the following conferences, which in hindsight were key events in the developement of this thesis:

$\bullet$ Mini-Workshop ``PBW-structures in Representation Theory", Oberwolfach, March 2016;

$\bullet$ Workshop ``Hot topics: Cluster algebras and wall-crossing", Berkeley, March 2016;

$\bullet$ Summer School ``The Geometry of Valuations", Frankfurt, July 2016;

\noindent
and the hospitality of the following institutions:

$\bullet$ MPI MiS Leipzig, Germany;

$\bullet$ Università degli Studi di Roma ``Tor Vergata", Italy\footnote{supported by QM${}^2$ through the Institutional Strategy of the University of Cologne (ZUK 81/1)};

$\bullet$ UNAM Oaxaca, Mexico.

\noindent
Further, I would like to thank the following people for their support, helpful discussions or insightful explanations:  Kiumars Kaveh, Christopher Manon, Man Wai Cheung, Diane Maclagan, Arkady Berenstein, Alex Küronya, Silvia Sabatini, Markus Reineke, Lauren Williams, Gleb Koshevoy.

Lastly and most importantly, I want to thank deeply my advisor Peter: without you none of this would have been possible. I am extremely grateful for all the opportunities and support that you offered me; for being there when I needed advice and also giving me the freedom to travel to conferences and choose my own projects.

\chapter{General Theory}\label{chap:prep}

\section{Representation Theory of \texorpdfstring{$SL_n(\mathbb C)$}{}}\label{sec:pre rep theory}
In this section we recall basic notions of the representation theory of $SL_n(\mathbb C)$ (or $SL_n$ for short) that we need throughout this thesis.

We fix as Borel subgroup the upper triangular matrices $B\subset SL_n$ and diagonal matrices as maximal torus $T\subset B$. 
We denote the Borel subgroup of lower triangulat matrices $B^-$ (it is also called the \emph{opposite} Borel subgroup of $B$). 
Inside of $B$ (resp. $B^-$) we have the subgroup of unipotent matrices $U$ (resp. $U^-$) with all diagonal entries being $1$. They are the \emph{unipotent radical} of $B$ (resp. $B^-$).

Consider the Lie algebra $\Lie(SL_n)=\lie{sl}_n=\{n\times n$-matrices with trace zero$\}$.
The Lie bracket $[\cdot,\cdot]:\mathfrak{sl}_n\times \mathfrak{sl}_n\to \mathfrak{sl}_n$ is given by the commutator
\[
[A,B]:=AB-BA.
\]
We fix the Cartan decomposition $\lie{sl}_n=\lie n^-\oplus \lie h\oplus \lie n^+$ with $\lie h$ diagonal matrices as Cartan (maximal abelian Lie-subalgebra) and $\lie n^+$ (resp. $\lie n^-$) upper (resp. lower) triangular matrices in $\lie{sl}_n$. 
Note that with these choices we have $\Lie(B)=\lie h\oplus \lie n^+=\lie b$, $\Lie(T)=\lie h$ and $\lie n^-=\Lie(U^-)$.
Let us denote the \emph{root system} of $SL_n$ by $R\subset \mathbb R^{n}$. It is the root system of type $\mathtt A_{n-1}$.
Denoting the standard basis of $\mathbb R^n$ by $\{\epsilon_i\}_{i=1,\dots,n}$ we fix the the simple roots of $R$ to be $\alpha_i=\epsilon_i-\epsilon_{i+1}$ for $i=1,\dots,n-1$. They generate the \emph{root lattice}. 
The positive roots are denoted $R^+=\{\beta\in R\mid \beta>0\}$. They are of form $\alpha_{i,j}:=\alpha_i+\dots+\alpha_j$ for $i\le j<n$. 
With our choice of simple roots we have $\alpha_{i,j}=\epsilon_i-\epsilon_{j+1}$. The number of positive roots is denoted by $N=\frac{n(n-1)}{2}$. 

For a positive root $\beta=\alpha_{i,j}$ let $f_{\beta}\in \lie n^-$ be the root vector of weight $-\beta$. 
In other words, $f_\beta$ is the lower triangular $n\times n$-matrix with all entries being zero besides the $(i+1,j)$'th entry, which is $1$. 
Similarly we have $e_\beta\in\mathfrak n^+$ a root vector for $\beta$ of weight $\beta$.
With our choice of $\lie b$ it is the transpose of $f_\beta$.
We define a third element in $\mathfrak{sl}_n$ associated to $\beta\in R^+$, namely $h_\beta=[e_\beta,f_\beta]\in\lie h$.

For the weight lattice we choose the notation $\Lambda$ with generators the fundamental weights being $\omega_1,\dots,\omega_{n-1}$. 

Let $\Lambda^+$ denote the dominant integral weights in $\Lambda$, i.e. those $\lambda=\sum_{i=1}^{n-1}a_i\omega_i$ with $a_i\in\mathbb Z_{\ge 0}$. 
Dominant integral weights are the lattice points in the \emph{dominant Weyl chamber}, the positive span of the fundamental weights.
By $\Lambda^{++}$ we denote the set of \emph{regular dominant weights}, i.e. those $\lambda=\sum_{i=1}^{n-1}a_i\omega_i$ with $a_i\in\mathbb Z_{>0}$.

The roots and weights live in the same space $\mathbb R^n$, to which we have associated the basis $\{\epsilon_i\}_i$. We can express $\omega_i$ as follows in this basis
\[
n\omega_i=\sum_{k=1}^{i}(n-i+1)\epsilon_k-2i\epsilon_{i+1}-\sum_{k=i+2}^ni\epsilon_k.
\]
Then $\frac{1}{2}(\alpha_1+\dots+\alpha_{n-1})=(\omega_1+\dots+\omega_{n-1})=:\rho\in\lambda^{++}$ is the smallest regular dominant weight.

For every $\lambda\in\Lambda^+$ there is a (finite-dimensional) irreducible representation of $\lie{sl}_n$ of highest weight $\lambda$, denote it by $V(\lambda)$. It is cyclically generated by a highest weight vector $v_\lambda\in V(\lambda)$ (unique up to scaling) over $U(\lie n^-)$, the \emph{universal enveloping algebra} of $\lie n^-$ defined as follows.

For $\lie g$ any Lie algebra, $U(\lie g)$ is a quotient of the tensor algebra $T(\lie g)=\bigoplus_{k\ge 0}\lie g^{\otimes k}$.
The ideal by which we quotient by is generated by relations induced by the Lie bracket, i.e. relations of form $w\otimes v-v\otimes w-[w,v]$ for $v,w\in\lie g$.
The \emph{PBW-basis-Theorem} states the following: let $\{v_1,\dots,v_d\}$ be a ordered basis of $\lie g$, then as a vector space $U(\lie g)$ is generated by monomials of the form
\[
v_1^{a_1}v_2^{a_2}\cdots v_{d-1}^{a_{d-1}}v_d^{a_d}, \ a_i\in \mathbb Z_{\ge 0}.
\]

As the irreducible highest weight representation $V(\lambda)$ for $\lambda\in P^+$ 
are cyclically generated by $v_\lambda$ over $U(\lie n^-)$, we are particularly interested in a PBW-basis for $U(\lie n^-)$. This is given, for example, by fixing an order on all positive roots, e.g. $\beta_1,\dots,\beta_N$. Then for a chosen highest weight vector $v_\lambda\in V(\lambda)$ we have
\begin{eqnarray}\label{eq: irr mod cyc gen}
V(\lambda)=U(\lie n^-)\cdot v_\lambda=\langle f_{\beta_1}^{m_1}\cdots f_{\beta_N}^{m_N}\cdot v_\lambda\mid m_i\in\mathbb Z_{\ge 0} \rangle_{\mathbb C}.
\end{eqnarray}

\begin{example}\label{exp: fund reps 1 and 2}
We have $V(\omega_1)=\mathbb C^n$ and $V(\omega_k)=\bigwedge^k\mathbb C^n$. The root operators $f_{\alpha_{i,j}}=f_{i,j}\in\lie n^-$ act on $\mathbb C^n$ with standard basis $\{e_i\}_{i=1,\dots,n}$ by $f_{i,j} \cdot e_l=\delta_{i,l}e_{j+1}$. The highest weight vector $v_{\omega_1}$ can be chosen as $e_1$. For $V(\omega_2)$ fix the basis $\{e_k\wedge e_l\mid 1\le k<l\le n\}$. Then the action of $\lie n^-$ is given by
\[
f_{i,j}\cdot (e_k\wedge e_l)=f_{i,j}\cdot e_k\wedge e_l + e_k\wedge f_{i,j}\cdot e_l=\left\{\begin{matrix} 
e_{j+1}\wedge e_l, & \text{ if } k=i,\\
e_k\wedge e_{j+1}, & \text{ if } l=i,\\
0, & \text{ otherwise.} 
\end{matrix}\right.
\]
We can chose $e_1\wedge e_2$ as the highest weight vector $v_{\omega_2}$.
\end{example}

The Weyl group of $SL_n$ is the symmetric group $S_n$ generated by the simple transpositions $s_i=(i,i+1)$ for $1\le i<n$. 
By $w_0$ we denote the longest element in $S_n$.
For every $w \in S_{n}$, we denote by $\ell(w)$ the minimal length of $w$ as a word in the generators $s_i$. Further, $\underline{w}$ denotes a reduced expression  
\[
\underline{w} = s_{i_1} \cdots s_{i_{\ell(w)}}.
\]
Such an expression is not unique. For any two reduced expressions of $w$ there is a sequence of local transformations leading from one to the other. These local transformations are either swapping orthogonal reflections $s_i s_j = s_j s_i$ if $|i - j | > 1$ or exchanging consecutive $s_is_{i+1}s_i = s_{i+1}s_is_{i+1}$.

The symmetric group acts on the weight lattice as follows. 
Consider $\lambda\in\Lambda$ and $s_i\in S_n$. 
Then $s_i(\lambda)\in\Lambda$ is obtained from $\lambda$ by reflection on the hyperplane $H_{\alpha_i}$ perpendicular to the simple root $\alpha_i=\epsilon_i-\epsilon_{i+1}$.

Fix $w\in S_n$ and $\lambda\in\Lambda^+$. 
Then the weight space of weight $w(\lambda)$ in $V(\lambda)$, denoted $V(\lambda)_{w(\lambda)}$, is called \emph{extremal} and it is one-dimensional.

\begin{definition}\label{def: Demazure module}
For $w\in S_n$ and $\lambda\in \Lambda^+$ we fix a generator $v_{w\lambda}\in V(\lambda)_{w\lambda}$ we consider
$U(\mathfrak b)\cdot v_{w\lambda}=:V_w(\lambda)$.
This is a $\mathfrak b$-module called the \emph{Demazure module}. 
\end{definition}

Note that though $V_w(\lambda)$ is a $\mathfrak b$-submodule of $V(\lambda)$, it is not an $\mathfrak{sl}_n$-module. 
Let $\w = s_{i_1} \cdots s_{i_{\ell(w)}}$ be a reduced expression of $w$. Then for any $\lambda \in \Lambda^+$ similarly to the PBW-basis theorem we have that
\begin{eqnarray}\label{eq: PBW for demazure}
\{ f_{\alpha_{i_1}}^{m_{i_1}} \cdots f_{\alpha_{i_{\ell(w)}}}^{m_{i_{\ell(w)}}}\cdot v_\lambda
\in V(\lambda) \mid m_{i_j} \geq 0 \}
\end{eqnarray}
forms a spanning set of $V_w(\lambda)$ as a vector space. In particular, if $w=w_0$ then   $V_{w_0}(\lambda)=V(\lambda)$.
For a Demazure module $V_w(\lambda)$ we denote by $V_w(\lambda)^{\perp}$ its orthogonal complement in $V(\lambda)^*$.

\section{Tropical Geometry}\label{sec:pre trop}

In this section we recall basic notions of tropical geometry that we assume throughout the rest of the thesis.
Tropical geometry comes in many flavours, our approach follows closely the book \cite{M-S} by Maclagan-Sturmfels and we invite the reader to have a look there for a more detailed introduction.
This approach to tropical geometry is closely related to Gröbner theory.

\begin{definition}\label{def:initial form and ideal}
Let $f=\sum a_{\bf u}x^{\bf u}$ with ${\bf u}\in \mathbb Z^{n},a_{\bf u}\in \mathbb C$ be a polynomial in $\mathbb C[x_1^{\pm 1},\dots, x_n^{\pm 1}]$, where $x^{\bf u}$ denotes the monomial $x_1^{u_1}\dots x_{n}^{u_n}$. The \emph{initial form} of $f$ with respect to a fixed weight vector ${\bf w}\in \mathbb R^n$ is given by
\begin{eqnarray}\label{eq: def initial form}
\init_{\bf w}(f):=\sum_{\begin{smallmatrix}{\bf w}^T\cdot{\bf u} \text{ is minimal},\\
a_{\bf u}\not=0\end{smallmatrix}}a_{\bf u}x^{\bf u}.
\end{eqnarray}
This definition can be extended to ideals. For an ideal $I\subset \mathbb C[x_1^{\pm 1},\dots,x_n^{\pm 1}]$ we have \emph{initial ideal} with respect to ${\bf w}\in \mathbb R^n$ 
\begin{eqnarray}\label{eq: def initial ideal}
\init_{\bf w}(I):=\langle \init_{\bf w}(f)\mid f\in I\rangle. 
\end{eqnarray}
\end{definition}

\begin{example}\label{exp: initial form and ideal}
Consider the ideal $I=\langle x_1^2+x_2,x_1-x_2 \rangle\subset \mathbb C[x_1,x_2]$ and ${\bf w}=(1,0)$. Then $\init_{\bf w}(x_1^2+x_2)=x_2$ and $\init_{\bf w}(x_1-x_2)=-x_2$. In particular, $\langle \init_{\bf w}(x_1^2+x_2),\init_{\bf w}(x_1-x_2) \rangle=\langle x_2\rangle\subset \mathbb C[x_1,x_2]$. But we also have $x_1^2+x_1=(x_1^2+x_2)+(x_1-x_2)\in I$, so by definition $\init_{\bf w}(x_1^2+x_1)=x_1\in \init_{\bf w}(I)$. We deduce
\begin{eqnarray}\label{eq: gen ideal not gen initial ideal}
I=\langle f_1,\dots,f_s\rangle \not\Rightarrow \init_{\bf w}(I)=\langle \init_{\bf w}(f_1),\dots, \init_{\bf w}(f_s)\rangle.
\end{eqnarray}
\end{example}

By \cite[Theorem 15.17]{Eisenbud},
there exists a flat family over $\mathbb{A}^1$ whose fiber over $t\neq 0$ is isomorphic
to $V(I)$ and whose fiber over $t=0$ is isomorphic to $V(\init_\mathbf{w}(I))$. For $t$ the coordinate in $\mathbb A^1$ it is given by the following family of ideals
\begin{eqnarray}\label{eq: groebner family}
\tilde{I}_t:=\left\langle t^{-\min_{{\textbf u}}\{{\textbf w}\cdot {\textbf u}\}} f(t^{w_1}x_1,\ldots,t^{w_n}x_n)\left\vert f=\sum a_{{\textbf u}}x^{{\textbf u}}\ \in I   \right.\right\rangle\subset \mathbb C[t,x_1^{\pm 1},\dots,x_n^{\pm 1}].
\end{eqnarray}

More precisely, for a projective variety $X=V(I)\subset \mathbb P^{n-1}$, where $I\subset \mathbb C[x_1,\dots,x_n]$ is a homogeneous ideal, there is a flat degeneration over $\mathbb A^1$ with generic fibre (i.e. fibre over $t\not =0$) isomorphic to $V(I)$ and special fibre (i.e. fibre over $t=0$) $V(\init_{\bf w}(I))$. 
Let $I_s$ denote the ideal $\tilde I_t\vert_{t=s}$. For $s\not=0$ the isomorphism $V(I_s)\cong V(I_1)=V(I)$ is given by a ring automorphism of $\mathbb C[x_1,\dots,x_n]$ sending $I_s$ to $I$.
If $\init_{\bf w}(I)$ is \emph{toric}, i.e. a binomial prime ideal, then $V(\init_{\bf w}(I))$ is a toric variety (see e.g. \cite[Lemma 2.4.14]{M-S}).

In order to look for these toric degenerations we study the tropicalization of $V(I)$. 

\begin{definition}\label{def: trop fct}
Let $f=\sum a_{\textbf u}x^{\textbf u}\in \mathbb C[x_1^{\pm1},\dots,x_n^{\pm1}]$. The \emph{tropicalization} of $f$ is the function $f^{\trop}:\mathbb R^{n}\to \mathbb R$ given by 
\[
f^{\trop}(\textbf w):=\min\{\textbf w\cdot \textbf u \mid \textbf u\in \mathbb Z^{n} \text{ and } a_{\textbf u}\neq 0\}.
\]
\end{definition}

If $\textbf w-\textbf {v}=m\cdot\textbf 1$, for some $\textbf{v},\textbf{w}\in\mathbb R^{n}$,  ${\bf 1}=(1,\ldots,1)\in\mathbb R^{n}$ and $m\in \mathbb R$, we have that the minimum in $f^{\trop}({\bf w})$ and $f^{\trop}({\bf v})$
is achieved for the same ${\bf u}\in \mathbb Z^{n}$ with $ a_{\textbf u}\neq 0$.

\begin{definition}(\cite[Definition~3.1.1 and Definition~3.2.1]{M-S})\label{def:trop hypersurf}
Let $f=\sum a_{\textbf u}x^{\textbf u}\in \mathbb C[x_1^{\pm1},\dots,x_n^{\pm1}]$ and $V(f)$ the associated hypersurface in the algebraic torus $T^n=(\mathbb C^*)^n$. Then the \emph {tropical hypersurface} of $f$ is
\[
\trop(V(f)):=\left\{ \textbf w\in \mathbb R^{n} \left| 
\begin{matrix}
\text{the minimum in }f^{\trop}(\textbf w)\\
\text{is achieved at least twice}
\end{matrix}\right.\right\} .
\]
Let $I$ be an ideal in $\mathbb C[x_1^{\pm1},\dots,x_n^{\pm1}]$. The \textit{tropicalization} of the variety $V(I)\subset T^n$ is defined as
\[
\trop(V(I)):=\bigcap_{f\in I}\trop(V(f))\subset \mathbb R^n.
\]
\end{definition}
For a projective variety $V(I)\subset \mathbb P^{n-1}$ with $I$ a homogeneous ideal in $\mathbb C[x_1,\dots,x_n]$ we consider the ideal $\hat I:=I\mathbb C[x_1^{\pm1},\dots,x_n^{\pm1}]$.
Then $V(\hat I)=V(I)\cap T^n$. 
We consider the tropicalization of projective varieties  defined as $\trop(V(I)):=\trop(V(\hat I))$.

By the \emph{Fundamental Theorem of Tropical Geometry} \cite[Theorem~3.2.3]{M-S} we have
\[
\trop(V(I))=
\left\{
    \mathbf w\in\mathbb R^n 
        \mid     \init_{\mathbf w}(I) \text{  is monomial-free} 
\right\}.
\]
Further, the \emph{Structure Theorem}\cite[Theorem~3.3.5]{M-S} tells us that if $X\subset T^n$ is an irreducible $d$-dimensional variety, then $\trop(X)$ is the support of a pure rational $d$-dimensional polyhedral complex, connected in codimension $1$. 
We do not recall notions from polyhedral geometry but refer the interested reader to \cite[\S2.3]{M-S}. To us, most importantly, the structure theorem implies that we can associate a fan-structure with $\trop(V(I))$. 
We choose the fan structure in such a way that $\trop(V(I))$ is a subfan of the Gröbner fan of $I$.
If $\bw,\bv$ lie in the relative interior of a cone $C$ (also denoted $C^\circ$) in the Gröbner fan, then $\init_{\bw}(I)=\init_{\bv}(I)$.
Adopting this fan structure for $\trop(V(I))$ we therefore use the notation $\init_C(I):=\init_{\bw}(I)$ for some $\bw\in C^\circ$.

For an ideal $I\subset \mathbb C[x_1^{\pm1},\dots,x_n^{\pm1}]$ there may exist some $\mathbf w\in\mathbb R^n$ with $\init_{\mathbf w}(I)=I$.
For example, if $I$ is homogeneous this is always the case for $\mathbf 1:=(1,\dots,1)\in\mathbb R^n$.
The linear subspace $L_I:=\{\mathbf w\in \mathbb R^n\mid \init_{\bw}(I)=I\}\subset\trop(V(I))$ is called the \emph{lineality space} of $I$.

In \S\ref{sec:BLMM} we tropicalize the \emph{flag variety} (see \S\ref{sec:pre flag}). Although the flag variety is a projective variety and hence, by the above we have a recipe to tropicalize it, for computational convenience we choose an embedding into a product of projective spaces (instead of just one projective space).
The procedure of tropicalization can also be done in this setting, replacing $\mathbb C[x_1^{\pm1},\dots,x_n^{\pm1}]$ by S, the \textit{total coordinate ring} (see \cite[page 207]{CLS11} for a definition) of
$\mathbb P^{k_1}\times\cdots\times \mathbb P^{k_s}$.
Then $S$ has a $\mathbb{Z}^s$-grading given by $\deg:\mathbb Z^{n}\to \mathbb Z^s$, where $k_1+\dots+k_s=n-1$.
An ideal $I\subset S$ defining an irreducible subvariety $V(I)$ of $\mathbb P^{k_1}\times\cdots\times \mathbb P^{k_s}$ is homogeneous with respect to this grading. 
The tropicalization of $V(I)$ is contained in $\mathbb R^{k_1+\ldots+k_s+s}/H$, where $H$ is an $s$-dimensional linear space spanned by the rows of a matrix $D$ defining $\deg$. 
Similarly to the projective case, if $V(I)$ is a $d$-dimensional irreducible subvariety of $\mathbb P^{k_1}\times\cdots\times \mathbb P^{k_s}$, then $\trop(V(I))$  is the support of a fan,  which is the quotient by $H$ of a rational $(d+s)$-dimensional subfan $F$ of the Gr\"obner fan of $I$. 
Here the Krull dimension of $S/I$ is $d+s$.

In the following we always consider $\trop(V(I))$ with the fan structure defined above. 

\begin{remark}
A detailed definition of the tropicalization of a  general  toric variety $X_{\Sigma}$ and of its subvarieties can be found in \cite[Chapter 6]{M-S}. Note that we only consider the  tropicalization of the intersection of $V(I)$ with the torus of $X_{\Sigma}$ while in \cite[Chapter 6]{M-S} they introduce a generalized version  of $\trop(V(I))$ which includes the tropicalization of  the intersection of $V(I)$ with each orbit of  $X_{\Sigma}$.
\end{remark}

Another property of $\trop(V(I))$ is that any fan structure on it can be \emph{balanced} assigning a positive integer weight to every maximal cell. We do not explain the notion of balancing in detail and we consider an adapted version of the multiplicity defined in \cite[Definition 3.4.3]{M-S}.

\begin{definition}\label{def:mult cone}
Let $I\subset S$ be a homogeneous ideal and $\Sigma$ a fan with support $|\Sigma|=|\trop(V(I))|$, such that for every cone $C\subset\Sigma$ the ideal  $\init_{\bw}(I)$ is constant for $\textbf{w}\in C^\circ$. For a maximal dimensional cone $C\subset \Sigma$ we define the \emph{multiplicity} as 
\[
\mult(C):=\sum_P \mult(P,\init_{C}(I)).
\]
Here the sum is taken over the minimal associated primes $P$ of $\init_{C}(I)$ that do not contain monomials (see \cite[\S 3]{Eisenbud} or \cite[\S 4.7]{Cox}).
\end{definition}

As we have seen, each cone of $\trop(V(I))$ corresponds to an initial ideal which contains no monomials. We now explain why \emph{good candidates} for toric degenerations are the initial ideals corresponding to the relative interior of  maximal cones in $\trop(V(I))$. We say a maximal cone is \emph{prime} if the corresponding initial ideal $\init_C(I)$ is a prime ideal.

\begin{lemma}\label{lem:multiplicity}
Let $I\subset S$ be a homogeneous ideal and  $C$ a maximal cone of $\trop(V(I))$.
If $\init_{C}(I)$ is toric then $C$ has multiplicity one.
Moreover, if $C$ has multiplicity one then $\init_{C}(I)$ has a unique toric ideal in its primary decomposition.
\end{lemma}

\begin{proof}
We first prove the lemma for $s=1$, i.e. $S$ the homogeneous coordinate ring of $\mathbb P^{n-1}$. Let $I'=\init_{C}(I) \mathbb C[x_1^{\pm 1},\ldots,x_n^{\pm 1}]$ and consider $V(I')\subset T^n$. 
Then by \cite[Remark 3.4.4]{M-S} the multiplicity of a maximal cone $C$ is counting the number of $d$-dimensional torus orbits whose union is $V(I')$. 
If $\init_{C}(I)$ is toric, then $V(I')$ is an irreducible toric variety, hence it has a unique $d$-dimensional torus orbit. So $C$ has multiplicity one.

Suppose now $C$ has multiplicity one. 
Then $\init_C(I)$ contains one associated prime $J$, not containing any monomials. 
The ideal $J$ is further binomial since it is the ideal of the unique $d$-dimensional torus orbit contained in $V(I')$.

When $s>1$ and so $S$ is the total coordinate ring of $\mathbb P^{k_1}\times\cdots\times \mathbb P^{k_s}$, the torus is given by $T^{k_1}\times \cdots \times T^{k_s}\cong T^{k_1+\cdots +k_s}$. 
We may assume that for each $i$, 
\[
T^{k_i}=\{[1:a_2:\ldots:a_{k_i}]\in \mathbb P^{k_i}\mid  a_j\neq 0\text{\ for\ all\ } j\}.
\]
The variables for $\mathbb P^{k_i}$ are denoted by  $x_{i,0},\dots,x_{i,k_i}$ for each $i$. We fix  the  Laurent polynomial ring 
\[
S'=\mathbb C[{x^{\pm 1}_{1,0}},\dots,x^{\pm 1}_{1,k_1},x^{\pm 1}_{2,0},\dots,x^{\pm 1}_{2,k_2},\dots,x^{\pm 1}_{s,0},\dots,x^{\pm 1}_{s,k_s}].
\]
Then consider the ideal $I'=\init_{C}(I)S'\subset S'$ and the variety $V(I')\subset T^{k_1+\ldots +k_s}$ and the proof proceeds as before.
\end{proof}

\begin{remark}\label{rem:from lemma1}
From Lemma~\ref{lem:multiplicity} we conclude the multiplicity being one is a necessary but not sufficient condition for toric initial ideals. A cone can have multiplicity one but its associated initial ideal might be neither prime nor binomial. 
There may be associated primes that contain monomials in the decomposition of $\init_{\textbf w}(I)$ and these do not contribute to the multiplicity. 
We list examples of such cones in $\trop(\Flag_5)$ (for more details see Theorem~\ref{flag5}).
\end{remark}

Let $I$ be a homogeneous ideal in $S$ such that the Krull dimension of $S/I$ is $d$. Consider $\trop(V(I))\subset\mathbb R^{n}/H$ and the $d$-dimensional subfan  $F\subset \mathbb R^{n}$ of the Gr\"obner fan of $I$ with $F/H\cong \trop(V(I))$. 
When $V(I)\subset \mathbb P^{k_1-1}\times\cdots\times \mathbb P^{k_s-1}$ the linear space $H$ is spanned by the rows of the matrix $D$. 
In particular, when $V(I)\subset \mathbb P^{n-1}$ we have that $H$ is equal to the span of $(1,\ldots,1)$.
We now describe some properties of the toric initial ideals corresponding to maximal cones of $\trop(V(I))$.
Let $C$ be a cone in $\trop(V(I))$ and $\{{\bf w}_1,\ldots,{\bf w}_d\}$ be $d$ linearly independent vectors in $F$ generating the maximal cone $C'$, such that $C'/H\cong C$.
We can assume that the ${\bf w}_i$'s have integer entries since $F$ is a rational fan. 
We define a matrix associated to $C$ by
\begin{equation}\label{def:W}
W_C:=[{\bf w}_1,\ldots,{\bf w}_d]^t.
\end{equation}

Consider a sublattice $L$ of $\mathbb Z^{n}$ and the standard basis $e_1,\dots, e_{n}$ of $\mathbb Z^{n}$. Given $\ell=(\ell_1,\dots,\ell_{n+1})\in L$ we set $\ell^+=\sum_{\ell_i>0}\ell_i e_i$ and $\ell^-=-\sum_{\ell_j<0}\ell_j e_j$. 
Note that $\ell=\ell^+-\ell^-$ and $\ell^+,\ell^-\in \mathbb N^{n+1}$. 
We use the same notation as in \cite[page 15]{CLS11}.

\begin{lemma}\label{toric:lemma}
Let $I$ be a homogeneous ideal in $S$ and $C$ a maximal cone in $\trop(V(I))$. If $\init_{C}(I)$ is toric, then there exists a sublattice $L$ of $\mathbb Z^{n}$ and constants $c_\ell\in\mathbb C^*$ with $\ell \in L$ such that 
\begin{eqnarray}\label{eq: def I(W_C)}
\init_{C}(I)=I(W_C):=\langle x^{\ell^+}-c_{\ell} x^{\ell^-}\mid \ell\in L \rangle. 
\end{eqnarray}
In particular, $L$  is  the kernel of the map $f:\mathbb Z^{n}\to \mathbb Z^{d}$ defined by the matrix $W_C$. 
If  $C$ has multiplicity one and  $\init_{C}(I)$ is not toric, then  the unique toric ideal in the primary decomposition of  $\init_{C}(I)$ is of the form $ I(W_C)$.
\end{lemma}

\begin{proof}
Let  $\init_{C}(I)\subset  S$ be a toric initial ideal and let $C'$ be the corresponding cone in $F$.
The fan structure is defined on  $\trop(V(I))$ so that for every $\textbf{w}',\textbf{w}$ in the relative interior of $ C'$ we have $\init_{\textbf{w}'}(I)=\init_{C}(I)=\init_{\textbf{w}}(I)$. 
This implies $\init_{C}(I)$ is $W_C$-homogeneous with respect to the $\mathbb {Z}^d$-grading on $S$ given by the matrix $W_C$. By \cite[Lemma~10.12]{Stu96} there exists an automorphism $\phi$ of $S$ sending $x_i$ to $\lambda_i x_i$ for some $\lambda_i\in \mathbb C$, such that the ideal $\init_{C}(I)$ is isomorphic to an ideal
\[
I_{L}:=\langle x^{\ell^+}-x^{\ell^-}\mid \ell\in L \rangle. 
\]
Here $L$ is the sublattice of $\mathbb Z^{n+1}$ given by the  kernel of the map $f:\mathbb Z^{n+1}\to \mathbb Z^{d}$. 
Applying $\phi^{-1}$ to $\init_{C}(I)$  we can write each toric initial ideal as
\[
\langle x^{\ell^+}-c_{\ell}x^{\ell^-}\mid \ell\in L \rangle=I(W_C), 
\]
for some $ c_\ell\in\mathbb C^*$, $L$ and $W_C$ as defined above. 

Let $C$ be a cone of multiplicity one and suppose $\init_{C}(I)$ is not prime. Then by Lemma~\ref{lem:multiplicity} there exists a unique toric ideal $J$ in the primary decomposition of $\init_{C}(I)$. This toric ideal $J$  contains $\init_{C}(I)$ and we show that it can be expressed as $I(W_C)$.  The variety $V(I)$ is considered as a subvariety of $\mathbb P^{n-1}$. As in Lemma~\ref{lem:multiplicity}, the case $V(I)\subset \mathbb P^{k_1}\times\cdots\times \mathbb P^{k_s}$ can be treated similarly.

The tropical variety depends only on the intersection of $V(I)$ with the torus,  and $J$ is equal to $\init_{C}(I)\mathbb C[x_1^{\pm 1},\ldots ,x_n^{\pm 1}]$. Hence, $J$ is a prime ideal that is homogeneous with respect to $W_C$ so we can proceed as above to show $J$ can be written as $\langle x^{\ell^+}-c_{\ell} x^{\ell^-}\mid \ell\in L \rangle=I(W_C)$.  
\end{proof}

\begin{remark}
Note that the lattice $L$ and the ideal $I(W_C)$ only depend on the linear space spanned by the rays of the cone $C'$. Hence they are the same for every set of $d$ independent vectors in $C'$ chosen to define $W_C$.
\end{remark}

\section{Valuations}\label{sec:pre val}

Another construction of toric degenerations can be obtained from valuations as we explain in this section. 
We recall basic notions of the theory of Newton-Okounkov bodies as presented in \cite{KK12}.

We fix a linear order $\prec$ on the additive abelian group $\mathbb Q^r$, where $r\le d$. 

\begin{definition}\label{def: valuation}
A map $\val: A\setminus\{0\}\to (\mathbb Q^r,\prec)$ is a \emph{valuation}, if it satisfies for all $f,g\in A\setminus\{0\}$ and $c\in \mathbb C^*$ 
\begin{enumerate}[(i)]
\item $\val(f+g)\succeq \min \{\val(f),\val(g)\}$, 
\item $\val(fg)=\val(f)+\val(g)$ and  
\item $\val(cf)=\val(f)$.
\end{enumerate}
If we replace (ii) by $\val(fg)\succeq \val(f)+\val(g)$ then $\val$ is called a \emph{quasi-valuation} (also called \emph{loose valuation} in \cite{T03}).
\end{definition}

It is not hard to show, that in (i) if $\val(f)\not=\val(g)$ then $\val(f+g)=\min_{\prec}\{\val(f),\val(g)\}$.

\begin{example}\label{exp: val Piusseux series}
Consider $\mathbb C\{\{t\}\}$ the field of \emph{Piusseux series}. Elements are formal power series $c(t)=c_1t^{a_1}+c_2t^{a_2}+\dots$, with $c_i\in\mathbb C$, $a_i\in\mathbb Q$ sharing a common denominator and increasingly ordered $a_1<a_2<\dots$. It is the algebraic closure of the field of Laurent series $\mathbb C((t))$ (see e.g. \cite[Theorem~2.1.5]{M-S}). Moreover, we have
\[
\mathbb C\{\{t\}\}=\bigcup_{n\in\mathbb Z_{>0}} \mathbb C((t^{\frac{1}{n}})).
\]
It comes with a natural valuation $\text{val}:\mathbb C\{\{t\}\}\setminus\{0\}\to\mathbb Q$ sending an element $0\not=c(t)\in\mathbb C\{\{t\}\}$ to the lowest exponent $a_1$ in the series expansion.
As the field of rational functions $\mathbb C(t)\subset\mathbb C\{\{t\}\}$ is a subfield, we can consider the restriction $\text{val}:\mathbb C(t)\setminus\{0\}\to\mathbb Q$. 
For $0\not=q(t)\in\mathbb C(t)$, the valuation $\text{val}(q(t))$ is the order of the zero (resp. pole) $q(t)$ has at $t=0$.
\end{example}

Let $\val:A\setminus\{0\}\to(\mathbb Z^r,\prec)$ be a (quasi-)valuation, where we replace $\mathbb Q^r$ by $\mathbb Z^r$ for simplicity. 
One naturally defines a $\mathbb Z^r$-filtration on $A$ by $F_{\val \succeq a}:=\{f\in A\setminus\{0\}\vert \val(f)\succeq a\}\cup \{0\}$ (and similarly $F_{\val\succ a}$). The associated graded algebra is defined as
\begin{eqnarray}\label{eq:def ass graded}
\gr_\val(A):=\bigoplus_{a\in \mathbb Z^r}F_{\val\succeq a}/F_{\val \succ a}.
\end{eqnarray}

For $f\in A\setminus\{0\}$ denote by $\overline f$ its image in the quotient $F_{\val\succeq \val(f)}/F_{\val\succ \val(f)}$, hence $\overline f\in\gr_\val(A)$. 
If the filtered components $F_{\val \succeq a}/F_{\val \succ a}$ are at most one-dimensional for all $a\in\mathbb Z^r$, we say $\val$ has \emph{one-dimensional leaves}.

The filtration induced by a valuation allows to define the following property of vector space bases for $A$. 
As stated it can be found in \cite{KM16}, but bases with this property are also studied, for example in \cite{FFL15} where they are called \emph{essential bases}. More details on essential bases can be found in \S\ref{sec:Bos}.

\begin{definition}\label{def:adapted basis}
A vector space basis $\mathbb B\subset A$ is called \emph{adapted} to a valuation $\val:A\setminus\{0\}\to \mathbb Z^r$, if for every $a\in\mathbb Z^r$ $F_{\val\succeq a}\cap\mathbb B$ is a vector space basis for $F_{\val\succeq a}$.
\end{definition}

If a valuation $\val$ has one-dimensional leaves, an adapted basis $\mathbb B$ is particularly useful as by \cite[Remark~2.19]{KM16} there is a bijection between $\mathbb B$ and the set of values $\val(\mathbb B)$ given by $b\mapsto \val(b)$. 
We use this fact in the proof of Theorem~\ref{thm: val and quasi val with wt matrix}.

The image $\{\val(f)\mid f\in A\setminus\{0\}\}\subset \mathbb Z^r$ forms by the definition an additive semi-group, that is a subsemi-group of $\mathbb Z^r$. 
We denote it by $S(A,\val)$ and refer to it as the \emph{value semi-group}. 
The \emph{rank} of the valuation is the rank of the sublattice generated by $S(A,\val)$ in $\mathbb Z^r$.
If $\rank(\val)=d$, we say $\val$ is of \emph{full rank}.
By \cite[Theorem~2.3]{KM16} the one-dimensional leaves property holds for valuations of full rank.
The value semi-group is of great interest because of the following Lemma that can be found, for example, in \cite[Remark~4.13]{BG09}.

\begin{lemma*}(\cite[Remark~4.13]{BG09})\label{lem: semi-group alg}
If $\val$ has one-dimensional leaves, then $\gr_\val(A)$ is isomorphic to the semi-group algebra $\mathbb C[S(A,\val)]$.
\end{lemma*}

The following defintion introduced by Kaveh and Manon in \cite{KM16} is closely related. It generalizes the notion of SAGBI basis (a Gröbner basis analogue for subalgebras of polynomial algebras).

\begin{definition}\label{def: Khovanskii basis}
A set of algebra generators $\mathcal B \subset A$ is called a \emph{Khovanskii basis} for $(A,\val)$ if the image of $\mathcal B$ in $\gr_\val(A)$ forms a set of algebra generators.
\end{definition}

If $\mathcal B$ is a Khovanskii basis for $(A,\val)$ then (independent of the one-dimensional leaves property) by \cite[Lemma~2.10]{KM16} the image $\val(\mathcal B)$ generates $S(A,\val)$. 

Assume for now that $\gr_\val(A)$ is finitely generated and that $\val$ has one-dimensional leaves. 
Hence, $\gr_\val(A)\cong  \mathbb C[S(A,\val)]$ by the above lemma.
Further, the value semigroup $S(A,\val)$ is generated by $\{\val(b_1),\dots,\val(b_n)\}$, for some $\{b_1,\dots,b_n\}$ forming a Khovanskii basis for $(A,\val)$.
In this case $\Proj(\gr_\val(A))=\Proj(\mathbb C[S(A,v)])$ is a toric variety.
In fact, $\Proj(\gr_\val(A))$ is a flat degeneratipon of $\Proj(A)$.
To describe the corresponding family, we use the following Lemma due to Caldero.

\begin{lemma*}(\cite[Lemma~3.2]{Cal02})
Let $S$ be a finite subset of $\mathbb Z^r$. Then there exists a linear form $e:\mathbb Z^r\to \mathbb Z_{\ge 0}$ such that
\[
\bm \prec \bn \Rightarrow e(\bn)<e(\bm).
\]
\end{lemma*}

In \cite{Cal02} the lemma is stated with $\mathbb N^r$ in place of $\mathbb Z^r$.
By adding a large multiple of $(1,\dots,1)$ to every element in $S$ we obtain the lemma as stated above.
Examples of such a linear forms can be found throughout the thesis, in particular in \S\ref{sec:string&FFLV}.

Consider a linear form $e:\mathbb Z^r\to \mathbb Z$ as in \cite[Lemma~3.2]{Cal02} for $S:=\{\val(b_1),\dots,\val(b_n))\}$.
We construct a $\mathbb Z_{> 0}$-filtration on $A$ by $\mathcal F_{\le m}:=\mathcal F^{\val,e}_{\le m}=\{f\in A\setminus\{0\}\mid e(\val(f))\le m\}\cup\{0\}$ for $m\in \mathbb Z_{>0}$.
The filtration $\{\mathcal F_{\le m}\}_m$ has the property that $\bigoplus_{m\ge 0}\mathcal F_{\le m}/\mathcal F_{<m}\cong \gr_\val(A)$ and we obtain a family of $\mathbb C$-algebras (see e.g. \cite[Proposition~5.1]{An13}) that can be defined as follows.

\begin{definition}\label{def: Rees algebra}
The \emph{Rees algebra} associated with the valuation $\val$ and the filtration $\{\mathcal F_{\le m}\}_m$ is the flat $\mathbb C[t]$-subalgebra of $A[t]$ defined as
\begin{eqnarray}\label{eq: def Rees}
R_{\val,e}:=\bigoplus_{m\ge 0} (\mathcal F_{\le m})t^m.
\end{eqnarray}
It has the properties that $R_{\val,e}/tR_{\val,e}\cong \gr_\val(A)$ and $R_{\val,e}/(1-t)R_{\val,e}\cong A$. 
In particular, it defines a flat family over $\mathbb A^1$ (the coordinate on $\mathbb A^1$ given by $t$) with generic fibre isomorphic to $\Proj(A)=X$ and special fibre the toric variety $\Proj(\gr_\val(A))$.
\end{definition}

More details on Rees algebras can be found in \cite[\S5]{An13}, \cite[\S2]{T03}, and \cite[\S7]{KM16}.

Introduced by Lazarsfeld-Musta\c{t}\v{a} \cite{LM09} and Kaveh-Khovanskii \cite{KK12} we recall the definition of Newton-Okounkov body. The way we present it follows closely \cite{KK12}.

\begin{definition}\label{def: val sg NO}
Let $\val:A\setminus \{0\}\to (\mathbb Z^r,\prec)$ be a valuation. The \emph{Newton-Okounkov cone} is
\begin{eqnarray}\label{eq: def NO cone}
C(A,\val):=\conv(S(A,\val)\cup\{0\})\subset \mathbb R^r. 
\end{eqnarray}
One defines the corresponding \emph{Newton-Okounkov body} as
\begin{eqnarray}\label{eq: def NO body}
\Delta(A,\val):=\overline{\bigcup_{i> 0}\{\val(f)/i \mid 0\not =f\in A_i \}}
\end{eqnarray}
\end{definition}

We are mostly interested in projective varieties of subvarieties of a product of projective spaces as seen in the last section.
Let $X$ be such a variety of dimension $d$ and $A$ its homogeneous coordinate ring.
Recall that the total coordinate ring $S$ of $\mathbb P^{k_1}\times \dots\times\mathbb P^{k_s}$ is of form  $S=\mathbb C[x_{1,0},\dots,x_{1,k_1},x_{2,0},\dots,x_{2,k_2},\dots,x_{s,k_s}]$. 
On coordinates the degree is given by $\deg x_{i,j}:= \varepsilon_i\in\mathbb Z^s$ (see e.g. \cite[Example~5.2.2]{CLS11}) for all $i\in[s],j\in[k_i]$, where $\{\varepsilon_i\}_{i\in[s]}$ denotes the standard basis on $\mathbb Z^s$.
For $f=\sum a_{\bu}x^{\bu} \in S$ we choose the lexicographic order on $\mathbb Z^s$ and set $\deg f:=\max_{\text{lex}}\{\deg x^{\bu}\mid a_{\bu}\not=0\}$.
The $\mathbb Z_{\ge 0}^s$-grading on the  induces a $\mathbb Z^s_{\ge 0}$-grading on the homogeneous coordinate ring $A$ of $X$, which we denote $A=\bigoplus_{\bm\in\mathbb Z^s_{\ge 0}}A_\bm$. 

It is sometimes desirable to have a valuation that encodes the grading of $A$, i.e.
\begin{eqnarray}\label{eq: def extended val}
\hat{\val}:A\setminus\{0\}\to\mathbb (Z^s_{\ge 0}\times \mathbb Z^{r-s},\prec) 
\ \text{ of form }\ \hat{\val}(f)=(\deg f,\cdot), \ \forall f\in A\setminus\{0\}.
\end{eqnarray}
Examples of valuations that have this form can be found in \S\ref{sec:BLMM} where we consider valuations constructed from maximal prime cones in $\trop(\Flag_n)$ (as in \cite{KM16}).
In this case the Newton-Okounkov cone $C(A,\hat\val)$ is contained in $\mathbb R^s_{\ge 0}\times \mathbb R^{r-s}$.
The Newton-Okounkov body $\Delta(A,\hat\val)$ can be defined as the intersection of $C(A,\hat\val)$ with the hyperplane $\{(1,\dots,1)\}\times\mathbb R^{r-s}$.
More generally let $P_\val(\lambda):=C(A,\hat\val)\cap\{\lambda\}\times\mathbb R^{r-s}$ for $\lambda\in\mathbb R_{\ge 0}^s$.

Dealing with polytopes throughout the thesis we need the notion of \emph{Minkowski sum}. For two polytopes $A,B\subset \mathbb R^r$ it is defined as
\begin{eqnarray}\label{eq:Mink sum}
A+B:=\{a+b\mid  a\in A,b\in B\}.
\end{eqnarray}
For example, if $A$ is $\mathbb Z_{\ge0}^s$-graded for $s>1$ and $\val$ is of form $\hat\val$ as in \eqref{eq: def extended val} it is an interesting question if
\[
P_\val(\varepsilon_1)+\dots +P_\val(\varepsilon_s)=\Delta(A,\val),
\]
where $\varepsilon_i$ are standard basis vectors in $\mathbb R^s$. We investigate this question in the case where $A$ is the homogeneous cooridnate ring of the flag variety in \S\ref{sec:string&FFLV}.
\smallskip

The main result and reason for the popularity of Newton-Okounkov bodies is the following Theorem. 
This version is closest to the one in \cite{KK12}, but the same result in varying generalities was obtained, for example, in \cite{An13} and \cite{LM09}.

\begin{theorem*}(\cite{KK12})
Let $A$ be the homogeneous coordinate ring of a projective variety $X$ of dimension $d$ and $\val$ a valuation with one-dimensional leaves on $A$. Then $\Delta(A,\val)$ is a convex body.
Moreover, if $S(A,\val)$ is finitely generated, then $\Delta(A,\val)$ is a rational polytope whose volume $\text{Vol}(\Delta(A,\val))$ (up to rescaling by $d!$) equals the degree of $X$. 
In this case, the normalization of the (not necessarily normal) toric variety $Y=\Proj(\mathbb C[S(A,\val)])$ is the toric variety associated to $\Delta(A,\val)$.
\end{theorem*}




\section{Quasi-valuations with weighting matrices}\label{sec:val and quasival}

We briefly recall some background on higher-dimensional Gröbner theory  and quasi-valuations with weighting matrices as in \cite[\S3.1\&4.1]{KM16}. 
Then we define for a given valuation an associated quasi-valuation with weighting matrix.
This enables us to use the Kaveh-Manon's machinery for more general valuations.
For example, we can test whether a given valuation has a finitely genarated value semi-group and if so, compute the associated Newton-Okounkov body.
A central result of this thesis is Theorem~\ref{thm: val and quasi val with wt matrix}.
It is proved in full generality here, but appears in more specialized formulations in the following chapters.
Most of our results are in fact implications of this theorem.

The notions of initial form and initial ideal with respect to a weight vector (as seen in \S\ref{sec:pre trop}) can be generalized to \emph{weighting matrices} as follows.

\begin{definition}\label{def: init wrt M}
Let $f=\sum a_{\bf u} x^{\bf u} \in\mathbb C[x_1,\dots,x_n]$ with ${\bf u}\in\mathbb Z^n$, where $x^{\bf u}=x_1^{u_1}\cdots x_n^{u_n}$. For $M\in\mathbb Q^{r\times n}$ and a linear order on $\prec$ on $\mathbb Z^r$ we define
\begin{eqnarray}\label{eq: def init form M }
\init_M(f):=\sum_{\begin{smallmatrix}
M\bm =\min_{\prec}\{M{\bf u}\mid a_{\bf u}\not=0\}
\end{smallmatrix}} a_{\bm} x^{\bm}.
\end{eqnarray}
Similar to the case of weight vectors we extend this definition to ideals $I\subset\mathbb C[x_1,\dots,x_n]$ by 
\begin{eqnarray}\label{eq: def init ideal M }
\init_M(I):=\langle \init_M(f) \mid f\in I \rangle.
\end{eqnarray}
\end{definition}

A weighting matrix $M\in\mathbb Q^{r\times n}$ lies in the \emph{Gröbner region} GR$^r(I)$ of an ideal $I\subset \mathbb C[x_1,\dots,x_n]$, if there exists a monomial order $<$ on $\mathbb C[x_1,\dots,x_n]$ such that 
\[
\init_{<}(\init_{M}(I))=\init_<(I).
\]
By a \emph{positive grading} we mean a $\mathbb Z^s_{\ge 0}$-grading for $s\ge 1$ as in the case of the total coordinate ring $S$ of a product of projective spaces (see below Remark~\ref{rem:from lemma1}) or the usual polynomial ring.
If an ideal $I$ is (multi-)homogeneous with respect to a positive grading
then the lineality space $L_I$ of $I$ contains $(1,\dots,1)\in\mathbb  R^n$.
Kaveh-Manon show (see \cite[Lemma~3.7]{KM16}) that in this case $\mathbb Q^{r\times n}$ is entirely contained in GR$^r(I)$. 

To a given matrix $M\in\mathbb Q^{r\times n}$ in \cite{KM16} they associate a quasi-valuation as follows. As above, fix a group ordering $\prec$ on $\mathbb Q^r$.

\begin{definition}\label{def: quasi val from wt matrix}
Let $\tilde f=\sum a_{\bu}x^{\bu}\in\mathbb C[x_1,\dots,x_n]$ and define $\tilde\val_M:\mathbb C[x_1,\dots, x_n]\setminus\{0\}\to (\mathbb Q^r,\prec) $ by
\[
\tilde \val_M(\tilde f):=\min{}_{\prec}\{M\bu\mid a_{\bu}\not=0\}.
\]
Let $A=\mathbb C[x_1,\dots,x_n]/I$ with $I$ the kernel of $\pi:\mathbb C[x_1,\dots,x_n]\to A$. Then by \cite[Lemma~4.2]{KM16} there exists a quasi-valuation $\val_M:A\setminus\{0\}\to (\mathbb Q^r,\prec)$ given for $f\in A$ by
\[
\val_M(f):=\max{}_{\prec}\{\tilde\val_M(\tilde f)\mid \tilde f\in\mathbb C[x_1,\dots,x_n], \pi(\tilde f)=f\}.
\]
It is called the \emph{quasi-valuation with weighting matrix} $M$.
\end{definition}

From the definition, it is usually hard to explicitly compute the values of a quasi-valuation $\val_M$.
The following proposition makes it more computable, given that $M$ lies in the Gröbner region of $I$.

\begin{proposition*}(\cite[Proposition~4.3]{KM16})
Let $M\in\text{GR}^r(I)$ and $\mathbb B\subset A$ be a standard monomial basis for the corresponding monomial order $<$ on $\mathbb C[x_1,\dots,x_n]$. Then $\mathbb B$ is adapted to $\val_M$. Moreover, for every element $f\in A$ written as $f=\sum b_\alpha a_\alpha$ with $\pi(x^\alpha)=b_\alpha\in \mathbb B$ and $a_\alpha\in \mathbb C$ we have
\[
\val_M(f)=\min{}_{\prec} \{ M\alpha\mid a_\alpha\not=0\}.
\]
\end{proposition*}

Recall from above how to associate a filtration to a (quasi-)valuation. We denote the associated graded algebra of a quasi-valuation $\val_M$ by $\gr_M(A)$.

From our point of view, quasi-valuations with weighting matrices are not the primary object of interest. In most cases we are given a valuation $\val:A\setminus\{0\}\to (\mathbb Q^r,\prec)$ whose properties we would like to know.
In particular, we are interested in the generators of the value semi-group and if there are only finitely many. 
The next definition establishes a connection between a given valuations and weighting matrices. It allows us later to apply techniques from Kaveh-Manon for quasi-valuations with weighting matrices to other valuations of our interest.

From now on let $A$ be a finitely generated algebra and domain with presentation a fixed $\pi:\mathbb C[x_1,\dots,x_n]\to A$, such that $A=\mathbb C[x_1,\dots,x_n]/\ker(\pi)$. Let $I:=\ker(\pi)$ and $\pi(x_i)=:b_i$ for $i\in[n]$.
The polynomial ring may be replaced by $S$ the total coordinate ring of the product of projective spaces, but for simplicity we just write $\mathbb C[x_1,\dots,x_n]$.

\begin{definition}\label{def: wt matrix from valuation}
Given a valuation $\val:A\setminus\{0\}\to (\mathbb Q^r,\prec)$. 
We define the \emph{weighting matrix of $\val$} by
\[
M_\val:=(\val(b_1),\dots,\val(b_n))\in \mathbb Q^{r\times n}.
\]
That is, the columns of $M_\val$ are given by the images $\val(b_i)$ for $i\in[n]$.
\end{definition}

Assume that the ideal $I$ is homogeneous with respect to a positive grading. 
We need the following key-lemma from \cite{KM16}.

\begin{lemma*}(\cite[Lemma~4.4]{KM16})
The associated graded algebra of the quasi-valuation with weighting matrix $M$ satisfies 
\begin{eqnarray}\label{eq: ass graded wt matrix val}
\gr_M(A)\cong \mathbb C[x_1,\dots,x_n]/\init_M(I).
\end{eqnarray}
\end{lemma*}

By a similar argument as in the proof of \cite[Proposition~5.2]{KM16} we obtain the following corollary with assumptions being as above.

\begin{corollary}\label{cor: NO body val_M}
Let $M\in\mathbb Q^{d\times n}$ with $d$ the Krull-dimension of $A$ (i.e. $\val_M$ has full rank).
If $\init_{M}(I)$ is prime, then $\val_M$ is a valuation whose value semi-group $S(A,\val_M)$ is generated by $\{\val_M(b_i)\}_{i\in[n]}$. 
In particular, the associated Newton-Okounkov body is given by
\[
\Delta(A,\val_M)=\conv(\val_M(b_i)\mid i\in [n]).
\]
\end{corollary}

\begin{proof}
As $\init_M(I)$ is prime, we have by \cite[Lemma~4.4]{KM16} $\gr_M(A)\cong\mathbb C[x_1,\dots,x_n]/\init_M(I)$ is a domain.
Assume $\val_{M}$ is not a valuation, i.e. there exist $f,g\in A\setminus\{0\}$ with $\val_{M}(fg)\succ\val_{M}(f)+\val_{M}(g)$. 
If $\val_{M}(f)=a$ and $\val_{M}(g)=b$, then $\overline{f}\in F_{\succeq a}/F_{\succ  a}$ and $\overline{g}\in F_{\succeq b}/F_{\succ  b}$, where $F_{\succ a}$ denotes a filtered piece of the filtration $\mathcal F_{\val_M}$ on $A$. 
Then $\overline{fg}=\overline{0}\in\gr_{M}(A)$ as by the grading we have $\overline{fg}\in F_{\succeq a+b}/F_{\succ a+b}$ but $\val_{M}(fg)\succ a+b$, a contradiction to being a domain.

In particular, $\gr_M(A)$ is generated by $\overline{b_i}=\overline{\pi(x_i)}$ for $i\in[n]$.
As by \cite[Theorem~2.3]{KM16} $\val_M$ has one-dimensional leaves, then by \cite[Proposition~2.4]{KM16} we have $\gr_M(A)\cong\mathbb C[S(A,\val_M)]$.
The rest of the claim follows.
\end{proof}

As mentioned before, we want to use the results on (quasi-)valuations with weighting matrices to analyze arbitrary given valuations. 
The next lemma makes a first connection between the (quasi-)valuation with weighting matrix $M_\val$ as in Definition~\ref{def: wt matrix from valuation} and the valuation $\val$ defining it.

Denote by $\{\varepsilon_i\}_{i\in[s]}$ the standard basis of $\mathbb R^s$. We consider the (partial) order $>$ on $\mathbb Z^s$:
\[
(m_1,\dots,m_s)>(n_1,\dots,n_s) \ :\Leftrightarrow \ \sum_{i=1}^s m_i>\sum_{i=1}^s n_i.
\]

\begin{lemma}\label{lem: val vs val_Mval}
Let $I$ be a homogeneous ideal with respect to a $\mathbb Z_{\ge 0}^s$-grading for $s\ge 1$ generated by elements $f\in I$ with $\deg f>\varepsilon_i$ for all $i\in[s]$. Then
 $\val_{M_\val}(b_i)=\val(b_i)$.
\end{lemma}

\begin{proof}
Denote by $\{e_i\}_{i\in[n]}$ the standard basis of $\mathbb R^n$.
Recall that $b_i=\pi(x_i)$ for all $i\in[n]$. Using the assumption that $I$ is homogeneous, we have by definition of $\val_{M_{\val}}$ 
\begin{eqnarray*}
\val_{M_\val}(b_i)&=& \max{}_{\prec}\{ \tilde \val_{M_\val}(x_i+f)\mid f\in I \}\\
&\stackrel{\tiny \deg f>\varepsilon_i\forall i}{=}& \max{}_{\prec}\{ \min{}_{\prec}\{M_\val e_{i},\tilde \val_{M}(f)\} \mid f\in I \}.
\end{eqnarray*}
As $\min_{\prec}\{M_\val e_{i},\tilde \val_{M}(f)\}\preceq M_\val e_{i}=\val(b_i)$ we deduce $\val_{M_\val}(b_i)=\val(b_i)$.
\end{proof}

The following theorem relating a given valuation $\val$ with the (quasi-)valuation with weighting matrix $M_\val$ is our main result on Newton-Okounkov bodies.
In the rest of the thesis we use it to prove our main results.
We apply it to valuations from birational sequences (\cite{FFL15}) for $\Gr(2,\mathbb C^n)$ in \S\ref{sec:Bos} and to Rietsch-Williams \cite{RW17} valuation in \S\ref{sec:BFFHL}. 
In \S\ref{sec:BLMM} we use it to make a connection between string valuations on the homogeneous coordinate ring of the flag variety and the topical flag variety.

For simplicity we assume the image of our valuation lies in $\mathbb Z^d$ instead of $\mathbb Q^d$. This is the case for all valuations we are interested in. 

\begin{theorem}\label{thm: val and quasi val with wt matrix}
Assume $I$ is homogeneous with respect to the $\mathbb Z_{\ge 0}^s$-grading and generated by elements $f\in I$ with $\deg f>\varepsilon_i$ for all $i\in[s]$.
Let $\val:A\setminus\{0\}\to (\mathbb Z^d,\prec)$ be a full-rank valuation
with $M_\val\in\mathbb Z^{d\times n}$ the weighting matrix of $\val$ and assume $\init_{M_\val}(I)$ is prime.

Then $S(A,\val)$ is generated by $\{\val(b_i)\}_{i\in[n]}$, where $b_i=\pi(x_i)$. In particular,
\[
\Delta(A,\val)=\conv(\val(b_i)\mid i\in[n]),
\]
and $\{b_1,\dots,b_n\}$ is a Khovanskii basis for $(A,\val)$.
\end{theorem}

\begin{proof}
As $\val$ and $\val_{M_\val}$ (by Corollary~\ref{cor: NO body val_M}) are full-rank valuations, they have one-dimensional leaves by \cite[Theorem~2.3]{KM16}.
We apply \cite[Proposition~2.4]{KM16} and obtain $\gr_{\val}(A)\cong\mathbb C[S(A,\val)]$ and $\gr_{M_{\val}}(A)\cong \mathbb C[S(A,\val_{M_{\val}})]$.
\smallskip

\emph{Claim:} For all $g\in A$ we have $\val(g)=\val_{M_\val}(g)$.
\smallskip

\noindent
That is, $S(A,\val)=S(A,\val_{M_\val})$.
The latter is generated by $\{\val(b_i)\}_{i\in[n]}$ by Corollary~\ref{cor: NO body val_M} and Lemma~\ref{lem: val vs val_Mval}. 
All other statements of the theorem are direct consequences.
\smallskip

\emph{Proof of claim:}
By Corollary~\ref{cor: NO body val_M} $\val_{M_\val}$ is a valuation.
In particular, then by Lemma~\ref{lem: val vs val_Mval} we have  $\val(b_{\bu})=\val_{M_\val}(b_{\bu})$ for monomials $b_{\bu}=b_1^{u_1}\cdots b_n^{u_n}\in A$.

As $I$ is homogeneous with respect to a positive grading, $M_\val$ lies in the Gröbner region of $I$. 
Let $\mathbb B\subset A$ be the standard monomial basis adapted to $\val_{M_\val}$ as in \cite[Proposition~4.3]{KM16} restated above.
Then we can write every $g\in A$ as $g=\sum_{i=1}^k b_{\alpha_i}a_i$ for $b_{\alpha_i}=\pi(x^{\alpha_i})\in \mathbb B$ and $a_i\in\mathbb C$.
We compute
\begin{eqnarray*}
\val_{M_\val}(g) &=& \val_{M_\val}\left( \sum_{i=1}^k b_{\alpha_i}a_i \right)\\
& \succeq & \min{}_\prec\{\val_{M_\val}(b_{\alpha_i})\mid a_i\not=0\} \\
&\stackrel{\text{Lemma~\ref{lem: val vs val_Mval}}}{=} & \min{}_\prec\{ \val(b_{\alpha_i})\mid a_i\not=0 \}\\
&\stackrel{\text{Def. }M_\val }{=}& \min{}_\prec \{M_\val \alpha_i\mid a_i\not=0\}\\
&\stackrel{\text{\cite[Proposition~4.3]{KM16}}}{=}& \val_{M_\val}(g).
\end{eqnarray*}
As $\val_{M_\val}$ has one-dimensional leaves $b\mapsto \val_{M_{\val}}(b)$ for $b\in\mathbb B$ (adapted to $\val_{M_\val}$) defines a bijection between $\mathbb B$ and the set of values of $\val_{M_\val}(\mathbb B)$ by \cite[Remark~2.29]{KM16}.
In particular, we have
$
\val(b_{\alpha_i})=\val_{M_\val}(b_{\alpha_i})\not= \val_{M_{\val}}(b_{\alpha_i}) =\val(b_{\alpha_j}) \text{ for all } i\not=j. 
$
This implies
\begin{eqnarray*}
\val(g) = \min{}_{\prec}\{\val(b_{\alpha_i})\mid a_i\not=0\} = \val_{M_\val}(g).
\end{eqnarray*}
\end{proof}

\subsubsection*{From weighting matrix to weight vector}

The assumption $\init_{M_\val}(I)$ being prime is quite strong, as this is in general hard to verify. However, taking the initial ideal with respect to a weighting matrix is closely related to taking the initial ideal with respect to a weight vector, which makes the computation easier.
For example, \cite[Proposition~3.10]{KM16} says that for every $M\in\mathbb Q^{r\times n}$ there exists $\bw\in\mathbb Q^n$ such that $\init_{M}(I)=\init_{\bw}(I)$.
We want to make this more explicit using \cite[Lemma~3.2]{Cal02} restated in \S\ref{sec:pre val}.
The lemma allows us to associate a weight vector to a weighting matrix as follows.

\begin{definition}\label{def: wt vector for wt matrix}
Let $M\in\mathbb Z^{r\times n}$ and choose $e:\mathbb Z^r\to \mathbb Z$ as in \cite[Lemma~3.2]{Cal02} for $S=\{M_1,\dots,M_n\}$ the set of columns of $M$.
We define the \emph{weight vector associated to $M$} as
\[
e(M):=(e(M_1),\dots,e(M_n))\in\mathbb Z^n.
\]
\end{definition}

The following lemma shows that for the initial ideal, $e(M)$ is independent of the choice of $e$. It is applied throughout the thesis whenever we apply Theorem~\ref{thm: val and quasi val with wt matrix}.

\begin{lemma}\label{lem: init wM vs init M}
Let $f\in \mathbb C[x_1,\dots,x_n]$. Then $\init_M(f)=\init_{e(M)}(f)$. In particular, for every ideal $I\subset\mathbb C[x_1,\dots,x_n]$ we have $\init_M(I)=\init_{e(M)}(I)$.
\end{lemma}

\begin{proof}
Let $M=(m_{i,j})_{i\in[r],j\in[n]}\in\mathbb Z^{r\times n}$ and $e:\mathbb Z^{r}\to\mathbb Z$ given by $e(x_1,\dots,x_r)=\sum_{i=1^r}x_ic_i$ for $c_i\in\mathbb Z_{>0}$.
Let $\cdot$ denote the usual dot-product in $\mathbb R^n$.
We compute for $\bu=(u_1,\dots, u_n)\in \mathbb Z_{\ge 0}^n$
\[
e(M\bu)=\sum_{i=1}^r\sum_{j=1}^n m_{ij}u_jc_i = e(M)\cdot \bu. 
\]
Now for $f=\sum a_{\bu}x^{\bu}\in\mathbb C[x_1,\dots,x_n]$ by defintion we have
\begin{eqnarray*}
\init_M(f) &=& \sum_{\bm:\ M\bm=\min_{\prec}\{M\bu\mid a_{\bu}\not=0\}} a_{\bm}x^{\bm}\\
&\stackrel{\text{\cite[Lemma~3.2]{Cal02}}}{=}& \sum_{\bm: \ e(M\bm)=\min\{e(M\bu)\mid a_{\bu}\not=0 \}} a_{\bm}x^{\bm} \\
&\stackrel{e(M)\cdot \bu=e(M\bu)}{=}& \sum_{\bm:\ e(M)\cdot\bm=\min\{e(M)\cdot \bu\mid a_{\bu}\not =0\}} a_{\bm}x^{\bm} \\
&=& \init_{e(M)}(f).
\end{eqnarray*}
\end{proof}

With assumptions as in the Lemma
let $S_M:=\{e(M)\mid e \text{ as in \cite[Lemma~3.2]{Cal02} }\}\cup\{0\}\subset \mathbb Q^n$. We define a polyhedral cone given the set $S_M$ and the lineality space $L_I\subset\mathbb R^n$ of the ideal $I\subset \mathbb C[x_1,\dots,x_n]$
by
\[
C_M:=\cone(S_M) +L_I\subset \mathbb R^n. 
\]
Then the following corollary is a reformulation of Lemma~\ref{lem: init wM vs init M}.
\begin{corollary}
There exists a cone $C$ in the Gröbner fan of $I$ with $C_M\subseteq C$. Moreover, if $\init_M(I)$ is monomial-free, $C_M\subset \trop(V(I))$.
\end{corollary}


\section{Cluster algebras}\label{sec: prep cluster}

We recall here the basic notions and definitions from cluster theory. 
This section follows \cite[\S2]{Wil14} for quivers and quiver mutation and \cite[\S2]{GHK15} for the review of $\mathcal A$- and $\mathcal X$-cluster varieties.

A \emph{quiver} $Q$ is a tupel $(Q_0,Q_1)$ with $Q_0$ a finite set of vertices and $Q_1$ a finite set of arrows between the vertices in $Q_0$. A \emph{loop} is an arrow whose source and target vertex coincide, a \emph{2-cycle} is an oriented cycle consisting of two arrows. We consider quivers with neither loops nor 2-cycles.

\begin{definition}\label{def:quiver mutation}
Let $Q$ be a finite quiver without loops and 2-cycles and $k\in Q_0$. Then we define $\mu_k(Q)$ to be the quiver obtained from $Q$ by the following recipe called \emph{mutation at vertex $k$}:

    \emph{Step 1:} for every configuration of arrows $i\to k\to j$ add a new arrow $i\to j$;
    
    \emph{Step 2:} reverse all arrows incident to $k$;
    
    \emph{Step 3:} delete a maximal set of 2-cycles that may have appeared as a result of  Steps~1\&2.
\end{definition}
It is a fact that mutation defines an involution and we have $\mu_k(\mu_k(Q))=Q$. For an example see Figure~\ref{fig:exp quiver mut}.

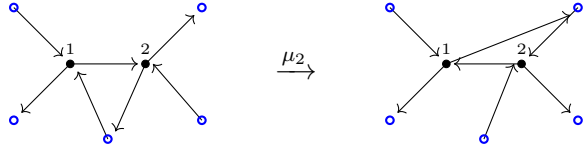
\begin{figure}[h]
\centering
\begin{tikzpicture}
    \draw[->] (-.75,.75) -- (-.1,.1);
    \draw[->] (0,0) -- (.9,0);
    \draw[->] (1,0) -- (1.65,.65);
    \draw[->] (1.75,-.75) -- (1.1,0);
    \draw[->] (1,0) -- (.6,-.9);
    \draw[->] (.5,-1) -- (.1,-.1);
    \draw[->] (0,0) -- (-.65,-.65);
    
    \draw[thick, blue] (-.75,.75) circle [radius=0.05];
    \draw[fill] (0,0) circle [radius=0.05];
    \draw[fill] (1,0) circle [radius=0.05];
    \draw[thick, blue] (1.75,.75) circle [radius=0.05];
    \draw[thick, blue] (1.75,-.75) circle [radius=0.05];
    \draw[thick, blue] (.5,-1) circle [radius=0.05];
    \draw[thick, blue] (-.75,-.75) circle [radius=0.05];
    
    \node[above] at (0,0) {\tiny 1};
    \node[above] at (1,0) {\tiny 2};
    \node at (3,0) {$\xrightarrow{\mu_2}$};
    
    \begin{scope}[xshift=5cm]
    \draw[->] (-.75,.75) -- (-.1,.1); 
    \draw[<-] (0.1,0) -- (1,0); 
    \draw[<-] (1.1,0.1) -- (1.75,.75); 
    \draw[<-] (1.65,-.65) -- (1,0); 
    \draw[<-] (.9,0) -- (.5,-1); 
    \draw[->] (0,0) -- (-.65,-.65);
    \draw[->] (0,0) -- (1.65,.65);
     
    \draw[thick, blue] (-.75,.75) circle [radius=0.05];
    \draw[fill] (0,0) circle [radius=0.05];
    \draw[fill] (1,0) circle [radius=0.05];
    \draw[thick, blue] (1.75,.75) circle [radius=0.05];
    \draw[thick, blue] (1.75,-.75) circle [radius=0.05];
    \draw[thick, blue] (.5,-1) circle [radius=0.05];
    \draw[thick, blue] (-.75,-.75) circle [radius=0.05];
    
    \node[above] at (0,0) {\tiny 1};
    \node[above] at (1,0) {\tiny 2};
    \end{scope}
\end{tikzpicture}
    \caption{An example of quiver mutation at the vertex $2$.}
    \label{fig:exp quiver mut}
\end{figure}

We divide the vertex set $Q_0=\{1,\dots,m\}$, into two parts $\{1,\dots n\}$ and $\{n+1,\dots, m\}$ for $n\le m$. We call $\{1,\dots, n\}$ \emph{mutable} vertices and $\{n+1,\dots, m\}$ \emph{frozen} vertices. From now on we only allow mutation at mutable vertices. Further, we ignore arrows between frozen vertices as they are irrelevant for the mutation.
To $Q$ we associate its incidence matrix $(\epsilon_{ik})_{i,k\in Q_0, k \text{ mutable}}\in M_{m\times n}$ given by
\begin{eqnarray}\label{eq: matrix for quiver}
\epsilon_{ik}:=\#\{\text{arrows } i\to k \in Q_1 \}- \#\{\text{arrows }k\to i\in Q_1\}.
\end{eqnarray}

We fix $\mathcal F$ as our ambient field of rational functions in $n$ variables defined over the field $\mathbb Q(A_{n+1},\dots,A_{m})$.

\begin{definition}\label{def:seed}
A \emph{labelled seed} in $\mathcal F$ is a pair $s:=({\bf A}_s,Q_s)$, where ${\bf A}_s:=(A_{1,s},\dots,A_{m,s})$ is a free generating set for $\mathcal F$ and $Q_s$ a quiver with mutable vertices $\{1,\dots, n\}$ and frozen vertices $\{n+1,\dots,m\}$.
We call ${\bf A}_s$ an \emph{extended cluster} with \emph{cluster variables} $\{A_{1,s},\dots,A_{n,s}\}$ and \emph{frozen variables} $\{A_{n+1,s},\dots,A_{m,s}\}$.
\end{definition}

\begin{definition}\label{def: A mutation}
Let $s=({\bf A}_s,Q_s)$ be a labelled seed in $\mathcal F$ and $k\in\{1,\dots,n\}$. We define the \emph{seed mutation} (also called $\A$-mutation) in direction $k$ to be the operation that takes $s$ to $s'=({\bf A}_{s'},Q_{s'})$, where $Q_{s'}=\mu_k(Q_s)$ and ${\bf A}_{s'}=(A_{1,s'},\dots,A_{m,s'})$ is given by
\begin{eqnarray}\label{eq: A mutation}
A_{k,s'}A_{k,s}:=\prod_{i\to k\in Q} A_{i,s} + \prod_{k\to j\in Q} A_{j,s}.
\end{eqnarray}
\end{definition}

Note that when $s'$ is obtained from $s$ by mutation at $k$, then also $s$ is obtained from $s'$ by mutation at $k$. 
That is mutation is an involution on seeds. 
Observe that frozen variables are not affected by mutation. For any two seeds $s$ and $s'$ we have $A_{k,s}=A_{k,s'}$ for all $k\in[n+1,m]$. 
We therefore drop the index of the seed from frozen variables and have ${\bf A}_s=(A_{1,s},\dots,A_{n,s},A_{n+1},\dots,A_m).$
If it is clear from the context which seed we are considering we also drop the $s$ completely in our notation.

Consider the n-regular infinite tree $\mathbb T_n$ whose edges at every vertex are labelled by $1,\dots,n$. 
An assignment of a seed $s_t$ to every vertex $t\in\mathbb T_n$ is called a \emph{seed pattern}, if two seeds $s_t,s_{t'}$ associated to adjacent vertices $t\stackrel{k}{--} t'$ in $\mathbb T_n$ are obtained from each other by mutation at $k$. 
Let $\mathcal V:=\bigcup_{t\in \mathbb T_n}\{A_{1,s_t},\dots,A_{n,s_t}\}$ be the union of all cluster variables for all seeds in the seed pattern. Note that allthough the tree has infinitely many vertices $\mathcal V$ might be a finite set as through repetition some seeds might coincide. 
For example, in Figure~\ref{fig: seed pattern T_2} there is a seed pattern for $\mathbb T_2$ and we observe that the cluster variables for $s_{t_{i-2}}$ coincide with those for $s_{t_{i+3}}$.

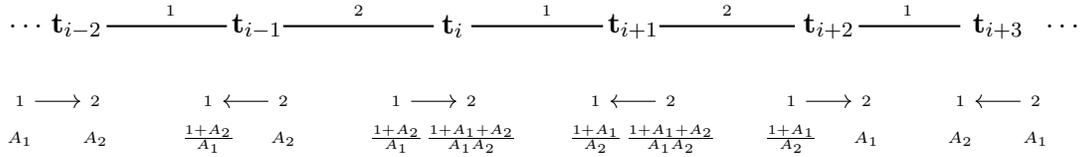
\begin{figure}[h]
    \centering
\begin{tikzpicture}[scale=2]
    \draw[thick] (1.325,0) -- (2.125,0);
    \node[above] at (1.75,0) {\tiny 1};
    \draw[thick] (2.5,0) -- (3.5,0);
    \node[above] at (3,0) {\tiny 2};
    \draw[thick] (3.75,0) -- (4.625,0);
    \node[above] at (4.25,0) {\tiny 1};
    \draw[thick] (5,0) -- (5.9,0);
    \node[above] at (5.45,0) {\tiny 2};
    \draw[thick] (6.325,0) -- (7,0);
    \node[above] at (6.65,0) {\tiny 1};
    
    \node[left] at (1,0) {$\dots$};
    \node at (1.125,0) {$\mathbf t_{i-2}$};
        \begin{scope}[yshift=-1.25cm]
        \node at (.75,.75) {\tiny 1};
        \draw[->] (.85,.75) -- (1.15,.75);
        \node at (1.25,.75) {\tiny 2};
        \node at (.75,.5) {\tiny $A_1$};
        \node at (1.25,.5) {\tiny $A_2$};
        \end{scope}
    \node at (2.325,0) {$\mathbf t_{i-1}$};
        \begin{scope}[xshift=1.25cm,yshift=-1.25cm]
        \node at (.75,.75) {\tiny 1};
        \draw[<-] (.85,.75) -- (1.15,.75);
        \node at (1.25,.75) {\tiny 2};
        \node at (.75,.5) {\tiny $\frac{1+A_2}{A_1}$};
        \node at (1.25,.5) {\tiny $A_2$};        
        \end{scope}
    \node at (3.625,0) { $\mathbf t_i$};
        \begin{scope}[xshift=2.5cm,yshift=-1.25cm]
        \node at (.75,.75) {\tiny 1};
        \draw[->] (.85,.75) -- (1.15,.75);
        \node at (1.25,.75) {\tiny 2};
        \node at (.75,.5) {\tiny $\frac{1+A_2}{A_1}$};
        \node at (1.25,.5) {\tiny $\frac{1+A_1+A_2}{A_1A_2}$};        
        \end{scope}
    \node at (4.825,0) {$\mathbf t_{i+1}$};
        \begin{scope}[xshift=3.825cm,yshift=-1.25cm]
        \node at (.75,.75) {\tiny 1};
        \draw[<-] (.85,.75) -- (1.15,.75);
        \node at (1.25,.75) {\tiny 2};
        \node at (.75,.5) {\tiny $\frac{1+A_1}{A_2}$};
        \node at (1.25,.5) {\tiny $\frac{1+A_1+A_2}{A_1A_2}$};        
        \end{scope}
    \node at (6.125,0) {$\mathbf t_{i+2}$};
        \begin{scope}[xshift=5.125cm,yshift=-1.25cm]
        \node at (.75,.75) {\tiny 1};
        \draw[->] (.85,.75) -- (1.15,.75);
        \node at (1.25,.75) {\tiny 2};
        \node at (.75,.5) {\tiny $\frac{1+A_1}{A_2}$};
        \node at (1.25,.5) {\tiny $A_1$};        
        \end{scope}  
    \node at (7.25,0) {$\mathbf t_{i+3}$};    
        \begin{scope}[xshift=6.25cm,yshift=-1.25cm]
        \node at (.75,.75) {\tiny 1};
        \draw[<-] (.85,.75) -- (1.15,.75);
        \node at (1.25,.75) {\tiny 2};
        \node at (.75,.5) {\tiny $A_2$};
        \node at (1.25,.5) {\tiny $A_1$};        
        \end{scope}        
    \node[right] at (7.5,0) {$\dots$};
\end{tikzpicture}
    \caption{A seed pattern for $\mathbb T_2$.}
    \label{fig: seed pattern T_2}
\end{figure}

\begin{definition}\label{def:cluster algebra}
The \emph{($\A$-)cluster algebra} associated with a given seed pattern is the algebra
\begin{eqnarray}\label{eq: def cluster algebra}
\mathcal Y({\bf A},Q):=\mathbb Z[A_{n+1},\dots,A_m][\mathcal V],
\end{eqnarray}
where $({\bf A},Q)$ is any seed in the given seed pattern. We say it has \emph{rank} $n$, as every cluster contains $n$ cluster variables. It is called a \emph{skew-symmetric cluster algebra of geometric type}.
We also define the \emph{upper cluster algebra} following \cite[Definition~1.6]{BFZ05} as the $\mathcal F$-subalgebra of all Laurent polynomials in the variables of any seed in the given seed pattern. We denote it by $\overline{\mathcal Y}(\mathbf A,Q)$.
\end{definition}

\begin{example}\label{exp: cluster Gr(2,4)}
Consider $\mathbb C[\Gr(2,4)]$ the homogeneous coordinate ring of the Grassmannian $\Gr(2,4)$. Recall (or see \S\ref{sec:pre-grass}) that
\[
\mathbb C[\Gr(2,4)]=\mathbb C[p_{12},p_{13},p_{23},p_{14},p_{24},p_{34}]/\langle p_{12}p_{34}-p_{13}p_{24}+p_{14}p_{23}\rangle.
\]
Then $\{p_{12},p_{13},p_{23},p_{14},p_{34}\}$ is a set of algebraically independent generators as
\[
p_{24}p_{13}^{-1}=p_{12}p_{34}+p_{14}p_{23}.
\]
Observe that this relation is strikingly reminiscent with the mutation formula in \eqref{eq: A mutation}.
In fact, considering the quiver $Q=(Q_0,Q_1)$ with $Q_0=\{1,\dots,5\}, Q_1=\varnothing$ and $1$ being the only mutable vertex we obtain a cluster algebra $\mathcal Y$ with $\mathcal Y\otimes_{\mathbb Z}\mathbb C\cong \mathbb C[\Gr(2,4)]$. The quiver $Q$ is of type $\mathtt A_1$.
A more general statement holds due to Scott \cite{Sco06} and Fomin-Zelevinsky \cite{FZ02}: $\mathbb C[\Gr(2,n)]$ has the structure of a cluster algebra of \emph{type} $\mathtt A_{n-3}$: i.e. among all mutation equivalent quivers defining the cluster algebra, there exists one whose full subquiver on all mutable vertices is an orientation of an $\mathtt A_{n-3}$-Dynkin diagram. For example, in Figure~\ref{fig:exp quiver mut} there are two quivers of type $\mathtt A_2$ for $\mathbb C[\Gr(2,5)]$.
\end{example}

Very important results in the theory of cluster algebras are the \emph{Laurent phenomenon} \cite[Theorem~3.1]{FZ02} and the \emph{Positivity of the Laurent phenomenon} \cite[Corollary~0.4]{GHKK14}. We state the latter below.

\begin{theorem*}(\cite[Corollary~0.4]{GHKK14})\label{thm:laurent phenom}
Each cluster variable of an $\A$-cluster algebra is a Laurent polynomial with nonnegative integer coefficients in the cluster variables of any given seed.
\end{theorem*}

In order to define cluster varieties we slightly change our perspective from this algebraic point of view to a more geometric one. To a seed $s$ we associate a lattice $N=\mathbb Z^m$ with basis $\{e_{1,s},\dots,e_{n,s},e_{n+1},\dots, e_m\}$.
We sometimes write $N_s$ to refer to $N$ with the associated basis.
It comes equipped with a (global) bilinear form on $N$. For a fixed seed $s$ we have
\begin{eqnarray}\label{eq: bilin form on lattice}
\{\cdot,\cdot\}_s: N\times N \to \mathbb Z,
\end{eqnarray}
is (locally) induced by the exchange matrix of $Q_s$ (for details see \cite[\S2]{GHK15}).
Let $M=\Hom(N,\mathbb Z)$ be the dual lattice with dual basis $\{f_{1,s},\dots,f_{n,s},f_{n+1},\dots, f_m\}$. 
To each lattice we associate a torus $T_N\cong (\mathbb C^*)^m\cong T_M$ by
\begin{eqnarray}\label{def: seed tori}
\X_s:=T_M=\Spec(\mathbb C[N]) \text{ and } \A_s:=T_N=\Spec(\mathbb C[M]).
\end{eqnarray}

We denote the coordinates on $\X_s$ by $X_{1,s},\dots,X_{m,s}$. Corresponding to the basis of the lattice we have $X_{i,s}:=z^{e_{i,s}}$. When the seed we are working in is clear we drop it from the notation. We define \emph{mutation} at $k$ on the basis $\{e_{i,s}\}$ of the lattice $N$ for seed $s$ by
\begin{eqnarray}\label{eq: def mutation lattice basis}
e_{i,s'}:=\left\{
    \begin{matrix}
    e_{i,s}+\max\{\epsilon_{ik},0\}e_{k,s}, & \text{ for } i\not =k,\\
    -e_{k,s},& \text{ for } i=k.
    \end{matrix}
\right.
\end{eqnarray}
Then $\{e_{1,s'},\dots,e_{n,s'},e_{n+1},\dots,e_m\}$ forms again a basis for $N$ associated with the seed $s'=\mu_k(s)$. The dual basis for $M$ transforms as
\[
f_{i,s'}:=\left\{
    \begin{matrix}
        -f_{i,s},& \text{ for } i\not=k, \\
    f_{k,s}+\sum_{j}\max\{-\epsilon_{kj},0\}f_{j,s}, & \text{ for } i =k.
    \end{matrix}
\right.
\]
Then $\{f_{1,s'},\dots,f_{n,s'},f_{n+1},\dots,f_{m}\}$ is the dual basis for $M$ associated with $s'=\mu_k(s)$. Mutation induces birational maps between the tori
\[
\mu_k:\X_s\to \X_{\mu_k(s)} \text{ and } \mu_k:\A_s\to \A_{\mu_k(s)}.
\]
defined by the pullback of functions. We have for $\X$-tori
\begin{eqnarray}\label{eq: def pullback X-mut}
\mu_k^*(z^n):=z^n(1+z^{e_{k,s}})^{-{\{n,e_{k,s}\}_s}}, \text{ for } n\in N. 
\end{eqnarray}
For the $\A$-tori the birational map is induced from the seed mutation defined in \eqref{eq: A mutation}, we recover
\[
\mu_k^*(A_{k,s'})= \left\{\begin{matrix}
    A_{i,s}, &\text{ for } i\not=k, \\
    \frac{\prod_{i\to k\in Q_s} A_{i,s} + \prod_{k\to j\in Q_s} A_{j,s}}{A_{k,s}}, &\text{ for } i=k.
    \end{matrix}\right.
\]
To be consistent with the $\X$-notation, we set $A_{i,s}=z^{f_{i,s}}$ for $1\le i\le n$ and $A_l=z^{f_l}$ for $n+1\le l\le m$ coordinates for $\A_s$. Consider again a given seed pattern, then  by \cite[Proposition~2.4]{GHK15} we can give the following definition.

\begin{definition}\label{def:cluster variety}
Given a seed pattern the $\X$- (resp. $\A$-) cluster variety is defined as the scheme
\begin{eqnarray}\label{def: A and X variety}
\X := \bigcup_{t\in \mathbb T_n}\X_{s_t} \ (\text{resp.  } \A:=\bigcup_{t\in\mathbb T_n} \A_{s_t})
\end{eqnarray}
obtained by glueing the tori $\X_{s_t}$ (resp. $\A_{s_t}$) along the birational maps induced by mutation.
\end{definition}

Sometimes $\X$ is called the \emph{Fock-Goncharov dual} to the cluster variety $\A$.
The relation to cluster algebras is the following. The global sections of the structure sheaf on $\A$ are related to the upper cluster algebra associated to the given seed pattern by 
\[
H^0(\A,\mathcal O(\A))=\overline{\mathcal{Y}}(s_t)\otimes_{\mathbb Z}\mathbb C.
\]

A natural (partial) compactification $\bar \A$ of $\A$ (an $\A$-cluster variety) is given by allowing the frozen variables $A_{n+1},\dots,A_m$ to vanish. 
We denote the resulting \emph{boudary disivor} in $\bar \A$ by
\begin{eqnarray}\label{eq: def boundary divisor}
D:=\sum_{f=n+1}^m D_f, \text{ where }D_{f}:=\{A_f=0\}\subset \bar \A.
\end{eqnarray}

\begin{example}\label{exp: bdy div Grass}
Recall Example~\ref{exp: cluster Gr(2,4)}. The $\A$-cluster variety for this type $\mathtt A_1$-cluster algebra with four frozen vertices is given by glueing two tori
\[
T_{p_{13},p_{12},p_{14},p_{23},p_{34}}\cup_{\mu} T_{p_{24},p_{12},p_{14},p_{23},p_{34}}
\]
along the birational map $\mu$ induced by mutation. The irreducible components of the boundary divisor are
\[
\{p_{12}=0\},\{p_{14}=0\},\{p_{23}=0\},\{p_{34}=0\}.
\]
One can show that up to codimension two $\bar\A$ is $\Gr(2,4)$. As the Picard group of $\Gr(2,4)$ has rank one, all four divisors are linearly equivalent and the boundary divisor $D$ is in fact the anticanonical divisor for $\Gr(2,4)$.
We recall later (in \S\ref{sec:pre flag}) how to associate very ample line bundles $L_\lambda$ on $SL_n/B$ to weights $\lambda\in\Lambda^{++}$.
The same construction works for $\Gr(2,4)$ and one obatians $\mathcal O(L_{4\omega_2})=D$ (up to linear equivalence).
\end{example}

Every component $D_f$ of the boundary divisor induces a (rank 1) valuation $\ord_{D_f}:\mathbb C[\A]\to\mathbb Z$ by sending a function $g\in\mathbb C[\A]$ to its order of vanishing along $D_f$. 
If $g$ has a pole along $D_f$, then $\ord_{D_f}(g)<0$ is the order of the pole.
These valuations are called \emph{divisorial discrete valuations} in \cite{GHKK14}.

A main result of \cite{GHKK14} is the definition and parametrization of the $\vartheta$-basis for $\mathbb C[\A]$. One central question is:
\emph{When is a basis element of $\mathbb C[\A]$ also a basis element for $\mathbb C[\bar \A]$?}

The \emph{full Fock-Goncharov conjecture} (see \cite[Definition~0.6]{GHKK14}) suggests that basis elements for $\mathbb C[\A]$ are parametrized by \emph{tropical points} in $\X^{\trop}(\mathbb Z)$ (see \cite[\S2]{GHKK14}). 
We don't go into detail about this tropical space due to the following fact: 
fixing a seed $s$ we have an isomorphism 
\[
\X^{\trop}(\mathbb Z)\vert_{s}\cong N_s\cong \mathbb Z^{m}.
\]
For the purpose of this thesis we always work in a fixed seed and therefore have an identification of lattice points in $N_s$ with basis elements for $\mathbb C[\A]$.
From now on we assume that the cluster variety $\A$ satisfies the full FG-conjecture, as this is the case for the cluster varieties we are interested in. 
For example, Magee showed in \cite{Mag15} that this is the case for the cluster variety inside $SL_n/U$ which are of interest in \S\ref{sec:BF}.
A number of criteria for the full Fock-Goncharov conjecture to hold are discussed in \cite[\S8.4]{GHKK14} and we refer the interested reader there for more details.

Associated to each component of the boundary divisor there exists a function $\vartheta_f$ on the dual cluster variety $\X$.
Assuming the full FG-conjecture we can compute and expression for $\vartheta_f$ in $\X_{s_0}$ ($s_0$ being a fixed initial seed) as described by the Algorithm~\ref{alg:superpot via opt seeds}, which we consider as definition.

\begin{definition}\label{def:superpotential}
Let $\A$ be a cluster variety associated to an $\A$-cluster algebra $\mathcal Y({\bf A},Q)$ satisfying the full Fock-Goncharov conjecture. 
Then we define the \emph{superpotential} $W:\X\to \mathbb C$ on the dual cluster variety $X$ as
\[
W:=\sum_{f \text{ frozen vertex in }Q }\vartheta_f.
\]
\end{definition}

\begin{algorithm}[h]
\SetAlgorithmName{Algorithm}{} 
\KwIn{\medskip {\bf Input:\ }  A cluster variety $\A$ with initial seed $s_0$ satisfying the full FG-conjecture.}
\BlankLine
\For{every frozen vertex $f\in Q_{s_0}$}{find a sequence of mutations $\overline{\mu}$ from $s_0$ to a seed $s_f$ where $f$ is a sink.\\
    \If{$s_f=s_0$}{{\bf Output:\ } $\vartheta_f\vert_{\X_{s_0}}=z^{-e_{f,s_0}}$.}
    \Else{apply the pullback of the reverse mutation sequence to $z^{-e_{f,s_f}}$. \\
    {\bf Output:\ } $\vartheta_f\vert_{\X_{s_0}}=(\overline{\mu})^*(z^{-e_{f,s_f}})$.}}
\BlankLine
{\bf Output:\ } The superpotential $W\vert_{\X_{s_0}}=\sum_{f\text{ frozen in }Q_{s_0}} \vartheta_f\vert_{\X_{s_0}}$. 
\label{alg:superpot via opt seeds}
\caption{Computing an expression for the superpotential in a given inital seed.}
\end{algorithm}

A seed $s_f$ for which a frozen vertex $f$ is a sink (as in the first step of Algorithm~\ref{alg:superpot via opt seeds}) is called \emph{optimized} for $f$.

\begin{remark}
Finding an optimized seed for a frozen vertex is in general a hard problem as there might be infinitely many seeds. Further, doing these computations by hand is already after a few mutation quite frustrating due to the recursive formulas. An excellent tool for such computations is provided by Keller's \emph{quiver mutation applet} \cite{QuiverApp}.
\end{remark}

Coming back to $\mathbb C[\A]$, note that a basis element $\vartheta\in\mathbb C[\A]$ gives an element in $\mathbb C[\bar \A]$ if $\ord_{D_f}(\vartheta)\ge 0$ for every component $D_f$ of the boundary divisor.
In particular,
\[
\vartheta \in \mathbb C[\bar \A] \ \text{ if and only if } \ \min_{f \text{ frozen}}\{\ord_{D_f}(\vartheta)\}\ge 0.
\]
Let $g_\vartheta\in N_s$ be the lattice point associated to $\vartheta$ for a fixed seed $s$. Then using the fact that $\vartheta_f^{\trop}(g_{\vartheta})=\ord_{D_f}(\vartheta)$, this translates to
\begin{eqnarray}\label{eq:pts in superpot cone}
\vartheta \in \mathbb C[\bar{\A}]\  \text{ if and only if } \ g_\vartheta\in \{\mathbf x\in \mathbb R^{m}\mid W\vert_{\X_s}^{\trop}(\mathbf x)\ge 0\}\cap N_s.
\end{eqnarray}
In particular, the lattice points in $\{\mathbf x\in \mathbb R^{m}\mid W\vert_{\X_s}^{\trop}(\mathbf x)\ge 0\}$ parametrize a basis for $\mathbb C[\bar \A]$.

\chapter{Grassmannians}\label{chap:Grass}

\section{Preliminary notions}\label{sec:pre-grass}

The Grassmannian $\Gr(k,\mathbb C^n)$ for integers $k\le n$ is the space of $k$-dimensional subspaces of $\mathbb C^n$. 
It has the structure of a projective variety given by the Pl\"ucker embedding $\Gr(k,\mathbb C^n)\hookrightarrow \mathbb P(\bigwedge^k\mathbb C^n)$ sending the generators $v_1,\dots,v_k\in \mathbb C^n$ of a $k$-dimensional vector subspace $V\subset \mathbb C^n$ to $[v_1\wedge \dots\wedge v_k]\in \mathbb P(\bigwedge^k\mathbb C^n)$. 
In many cases it is useful to describe $\Gr(k,\mathbb C^n)$ as a vanishing set $V(I_{k,n})$. We denote the standard basis of $\mathbb C^n$ by $\{e_1,\dots, e_n\}$ and choose a subset $I=\{i_1,\dots,i_k\}$ of $\{1,\dots,n\}=:[n]$. 

\begin{definition}\label{def: Pluecker coord}
The \emph{Pl\"ucker coordinate} $\bar p_I$ is the basis element in $(\bigwedge^k\mathbb C^n)^*$ dual to $e_{i_1}\wedge\dots\wedge e_{i_k}$. 
\end{definition}
Plücker coordinates generate the homogeneous coordinate ring of $\Gr(k,\mathbb C^n)$ satisfying certain relations.
We want to express $\mathbb C[\Gr(k,\mathbb C^n)]=:A_{k,n}$ as a quotient of the polynomial ring $\mathbb C[p_J\mid J\in\binom{[n]}{k}]$ by a prime ideal encoding these relations.
We define for $K\in\binom{[n]}{k-1}$ and $L\in\binom{[n]}{k+1}$ the sign $\sgn(j;K,L):=(-1)^{\#\{l\in L\mid j<l \}+\#\{k\in K\mid k>j\}}$. 
The following definition can be found for example in \cite[p. 170]{M-S}. 
\begin{definition}
The \emph{Plücker relation} $R_{K,L}\in\mathbb C[p_J\mid J\in\binom{[n]}{k}]$ for $K\in\binom{[n]}{k-1}$ and $L\in\binom{[n]}{k+1}$ is
\begin{eqnarray}\label{eq: def plucker rel}
R_{K,L}:=\sum_{j\in L} \sgn(j;K,L)p_{K\cup \{j\}}p_{L\setminus\{j\}}.
\end{eqnarray}
The \emph{Plücker ideal} $I_{k,n}\subset \mathbb C[p_J\mid J\in\binom{[n]}{k}]$ is generated by $R_{K,L}$ for all $K\in\binom{[n]}{k-1}$ and $L\in\binom{[n]}{k+1}$ and $\mathbb C[\Gr(k,\mathbb C^n)]=\mathbb C[p_J\mid J\in\binom{[n]}{k}]/I_{k,n}$.
\end{definition}

In the special case of $k=2$, Plücker relations are of a particularly nice form. We simplify the notation in this case to $R_{\{i\},\{j,k,l\}}=:R_{i,j,k,l}\in\mathbb C[p_{I}\mid I\in\binom{[n]}{2}]$, where for $1\le i<j<k<l\le n$ we have
\[
R_{i,j,k,l}=p_{ij}p_{kl}-p_{ik}p_{jl}+p_{il}p_{jk}\in I_{2,n}.
\]
By setting $p_{ij}=-p_{ji}$ we see that it is enough to consider $R_{i,j,k,l}$ with $1\le i<j<k<l\le n$ as generators for the ideal $I_{2,n}$ as up to sign these are all relations.
We denote the polynomial ring $\mathbb C[p_{I}\mid I\in\binom{[n]}{2}]$ by $\mathbb C[p_{ij}]_{ij}$ for short, if it clear which $n$ we are considering. 
To distinguish between the polynomial generators $p_{ij}$ (also called \emph{Plücker variables}) and the Plücker coordinates in $\mathbb C[\Gr(2,\mathbb C^n)]=\mathbb C[p_{ij}]_{ij}/I_{2,n}$ we denote the Plücker coordinate by $\bar p_{ij}\in \mathbb C[\Gr(2,\mathbb C^n)]=A_{2,n}$.
When there is no risk of confusion we drop this distinction.

The Grassmannian $\Gr(k,\mathbb C^n)$ can be realized as a quotient of the algebraic group $SL_n$ over $\mathbb C$. Recall the basic notations from \S\ref{sec:pre rep theory}.

Consider $P_k\subset SL_n$ the parabolic subgroup of block upper traingular matrices with blocks of size $k\times k$ and $(n-k)\times (n-k)$ along the diagonal. 
Naturally, it contains $B$. 
Set $I_k:=[n-1]\setminus k$ and consider the subgroup $W_{I_k}:=\langle s_i\mid i\in I_k\rangle$ of $S_n$. 
We choose a representative $w_k$ in the coset of $w_0$ in the quotient $S_n/W_{I_k}$. 
Then by the identification $S_n=N_{SL_n}(T)/T$ (here $N_{SL_n}(T)$ is the normalizer of $T$ in $SL_n$) we have $P_k=\overline{Bw_kB}$.
The Grassmannian is then the quotient
\[
SL_n/P_k=\Gr(k,\mathbb C^n).
\]
Similarly to $R^+$, let $R^+_{k}=\{\beta\in R^+\mid w_k(\beta)<0\}$ be the set of positive roots for $SL_n/P_k$.
In fact, we have $R_k^+=\{\alpha_{i,j}\in R^+\mid i\le k\le j\}$.
We also have $\lie n_k^-=\langle f_\beta\mid \beta\in R_k^+\rangle\subset \lie n^-$ a Lie subalgebra and denote by $U_k^-\subset B^-$ the corresponding subgroup with $\Lie U_k^-=\lie n_k^-$. It consists of lower triangular matrices with $1$s on the diagonal and non-zero entries only in positions $(i,j)$ with $k\le i\le n$ and $1\le j\le k$. 

\begin{example}
For $\Gr(2,\mathbb C^4)$ we have $I_2=\{1,3\}\subset[3]$ and consider $S_n/\langle s_1,s_3 \rangle$. As representative of $w_0$ in the quotient we can chose $w_2$ with reduced expression $s_2s_1s_3s_2$. Then we compute $R_2^+=\{\alpha_2,\alpha_{1,2},\alpha_{2,3},\alpha_{1,3}\}=\{\epsilon_i-\epsilon_j\vert 1\le i\le 2<j\le n\}$. The corresponding subgroups of $SL_4$ are
\[
P_2=\left\{\left(
\begin{smallmatrix}
x_{1,1} & x_{1,2} & x_{1,3} & x_{1,4} \\
x_{2,1} & x_{2,2} & x_{2,3} & x_{2,4} \\
0 & 0 & x_{3,3} & x_{3,4} \\
0 & 0 & x_{4,3} & x_{4,4} \\
\end{smallmatrix}
\right)
\right\} \text{ and }
U_2^-=\left\{\left(
\begin{smallmatrix}
1 & 0 & 0 & 0 \\
0 & 1 & 0 & 0 \\
x_{3,1} & x_{3,1} & 1 & 0 \\
x_{4,1} & x_{4,2} & 0 & 1 \\
\end{smallmatrix}
\right)
\right\}.
\]
\end{example}

Note that $U_k^-$ is open and dense in $\Gr(k,\mathbb C^n)$ and we have an isomorphism of fields of rational functions $\mathbb C(\Gr(k,\mathbb C^n))\cong \mathbb C(U_k^-)$.
We see in \S\ref{sec:pre flag} that $\mathbb C[SL_n/B]=\bigoplus_{r\ge 1}V(r\lambda)^*$ for every $\lambda\in\Lambda^{++}$.
Having $\Gr(k,\mathbb C^n)=SL_n/P_k$ similary we have for the homogeneous coordinate ring of the Grassmannian
\begin{equation}\label{eq:hom coord ring Gr}
    \mathbb C[\Gr(k,\mathbb C^n)]=\bigoplus_{r\ge 1}V(r\omega_k)^*,
\end{equation}
where $\omega_k\in\Lambda^+$ is the $k$th fundamental weight (see \S\ref{sec:pre rep theory}).

\subsection{The tropical Grassmannian}\label{sec:pre_trop}\label{grassmann}

In this section we recall results on the tropical Grassmannian due to Speyer and Sturmfels in \cite{SS04} and \cite[\S4.3]{M-S}. For computations in small cases we rely on \emph{Macaulay2}\cite{M2} and \emph{gfan}\cite{Gfan}. 

\begin{definition}
The \emph{tropical Grassmannian}, denoted $\trop(\Gr(k,\mathbb C^n))\subset \mathbb R^{\tbinom{n}{k}}$ is the tropical variety of the Plücker ideal $I_{k,n}$.
By \cite[Corollary~3.1]{SS04} it is a $k(n-k)+1$-dimensional polyhedral fan whose maximal cones are all of this dimension.
\end{definition}

By what we have seen in \S\ref{sec:pre trop} $\trop(\Gr(k,\mathbb C^n))$ is the subfan of the Gr\"obner fan of $I_{k,n}$ consisting of those ${\bf w}$, such that $\init_{\bf w}(I_{k,n})$ is monomial-free. 
Recall that for a fixed cone $C$ of $\trop(\Gr(k,\mathbb C^n))$ each two points ${\bf v,w}$ in its relative interior yield the same initial ideal, i.e. $\init_{\bf w}(I_{k,n})=\init_{\bf v}(I_{k,n})$ and we use the notation $\init_C(I_{k,n})$.
Recall that a maximal cone $C$ of $\trop(\Gr(k,\mathbb C^n))$ by definition is prime, if $\init_C(I_{k,n})$ is a prime ideal. 

We mainly focus on the tropicalization of $\Gr(2,\mathbb C^n)$ which has a very nice properties.

\begin{corollary*}(\cite[Corollary~4.4]{SS04})\label{cor: SS all cones prime}
Every initial ideal $\init_C(I_{2,n})$ associated to a maximal cone $C$ in $\trop(\Gr(2,\mathbb C^n))$ is prime.
\end{corollary*}

Recall that $\trop(V(I))\subset\mathbb R^n$ for $I$ a homogeneous ideal in $\mathbb C[x_1,\dots,x_n]$ contains a linear subspace $L_I$ called \emph{lineality space}. The elements $l\in L_I$ have the property that $\init_l(I)=I$. In particular, $\mathbb R(1,\dots,1)\subset\mathbb R^n$ is contained in $L_I$.

\begin{theorem*}(\cite[Theorem~3.4]{SS04})\label{thm: SS space of trees is trop Gr} 
The quotient $\trop(\Gr(2,\mathbb C^n))/L_{I_{2,n}}\subset \mathbb R^{\binom{n}{2}}/\mathbb R^{n-3}$ intersected with the unit sphere is, up to sign, the \emph{space of phylogenetic trees} \cite{BHV01}. 
\end{theorem*}

We explain the implications of the theorem in more detail. In particular, it implies that every maximal prime cone $C$ can be associated with a \emph{labelled trivalent tree} with $n$ leaves. The set of all labels trivalent trees with $n$ leaves is denoted by $\mathcal T_n$.
A \emph{trivalent tree} is a graph with internal vertices of valency three and no loops or cycles of any kind. 
Non-internal vertices are called \emph{leaves} and the word \emph{labelled} refers to labelling the leaves by $1,\dots, n$.
We call an edge \emph{internal}, if it connects two internal vertices.

We label the standard basis of $\mathbb R^{\tbinom{n}{2}}$ by pairs $(i,j)$ with $1\le i<j\le n$ corresponding to Plücker coordinates. The following definition shows how we can get a point in the relative interior of a maximal cone in $\trop(\Gr(2,\mathbb C^n))$ from a labelled trivalent tree. It follows from \cite[Theorem~3.4]{SS04}.

\begin{definition}\label{def:treedeg}
Let $T$ be a labelled trivalent tree with $n$ leaves.  Then the $(i,j)$'th entry of the weight vector ${\bf w}_T\in \trop(\Gr(2,\mathbb C^n))$ is 
\[
-\#\{\text{internal edges on path from leaf } i \text{ to leaf } j \text{ in }T\}.
\]
For notational convenience we set
$\init_T(I_{2,n}):=\init_{{\bf w}_T}(I_{2,n})$. The corresponding maximal cone in $\trop(\Gr(2,\mathbb C^n))$ is denoted $C_T$. 
\end{definition}

Later in \S\ref{sec:BFFHL} we refer to the entries of $-{\bf w}_T$ (note the sign change) as \emph{tree degrees}, we denote $\deg_T p_{i,j}=(-{\bf w}_T)_{(i,j)}$.
Combining the above, we conclude that every trivalent labelled tree induces a toric degeneration of $\Gr(2,\mathbb C^n)$ with flat family given as in \eqref{eq: groebner family}.

The symmetric group $S_n$ acts on $\mathcal T_n$ by permuting the labels of the leaves of trees.
We also have a $S_n$-action on Plücker coordinates given by
\[
\sigma(p_{ij})={\text{sgn}(\sigma)}p_{\sigma^{-1}(i),\sigma^{-1}(j)} \text{ for }\sigma\in S_n.
\]
This action induces a ring automorphism of $\mathbb C[p_{i,j}]_{ij}$ for every $\sigma\in S_n$ that sends $\init_T(I_{2,n})$ to $\init_{\sigma(T)}(I_{2,n})$ for every trivalent labelled tree $T$.
Denote by $\mathtt T$ the equivalence class of $T\in \mathcal T_n$. It is uniquely determined by the underlying \emph{(unlabelled) trivalent tree} with $n$ leaves, see for example Figure~\ref{fig:gr(2,4)}. We denote the set of trivalent tree by $\mathcal T_n/S_n$

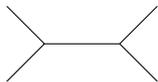
\begin{figure}[ht]
\begin{center}
\begin{tikzpicture}[scale=.5]

\draw (0,0) -- (1,1) -- (0,2);
\draw (1,1) -- (3,1) -- (4,2);
\draw (3,1) -- (4,0);

\end{tikzpicture}
\end{center}\caption{A trivalent tree with $4$ leaves.}\label{fig:gr(2,4)}
\end{figure}

Consider a trivalent tree $\mathtt T\in\mathcal T_n/S_n$. If there are two non-internal edges connected to the same internal vertex $c$, then we say $\mathtt T$ has a \emph{cherry} at vertex $c$.

\begin{lemma}\label{lem:cherry}
Every trivalent tree with $n\ge4$ leaves has a cherry.
\end{lemma}

\begin{proof}
We use induction on $n$. For $n=4$ Figure~\ref{fig:gr(2,4)} displays the only trivalent tree in $\mathcal T_4/S_4$ and we see, it has two cherries. Now consider a trivalent tree $\mathtt T'\in\mathcal T_{n+1}/S_{n+1}$. We remove one edge connected to a leaf and obtain a tree $\mathtt T\in\mathcal T_n/S_n$. By induction, $\mathtt T$ has a cherry at some vertex $c$. Adding the removed edge back there are two possibilities: either we add it to an internal edge, then the cherry also exists in $\mathtt T'$. Or we add it at an edge with a leaf, hence create a new cherry. 
\end{proof}

\subsection{Cluster structure on \texorpdfstring{$\mathbb C[\Gr(2,\mathbb C^n)]$}{}}\label{subsec: cluster gr2n}

We have seen in Example~\ref{exp: cluster Gr(2,4)} the cluster structure on $\mathbb C[\Gr(2,\mathbb C^4)]$. In this subsection we want to recall the cluster structure on $\mathbb C[\Gr(2,\mathbb C^n)]$ following \cite{FZ02} and \cite{Sco06}.

Let $D_{n}$ be a disk with $n$ marked points on its boundary $\partial D_n$ labelled by $[n]$ in counterclockwise order. We define an \emph{arc} in $D_n$ as a line connecting two marked points. A \emph{triangulation} $\Delta$ of $D_n$ is a maximal collection of non-crossing arcs. We call arcs that intersect $D_n^\circ:=D_n\setminus\partial D_n$ \emph{internal arc} and those along $\partial D_n$ \emph{boundary arc}.
Note that every triangulation consists of $n$ boundary arcs and $n-3$ internal arcs.
A collection of three arcs $\{d_1,d_2,d_3\}$ in $\Delta$ is a \emph{triangle} if pairwise they have one adjacent marked point in common.

\begin{algorithm}[h]
\SetAlgorithmName{Algorithm}{} 
\KwIn{\medskip {\bf Input:\ }  A triangulation $\Delta$ of $D_n$.}
\BlankLine
\For{every internal arc $d$ in $\Delta$}{create a mutable vertex $v_d\in Q_0$;}
\For{every boundary arc $b$ in $\Delta$}{create a frozen vertex $v_b\in Q_0$;}  
\For{every triangle $\{d_1,d_2,d_3\}$ in $\Delta$}{draw three arrows in $Q_1$ between the vertices $v_{d_1},v_{d_2},v_{d_3}\in Q_0$ creating a counterclockwise oriented $3$-cycle in $Q$;}
\BlankLine
{\bf Output:\ } The quiver $Q_{\Delta}:=(Q_0,Q_1)$. 
\label{alg:quiver assoc to triang}
\caption{Associating a quiver with a triangulation of $D_n$.}
\end{algorithm}

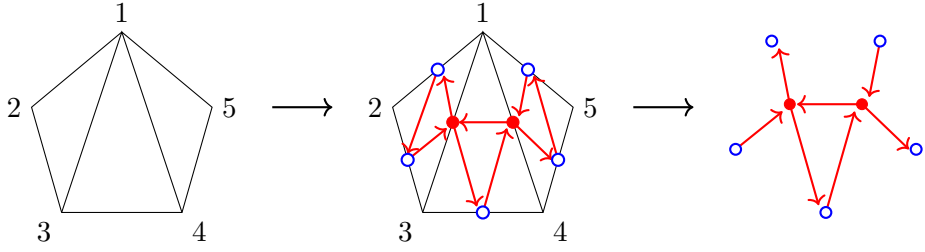
\begin{figure}
    \centering
\begin{tikzpicture}[scale=.8]
\draw (0,0) -- (2,0) -- (2.5,1.75) -- (1,3) -- (-.5,1.75) -- (0,0);
\draw (0,0) -- (1,3) -- (2,0);

\node[above] at (1,3) {1};
\node[left] at (-.5,1.75) {2};
\node[right] at (2.5,1.75) {5};
\node[below right] at (2,0) {4};
\node[below left] at (0,0) {3};

\draw[thick,->] (3.5,1.75)-- (4.5,1.75);

\begin{scope}[xshift=6cm]
\draw (0,0) -- (2,0) -- (2.5,1.75) -- (1,3) -- (-.5,1.75) -- (0,0);
\draw (0,0) -- (1,3) -- (2,0);

\node[above] at (1,3) {1};
\node[left] at (-.5,1.75) {2};
\node[right] at (2.5,1.75) {5};
\node[below right] at (2,0) {4};
\node[below left] at (0,0) {3};

\draw[->, thick, red] (1,0) -- (1.4,1.4); 
\draw[->, thick, red] (1.5,1.5) -- (.6,1.5); 
\draw[->, thick, red] (.5,1.5) -- (.9,.1); 

\draw[thick, red,->] (1.5,1.5) -- (2.15,.875); 
\draw[thick, red,->] (2.25,.875) -- (1.85,2.275); 
\draw[thick, red,->] (1.75,2.375) -- (1.6,1.6); 

\draw[thick, red,->] (.25,2.375) -- (-.25,.975); 
\draw[thick, red,->] (-.25,.875) -- (.4,1.4); 
\draw[thick, red,->] (.5,1.5) -- (.35,2.275); 

\draw[fill, white] (1,0) circle [radius=.1]; 
\draw[thick, blue] (1,0) circle [radius=.1];
\draw[fill, white] (2.25,.875) circle [radius=.1]; 
\draw[thick, blue] (2.25,.875) circle [radius=.1];
\draw[fill, white] (-.25,.875) circle [radius=.1]; 
\draw[thick, blue] (-.25,.875) circle [radius=.1];
\draw[fill, white] (1.75,2.375) circle [radius=.1]; 
\draw[thick, blue] (1.75,2.375) circle [radius=.1];
\draw[fill, white] (.25,2.375) circle [radius=.1]; 
\draw[thick, blue] (.25,2.375) circle [radius=.1];
\draw[fill, red] (.5,1.5) circle [radius=.1]; 
\draw[fill, red] (1.5,1.5) circle [radius=.1]; 

\draw[thick,->] (3.5,1.75)-- (4.5,1.75);

\begin{scope}[xshift=5.5cm, scale=1.2]
\draw[->, thick,red] (1,0) -- (1.4,1.4); 
\draw[->, thick,red] (1.5,1.5) -- (.6,1.5); 
\draw[->, thick,red] (.5,1.5) -- (.9,.1); 

\draw[thick,->,red] (1.5,1.5) -- (2.15,.875); 
\draw[thick,->,red] (1.75,2.375) -- (1.6,1.6); 

\draw[thick,->,red] (-.25,.875) -- (.4,1.4); 
\draw[thick,->,red] (.5,1.5) -- (.35,2.275); 

\draw[fill, white] (1,0) circle [radius=.075]; 
\draw[thick, blue] (1,0) circle [radius=.075];
\draw[fill, white] (2.25,.875) circle [radius=.075]; 
\draw[thick, blue] (2.25,.875) circle [radius=.075];
\draw[fill, white] (-.25,.875) circle [radius=.075]; 
\draw[thick, blue] (-.25,.875) circle [radius=.075];
\draw[fill, white] (1.75,2.375) circle [radius=.075]; 
\draw[thick, blue] (1.75,2.375) circle [radius=.075];
\draw[fill, white] (.25,2.375) circle [radius=.075]; 
\draw[thick, blue] (.25,2.375) circle [radius=.075];
\draw[fill,red] (.5,1.5) circle [radius=.075]; 
\draw[fill,red] (1.5,1.5) circle [radius=.075]; 
\end{scope}

\end{scope}
\end{tikzpicture}
    \caption{Visualizing Algorithm~\ref{alg:quiver assoc to triang} for a triangulation of $D_5$.}
    \label{fig: quiver from triang}
\end{figure}

\begin{definition}
To a triangulation $\Delta$ of $D_n$ we associate the quiver $Q_\Delta$ that is the output of Algorithm~\ref{alg:quiver assoc to triang} and set ${\bf A}_\Delta:=(A_{1,\Delta},\dots,A_{n-3,\Delta},A_{n-2},\dots,A_{2n-3})$. 
Then $\Delta$ determines the cluster algebra $\mathcal Y_\Delta:=\mathcal{Y}({\bf A}_\Delta,Q_\Delta)$.
\end{definition}

Given a triangulation $\Delta$ of $D_n$ we create a new triangulation $\Delta'$ by flipping a diagonal. 
More precisely, consider two adjacent triangles in $\Delta$ forming a quadrilateral with vertices the marked point $i,j,k,l$ in circular order along $\partial D_n$ and diagonal $d=[i,k]$. 
Then \emph{flipping} $d$ refers to replacing it with $d'=[j,l]$ (see Figure~\ref{fig: flip}).
The outcome is a new triangulation $\Delta'$ which only differs from $\Delta$ by $d$.
Given this definition the next proposition has a straightforward proof.

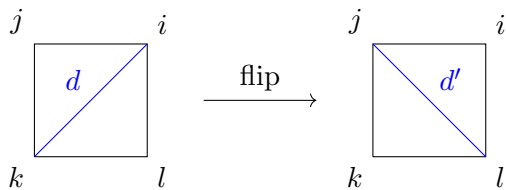
\begin{figure}
    \centering
\begin{tikzpicture}[scale=1.5]
\draw (0,0) -- (1,0) -- (1,1) -- (0,1) -- (0,0);
\draw[blue] (0,0) -- (1,1);
\node[above right] at (1,1) {$i$};
\node[above left] at (0,1) {$j$};
\node[below left] at (0,0) {$k$};
\node[below right] at (1,0) {$l$};
\node[above left,blue] at (.5,.5) {$d$};

\draw[->] (1.5,.5)--(2.5,.5);
\node[above] at (2,.5) {flip};

\begin{scope}[xshift=3cm]
\draw (0,0) -- (1,0) -- (1,1) -- (0,1) -- (0,0);
\draw[blue] (0,1) -- (1,0);
\node[above right] at (1,1) {$i$};
\node[above left] at (0,1) {$j$};
\node[below left] at (0,0) {$k$};
\node[below right] at (1,0) {$l$};
\node[above right,blue] at (.5,.5) {$d'$};
\end{scope}
\end{tikzpicture}
    \caption{Flipping an internal arc.}
    \label{fig: flip}
\end{figure}

\begin{proposition}
Let $\Delta$ and $\Delta'$ be two triangulations of $D_n$ related to each other by flipping the (internal) arc $d\in\Delta$. Then the quivers $Q_\Delta$ and $Q_{\Delta'}$ are related to each other by quiver mutation (see Definition~\ref{def:quiver mutation}).
Moreover, the cluster algebras $\mathcal Y_\Delta$ and $\mathcal{Y}_\Delta'$ are isomoprhic.
\end{proposition}

The main result is then the following.

\begin{proposition*}(\cite{FZ02},\cite[Proposition~2]{Sco06})
For $n\ge 5$ the homogeneous coordinate ring $\mathbb C[\Gr(2,\mathbb C^n)]$ is isomorphic to $\mathcal{Y}_\Delta\otimes_{\mathbb Z}\mathbb C$ for any triangulation $\Delta$ of $D_n$.
\end{proposition*}

Let $\Delta$ be a triangulation of $D_n$ with extended cluster ${\bf A}_\Delta$.
Then the cluster variables $A_{1,\Delta},\dots A_{n-3,\Delta}$ correspond to internal arcs of $\Delta$, each connecting two marked points. 
If $A_{k,\Delta}$ corresponds to the arc connecting $i$ and $j$, the isomorphism in \cite[Proposition~2]{Sco06} identifies $A_{k,\Delta}$ with the Plücker coordinate $\bar p_{ij}\in\mathbb C[\Gr(2,\mathbb C^n)]$.
The frozen variables $A_{n-2},\dots,A_{2n-3}$ correspond to arcs connecting successive marked points $i$ and $i+1  \mod n$. 
They are identified with the corresponding Plücker coordinates $\bar p_{i,i+1}$.

\newpage
\section{Birational sequences for Grassmannians and \texorpdfstring{$\trop(\Gr(2,\mathbb C^{n}))$}{}}\label{sec:Bos}

\noindent
We study birational sequences due to Fang, Fourier, and Littelmann \cite{FFL15} for Grassmannians and introduce the class of iterated birational sequences. We show that toric degenerations of $\Gr(2,\mathbb C^{n})$ constructed using the tropical Grassmannian $\trop(\Gr(2,\mathbb C^{n}))$ due to Speyer and Sturmfels \cite{SS04} can also be obtained using (iterated) biratinal sequences.

\subsection{Birational sequences}\label{sec:pre_birat}

We start the section by recalling some results due to Fang, Fourier, and Littelmann in \cite{FFL15} regarding birational sequences and associated valuations. 
After proving Lemma~\ref{lem:birat},which is central in this section we define a new class of birational sequences caled \emph{iterated} in Definition~\ref{def:itseq}.
\smallskip

Consider a positive root $\beta\in R^+$, then the \emph{root subgroup} corresponding to $\beta$ is given by 
\[
U_{-\beta}:=\{\exp(zf_{\beta})\mid z\in \mathbb C\}\subset U^-.
\]

\begin{definition}(\cite{FFL15})\label{def:birat}
Let $S=(\beta_1,\dots,\beta_{k(n-k)})$ be a sequence of positive roots. Then $S$ is called a \emph{birational sequence} for $\Gr(k,\mathbb C^n)$ if the product map induced by multiplication is birational:
\[
U_{-\beta_1}\times\dots\times U_{-\beta_{k(n-k)}}\to U_k^-.
\]
\end{definition}

\begin{example}\label{exp: birat seq, PBW and string}
The following are two first (and motivating) examples of birational sequences that we encounter again later in \S\ref{sec:BLMM}.
\begin{enumerate}
    \item The product map $\pi:\prod_{\beta \in R_k^+} U_{-\beta}\to U_k^-$ is birational, which makes any sequence containing all roots in $R_k^+$ (in arbitrary order) a birational sequence called \emph{PBW-sequence} (see \cite[Example~1 and page 131]{FFL15}). We distinguish between PBW-sequences $S$ and $S'$ when the roots in both appear in different order.
    \item Another example is given by a reduced decomposition $\w_k=s_{i_1}\dots s_{i_{k(n-k)}}$ of $w_k$, a coset representative of $w_0$ in $S_n/W_{I_k}$ (see \S\ref{sec:pre-grass}). 
    Let $S=(\alpha_{i_1},\dots,\alpha_{i_{k(n-k)}})$ be the corresponding sequence of simple roots. Then $S$ is a birational sequence called \emph{the reduced decomposition case} (see \cite[Example~2]{FFL15}). 
\end{enumerate}
\end{example}
\noindent
The second example shows that repetitions of positive roots may occur in birational sequences. 
Our aim is to shed some light on sequences that are \emph{neither} PBW \emph{nor} associated to reduced decompositions for Grassmannians.
The following lemma allows us to construct such sequences for $\Gr(k,\mathbb C^{n+1})$ from sequences for $\Gr(k,\mathbb C^n)$. 

\begin{lemma}\label{lem:birat}
Let $S=(\beta_{1},\dots,\beta_{k(n-k)})$ be a birational sequence for $\Gr(k,\mathbb C^n)$. Then extending it to the left by $\alpha_{i_1,n},\dots, \alpha_{i_k,n}$ for distinct $i_1,\dots,i_k\le n$ yields the following birational sequence for $\Gr(k,\mathbb C^{n+1})$
\begin{eqnarray}\label{seq}
S'=(\alpha_{i_1,n},\dots,\alpha_{i_k,n},\beta_{1},\dots,\beta_{k(n-k)}).
\end{eqnarray}
\end{lemma}

\begin{proof}
The sequence $S'$ yields the product of root subgroups $\mathcal G'=U_{-\alpha_{i_1,n}}\times\dots\times U_{-\alpha_{i_k,n}}\times U_{-\beta_{1}}\times \dots \times U_{-\beta_{k(n-k)}}\subset SL_{n+1}$, where for $y_1,\dots,y_k,z_1,\dots,z_{k(n-1)}\in\mathbb C$ elements are of form
\begin{eqnarray*}
\exp(y_{1}f_{\alpha_{i_1,n}})\cdots \exp(y_{k}f_{\alpha_{i_k,n}})\exp(z_1f_{\beta_1})\cdots \exp(z_{k(n-k)}f_{\beta_{k(n-k)}})\\
=\begin{pmatrix}
1 & 0 &0 && \dots& & 0&0\\[-1.5ex]
a^1_2 & 1 & 0 && \dots& &0&0\\[-1.5ex]
a^1_3 & a^2_3 & 1 && \dots& &0&0\\[-1.5ex]
\vdots & \vdots & \vdots &\ddots& &&0&0\\[-1.5ex]
a^1_{k+1}&a^2_{k+1}&\dots &a_{k+1}^k&1&\dots& 0&0\\[-1.5ex]
\vdots & \vdots & &\vdots&&\ddots  &\vdots&\vdots\\[-1.5ex]
a^1_n & a^2_n & \dots&a^k_n &*&\dots &1&0\\[-1.5ex]
a^1_{n+1}&a^2_{n+1} &\dots &a^k_{n+1} &*&\dots &*&1
\end{pmatrix}\in \mathcal G'.
\end{eqnarray*}
Here $a^j_{n+1}=y_{1}a^j_{i_1}+\dots + y_{k}a^j_{i_k}$ for $1\le j\le k$. Denote the $i$-th row of a fixed element $A\in \mathcal G'$ by $a_i=(a^1_i,a^2_i,\dots,a_i^k,*,\dots,*)\in\mathbb C^{n+1}$. Set $a^j_j=1$ for $1\le j\le n$ and $a^j_i=0$ for $i<j$. The coefficient of $e_{j_1}\wedge \dots \wedge e_{j_k}$ in $A(e_1\wedge \dots \wedge e_k)$ is the minor
$\bar p_J(A)=\det \bigg( \begin{smallmatrix}
a_{j_1}\\[-1ex]
\vdots\\[-1ex]
a_{j_k}
\end{smallmatrix} \bigg)$ with $J=\{j_1,\dots,j_k\}\subset [n+1]$. Now assume $J'=\{j_1,\dots,j_{k-1},n+1\}$. As $a_{n+1}=y_{1}a_{i_1}+\dots +y_{k}a_{i_k}$ then
\[
p_{J'}=\det \begin{pmatrix}
a_{j_1}\\[-1.5ex]
\vdots\\[-1.5ex]
a_{j_{k-1}}\\[-1.5ex]
a_{n+1}
\end{pmatrix} =y_{1}\det \begin{pmatrix}
a_{j_1}\\[-1.5ex]
\vdots\\[-1.5ex]
a_{j_{k-1}}\\[-1.5ex]
a_{i_1}
\end{pmatrix} +\dots + y_{k}\det  \begin{pmatrix}
a_{j_1}\\[-1.5ex]
\vdots\\[-1.5ex]
a_{j_{k-1}}\\[-1.5ex]
a_{i_k}
\end{pmatrix}. 
\]

We define the map $\varphi':\mathbb C(\mathbb A^{k(n-k+1)})\to \mathbb C(\Gr(k,\mathbb C^{n+1}))\cong\mathbb C(U_k^-)$ as extension of the birational map induced by $S$ on the function fields, which we denote by $\varphi:\mathbb C(\mathbb A^{k(n-k)})\to \mathbb C(\Gr(k,\mathbb C^n))$. 
For $I=\{i_1,\dots,i_k\}\subset[n]$ and for $1\le j\le k$ we define
\[
\varphi'(y_{j}):=\frac{\bar p_{I\setminus \{i_j\} \cup\{n+1\}}}{\bar p_I}.
\]
In order to prove that $S'$ is birational it suffices to find a map $\psi':\mathbb C(\Gr(k,\mathbb C^{n+1}))\to \mathbb C(\mathbb A^{k(n-k+1)})$ that is inverse to $\varphi'$. 
Let $\psi:\mathbb C(\Gr(k,\mathbb C^n))\to \mathbb C(\mathbb A^{k(n-k)})$ be the inverse of $\varphi$. We define $\psi'$ to as the extension of $\psi$ given by
\[
\psi'(p_{J'})=y_{1}\psi(\bar p_{J'\setminus\{n+1\}\cup\{i_1\}})+\dots+y_{k}\psi(\bar p_{J'\setminus\{n+1\}\cup\{i_k\}}).
\]
A straightforward computation then reveals that $\psi'$ and $\varphi'$ are indeed inverse to each other.
Therefore $S'$ is a birational sequence for $\Gr(k,\mathbb C^{n+1})$.
\end{proof}

\begin{definition}\label{def:itseq}
For $k<n$ consider a birational sequence for $\Gr(k,\mathbb C^{k+1})$. Now extend it as in (\ref{seq}) to a birational sequence for $\Gr(k,\mathbb C^{k+2})$. Repeat this process until the outcome is a birational sequence for $\Gr(k,\mathbb C^n)$. Birational sequences of this form are called \emph{iterated}.
\end{definition}

We explain how to obtain a valuation from a fixed birational sequence $S=(\beta_1,\dots,\beta_{d})$ for $\Gr(k,\mathbb C^n)$ as constructed in \cite{FFL15}. 
Let $d:=k(n-k)$ and define the \emph{height function} $\height:R^+\to \mathbb Z_{\ge 0}$ by sending a positive root to the number of its simple summands, i.e. $\height(\alpha_{i,j})=j-i+1$. 
Then the \emph{height weighted function} $\Psi: \mathbb Z^{d}\to \mathbb Z$ is given by 
\[
\Psi(m_1,\dots,m_{d}):=\sum_{i=1}^{d}m_i\height(\beta_i).
\] 
Let $<_{lex}$ be the lexicographic order on $\mathbb Z^{d}$. 
Then we define the \emph{$\Psi$-weighted reverse lexicographic order} $\prec_\Psi$ on $\mathbb Z^d$  by setting for ${\bf m,m'}\in \mathbb Z^{d}$
\begin{eqnarray}\label{eq: def psi wt order}
{\bf m}\prec_{\Psi}{\bf m'} :\Leftrightarrow \Psi({\bf m})<\Psi({\bf m'}) \text{ or } \Psi({\bf m})=\Psi({\bf m'}) \text{ and } {\bf m}>_{lex}{\bf m'}.
\end{eqnarray}

\begin{definition}[\cite{FFL15}]\label{def:valseq}
Let $f=\sum a_{\bf u}x^{\bf u}$ with $\bu\in\mathbb Z^d_{\ge0}$ be a non-zero polynomial in $\mathbb C[x_1,\dots,x_{d}]$. 
The valuation $\val_S:\mathbb C[x_1,\dots,x_{d}]\setminus \{0\}\to(\mathbb Z_{\ge 0}^{d},\prec_{\Psi})$ associated to $S$ is defined as
\begin{eqnarray}\label{eq:valseq}
\val_S(f):=\min{}_{\prec_\Psi}\{{\bf u}\in\mathbb Z_{\ge 0}^{k(n-k)}\mid a_{\bf u}\ne 0\}.
\end{eqnarray}
We extend $\val_S$ to a valuation on $\mathbb C(x_1,\dots,x_{d})\setminus\{0\}$ by setting for $h=\frac{f}{g}$ a rational function $\val_S(h):=\val_S(f)-\val_S(g)$.
\end{definition}

Valuations of form \eqref{eq:valseq} are usually called \emph{lowest term valuations}.
As $S$ is a birational sequence, for every element in $\mathbb C(\Gr(k,\mathbb C^n))$ there exists a unique element $f\in\mathbb C(x_1,\dots,x_{d})$ associated to it by the isomorphism $\psi:\mathbb C(\mathbb A^{d})\to \mathbb C(\Gr(k,\mathbb C^n))$. 
Hence, we have a valuation on $\mathbb C(\Gr(k,\mathbb C^n))\setminus\{0\}$. 
Further, as $\mathbb C[\Gr(k,\mathbb C^n)]\setminus\{0\}\subset \mathbb C(\Gr(k,\mathbb C^n))\setminus\{0\}$ we can restrict to obtain
\[
\val_S:\mathbb C[\Gr(k,\mathbb C^n)]\setminus\{0\}\to (\mathbb Z^d_{\ge 0},\prec_\Psi).
\]
We denote as in \S\ref{sec:pre val} by $S(\mathbb C[\Gr(k,\mathbb C^n)],\val_S)$ the associated value semi-group and the associated graded algebra by $\gr_S(\mathbb C[\Gr(k,\mathbb C^n)])$.
For the images of Plücker coordinates $\bar p_J\in\mathbb C[\Gr(k,\mathbb C^n)]$ we chose as before the notation $\overline{p_J}\in\gr_S(\mathbb C[\Gr(k,\mathbb C^n)])$ for $J\in\binom{[n]}{k}$.

\medskip

We are interested in toric degenerations of $\Gr(k,\mathbb C^n)$ from the above defined valuations using the Rees algebra construction \eqref{eq: def Rees}.
We would like to apply Theorem~\cite{KK12} stated in \S\ref{sec:pre val} and therefore need to show that the value semi-group $S(\mathbb C[\Gr(k,\mathbb C^n)],\val_S)$ is finitely generated. 
The following representation theoretic point of view on the valuation $\val_S$ from \cite[\S8 and \S9]{FFL15} is useful to do so for $\Gr(2,\mathbb C^n)$ in \S\ref{sec:proof} below.

\smallskip

A birational sequence $S=(\beta_1,\dots,\beta_d)$ for $\Gr(k,\mathbb C^n)$ together with the total order $\prec_{\Psi}$ on $\mathbb Z^{d}$ induces a filtration on the universal enveloping algebra $U(\lie n_k^-)$ for $0\not=\bm\in \mathbb Z^{d}_{\ge 0}$ by
\begin{eqnarray}\label{eq: def ind filtr univ env}
U(\lie n_k^-)_{\preceq_{\Psi}\bm}:=\langle \mathbf f^{\bk}=f_{\beta_1}^{k_1}\cdots f_{\beta_d}^{k_d}\mid \bk\in\mathbb Z_{\ge 0},\bk\preceq_{\Psi}\bm \rangle.
\end{eqnarray}
We define similarly  $U(\lie n_k^-)_{\prec_{\Psi}\bm}$.
Recall that the highest weight module $V(\lambda)$ for $\lambda=r\omega_k$ with $r\ge1$ is cyclically generated by a highest weight vector $v_{\lambda}\in V(\lambda)$ over $U(\lie n_k^-)$.
We therefore have an induced filtration for $0\not=\bm\in \mathbb Z^{d}_{\ge 0}$ defined by 
\begin{eqnarray}\label{eq: def ind filtr rep}
V(\lambda)_{\preceq_{\Psi}}:=U(\lie n_k^-)_{\preceq_{\Psi}\bm}\cdot v_{\lambda},
\end{eqnarray}
Similarly we define $V(\lambda)_{\prec_{\Psi}}$. 
Then $V^{\gr}(\lambda):=\bigoplus_{0\not=\bm\in\mathbb Z^d_{\ge 0}}V(\lambda)_{\preceq_{\Psi}}/V(\lambda)_{\prec_{\Psi}}$ is the associated graded vector space.
This leads to the following definition of \emph{essential sets} (see \cite[Definition~7]{FFL15})
\begin{eqnarray}\label{eq: def ess set}
\es_S(\lambda):= \{\bm\in\mathbb Z^{d}_{\ge 0}\mid V(\lambda)_{\preceq_{\Psi}}/V(\lambda)_{\prec_{\Psi}}\not=0\}.
\end{eqnarray}
These sets are of particular importance as $\{\bff^{\bm}\cdot v_\lambda\mid \bm\in\es_S(\lambda)\}$ forms a basis for $V(\lambda)$ and hence $\bigcup_{r\ge 1}\{\bff^{\bm}\cdot v_{r\omega_k} \mid\bm \in\es_S(r\omega_k)\}$ a basis for $\mathbb C[\Gr(k,\mathbb C^n)]$.

The connection between the valuation $\val_S$ introduced above and the essential sets is the following.

\begin{proposition*}(\cite[Proposition~2]{FFL15})
For every birational sequence $S$ for $\Gr(k,\mathbb C^n)$ we have $\bigcup_{r\ge 1}\es_S(r\omega_k)=S(A_{k,n},\val_S)$. 
\end{proposition*}

The proposition (resp. its proof) implies that $\es_S(\omega_k)=\{\val_S(\bar p_{J})\mid J\in\binom{[n]}{k}\}$ by the counting argument 
\[
\vert\es_S(\omega_k)\vert = \dim_{\mathbb C}V(\omega_k)=\binom{n}{k}= 
\left\vert
    \left\{\val_S(\bar p_J)
        \left\vert J\in\tiny{\binom{[n]}{k}}
        \right.
    \right\}
\right\vert.
\]
Consider for $r\ge 1$ the set $r\es_S(\lambda):=\{\sum_{j=1}^r\bm_j\mid \bm_j\in\es_S(\lambda)\forall j\}$. Then by construction we have
\[
r\es_S(\omega_k)\subseteq \es_S(r\omega_k).
\]
If equality holds for all $r\ge 1$ by \cite[Proposition~2]{FFL15} the value semi-group $S(\mathbb C[\Gr(k,\mathbb C^n)],\val_S)$ is generated by $\{\val_S(\bar p_{J})\mid J\in\binom{[n]}{k}\}$. Hence, the Plücker coordinates form a Khovanskii basis for $\val_S$ and we can apply Theorem~\cite{KK12} to get a toric degenration of $\Gr(k,\mathbb C^n)$. 
Our aim is to show that this is the case when $S$ is an iterated sequences for $\Gr(2,\mathbb C^n)$.

\subsection{Iterated sequences for \texorpdfstring{$\Gr(2,\mathbb C^n)$}{}}\label{sec:proof}
In this subsection we prove Theorem~\ref{thm:main birat} stated in the introduction. 
After proving Proposition~\ref{prop: init M_S= init C_S} it follows from Theorem~\ref{thm: val and quasi val with wt matrix} stated in \S\ref{sec:val and quasival}.
We focus on iterated sequences for $\Gr(2,\mathbb C^n)$ and start by making the above definitions precise.

Let $S=(\beta_1,\dots,\beta_d)$ be a birational sequence for $\Gr(2,\mathbb C^n)$.
With notation as in the previous subsection we have $I_k=I_2=[n-1]\setminus 2$ and $\ell(w_2)=2(n-2)=d$. 
For $\lie n^-_{2}=\Lie(U_2^-)$, by \cite[Lemma~2]{FFL15} $U(\lie n^-_2)$ is generated by monomials of form $f_{\beta_1}^{m_1}\dots f_{\beta_d}^{m_{d}}$.
We consider the irreducible highest weight representation $V(\omega_2)=\bigwedge^2\mathbb C^n$ of highest weight $\omega_2\in\Lambda^+$. 
It is cyclically generated over $U(\lie n^-_2)$ by a highest weight vector $v_{\omega_2}$, which we chose to be $e_1\wedge e_2$ as in Example~\ref{exp: fund reps 1 and 2}. 
The Pl\"ucker coordinate $\bar p_{ij}$ is the dual basis vector to $e_i\wedge e_j$ for $1\le i<j\le n$ in $(\bigwedge^2\mathbb C^n)^*$. 
There exists at least one monomial of form ${\bf f^m}=f_{\beta_1}^{m_1}\dots f_{\beta_{d}}^{m_{d}}$ with the property ${\bf f^m}(e_1\wedge e_2)=e_i\wedge e_j$ for all $i,j\in[n]$. 
Then by \cite[Proposition~2]{FFL15} we have
\begin{eqnarray}\label{eq: val seq on plucker}
\val_s(\bar p_{ij})=\min{}_{\prec_{\Psi}}\{{\bf m}\in \mathbb Z^{d}_{\ge 0}\mid {\bf f^m}(e_1\wedge e_2)=e_i\wedge e_j\}.
\end{eqnarray}

\begin{table}
\begin{center}
\begin{tabular}{ |c|c|c|}
\hline
Pl\"ucker & $\val_S$ & $\val_{S'}$\\
 \hline
$\bar p_{12}$ & $(0,0,0,0)$ & $(0,0,0,0)$\\[-1.5ex]
$\bar p_{13}$ & $(0,0,0,1)$ & $(0,0,0,1)$\\[-1.5ex]
$\bar p_{23}$ & $(0,0,1,0)$ & $(0,0,1,0)$\\[-1.5ex]
$\bar p_{14}$ & $(0,1,0,0)$ & $(1,0,0,1)$\\[-1.5ex]
$\bar p_{24}$ & $(1,0,0,0)$ & $(1,0,1,0)$\\[-1.5ex]
$\bar p_{34}$ & $(1,0,0,1)$ & $(0,1,1,0)$\\
\hline
\end{tabular}
\end{center}\caption{Images of Pl\"ucker coordinates under the valuations $\val_S,\val_{S'}$ associated to $S=(\alpha_{1,3},\alpha_{2,3},\alpha_{1,2},\alpha_2)$ and  $S'=(\alpha_3,\alpha_{2,3},\alpha_{1,2},\alpha_2)$ for $\Gr(2,\mathbb C^4)$.}\label{tab:val(2,4)}
\end{table}

\begin{example}\label{exp:seq}
Consider $\Gr(2,\mathbb C^4)$ with iterated sequences $S=(\alpha_{1,3},\alpha_{2,3},\alpha_{1,2},\alpha_2)$ and $S'=(\alpha_3,\alpha_{2,3},\alpha_{1,2},\alpha_2)$. 
They are birational by Lemma~\ref{lem:birat}, as $(\alpha_{1,2},\alpha_2)$ is of PBW type for $\Gr(2,\mathbb C^{3})$. 
We compute the valuation $\val_S$ on Pl\"ucker coordinates. 
There are two monomials sending $e_1\wedge e_2$ to $e_3\wedge e_4$, namely 
\[
{\bf f}^{(1,0,0,1)}\cdot e_1\wedge e_2={\bf f}^{(0,1,1,0)}\cdot e_1\wedge e_2=e_1\wedge e_4.
\]
We have $\Psi(1,0,0,1)=\Psi(0,1,1,0)=4$, but $(1,0,0,1)>_{lex}(0,1,1,0)$. Hence, $\val_S(\bar p_{34})=(1,0,0,1)$. 

For $\val_{S'}$ we compute ${\bf f}^{(1,0,0,1)}\cdots e_1\wedge e_2= {\bf f}^{(0,1,0,0)}\cdots e_1\wedge e_2= e_1\wedge e_4$
Again, we have $\Psi(1,0,0,1)=\Psi(0,1,0,0)=2$, but as $(1,0,0,1)>_{lex} (0,1,0,0)$ it follows $\val_{S'}(\bar p_{14})=(1,0,0,1)$.
In Table~\ref{tab:val(2,4)} you can find the images of all Pl\"ucker coordinates under $\val_S$ and $\val_{S'}$.
\end{example}

\begin{figure}
\begin{center}
\begin{tikzpicture}[scale=.25]
\node at (-5,-0.5) {$T_3=$};

\draw (0,0) -- (0,-3);
\draw (0,0) -- (3,2);
\draw (0,0) -- (-3,2);

\node at (0,-3.6) {\tiny $3$};
\node at (3.5,2.25) {\tiny $2$};
\node at (-3.5,2.25) {\tiny $1$};
\end{tikzpicture}
\end{center}
\caption{Labelled trivalent tree with three leaves.}\label{fig:3leaves}
\end{figure}
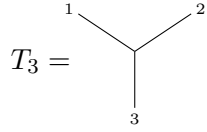

From now on we consider an iterated birational sequence $S=((i_n,n),(j_n,n),\dots,(i_3,3),(j_3,3))$ for $\Gr(2,\mathbb C^n)$, where $(i_k,k)$ represents the positive root $\alpha_{i_k,k-1}=\epsilon_{i_k}-\epsilon_k$.
We chose this notation as it easily encodes the action of $f_{\alpha_{i_k,k-1}}\in\lie n^-$ on $\mathbb C^n$ (see Example~\ref{exp: fund reps 1 and 2}), which we need to compute $\val_S$ on Plücker coordinates.
The following algorithm associates to $S$ a trivalent tree $T_S$ with $n$ leaves labelled by $[n]$.

\begin{algorithm}[h]
\SetAlgorithmName{Algorithm}{} 
\KwIn{\medskip {\bf Input:\ }  An iterated birational sequence $S=((i_n,n),(j_n,n),\dots,(i_3,3),(j_3,3))$, the trivalent tree $T_3$ as in Figure~\ref{fig:3leaves}.}
\BlankLine
{\bf Initialization:} Set $k=4$, $T^S_{3}:=T_3$.\\
\For{k}{Construct a tree $T_k^S$ from $T^S_{k-1}$ by replacing the edge with leaf $i_k$ in $T_{k-1}^S$ by three edges forming a cherry with leaves labelled by $i_k$ and $k$.\\
    \If{k=n}
    {{\bf Output:\ } The tree $T^S_n$.}
    \Else{Replace $k$ by $k+1$, $T^S_{k-1}$ by $T^S_k$ and start over.}}
\BlankLine
{\bf Output:\ } The tree $T_S:=T^S_n$ and the sequence $\mathbb T_S:=(T_n^S,\dots,T_3^S)$ of trees.
\label{alg:tree assoc to seq}
\caption{Associating a trivalent tree $T_S$ with an iterated sequence $S$.}
\end{algorithm}

\begin{definition}\label{def: T_S and C_S}
To an iterated sequence $S$ we associate the trivalent tree $T_S$ and the sequence of trees $\mathbb T_S=(T^S_n,\dots,T^S_3)$ that are the output of Algorithm~\ref{alg:tree assoc to seq}. 
Denote by $C_S$ the maximal cone in $\trop(\Gr(2,\mathbb C^n))$ corresponding to the tree $T_S$ by \cite[Theorem~3.4]{SS04} restated in \S\ref{sec:pre_trop}.
\end{definition}

\begin{example}\label{exp:treeseq}
Consider $S=((4,6),(5,6),(2,5),(3,5),(2,4),(3,4),(1,3),(2,3))$, an iterated sequence for $\Gr(2,\mathbb C^6)$ .
We construct the trees $\mathbb T_S=(T^S_{3},T_4^S,T_5^S,T^S_6)$ by Algorithm~\ref{alg:tree assoc to seq}. 
Figure~\ref{fig:treeseq} shows the obtained sequence of trees.

\begin{figure}[ht]
\begin{center}
\begin{tikzpicture}[scale=.25]
\draw (0,0) -- (0,-2.5);
\draw (0,0) -- (2.5,2);
\draw (0,0) -- (-2.5,2);
\node at (0,-3.1) {\tiny $1$};
\node at (-3,2.5) {\tiny $2$};
\node at (3,2.5) {\tiny $3$};
\draw[->,thick] (3.5,0) -- (6,0);

\begin{scope}[xshift=10cm] 
\draw (-3,2)-- (-1.5,0) -- (1.5,0) -- (3,2);
\draw (-3,-2) -- (-1.5,0);
\draw (1.5,0) -- (3,-2);
\node at (3.5,-2.5) {\tiny $1$};
\node at (-3.5,2.5) {\tiny $2$};
\node at (3.5,2.5) {\tiny $3$};
\node at (-3.5,-2.5) {\tiny $4$};
\draw[->,thick] (5,0) -- (7.5,0);

\begin{scope}[xshift=12.5cm]
\draw (-3,2)-- (-1.5,0) -- (4.5,0) -- (6,2);
\draw (-3,-2) -- (-1.5,0);
\draw (4.5,0) -- (6,-2);
\draw (1.5,0) -- (1.5,-2);
\node at (6.5,-2.5) {\tiny $1$};
\node at (-3.5,2.5) {\tiny $2$};
\node at (6.5,2.5) {\tiny $3$};
\node at (1.5,-2.6) {\tiny $4$};
\node at (-3.5,-2.5) {\tiny $5$};
\draw[->,thick] (8,0) -- (10.5,0);

\begin{scope}[xshift=15cm]
\draw (-3,2)-- (-1.5,0) -- (4.5,0) -- (6,2);
\draw (-3,-2) -- (-1.5,0);
\draw (4.5,0) -- (6,-2);
\draw (1.5,0) -- (1.5,-2) -- (0,-4);
\draw (1.5,-2) -- (3,-4);
\node at (6.5,-2.5) {\tiny $1$};
\node at (-3.5,2.5) {\tiny $2$};
\node at (6.5,2.5) {\tiny $3$};
\node at (-.5,-4.5) {\tiny $4$};
\node at (-3.5,-2.5) {\tiny $5$};
\node at (3.5,-4.5) {\tiny $6$};
\end{scope}
\end{scope}
\end{scope}
\end{tikzpicture}
\end{center}
\caption{The sequence $\mathbb T_S$ for $S$ as in Example~\protect{\ref{exp:treeseq}}.}\label{fig:treeseq}
\end{figure}
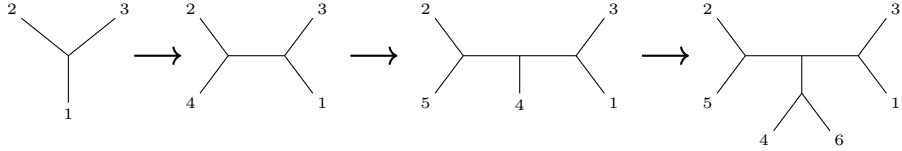
\end{example}

\begin{definition}\label{def: M_S}
We define the \emph{weighting matrix} $M_S\in\mathbb Z^{d\times \binom{n}{2}}$ associated to $S$ as the matrix whose columns are $\val_S(\bar p_{ij})$ for $\{i,j\}\in\binom{[n]}{2}$.
\end{definition}

Following \cite[\S3.1]{KM16} we want to compute $\init_{M_S}(I_{2,n})$ to apply Theorem~\ref{thm: val and quasi val with wt matrix}. Recall the Definition~\ref{def: init wrt M} from \S\ref{sec:val and quasival}.

\begin{proposition}\label{prop: init M_S= init C_S}
For every interated sequence $S$ we have $\init_{M_S}(I_{2,n})=\init_{C_S}(I_{2,n})$.
\end{proposition}
\begin{proof}
Recall that $\init_{C_S}(I_{2,n})=\langle \init_{C_S}(R_{i,j,k,l})\mid i,j,k,l\in[n] \rangle$ by \cite[Proof of Theorem~3.4]{SS04}. 
This implies that it is enough to prove the following claim.

\emph{Claim:} For every Plücker relation $R_{i,j,k,l}$ with $i,j,k,l\in[n]$ we have $\init_{C_S}(R_{i,j,k,l})=\init_{M_S}(R_{i,j,k,l})$.

Let $\{e_{ij}\}_{\{i,j\}\in\binom{[n]}{2}}$ be the stansard basis for $\mathbb R^{\binom{n}{2}}$. Adopting the notation for monomials in the polynomial ring $\mathbb C[p_{ij}]_{ij}$ we have
\[
R_{i,j,k,l}=p_{ij}p_{kl}-p_{ik}p_{jl}+p_{il}p_{jk}=\p^{e_{ij}+e_{kl}}-\p^{e_{ik}+e_{jl}}+\p^{e_{il}+e_{jk}}. 
\]
In particular, $M_S(e_{ij}+e_{kl})=\val_S(\bar p_{ij})+\val_S(\bar p_{kl})=\val_S(\bar p_{ij}\bar p_{kl})$ and $\init_{M_S}(R_{i,j,k,l})$ is the sum of those monomials in $R_{i,j,k,l}$ for which the valuation $\val_S$ of the corresponding monomials in $A_{2,n}$ is minimal with respect to $\prec_{\Psi}$.

\emph{Proof of claim:} We proceed by induction. For $n=4$ let $S=((i,4),(j,4),(i_3,3),(j_3,3))$, i.e. the tree $T_S$ has a cherry labelled by $i$ and $4$. 
Consider the Plücker relation $R_{i,j,k,4}=p_{ij}p_{k4}-p_{ik}p_{j4}+p_{i4}p_{jk}$ with $\{i,j,k\}=[3]$. Then
\[
\init_{C_S}(R_{i,j,k,4})=p_{ij}p_{k4}-p_{ik}p_{j4}.
\]
Let $S'=((i_3,3),(j_3,3))$ be the sequence for $\Gr(2,\mathbb C^{3})$ and denote by $\hat p_{rs}$ with $r,s\in[3]$ the Plücker coordinates in $A_3$. 
For $\bm\in\mathbb Z^{d-2}$ and $m_d,m_{d-1}\in\mathbb Z$ write $(m_d,m_{d-1},\bm):=(m_d,m_{d-1},m_{d-2},\dots,m_1)$
We compute
\[
\val_S(\bar p_{i4})=(0,1,\val_{S'}(\hat p_{ij})),\ \
\val_S(\bar p_{j4})=(1,0,\val_{S'}(\hat p_{ij})),\ \text{  and  } \ 
\val_S(\bar p_{k4})=(1,0,\val_{S'}(\hat p_{ik})).
\]
This implies $\val_S(\bar p_{i4}\bar p_{jk})\succ_{\Psi}\val_S(\bar p_{ij}\bar p_{k4})=\val_S(p_{ik}p_{j4})$, and hence
$\init_{M_S}(R_{i,j,k,4})=\init_{C_S}(R_{i,j,k,4}).$

Assume the claim is true for $n-1$ and let $S=((i_n,n),(j_n,n),\dots,(i_3,3),(j_3,3))$ be an iterated sequence for $\Gr(2,\mathbb C^n)$.  
Then $S'=((i_{n-1},n-1),(j_{n-1},n-1),\dots,(i_3,3),(j_3,3))$ is an iterated sequence for $\Gr(2,\mathbb C^{n-1})$.
Denote by $\hat p_{ij}$ with $i,j\in[n-1]$ the Plücker coordinates in $A_{n-1}$.
As $\val_S(\bar p_{ij})=(0,0,\val_{S'}(\hat p_{ij}))$ for $i,j<n$ we deduce $\init_{C_S}(R_{i,j,k,l})=\init_{M_S}(R_{i,j,k,l})$ with $i,j,k,l<n$ by induction.
Consider the Plücker relation $R_{i,j,k,n}$. Then
\[
\val_S(\bar p_{rn})=\left\{\begin{matrix}
(1,0,\val_{S'}(\hat p_{ri_n})), & \text{ if } r\not=i_n\\
(0,1,\val_{S'}(\hat p_{i_nj_n})),& \text{ if } r=i_n.
\end{matrix}\right.
\]
As $(i_n,n)$ is a cherry in $T_S$ we observe that the associated weight vector $\bw_{T_S}\in C_S^\circ\subset\trop(\Gr(2,\mathbb C^n))$ satisfies $(\bw_{T_S})_{rn}=(\bw_{T_S})_{ri_n}=(\bw_{T_{S'}})_{ri_n}-1$.
In particular, for $i,j,k\not=i_n$ we deduce by induction $\init_{M_S}(R_{i,j,k,n})=\init_{C_S}(R_{i,j,k,n})$.
The only relations left to consider are of form $R_{i_n,j,k,n}$  for $j,k\in[n-1]\setminus\{i_n\}$. For $M_S$ we compute by the above
\[
\val_S(\bar p_{i_nj}\bar p_{kn})=\val_S(\bar p_{i_nk}\bar p_{jn})\succ_{\Psi} \val_S(\bar p_{i_nn}\bar p_{jk}).
\]
Hence, $\init_{M_S}(R_{i_n,j,k,n})=p_{i_nj}p_{kn}-p_{i_nk}p_{jn}$. As $(i_n,n)$ is a cherry in $T_S$ we obtain $\init_{C_S}(R_{i_n,j,k,n})=\init_{M_S}(R_{i_n,j,k,n})$.
\end{proof}

As $\init_{C_S}(I_{2,n})$ is prime, Proposition~\ref{prop: init M_S= init C_S} allows us to apply Theorem~\ref{thm: val and quasi val with wt matrix} from \S\ref{sec:val and quasival}.
Let us have a look at the other necessary assumptions before we formulate the statements from \S\ref{sec:val and quasival} in the context of valuations from iterated sequences for $\Gr(2,\mathbb C^n)$ below. 
For completeness we also include the proofs in this case, although this would be not necessary given that we can apply the general theorem. 
They just serve as an example to obtain a better understanding of the general theory.

We consider the algebra $A_{2,n}$, the homogeneous coordinate ring of $\Gr(2,\mathbb C^n)$.
We have fixed the Plücker embedding, that yields a presentation $\pi: \mathbb C[p_{ij}]_{ij}\to A_{2,n}$ with $A_{2,n}=\mathbb C[p_{ij}]_{ij}/\ker(\pi)$.
More precisely, $\ker(\pi)=I_{2,n}$ is the Plücker ideal. The candidate for a Khovanksii basis is therefore $\{\pi(p_{ij})=\bar p_{ij}\}_{ij}\subset A_{2,n}$.
As $\mathbb C[p_{ij}]_{ij}$ is postivily graded by $\mathbb Z_{\ge 0}$ (we have $\deg p_{ij}=1$) and $I_{2,n}$ is homogeneous with respect to this grading generated by Plücker relations of degree 2, by Lemma~\ref{lem: val vs val_Mval} we have $\val_S(\bar p_{ij})=\val_{M_S}(\bar p_{ij})$. 
Here $\val_S:A_{2,n}\setminus\{0\}\to \mathbb Z^d$ is the valuation induced by the iterated sequence $S$ and $\val_{M_S}:A_{2,n}\setminus\{0\}\to \mathbb Z^d$ the (quasi-)valuation defined by the weighting matrix $M_S$ of $\val_S$ as above (see Definition~\ref{def: quasi val from wt matrix}).
 
Inspired by the proof of \cite[Proposition~5.2]{KM16} we obtain the next proposition.
The proof is analogous to the one of Corollary~\ref{cor: NO body val_M} in \S\ref{sec:val and quasival}.

\begin{proposition}\label{prop: ass gr M_S}
For every iterated sequence $S$ the associated quasi-valuation $\val_{M_S}$ with weighting matrix $M_S$ satisfies $\gr_{M_S}(A_{2,n})\cong \mathbb C[p_{ij}]_{ij}/\init_{C_S}(I_{2,n})$.
Moreover, $\val_{M_S}$ is a valuation with value semi-group $S(A_{2,n},\val_{M_S})$ generated by $\val_{M_S}(\bar p_{ij})$ for $1\le i<j\le n$.
\end{proposition}
\begin{proof}
As $I_{2,n}$ is homogeneous with respect to a positive grading, we have $M_S\in\text{GR}^d(I_{2,n})$ and hence, can apply \cite[Lemma~4.4]{KM16} to get $\gr_{M_S}(A_{2,n})\cong \mathbb C[p_{ij}]_{ij}/\init_{M_S}(I_{2,n})$. 
Then the first part of the claim follows from Proposition~\ref{prop: init M_S= init C_S}.

As $\init_{C_S}(I_{2,n})$ is prime, $\gr_{M_S}(A_{2,n})\cong \mathbb C[p_{ij}]_{ij}/\init_{C_S}(I_{2,n})$ is a domain.
The rest of the proof is exactly the same as the proof of Corollary~\ref{cor: NO body val_M} in \S\ref{sec:val and quasival}.
\end{proof}

\begin{corollary}\label{cor: NO body M_S}
For every iterated sequence $S$ the Newton-Okounkov body associated with the weight valuation $\val_{M_S}$ is given by
\[
\Delta(A_{2,n},\val_{M_S})=\conv(\val_{S}(\bar p_{ij})\mid 1\le i<j\le n).
\]
\end{corollary}
\begin{proof}
By Proposition~\ref{prop: ass gr M_S} we have $\Delta(A_{2,n},\val_{M_S})=\conv(\val_{M_S}(\bar p_{ij})\mid 1\le i<j\le n)$. Therefore it remains to show $\val_{M_S}(\bar p_{ij}))=\val_S(\bar p_{ij})$ for all $i,j\in[n]$. This follows exactly by the argument in the proof of Lemma~\ref{lem: val vs val_Mval}.
\end{proof}

\begin{theorem}\label{thm:main birat}
For every iterated sequence $S$ we have $\gr_S(A_{2,n})\cong \mathbb C[p_{ij}]_{ij}/\init_{C_S}(I_{2,n})$.
Moreover, for every maximal prime cone $C$ of $\trop(\Gr(2,\mathbb C^n))$ there exists a birational sequence $S$, such that $\mathbb C[p_{ij}]_{ij}/\init_C(I_{2,n})\cong \gr_S(A_{2,n})$.
\end{theorem}
\begin{proof}
The first part of the claim follows from Theorem~\ref{thm: val and quasi val with wt matrix}. We give an alternative proof here, using the essential sets to illustrate another point of view on the general theorem.

We show that $S(A_{2,n},\val_S)$ is generated by $\val_S(p_{ij})$ for $1\le i<j\le n$. This implies $S(A_{2,n},\val_S)=S(A_{2,n},\val_{M_S})$. As $\val_S$ and $\val_{M_S}$ are full rank and hence have one-dimensional leaves by \cite[Remark~4.13]{BG09} (see also \S\ref{sec:pre val}) we have $\gr_S(A_{2,n})\cong \mathbb C [S(A_{2,n},\val_S)]$ and $\gr_{M_s}(A_{2,n})\cong \mathbb C [S(A_{2,n},\val_{M_S})]$. Therefore,  $\gr_{M_s}(A_{2,n})\cong \gr_{M_S}(A_{2,n})$. Then the first claim follows by Proposition~\ref{prop: ass gr M_S}.
In order to do so, we use \cite[Proposition~2]{FFL15} restated above and show $\es_S(k\omega_2)=k\es_S(\omega_2)$ for all $k\ge 1$.

As $\Delta(A_{2,n},\val_{M_S})$ is integral by Corollary~\ref{cor: NO body M_S}, all lattice points in the $k$th dilation $k\Delta(A_{2,n},\val_{M_S})$ are sums of $k$ lattice points in $\Delta(A_{2,n},\val_{M_S})$.
We have $\es_S(\omega_2)=\{\val_S(\bar p_{ij})\mid 1\le i<j\le n\}$ by \cite[Proposition~2]{FFL15}, which are the lattice points in $\Delta(A_{2,n},\val_{M_S})$.
Then by Corollary~\ref{cor: NO body M_S} for $f\in A_{2,n}$ with $\deg f=k\ge1$ there exist $\bm_1,\dots,\bm_k\in\es_S(\omega_2)$ such that $\val_{M_S}(f)=\sum_{j=1}^k \bm_j$.
In particular, $\val_{M_S}(f)\in k\es_S(\omega_2)$.
Hence, we count
\[
\vert k\es_S(\omega_2)\vert =\dim_{\mathbb C} V(k\omega_2)=\vert \es_S(k\omega_2)\vert
\]
and the claim follows.

For the second part, note that by Algorithm~\ref{alg:tree assoc to seq} for every shape of tree we can find an iterated sequence, such that the output has the desired shape (ignoring the labelling for now).
Therefore, for a given maximal prime cone $C\subset \trop(\Gr(2,\mathbb C^n))$ consider the corresponding tree $T_C$ and its shape $\mathtt T_C$.
Then find S with $T_S$ of shape $\mathtt T_C$ (see also Corollary~\ref{cor:treegraph} below).
The action of the symmetric group induces an isomorphism $\mathbb C[p_{ij}]_{ij}/\init_{T_S}(I_{2,n})\cong \mathbb C[p_{ij}]_{ij}/\init_{T_C}(I_{2,n})$.
The rest follows then by the first part.
\end{proof}

\begin{remark}
Note that the essential basis for $A_{2,n}$ (see \cite[Remark~5]{FFL15}) induced by $\bigcup_{k\ge 1}\es(k\omega_2)$ is an adapted basis for the valuation $\val_S$ and therefore also for $\val_{M_S}$. Having the notion of essential sets in this context allowed us to use this (more concrete) basis instead of the (more abstract) standard monomial basis for $\val_{M_S}$ (that exists as $M_S$ lies in the Gröbner region) used in the proof of Theorem~\ref{thm: val and quasi val with wt matrix}.
\end{remark}



For an iterated sequence $S$ for $\Gr(2,\mathbb C^n)$ denote by $\mathtt T_i^S$ the (non-labelled) trivalent tree underlying the labelled trivalent tree $T^S_i$ with $i$ leaves in the tree sequence $\mathbb T_S$.
The Algorithm~\ref{alg:tree assoc to seq} provides a tool for comparing whether two iterated sequences induce isomorphic flat toric degenerations. 
Construct $\mathbb T_{S_1},\mathbb T_{S_2}$ for two such sequences $S_1,S_2$ and consider $\mathtt T^{S_1}_n$ and $\mathtt T^{S_2}_n$. If $\mathtt T^{S_1}_n$ and $\mathtt T^{S_2}_n$ coincide then 
\[
\init_{T_{n}^{S_1}}(I_{2,n})\cong\init_{T_n^{S_2}}(I_{2,n}).
\]
The following definition allows us to interpret iterated sequences for $\Gr(2,\mathbb C^n)$ in a combinatorial way in Corollary~\ref{cor:treegraph} below.

\begin{figure}[h]
\begin{center}
\begin{tikzpicture}[scale=.14]

\draw (0,0) -- (0,-2.5);
\draw (0,0) -- (2.5,2);
\draw (0,0) -- (-2.5,2);

\draw[thick,->] (0,-3.5) -- (0,-6);

\begin{scope}[yshift=-8cm] 
\draw (-3,2)-- (-1.5,0) -- (1.5,0) -- (3,2);
\draw (-3,-2) -- (-1.5,0);
\draw (1.5,0) -- (3,-2);

\draw[thick,->] (0,-3.5) -- (0,-6);
\end{scope}

\begin{scope}[yshift=-16cm,xshift=-1.5cm] 
\draw (-3,2)-- (-1.5,0) -- (4.5,0) -- (6,2);
\draw (-3,-2) -- (-1.5,0);
\draw (4.5,0) -- (6,-2);
\draw (1.5,0) -- (1.5,-2);

\draw[thick,->] (7,-3.5) -- (8.25,-6);
\draw[thick,->] (-4.25,-3.5) -- (-5.75,-6);
\end{scope}

\begin{scope}[yshift=-24cm,xshift=-10cm] 
\draw (-3,2)-- (-1.5,0) -- (6.5,0) -- (8,2);
\draw (-3,-2) -- (-1.5,0);
\draw (6.5,0) -- (8,-2);
\draw (1,0) -- (1,-2);
\draw (4,0) -- (4,-2);
\draw[thick,->] (1.5,-3) -- (-.5,-6);
\draw[thick, ->] (9.5,-3) -- (12,-6);
\begin{scope}[xshift=16cm] 
\draw (-3,2)-- (-1.5,0) -- (4.5,0) -- (6,2);
\draw (-3,-2) -- (-1.5,0);
\draw (4.5,0) -- (6,-2);
\draw (1.5,0) -- (1.5,-2) -- (0,-4);
\draw (1.5,-2) -- (3,-4);
\draw[thick,->] (4.5,-3.5) -- (6,-6);
\end{scope}
\end{scope}

\begin{scope}[yshift=-33cm,xshift=-13cm]
\draw (-3,2)-- (-1.5,0) -- (6.5,0) -- (8,2);
\draw (-3,-2) -- (-1.5,0);
\draw (6.5,0) -- (8,-2);
\draw (.5,0) -- (.5,-2);
\draw (2.5,0) -- (2.5,-2);
\draw (4.5,0) -- (4.5,-2);
\draw[thick,->] (-4.5,-3) -- (-7,-6);
\draw[thick,->] (2.5,-3) -- (3.5,-6);
\draw[thick, ->] (9.5,-2.5) -- (17,-6);
\begin{scope}[xshift=20cm] 
\draw (-3,2)-- (-1.5,0) -- (6.5,0) -- (8,2);
\draw (-3,-2) -- (-1.5,0);
\draw (6.5,0) -- (8,-2);
\draw (1,0) -- (1,-2);
\draw (4,0) -- (4,-2);
\draw (4,-2) -- (2.5,-4);
\draw (4,-2) -- (5.5,-4);
\draw[thick,->] (.75,-3.5) -- (-.5,-6.5);
\draw[thick,->] (-3.5,-2.5) -- (-10.5,-6);
\draw[thick, ->] (9.5,-3) -- (12,-6);
\end{scope}
\end{scope}

\begin{scope}[yshift=-42cm,xshift=-23cm]
\draw (-3,2)-- (-1.5,0) -- (6.5,0) -- (8,2);
\draw (-3,-2) -- (-1.5,0);
\draw (6.5,0) -- (8,-2);
\draw (0.1,0) -- (0.1,-2);
\draw (1.7,0) -- (1.7,-2);
\draw (3.2,0) -- (3.2,-2);
\draw (4.9,0) -- (4.9,-2);
\node at (2.5,-4) {$\vdots$};
\begin{scope}[xshift=14cm] 
\draw (-3,2)-- (-1.5,0) -- (6.5,0) -- (8,2);
\draw (-3,-2) -- (-1.5,0);
\draw (6.5,0) -- (8,-2);
\draw (.5,0) -- (.5,-2);
\draw (2.5,0) -- (2.5,-2);
\draw (4.5,0) -- (4.5,-2);
\draw (4.5,-2) -- (3,-4);
\draw (4.5,-2) -- (6,-4);
\node at (1.5,-4) {$\vdots$};
\end{scope}

\begin{scope}[xshift=27cm]
\draw (-3,2)-- (-1.5,0) -- (6.5,0) -- (8,2);
\draw (-3,-2) -- (-1.5,0);
\draw (6.5,0) -- (8,-2);
\draw (.5,0) -- (.5,-2);
\draw (2.5,0) -- (2.5,-2);
\draw (2.5,-2) -- (1,-4);
\draw (2.5,-2) -- (4,-4);
\draw (4.5,0) -- (4.5,-2);
\node at (2.5,-4.5) {$\vdots$};
\begin{scope}[xshift=14cm] 
\draw (-3,2)-- (-1.5,0) -- (6.5,0) -- (8,2);
\draw (-3,-2) -- (-1.5,0);
\draw (6.5,0) -- (8,-2);
\draw (.5,0) -- (.5,-2);
\draw (.5,-2) -- (-1,-4);
\draw (.5,-2) -- (2,-4);
\draw (4.5,0) -- (4.5,-2);
\draw (4.5,-2) -- (3,-4);
\draw (4.5,-2) -- (6,-4);
\node at (2.5,-4.5) {$\vdots$};
\end{scope}
\end{scope}
\end{scope}

\end{tikzpicture}
\end{center}
\caption{The tree graph $\mathcal T$ from level ($\#$of leaves) 3 to 8.}\label{fig:treegraph}
\end{figure}
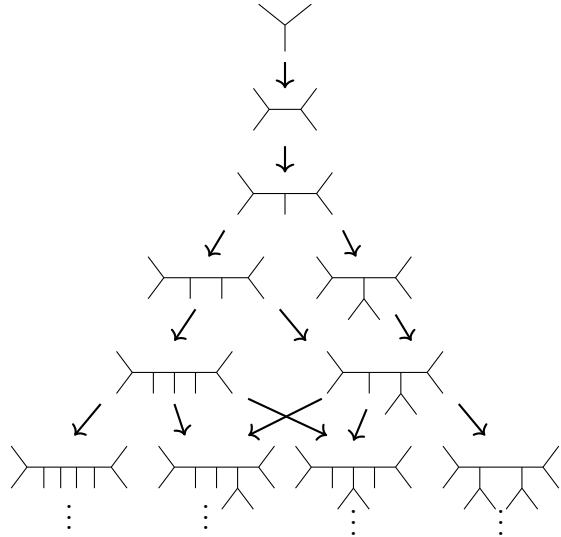

\begin{definition}
The \emph{tree graph} $\mathcal T$ is an infinite graph whose vertices at level $i\ge 3$ correspond to trivalent trees with $i$ leaves. There is an arrow $\mathtt T\to \mathtt T'$, if $\mathtt T$ has $i$ leaves, $\mathtt T'$ has $i+1$ leaves and $\mathtt T'$ can be obtained from $\mathtt T$ by attaching a new boundary edge in the middle of some edge of $\mathtt T$. There is a unique source $\mathtt T_3$ at level $3$. See Figure~\ref{fig:treegraph}.
\end{definition}

\begin{corollary}\label{cor:treegraph}
Every iterated sequence $S$ for $\Gr(2,\mathbb C^n)$ corresponds to a path from $\mathtt T_3$ to $\mathtt T^S_n$ in the tree graph $\mathcal T$.
\end{corollary}

\begin{proof}
The underlying unlabelled trees in the sequence $\mathbb T_S=(T_3,T_4^S,\dots,T_n^S)$ associated to $S$ define the path $\mathtt T_3\to\mathtt T_4^S\to\dots\to\mathtt T_n^S$ in $\mathcal T$.
\end{proof}


\newpage
\section{Toric degenerations via plabic graphs}\label{sec:BFFHL}

\noindent
In this section we apply Theorem~\ref{thm: val and quasi val with wt matrix} from \S\ref{sec:val and quasival} to the valuation defined by Rietsch-Williams in \cite{RW17} on $\mathbb C[\Gr(k,\mathbb C^n)]$ using the cluster structer.
Theorem~\ref{thm: RW val and wt vect} specifies when their toric degeneration can be realized as a Gröbner toric degeneration having a Khovanskii basis in terms of Plücker coordinates.
Further, Corollary~\ref{cor: integral NO vs prime} establishes a connection between the integrality of their associated Newton-Okounkov body and the weighting matrix of the valuation (as in Definition~\ref{def: wt matrix from valuation}).

Moreover, we show that the weight vector defined for plabic graphs in joint work with Fang, Fourier, Hering, and Lanini in \cite{BFFHL} is closely realted to the weighting matrix.
The subsection \S\ref{sec:case gr2n} is based on this joint work, where
we establish an explicit bijection between the toric degenerations of the Grassmannian $\Gr(2,\mathbb C^{n})$ arising from maximal cones in tropical Grassmannians and the ones coming from plabic graphs corresponding to $\Gr(2,\mathbb C^{n})$. 

\begin{figure}[h]
\centering
\begin{center}
\begin{tikzpicture}[scale=.7]
    \draw (-1,1.5) -- (-.5,1) -- (.5,1) -- (1,1.5);
    \draw (-.5,1) -- (-1,0) -- (-2,-.5);
    \draw (-1,0) -- (-.5,-1) -- (-.5,-1) -- (-.5,-2);
    \draw (-.5,-1) -- (1,-.5) -- (1.75,-.5);
    \draw (1,-.5) -- (0,0) -- (-1,0);
    \draw (0,0) -- (.5,1);

    \draw (-1,1.5) to [out=25,in=155] (1,1.5)
    to [out=-45,in=95] (1.75,-0.5) to [out=-95,in=0] (-.5,-2) to [out=180,in=-80] (-2,-.5) to [out=90,in=-145] (-1,1.5);

    \node[above] at (-1,1.5) {4};
    \node[above, right] at (1,1.5) {3};
    \node[right] at (1.75,-.5) {2};
    \node[below] at (-.5,-2) {1};
    \node[left] at (-2,-.5) {5};

    \draw[fill] (-.5,1) circle [radius=0.1];
    \draw[fill] (0,0) circle [radius=0.1];
    \draw[fill] (-.5,-1) circle [radius=0.1];
    \draw[fill, white] (.5,1) circle [radius=.115];
        \draw (.5,1) circle [radius=.115];
    \draw[fill, white] (1,-.5) circle [radius=.115];
        \draw (1,-.5) circle [radius=.115];
    \draw[fill, white] (-1,0) circle [radius=.115];
        \draw (-1,0) circle [radius=.115];
\end{tikzpicture}
\end{center}
\label{ExamplePlabic}
\caption{A plabic graph.}
\end{figure}
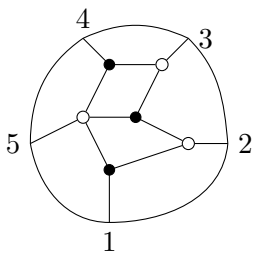

\subsection{Plabic graphs}\label{subsec: plabic graphs}

We review the definition of plabic graphs due to Postnikov \cite{Pos06}. This section is closely oriented towards Rietsch and Williams \cite{RW17}.

\begin{definition}\label{def: plabic graph}
A \textit{plabic graph} $\mathcal{G}$ is a planar bicolored graph embedded in a disk. It has $n$ boundary vertices numbered $1,\dots, n$ in a counterclockwise order. Boundary vertices lie on the boundary of the disk and are not colored. Additionally there are internal vertices colored black or white. Each boundary vertex is adjacent to a single internal vertex.
\end{definition}

For our purposes we assume that plabic graphs are connected and that every leaf of a plabic graph is a boundary vertex. We first recall the four local moves on plabic graphs.

\begin{center}
    \begin{figure}[h]
    \centering
    \begin{tikzpicture}[scale=0.7]
    \draw (0,0) -- (1,1) -- (2,1) -- (3,0);
    \draw (0,3) -- (1,2) -- (2,2) -- (3,3);
    \draw (1,2) -- (1,1);
    \draw (2,2) -- (2,1);

    \draw[->] (3.5,1.5) -- (4.5,1.5);
    \draw[->] (4.5,1.5) -- (3.5,1.5);

    \draw (5,0) -- (6,1) -- (7,1) -- (8,0);
    \draw (5,3) -- (6,2) -- (7,2) -- (8,3);
    \draw (6,2) -- (6,1);
    \draw (7,2) -- (7,1);

    \draw[fill] (1,1) circle [radius=0.1];
    \draw[fill] (2,2) circle [radius=0.1];
    \draw[fill, white] (1,2) circle [radius=0.1];
    \draw[fill, white] (2,1) circle [radius=0.1];
    \draw (1,2) circle [radius=0.1];
    \draw (2,1) circle [radius=0.1];

    \draw[fill] (6,2) circle [radius=0.1];
    \draw[fill] (7,1) circle [radius=0.1];
    \draw[fill, white] (6,1) circle [radius=0.1];
    \draw[fill, white] (7,2) circle [radius=0.1];
    \draw (6,1) circle [radius=0.1];
    \draw (7,2) circle [radius=0.1];
    \end{tikzpicture}
    \label{fig: plabic move M1}
    \caption{Square move (M1)}
    \end{figure}
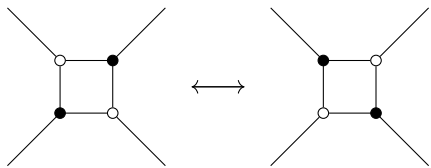
    \end{center}

\begin{itemize}
\item[(M1)] If a plabic graph contains a square of four internal vertices with alternating colors, each of which is trivalent, then the colors can be swapped. So every black vertex in the square becomes white and every white vertex becomes black (see Figure~\ref{fig: plabic move M1}).

    \begin{center}
    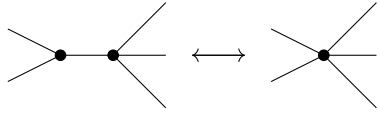
\begin{figure}[h]
    \centering
    \begin{tikzpicture}[scale=0.7]
    \draw (0,0.5) -- (1,1) -- (0,1.5);
    \draw (1,1) -- (2,1) -- (3,2);
    \draw (3,1) -- (2,1) -- (3,0);

    \draw[->] (3.5,1) -- (4.5,1);
    \draw[->] (4.5,1) -- (3.5,1);

    \draw (5,0.5) -- (6,1) -- (5,1.5);
    \draw (7,2) -- (6,1) -- (7,1);
    \draw (6,1) -- (7,0);

    \draw[fill] (6,1) circle [radius=0.1];
    \draw[fill] (1,1) circle [radius=0.1];
    \draw[fill] (2,1) circle [radius=0.1];
    \end{tikzpicture}
    \label{fig: plabic move M2}
    \caption{Merge vertices of same color (M2)}
    \end{figure}
    \end{center}

\item[(M2)] If two internal vertices of the same color are connected by an edge, the edge can be contracted and the two vertices can be merged. Conversely, any internal black or white vertex can be split into two adjacent vertices of the same color (see Figure~\ref{fig: plabic move M2}).

    \begin{center}
    \begin{figure}[h]
    \centering
    \begin{tikzpicture}[scale=0.7]
    \draw (0,0) -- (1,0) -- (2,0);

    \draw[->] (2.5,0) -- (3.5,0);
    \draw[->] (3.5,0) -- (2.5,0);

    \draw (4,0) -- (6,0);

    \draw[fill, white] (1,0) circle [radius=0.1];
    \draw (1,0) circle [radius=0.1];
    \end{tikzpicture}
    \label{fig: plabic move M3}
    \caption{Insert/remove degree two vertex (M3)}
    \end{figure}
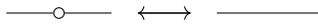
    \end{center}

\item[(M3)] If a plabic graph contains an internal vertex of degree $2$, it can be removed. Equivalently, an internal black or white vertex can be inserted in the middle of any edge (see Figure~\ref{fig: plabic move M3}).

    \begin{center}
    \begin{figure}[h]
    \centering
    \begin{tikzpicture}[scale=0.7]
    \draw (0,0) -- (1,0);
    \draw (2,0) -- (3,0);

    \draw (1,0.06) -- (2,0.06);
    \draw (1,-.06) -- (2,-.06);

    \draw[fill] (1,0) circle [radius=0.1];
    \draw[fill, white] (2,0) circle [radius=0.1];
    \draw (2,0) circle [radius=0.1];

    \draw[->] (3.5,0) -- (4.5,0);

    \draw (5,0) -- (7,0);
    \end{tikzpicture}
    \label{fig: plabic move R}
    \caption{Reducing parallel edges (R)}
    \end{figure}
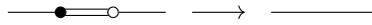
    \end{center}

\item[(R)] If two internal vertices of opposite color are connected by two parallel edges, they can be reduced to only one edge. This can not be done conversely (see Figure~\ref{fig: plabic move R}).
\end{itemize}

The \textit{equivalence class} of a plabic graph $\mathcal G$ is defined as the set of all plabic graphs that can be obtained from $\mathcal G$ by applying (M1)-(M3). 
If in the equivalence class there is no graph to which (R) can be applied, we say $\mathcal G$ is \textit{reduced}. 
From now on we only consider reduced plabic graphs.

\begin{definition}\label{def: trip permutation}
Let $\mathcal G$ be a reduced plabic graph with boundary vertices $v_1,\dots, v_n$ labelled in a counterclockwise order. We define the \textit{trip permutation} $\pi_{\mathcal G}$ as follows. We start at a boundary vertex $v_i$ and form a path along the edges of $\mathcal G$ by turning maximally right at an internal black vertex and maximally left at an internal white vertex. We end up at a boundary vertex $v_{\pi(i)}$ and define $\pi_{\mathcal G}=[\pi(1),\dots,\pi(n)]\in S_n$.
\end{definition}

It is a fact that plabic graphs in one equivalence class have the same trip permutation. Further, it was proven by Postnikov in \cite[Theorem 13.4]{Pos06} that plabic graphs with the same trip permutation are connected by moves (M1)-(M3) and are therefore equivalent.
Let $\pi_{k,n}=(n-k+1,n-k+2,\dots,n,1,2,\dots,n-k)$. From now on we focus on plabic graphs $\mathcal G$ with trip permutation $\pi_{\mathcal G}=\pi_{k,n}$. Each path $v_i$ to $v_{\pi_{k,n}(i)}$ defined above, divides the disk into two regions. 
We label every face in the region to the left of the path by $i$. After repeating this for every $1\le i\le n$, all faces have a labelling by an $(n-k)$-element subset of $[n]$. 
We denote by $\mathcal{P_G}$ the set of all such subsets for a fixed plabic graph $\mathcal G$.

A face of a plabic graph is called \emph{internal}, if it does not intersect with the boundary of the disk. Other faces are called \emph{boundary faces}.
Following \cite{RW17} we define an orientation on a plabic graph. 
This is the first step in establishing the \emph{flow model} introduced by Postnikov, which we use to define plabic degrees on the Pl\"ucker coordinates.

\begin{definition}\label{def: perfect orientation}
An orientation $\mathcal O$ of a plabic graph $\mathcal G$ is called \emph{perfect}, if every internal white vertex has exactly one incoming arrow and every internal black vertex has exactly one outgoing arrow. 
The set of boundary vertices that are sources is called the \emph{source set} and is denoted by $I_{\mathcal O}$.
\end{definition}

Postnikov showed in \cite{Pos06} that every reduced plabic graph with trip permutation $\pi_{k,n}$ has a perfect orientation with source set of order $k$. See Figure \ref{fig:ExamplePerfectOrientation} for a plabic graph with trip permutation $\pi_{2,5}$.

\begin{center}
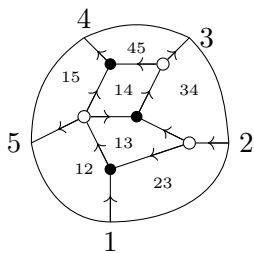
\begin{figure}[h]
\centering
\begin{tikzpicture}[scale=.7]
    \draw (-1,1.5) -- (-.5,1) -- (.5,1) -- (1,1.5);
    \draw (-.5,1) -- (-1,0) -- (-2,-.5);
    \draw (-1,0) -- (-.5,-1) -- (-.5,-1) -- (-.5,-2);
    \draw (-.5,-1) -- (1,-.5) -- (1.75,-.5);
    \draw (1,-.5) -- (0,0) -- (-1,0);
    \draw (0,0) -- (.5,1);

    \draw[->] (-.5,1) -- (-.75,1.25);
    \draw[->] (-1,0) --(-1.5,-.25);
    \draw[->] (-1,0) -- (-.75,.5);
    \draw[->] (-1,0) -- (-.5,0);
    \draw[->] (-.5,-1) -- (-.75,-0.5);
    \draw[->] (-.5,-2) -- (-.5,-1.5);
    \draw[->] (1,-.5) -- (.25,-0.75);
    \draw[->] (1.75,-0.5) -- (1.35,-0.5);
    \draw[->] (1,-0.5) -- (.5,-0.25);
    \draw[->] (0,0) -- (.25,0.5);
    \draw[->] (.5,1) -- (.75,1.25);
    \draw[->] (.5,1) -- (0,1);

    \draw (-1,1.5) to [out=25,in=155] (1,1.5)
    to [out=-45,in=95] (1.75,-0.5) to [out=-95,in=0] (-.5,-2) to [out=180,in=-80] (-2,-.5) to [out=90,in=-145] (-1,1.5);

    \node[above] at (-1,1.5) {4};
    \node[above, right] at (1,1.5) {3};
    \node[right] at (1.75,-.5) {2};
    \node[below] at (-.5,-2) {1};
    \node[left] at (-2,-.5) {5};

    \node at (0,1.3) {\tiny{$45$}};
    \node at (-.25,.5) {\tiny{$14$}};
    \node at (1,.5) {\tiny{$34$}};
    \node at (0.5,-1.25) {\tiny{$23$}};
    \node at (-.25,-.5) {\tiny{$13$}};
    \node at (-1,-1) {\tiny{$12$}};
    \node at (-1.25,.75) {\tiny{$15$}};

    \draw[fill] (-.5,1) circle [radius=0.1];
    \draw[fill] (0,0) circle [radius=0.1];
    \draw[fill] (-.5,-1) circle [radius=0.1];
    \draw[fill, white] (.5,1) circle [radius=.115];
        \draw (.5,1) circle [radius=.115];
    \draw[fill, white] (1,-.5) circle [radius=.115];
        \draw (1,-.5) circle [radius=.115];
    \draw[fill, white] (-1,0) circle [radius=.115];
        \draw (-1,0) circle [radius=.115];
\end{tikzpicture}
\caption{A plabic graph with a perfect orientation and source set $\{1,2\}$.}
\label{fig:ExamplePerfectOrientation}
\end{figure}
\end{center}

Index the standard basis of $\mathbb Z^{d+1}$ by the faces of the plabic graph $\mathcal G$, where $d=k(n-k)$.
Given a perfect orientation $\mathcal O$ on $\mathcal G$, every directed path $\p$ from a boundary vertex in the source set to a boundary vertex that is a sink, divides the disk in two parts. 
The \emph{weight} $\wei(\p)\in\mathbb Z_{\ge 0}^{d+1}$ has entry $1$ in the position corresponding to a face $F$ of $\mathcal G$, if $F$ is to the left of $\p$ with respect to the orientation.
The \emph{degree} $\deg_{\mathcal G}(\p)$ is defined the number of internal faces to the left of the path. 

For a set of boundary vertices $J$ with $\vert J\vert =\vert I_{\mathcal O}\vert$, we define a \emph{$J$-flow} as a collection of self-avoiding, vertex disjoint directed paths with sources $I_{\mathcal O}-(J\cap I_{\mathcal O})$ and sinks $J-(J\cap I_{\mathcal O})$. 
Let $I_{\mathcal O}-(J\cap I_{\mathcal O})=\{j_1,\dots,j_r\}$ and $\bfff=\{\p_{j_1},\dots,\p_{j_r}\}$ be a flow where each path $\p_{j_i}$ has sink $j_i$.
Then the \emph{weight of the flow} is $\wei(\bfff):=\wei(\p_{j_1})+\dots+\wei(\p_{j_r})$.
Similarly, we define the \emph{degree of the flow} as $\deg_\G(\bfff)=\deg_\G(\p_{j_1})+\dots +\deg_\G(\p_{j_r})$.
By $\mathcal{F}_J$ we denote the set of all $J$-flows in $\mathcal G$ with respect to $\mathcal O$.

\subsubsection*{Valuation and plabic degree}

In \cite{RW17} Rietsch-Williams use the cluster structure on $\Gr(k,\mathbb C^n)$ (due to Scott, see \cite{Sco06}) to define a valuation on $\mathbb C(\Gr(k,\mathbb C^n))\setminus \{0\}$ for every seed.
In fact, a plabic graph $\G$ defines a seed in the corresponding cluster algebra. 
A combinatorial algorithm associated a quiver with $\G$ (see e.g. \cite[Definition~3.8]{RW17}).
The corresponding $\mathcal A$-cluster is given in terms of Plücker coordinates $p_J$, where $J$ is a face label in $\G$ as described above.

As we are only interested in the values on Plücker coordinates of this valuation, we do not recall it in detail hear but refer the reader to \cite[\S7]{RW17}.
The main idea to construct the valuation is to use $\mathcal X$-cluster variables as coordinates for $\mathbb C[\Gr(k,\mathbb C^n)]$ and send a Laurent polynomial in those to its lexicographically minimal term.
This is another instance of a lowest term valuation that we have encountered already in \S\ref{sec:Bos} in the context of birational sequences.

Let $A_{k,n}:=\mathbb C[\Gr(k,n)]$, then the Rietsch-Williams valuation in \cite[Definition~7.1]{RW17} can be restricted to a valuation $\val_\G:A_{k,n}\setminus \{0\}\to (\mathbb Z^{d+1},\prec)$. 
The total order $\prec$ on $\mathbb Z^{d+1}$ is the lexicographic order with respect to a fixed order on the coordinates (see \cite[Definition~7.1]{RW17}).
For $J\in\binom{[n]}{k}$ let $\bfff_J\in\ff_J$ be such that $\deg_G(\bfff_J)=\min\{\deg_\G(\bfff)\mid \bfff\in\ff_J\}$.
Then on a Plücker coordinate $\bar p_J\in A_{k,n}$ the valuation $\val_\G$ is given by
\[
\val_\G(\bar p_J)=\wei(\bfff_J).
\]

\begin{remark}
This is for us, given the notion of degree, the most convenient way to write it and in fact equivalent to how it is described in \cite{RW17}.
\end{remark}

We define closely related to the valuation the following notion of degree for Plücker variables in $\mathbb C[p_J]_{J}$ and associate a weight vector in $\mathbb R^{\binom{n}{k}}$.

\begin{definition}\label{def: plabic deg}
For $J\in\binom{[n]}{k}$ and a plabic graph $\mathcal{G}$, the \emph{plabic degree} of the Pl\"ucker variable $p_J$ is defined as
\[
\deg_{\mathcal G}(p_J)=\min\{\deg_{\mathcal G}(\bfff)\mid \bfff\in \mathcal F_J\}.
\]
It gives rise to a weight vector ${\bf w}_{\mathcal G}\in\mathbb R^{\binom{n}{k}}$ by setting $({\bf w}_{\mathcal G})_J=\deg_{\mathcal G}(J)$.
\end{definition}

By \cite[Lemma~3.2]{PSW09} and its proof, the plabic degree is independent of the choice of the perfect orientation. We therefore fix the perfect orientation by choosing the source set $I_{\mathcal O}=[k]$. 
The following proposition guarantees that the degree (and the valuation) are well-defined.
It is a reformulation of the original statement adapted to our notion degree.

\begin{proposition*}(\cite[Corollary~11.4]{RW17})\label{prop: unique min flow}
There is a unique $J$-flow in $\mathcal G$ with respect to $\mathcal O$ with degree equal to $\deg_{\mathcal{G}}( p_J)$.
\end{proposition*}



\begin{example}
Consider the plabic graph $\mathcal{G}$ with perfect orientation from Figure \ref{fig:ExamplePerfectOrientation} and source set is $I_{\mathcal O}=[2]$. 
We compute $\deg_\G(p_J)$ and $\val_\G(\bar p_J)$ for all $J\in\binom{[5]}{2}$.

Order the faces of $\G$ by
\[
F_{23},F_{34},F_{45},F_{15},F_{12}, F_{13}, F_{14}.
\]
For example, consider $J=\{1,4\}$.
There are two flows, $\bfff_1$ and $\bfff_2$ from $I_{\mathcal O}$ to $J=\{1,4\}$. Both consist of only one path from $2$ to $4$. 
One of them, say $\bfff_1$, has faces labelled by $\{2,3\}$, $\{1,5\}$ and $\{1,2\}$ to the left, and $\bfff_2$ has faces $\{2,3\}$, $\{1,5\}$, $\{1,2\}$ and $\{1,3\}$ to the left. 
Then with respect to the order of coordinates on $\mathbb Z^7$ we have
\[
\wei(\bfff_1)=(1,0,0,1,1,0,0) \ \text{ and } \ \wei(\bfff_2)=(1,0,0,1,1,1,0). 
\]
As $\deg_\G(\bfff_1)=0$ and $\deg_\G(\bfff_2)=1$, we have $\val_\G(\bar p_{14})=(1,0,0,1,1,0,0)$ and $\deg_\G(p_{14})=0$.
All other $\val_\G(\bar p_J)$ and $\deg_\G(p_J)$ can be found in Table~\ref{tab:exp Rw val and deg}.
\end{example}

\begin{table}[]
    \centering
    \begin{tabular}{c|c c c c c c c| c | c}
        $p_J$ & $F_{23}$ & $F_{34}$ & $F_{45}$ & $F_{15}$ & $F_{12}$ & $F_{13}$ & $F_{14}$ & $\deg_\G(p_J)$ & $e(M_\G)$ \\ \hline
        $p_{12}$ & $0$ & $0$ & $0$ & $0$ & $0$ & $0$ & $0$ & $0$ & $0$ \\[-1ex]
        $p_{13}$ & $1$ & $0$ & $1$ & $1$ & $1$ & $0$ & $0$ & $0$ & $13$ \\[-1ex]
        $p_{14}$ & $1$ & $0$ & $0$ & $1$ & $1$ & $0$ & $0$ & $0$ & $10$ \\[-1ex]
        $p_{15}$ & $1$ & $0$ & $0$ & $0$ & $1$ & $0$ & $0$ & $0$ & $6$ \\[-1ex]
        $p_{23}$ & $0$ & $0$ & $1$ & $1$ & $1$ & $0$ & $0$ & $0$ & $12$ \\[-1ex]
        $p_{24}$ & $0$ & $0$ & $0$ & $1$ & $1$ & $0$ & $0$ & $0$ & $9$\\[-1ex]
        $p_{25}$ & $0$ & $0$ & $0$ & $0$ & $1$ & $0$ & $0$ & $0$ & $5$\\[-1ex]
        $p_{34}$ & $1$ & $0$ & $1$ & $2$ & $2$ & $1$ & $1$ & $2$ & $22$\\[-1ex]
        $p_{35}$ & $1$ & $0$ & $1$ & $1$ & $2$ & $1$ & $1$ & $2$ & $18$\\[-1ex]
        $p_{45}$ & $1$ & $0$ & $0$ & $1$ & $2$ & $1$ & $1$ & $2$ & $15$
    \end{tabular}
    \caption{The valuation $\val_\G$ for $\G$ as in Figure~\protect{\ref{fig:ExamplePerfectOrientation}} on Plücker coordinates, the plabic degrees and an example for a weight vector $e(M_\G)$ as in Proposition~\ref{prop: plabic lin form} (the multiplicities in the proof of Proposiiton~\ref{prop: plabic lin form} are chosen as $r_i=i$ and $q=1$, the columns are ordered as below).}
    \label{tab:exp Rw val and deg}
\end{table}

\subsection{The valuation \texorpdfstring{$\val_\G$}{} and the weighting matrix \texorpdfstring{$M_{\G}$}{}}

In this section we apply Theorem~\ref{thm: val and quasi val with wt matrix} from \S\ref{sec:val and quasival} to the valuation $\val_\G$ by Rietsch-Williams as seen in the last section.
We show that the weight vector defined by the plabic degree on Plücker coordinates is closely related to the valuation: in fact, taking the initial ideal of the Plücker ideal with respect to the weighting matrix of $\val_\G$ coincides with the initial ideal with respect to the plabic weight vector (see Proposition~\ref{prop: plabic lin form}).
In particular, we obtain that the associated graded algebra for $\val_\G$ is the quotient of the polynomial ring in Plücker coordinates by the initial ideal, given it is prime.
Moreover, in Corollary~\ref{cor: integral NO vs prime} we relate the property of the initial ideal being prime to integrality of the Newton-Okounkov body $\Delta(A_{k,n},\val_\G)$ studied in \cite[\S8]{RW17}.
We exhibit the case of $\Gr(3,\mathbb C^6)$ in detail below and dedicate \S\ref{sec:case gr2n} to analyzing the case of $\Gr(2,\mathbb C^n)$.

\medskip

Let $\G$ be a plabic graph for $\Gr(k,\mathbb C^n)$ with perfect orientation chosen such that $[k]$ is the source set as above.
Consider the weighting matrix  $M_{\mathcal G}:=M_{\val_{\mathcal G}}$ of $\val_{\mathcal G}$ as in Definition~\ref{def: wt matrix from valuation}.
That is, the columns of $M_{\mathcal G}$ are $\val_\G(p_{J})$ for $J\in\binom{[n]}{k}$ and the rows $M_1,\dots,M_{d+1}$ are indexed by the faces of the plabic graph $\mathcal G$, where $d=k(n-k)$. 
Denote the boundary faces of $\G$ by $F_1,\dots,F_n$, where $F_i$ is adjacent to the boundary vertices $i$ and $i+1$.
Hence, $F_k=F_{\varnothing}$ and $M_k=(0,\dots,0)$.
Order the rows of $M_\G$ such that $M_i$ is the row corresponding to the face $F_i$ in $\G$.

\begin{lemma}\label{lem: col corresp to bdy face}
Let $r\in[n]$ and $J=\{j_1,\dots,j_k\}\in\binom{[n]}{k}$ with $j_1<\dots j_s\le k<j_{s+1}<\dots <j_k$. Set $[k]\setminus\{j_1,\dots,j_s\}=\{i_1,\dots,i_{k-s}\}$ with $i_1<\dots<i_{k-s}$.
Then
\[
(M_r)_J=\#\{l\mid r\in[j_l,i_{k-l+1}]\},
\]
where $[j_l,i_{k-l+1}]$ is the cyclic interval ($j_l>i_{k-l+1}$).
\end{lemma}
\begin{proof}
Let $\bfff=\{\p_{j_1},\dots,\p_{j_k}\}$ be a flow from $[k]$ to $J$, where $\p_{j_i}$ denotes the path with sink $j_i$. 
The paths $\p_{j_r}$ for $r\le s$ are ``lazy paths", starting and ending at $j_r$ without moving.
Let $[k]\setminus \{j_1,\dots,j_s\}=\{i_1,\dots,i_{k-s}\}$ with $i_1<\dots<i_{k-s}$.
Hence, for $l>s$ the path $\p_{j_l}$ has source $i_{k-l+1}$ and sink $j_l$.
To its left are all boundary faces $F_r$ with $r$ in the cyclic interval $[j_l,i_{k-l+1}]$. 
Note that $i_{k-l+1}<k<j_l$, hence $k\not\in[j_l,i_{k-l+1}]$.
In particular, the claim follows
\end{proof}

\begin{corollary}\label{cor: col in lin sp}
Recall that $L_{I_{k,n}}\subset \trop(\Gr(k,\mathbb C^n))$ is the lineality space of the Plücker ideal $I_{k,n}$. For all $r\in[n]$ we have
\[
M_r\in L_{I_{k,n}}.
\] 
\end{corollary}
\begin{proof}
Consider a Plücker relation $R_{K,L}$ with $K\in\binom{[n]}{k-1}$ and $L\in\binom{[n]}{k+1}$ of form \eqref{eq: def plucker rel}. 
Every term in $R_{K,L}$ equals $\pm p_Jp_{J'}$ for some $J,J'\in\binom{[n]}{k}$.
Let $J=\{j_1,\dots,j_k\}\in\binom{[n]}{k}$ with $j_1<\dots j_s\le k<j_{s+1}<\dots <j_k$ and $J'=\{j'_1,\dots,j'_k\}\in\binom{[n]}{k}$ with $j'_1<\dots j'_{s'}\le k<j'_{s'+1}<\dots <j'_k$.
Set $[k]\setminus\{j_1,\dots,j_s\}=\{i_1,\dots,i_{k-s}\}$ with $i_1<\dots<i_{k-s}$ and $[k]\setminus\{j'_1,\dots,j'_{s'}\}=\{i'_1,\dots,i'_{k-s'}\}$ with $i'_1<\dots<i'_{k-s'}$.
We further denote $J\cup J'\setminus ([k]\cap (J\cup J'))=\{l_1,\dots,l_m\}$.
That is $\{l_1,\dots,l_m\}=\{j_{s+1},\dots,j_k,j'_{s'+1},\dots,j'_k\}$ and $m=2k-s-s'$. 
Note that there might be repetitions among the $l_i$.
Define for $q\in[m]$
\[
i_q:=\left\{\begin{matrix}
    i_{k-l+1}, & \text{ if } i_r=j_l,\\
    i'_{k-l'+1},& \text{ if } i_r=j'_{l'}.
\end{matrix}\right.
\]
With this notation we have
\begin{eqnarray*}
 \{j_l\in\{j_{s+1},\dots,j_k\}\mid r\in[j_l,i_{k-l+1}]\} &\cup& \{j'_{l'}\in\{j'_{s'+1},\dots,j'_k\}\mid r\in[j'_{l'},i'_{k-l'+1}]\}\\
&=& \{l_q\in\{l_1,\dots,l_m\}\mid r\in[l_q,i_q] \}.
\end{eqnarray*}
Consider $M_r\in\mathbb R^{\binom{n}{k}}$ as a weight vector for $\mathbb C[p_J]_J$.
Then the $M_r$-weight on $\pm p_Jp_{J'}$ is
\begin{eqnarray*}
(M_r)_J+(M_r)_{J'} &=& \#\{l_q\in\{l_1,\dots,l_m\}\mid r\in[l_q,i_q] \}.
\end{eqnarray*}
As $\#\{l_q\in\{l_1,\dots,l_m\}\mid r\in[l_q,i_q] \}$ depends only on $J\cup J'$, which is equal for all monomials in $R_{K,L}$ we deduce
\[
\init_{M_r}(R_{K,L})=R_{K,L},
\]
and the claim follows.
\end{proof}

Recall the plabic weight vector $\bw_\G$ from Definition~\ref{def: plabic deg}. The following proposition establishes the connection to what we have seen in \S\ref{sec:val and quasival}.
In terms of the weighting matrix $M_\G$, we observe
\[
\bw_\G=\sum_{j=k+1}^{d+1} M_j,
\]
where the sum contains exactly those $M_j$ corresponding to interior faces of $\G$.
Let $e:\mathbb Q^{d+1}\to\mathbb Q$ be a linear form. Using the notation as in \S\ref{sec:val and quasival} we have $e(M_\G)=(e(\val_\G(p_J))_{J\in\binom{[n]}{k}}\in\mathbb Q^{\binom{n}{k}}$.

\begin{proposition}\label{prop: plabic lin form}
For every plabic graph $\G$ there exists a linear form $e:\mathbb Q^{d+1}\to \mathbb Q$ satisfying $\val_\G(\bar p_I)\prec \val_\G(\bar p_J)$ implies $ e(\val_\G(\bar p_I))<e(\val_\G(\bar p_J))$ for $I,J\in\binom{[n]}{k}$, such that for the plabic weight vector $\bw_\G$ we have
\[
\init_{e(M_\G)}(I_{k,n})=\init_{\bw_G}(I_{k,n}).
\]
\end{proposition}
\begin{proof}
We use the following two observations following by definition from the fan structure we chose on $\trop(\Gr(k,\mathbb C^n))$. Firstly, for every $q\in\mathbb Z_{>0}$ we have $\init_{q\bw_\G}(I_{k,n})=\init_{\bw_\G}(I_{k,n})$.

Secondly, by Corollary~\ref{cor: col in lin sp}, and the fan structure on $\trop(\Gr(k,\mathbb C^n))$ we have for $r_1,\dots,r_k\in\mathbb Z_{\ge 0}$
that $\init_{\bw_\G+r_1M_1+\dots +r_kM_k}(I_{k,n})=\init_{\bw_\G}(I_{k,n})$.

In particular, these observations imply that it is enough to find $q,r_1,\dots,r_k\in\mathbb Z_{\ge 0}$ such that $e(x_1,\dots,x_{d+1}):=\sum_{i=1}^kr_ix_i+\sum_{i=k+1}^{d+1}qx_i$ satisfies
\[
\val_\G(\bar p_I)\prec \val_\G(\bar p_J) \Rightarrow e(\val_\G(\bar p_I))<e(\val_\G(\bar p_J)) \text{ for } I,J\in\binom{[n]}{k}.
\]
As $\val_\G(p_J)\in\mathbb Z_{\ge 0}^{d+1}$, it suffices to find $q,r_1,\dots,r_k\in\mathbb Z_{\ge 0}$ such that all $e(\val_\G(p_J))$ are distinct.
If we choose $r_1,\dots,r_k\in\mathbb Z_{>0}$ big enough with $\vert r_i-r_j\vert$ big enough and $q\in\mathbb Z_{>0}$ relatively small, this is the case and the claim follows.
\end{proof}

\begin{example}
Consider the plabi graph $\G$ with perfect orientation as in Figure~\ref{fig:ExamplePerfectOrientation}. We have seen the images of $\val_\G$ in Table~\ref{tab:exp Rw val and deg} above. Note that the columns corresponding to the faces of $\G$ are ordered as we fixed above. We have $F_{23}=F_1,F_{34}=F_2,\dots,F_{12}=F_5$.
In particular, in the row for $p_J$ the entries in columns $F_{23},\dots, F_{14}$ give $\val_\G(\bar p_J)$. For example, for $p_{14}$ these entires are $1,0,0,1,1,0,0$ and $\val_\G(\bar p_{14})=(1,0,0,1,1,0)$.
The matrix $M_\G$ is the transpose of the matrix with columns $F_i$ as in the table.
The last column $e(M_\G)$ corresponds to the linear form $e:\mathbb Z^7\to\mathbb Z$ given by
\[
e(x_1,\dots,x_7)=x_1+2x_2+3x_3+4x_4+5x_5+x_6+x_7.
\]
It is an example of a linear form as in the proof of Proposition~\ref{prop: plabic lin form} with $r_i=i$ and $q=1$.
\end{example}

\begin{theorem}\label{thm: RW val and wt vect}
If $\init_{\bw_\G}(I_{k,n})$ is prime we have 
\[
\gr_{\val_\G}(A_{k,n})\cong \mathbb C[p_J]_J/\init_{\bw_\G}(I_{k,n}).
\]
Moreover, $\Delta(A_{k,n},\val_\G(p_J))=\conv(\val_\G(\bar p_J)\mid J\in\binom{[n]}{k})$ and the Plücker coordinates form a Khovanskii basis for $(A_{k,n},\val_\G)$.
\end{theorem}
\begin{proof}
Consider $e:\mathbb Q^{d+1}\to\mathbb Q$ as in Proposition~\ref{prop: plabic lin form}. Then by Lemma~\ref{lem: init wM vs init M} we have
\[
\init_{\bw_\G}(I_{k,n})=\init_{e(M_\G)}(I_{k,n})=\init_{M_\G}(I_{k,n}).
\]
In particular, $\init_{M_\G}(I_{k,n})$ is prime. Moreover, as $I_{k,n}$ is homogeneous with respect to the usual grading on the polynomial ring $\mathbb C[p_J]_J$ and generated by elements of degree 2, we can apply Theorem~\ref{thm: val and quasi val with wt matrix} and obtain $\Delta(A_{k,n},\val_\G)=\conv(\val_\G(p_J)\mid J\in\binom{[n]}{k})$.
Recall the (quasi-)valuation $\val_{M_\G}$ with weighting matrix $M_\G$.
From the proof of Theorem~\ref{thm: val and quasi val with wt matrix} we further deduce that
\[
S(A_{k,n},\val_\G)=S(A_{k,n},\val_{M_\G}).
\]
Hence, by \cite[Lemma~4.4]{KM16} $\gr_{\val_\G}(A_{k,n})=\gr_{M_\G}(A_{k,n})\cong \mathbb C[p_J]_J/\init_{M_\G}(I_{k,n})$ and the claim follows.
\end{proof}

\begin{corollary}\label{cor: integral NO vs prime}
If $\Delta(A_{k,n},\val_\G)$ is not integral, then $\init_{\bw_\G}(I_{k,n})$ is not prime.
\end{corollary}
\begin{proof}
If $\Delta(A_{k,n},\val_G)$ is not integral it is in particular not the convex hull of $\val_\G(p_J)\in\mathbb Z^{d+1}$ for $J\in\binom{[n]}{k}$. Then the semigroup $S(A_{k,n},\val_\G)$ is also not generated by $\val_\G(p_J)$ for $J\in\binom{[n]}{k}$.
By Theorem~\ref{thm: RW val and wt vect} in hence follows, that $\init_{\bw_\G}(I_{k,n})$ cannot be prime.
\end{proof}

\subsubsection*{The case of $\Gr(2,\mathbb C^n)$}

We show in \S\ref{sec:case gr2n} that $\init_{\bw_\G}(I_{2,n})$ is prime for every plabic graph $\G$ for $\Gr(2,\mathbb C^n)$ (see Theorem~\ref{Thm:main}).
Further, we show that for every isomorphism class of maximal prime cones in $\trop(\Gr(2,\mathbb C^n))$ there exists a plabic graph $\G$ such that $\init_{\bw_\G}$ coincides with an initial ideal from a prime cone in the equivalence class.
Then Theorem~\ref{thm: RW val and wt vect} yields the following corollary.

\begin{corollary}\label{cor: RW val gr2n}
For every plabic graph $\G$ for $\Gr(2,\mathbb C^n)$ there exists a maximal prime cone $C\subset \trop(\Gr(2,\mathbb C^n))$ such that
\[
\gr_{\val_G}(A_{2,n})=\mathbb C[p_{ij}]_{ij}/\init_C(I_{2,n}).
\]
In particular, the special fibre of the toric degeneration of $\Gr(2,\mathbb C^n)$ given in Rietsch-Williams by $\val_\G$ occurs also as a special fibre in a Gröbner toric degeneration as in \eqref{eq: groebner family}.
Moreover, for every maximal cone $C\subset\trop(\Gr(2,\mathbb C^n))$ there exists a plabic graph $\G$ such that the toric variety $\Proj(\gr_{\val_\G}(A_{2,n}))$ is isomorphic to $V(\init_C(I_{2,n}))$.
\end{corollary}

Note that the Corollary implies in particular, that $\bw_\G\in\trop(\Gr(2,\mathbb C^n))$, which is not clear by definition. However, we observe in the next subsection that this is also the case for $\Gr(3,\mathbb C^6)$.

\subsubsection*{The case of $\Gr(3,\mathbb C^{6})$\footnote{The computations were done in joint work \cite{BFFHL}. The code is provided by Hering in \cite{M2c}.}}\label{Gr(3,6)}

For $\Gr(3,\mathbb C^6)$ there are (up to moves (M2) and (M3)) 34 plabic graphs. 
We compute for each of them the plabic weight vector $\bw_\G$ and the initial ideal $\init_{\bw_\G}(I_{2,n})$.
Before stating the results of our computation we review what is known about $\trop(\Gr(3,\mathbb C^6))$ from \cite{SS04}.
There are 7 isomorphism classes of  maximal cones in $\trop(\Gr(3,\mathbb C^{6}))$,
labelled by 
\[
\text{FFGG, EEEE, EEFF1, EEFF2, EEFG, EEEG, EEFG.}
\]
The last six are prime cones, while the initial ideals for cones in the isomorphism classes FFGG are not prime.
The following Theorem summarizes our results. The weight vectors $\bw_\G$ for all 34 plabic graphs can be found in Table~\ref{tab:matching} in Appendix~\ref{sec:app Gr}.

\begin{theorem}\label{thm: gr36 plabic}
For every plabic graph $\G$ for $\Gr(3,\mathbb C^6)$ the plabic weight vector $\bw_\G$ lies in $\trop(\Gr(3,\mathbb C^6))$.
Up to isomorphism, there are six distinct initial ideals $\init_{\bw_\G}(I_{3,6})$, five of which correspond to ideals from cones in the isomorphism classes
\[
\text{EEFF1, EEFF2, EEFG, EEEG, EEFG.}
\]
Two plabic graphs yield a weight vector lying on a ray of type GG of a maximal (non-prime) cone of type FFGG. These are in fact those plabic graphs, for which $\Delta(A_{3,6},\val_\G)$ is not integral (see \cite[\S8]{RW17}).
\end{theorem}

In \cite{BCL} they study a combinatorial model for cluster algebras of type $D_4$. 
The 50 seeds are given by centrally symmetric pseudo-triangulations of a once punctured disk with $8$ marked points. 
In the paper they analyze symmetries among the cluster seeds and associate each seed to an isomorphism class of maximal cones in $\trop(\Gr(3,\mathbb C^{6}))$. 
Although they consider all 50 cluster seeds, the outcome is similar to ours: they recover only six of the seven types of maximal cones, missing the cone of type EEEE.
In Table~\ref{tab:matching} we indicate to which seeds (using their labelling) the 34 plabic graphs correspond.
We observe that our findings match theirs in the sense that our weight vectors lie in the relative interior of those cones in $\trop(\Gr(3,\mathbb C^6))$ they identified with the corresponding seed through symmetries.



\subsection{Main Theorem for \texorpdfstring{$\Gr(2,\mathbb C^n)$}{}}\label{sec:case gr2n}
From now on we focus on $\Gr(2,\mathbb C^n)$. The main result of this section is Theorem~\ref{Thm:main} in which we show that the initial ideal with respect to the plabic weight vector coincides with the initial ideal corresponding to a certain trivalent tree.
In fact, the plabic graph and the tree are related combinatorially: they can both be obtained from the same triangulation.
\medskip

Recall the cluster structure of $\Gr(2,\mathbb C^n)$ from \S\ref{subsec: cluster gr2n}.
For $n\geq 4$, let $D_n$ be a disk with $n$ marked points on the boundary labelled by $1,\ldots,n$ in the counterclockwise order (or $n$-gon for short). For $1\leq i,j\leq n$, let $(i,j)$ be the arc connecting the points $i$ and $j$.

Fix $\Delta=\Delta_e\cup\Delta_d$ a triangulation of the $D_n$ as in \S\ref{subsec: cluster gr2n}.
We define
$$\Delta_d=\{(a_1,b_1),(a_2,b_2),\ldots,(a_{n-3},b_{n-3})\}$$
as the set of internal arcs and
$$\Delta_e=\{(1,2),(2,3),\ldots,(n-1,n),(n,1)\}$$
the set of boundary arcs connecting marked points.

A \textit{rooted (labelled) tree} is a trivalent tree on $n$ leaves with root $1$ and the other leaves labelled counterclockwise with $2,\ldots,n$. Each triangulation $\Delta$ of $D_n$ gives such a labelled tree $T_\Delta$ by considering the dual graph to $\Delta$. More precisely $T_\Delta$ can be constructed using the following Algorithm~\ref{alg:tree from triang}. See Figure \ref{triangulation:tree} for an example.

\begin{algorithm}[h]
\SetAlgorithmName{Algorithm}{} 
\KwIn{\medskip {\bf Input:\ }  A triangulation $\Delta$ of $D_n$.}
\BlankLine
\For{every triangle $t$ in $\Delta$}{draw a vertex $v_t$.}
\For{every two adjacent triangles $t,t'$ in $\Delta$}{connect the vertices $v_t$ and $v_{t'}$ by an edge.}
\For{every boundary arc $(i,i+1)\mod n$ of $D_n$}{draw a vertex labelled $i+1$ and connect it to the unique vertex $v_t$ for which $(i,i+1)\mod n$ is an edge of the triangle $t$.}
\BlankLine
{\bf Output:\ } a rooted tree $T_\Delta$. 
\label{alg:tree from triang}
\caption{Constructing a rooted tree $T_\Delta$ from a triangulation $\Delta$ of $D_n$.}
\end{algorithm}

\begin{center}
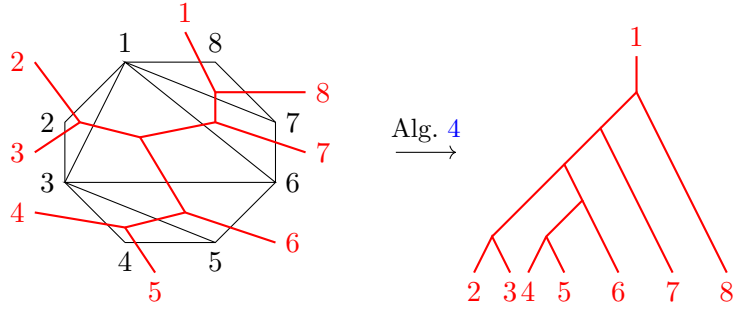
\begin{figure}
\centering
\begin{tikzpicture}[scale=0.4]
\node[below] at (3,1) {4};
\node[below] at (6,1) {5};
\node[right] at (8,3) {6};
\node[right] at (8,5) {7};
\node[above] at (6,7) {8};
\node[above] at (3,7) {1};
\node[left] at (1,5) {2};
\node[left] at (1,3) {3};

\node[below, red] at (4,0) {5};
\node[right, red] at (8,1) {6};
\node[right, red] at (9,4) {7};
\node[right, red] at (9,6) {8};
\node[above, red] at (5,8) {1};
\node[left, red] at (0,7) {2};
\node[left, red] at (0,4) {3};
\node[left, red] at (0,2) {4};

\draw (3,1) --(6,1) -- (8,3) -- (8,5) -- (6,7) -- (3,7) -- (1,5) -- (1,3) -- (3,1);
\draw (6,1) -- (1,3) -- (8,3) -- (3,7) -- (8,5);
\draw (1,3) -- (3,7);
\draw[red, thick] (4,0) -- (3,1.5) -- (0,2);
\draw[red, thick] (3,1.5) -- (5,2) -- (8,1);
\draw[red, thick] (5,2) -- (3.5,4.5) -- (1.5,5) -- (0,4);
\draw[red, thick] (1.5,5) -- (0,7);
\draw[red, thick] (3.5,4.5) -- (6,5) -- (9,4);
\draw[red, thick] (6,5) -- (6,6) -- (9,6);
\draw[red, thick] (6,6) -- (5,8);

\draw[->] (12,4) -- (14,4);
\node[above] at (13,4) {\small Alg.~\ref{alg:tree from triang}};

\begin{scope}[xshift=14cm, scale=0.6]
\node[above, red] at (10,12) {1};
\node[below, red] at (1,0) {2};
\node[below, red] at (3,0) {3};
\node[below, red] at (4,0) {4};
\node[below, red] at (6,0) {5};
\node[below, red] at (9,0) {6};
\node[below, red] at (12,0) {7};
\node[below, red] at (15,0) {8};

\draw[red, thick] (1,0) -- (2,2) -- (3,0);
\draw[red, thick] (2,2) -- (6,6) -- (7,4) -- (5,2) -- (4,0);
\draw[red, thick] (5,2) -- (6,0);
\draw[red, thick] (7,4) -- (9,0);
\draw[red, thick] (6,6) -- (8,8) -- (12,0);
\draw[red, thick] (8,8) -- (10,10) -- (15,0);
\draw[red, thick] (10,10) -- (10,12);
\end{scope}
\end{tikzpicture}
\caption{A triangulation of $D_8$ and the corresponding rooted tree (the output of Algorithm~\protect{\ref{alg:tree from triang}}) after rescaling.}
\label{triangulation:tree}
\end{figure}
\end{center}

For a Pl\"ucker variable $p_{ij}$, recall that the \emph{tree degree} $\ddeg_{T_\Delta}(p_{ij})$ is the number of internal edges between leaves $i$ and $j$ (see Definition~\ref{def:treedeg}). The tree degree admits an alternative description in terms of the corresponding triangulation as follows.

We adopt the following notations on cyclic intervals: $[ i,i ]=\{i\}$ and for $1\leq i<j\leq n$, let $[ i,j ]=\{i,i+1,\ldots,j\}$ and $[ j,i ]=\{j,j+1,\ldots,n\}\cup\{1,2,\ldots,i\}.$
The \emph{A-degree} on Pl\"ucker coordinates is defined for $i<j$ as
\begin{equation}\label{eq: def A-degree}
a_{ij}=\ddeg_A(p_{ij}):=\#\{(a_r, b_r) \in \Delta_d \mid \{a_r,b_r\}\cap[ i,j-1]\ \ \text{has cardinality $1$}\}.
\end{equation}
By definition the following proposition holds.
\begin{proposition}\label{Prop:treeA}
For any $1\leq i<j\leq n$, $\ddeg_{T_\Delta}(p_{ij})=\ddeg_A(p_{ij})$.
\end{proposition}
\begin{definition}\label{def: connection number}
An internal arc $(a,b)$ in the triangulation $\Delta$ is called \emph{connecting} $[ p,q]$ and $[ s,t]$, if $a\in [ p,q]$ and $b \in [ s,t]$ or vice versa. The number of such internal arcs are denoted by $C_{p,q}^{s,t}$ and called \emph{connection number}.
\end{definition}
Using this notation, the A-degree on the Pl\"ucker coordinates can be written as:
\begin{equation}\label{Eq:Adeg}
a_{ij}=C_{i,j-1}^{j,i-1}.
\end{equation}

This alternative description allows us to state the following proposition.
\begin{proposition}\label{a-prop}
For $1\leq i<j<k<l\leq n$,
\begin{enumerate}[(i)]
\item $a_{ij}+a_{kl}=a_{ik}+a_{jl}$ if and only if $C_{j,k-1}^{l,i-1}=0$; when this is the case, $a_{ij}+a_{kl}>a_{il}+a_{jk}$.
\item $a_{il}+a_{jk}=a_{ik}+a_{jl}$ if and only if $C_{i,j-1}^{k,l-1}=0$;
when this is the case, $a_{ij}+a_{kl}<a_{il}+a_{jk}$.
\end{enumerate}
\end{proposition}
To prove the proposition we need the following properties of the connection numbers that follow directly from their definition.

\begin{lemma}\label{Lem:C-nb}
Suppose that $1\leq p,q,s,t\leq n$, the following statements hold.
\begin{enumerate}[(i)]
\item $C_{p,q}^{s,t}=C_{s,t}^{p,q}$.
\item Suppose that $[ s,t]\cap [ p,q]=\emptyset$. For any $r\in[ s,t]$ such that $r\neq s$, $C_{p,q}^{s,t}=C_{p,q}^{s,r-1}+C_{p,q}^{r,t}$; if $r\neq t$, $C_{p,q}^{s,t}=C_{p,q}^{s,r}+C_{p,q}^{r+1,t}$.
\item For $t\in [ s,q]$ such that $t\neq q$, $C_{s,q}^{s,t}=C_{s,t}^{s,t}+C_{t+1,q}^{s,t}$.
\item $C_{s,q}^{s,q}=C_{s,s}^{s+1,q}+C_{s+1,q}^{s+1,q}$.
\end{enumerate}
\end{lemma}

\begin{proof}[Proof of Proposition~\ref{a-prop}]
By (\ref{Eq:Adeg}) and Lemma \ref{Lem:C-nb}(ii), we have
\begin{eqnarray*}
a_{ij}+a_{kl}& =& C_{i,j-1}^{j,k-1}+C_{i,j-1}^{k,l-1}+C_{i,j-1}^{l,i-1}+C_{k,l-1}^{l,i-1}+C_{k,l-1}^{i,j-1}+C_{k,l-1}^{j,k-1}\\
a_{ik}+a_{jl} &=& C_{i,j-1}^{k,l-1}+C_{j,k-1}^{k,l-1}+C_{i,j-1}^{l,i-1}+C_{j,k-1}^{l,i-1}+C_{j,k-1}^{l,i-1}+C_{k,l-1}^{l,i-1}+C_{j,k-1}^{i,j-1}+C_{k,l-1}^{i,j-1}\\
a_{il}+a_{jk} &= & C_{i,j-1}^{l,i-1}+C_{j,k-1}^{l,i-1}+C_{k,l-1}^{l,i-1}+C_{j,k-1}^{k,l-1}+C_{j,k-1}^{l,i-1}+C_{j,k-1}^{i,j-1}.
\end{eqnarray*}
Notice that in a triangulation, $C_{j,k-1}^{l,i-1}$ and $C_{i,j-1}^{k,l-1}$ can not both be zero. The proposition follows from comparing the terms.
\end{proof}

We continue by defining $X$-degrees in terms of connection numbers that coincide with the plabic degrees. 
Having both, plabic and tree degrees, in terms of connection numbers allows us to directly compare them on the triangulation and prove our main theorem.
Before we can do so we need a combinatorial tool relating triangulations and plabic graphs.
Kodama and Williams associate to a triangulation $\Delta$ a plabic
graph $\mathcal{G}_{\Delta}$ in \cite[Algorithm~12.1]{KW14}. 
We recall the algorithm below.

\begin{algorithm}[h]
\SetAlgorithmName{Algorithm}{} 
\KwIn{\medskip {\bf Input:\ }  A triangulation $\Delta$ of $D_n$.}
\BlankLine
\For{every triangle $t$ in $\Delta$}{draw a black vertex $b_t$ and connect it to the vertices of the $t$.}
\For{every every marked point $m\in\partial D_n$ }{
    \If{exist an arc $(m,k)\in \Delta$}{draw a white vertex $w_m$ for $m$}
    \Else{draw a black vertex $b_m$ for $m$}}
Erase the arcs $\Delta_d$ and $\Delta_e$, contract adjacent vertices of the same color. 

{\bf Output:\ } a graph $D_n^\Delta$ with $n$  boundary vertices

\For{every boundary vertex $m$ of $D_n^\Delta$}{add an edge $e_m$ such that no two $e_{m_1},e_{m_2}$ intersect for $m_1,m_2$ boundary vertices of $D_n^\Delta$ }
Embed the resulting graph in a disk such that the new vertices of edges $e_m$ lie on the boundary.
\BlankLine
{\bf Output:\ } a plabic graph $\mathcal G_\Delta$. 
\label{alg:plabic from triang}
\caption{Constructing a plabic graph $\mathcal G_\Delta$ from a triangulation $\Delta$ of $D_n$.}
\end{algorithm}

We call $D_n^\Delta$ (the first output of Algorithm~\ref{alg:plabic from triang}) the \emph{plabic $n$-gon} associated to the triangulation $\Delta$. Boundary vertices of $D_n^\Delta$ are colored by black or white.

Fix a perfect orientation on $\mathcal{G}_\Delta$ such that the source set is $\{1,2\}$. Recall the definition of plabic degree. The plabic degree has an alternative description in terms of the corresponding triangulation as follows.
For a fixed triangulation $\Delta$ of $D_n$ the \emph{X-degrees}  of the Pl\"ucker coordinates are defined by
\begin{equation}\label{eq: def X-deg}
    x_{ij} = \deg_X(p_{ij})=\left\{  
    \begin{matrix}
    0, & \text{ if } i=1,j=2 \\
    C_{j,1}^{j,1}+C_{1,1}^{2,j-1}, & \text{ if } i=1,2< j\leq n \\
    C_{j,1}^{j,1}, & \text{ if } i=2,2< j\leq n \\
    C_{i,1}^{i,1}+C_{j,i-1}^{j,1}, & \text{ otherwise.}
    \end{matrix}
    \right.
\end{equation}

\begin{theorem}\label{Thm:mainDeg}
For any $1\leq i<j\leq n$, $\ddeg_{\mathcal{G}_{\Delta}}(p_{ij})=\ddeg_X(p_{ij})$.
\end{theorem}

\begin{example}\label{exp:caterpillar2}
We examine Theorem \ref{Thm:mainDeg} in the case where the triangulation $\Delta$ is given by $\Delta_d=\{(2,4),(2,5),\ldots,(2,n)\}$.

\medskip
\noindent
\emph{Claim:}
For every $2<j\leq n$, $\ddeg_{\mathcal{G}_{\Delta}}(P_{1j})=\ddeg_{\mathcal{G}_{\Delta}}(P_{2j})=0$ and for every $2<i<j\leq n$, $\ddeg_{\mathcal{G}_{\Delta}}(P_{ij})=n-j+1$.
\begin{proof}[Proof of Claim]
The first statement follows from the fact that for $i\in\{1,2\}$ and for $2<j\leq n$ there is a path from $i$ to $j$ having only boundary faces of $\mathcal{G}_\Delta$ to its left (see Figure~\ref{fig: palm at vertex 2 and plabic}). 
For the second part of the claim we observe that there there is a (unique) path from $1$ to $j$ having only boundary faces of $\mathcal{G}_\Delta$ to its left. Further, there is a unique minimal path (with respect to the number of faces on its left) from $2$ to $i$ which does not intersect the one from $1$ to $j$. We count that the second path has $n-j+1$ faces on its left, namely those coming from the internal arcs $(2,n),\dots,(2,j)$.
\end{proof}

By straightforward computations, Theorem~\ref{Thm:mainDeg} holds in this case.
\end{example}

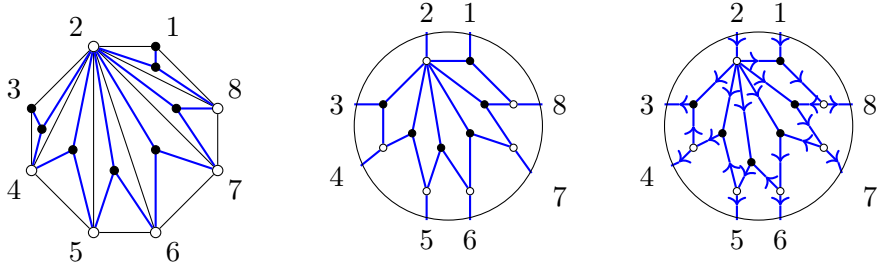
\begin{figure}[h]
\centering
\begin{center}
\begin{tikzpicture}[scale=0.275]
\draw (3,0) -- (6,0) -- (9,3) -- (9,6) -- (6,9) -- (3,9) -- (0,6) -- (0,3) -- (3,0);
\draw (3,9) -- (9,6);
\draw (3,9) -- (9,3);
\draw (3,9) -- (6,0);
\draw (3,9) -- (3,0);
\draw (3,9) -- (0,3);

\draw[blue, thick] (3,9) -- (6,8) -- (6,9);
\draw[blue, thick] (6,8) -- (9,6) -- (7,6) -- (3,9);
\draw[blue, thick] (7,6) -- (9,3) -- (6,4) -- (3,9);
\draw[blue, thick] (6,4) -- (6,0) -- (4,3) -- (3,9);
\draw[blue, thick] (4,3) -- (3,0) -- (2,4) -- (3,9);
\draw[blue, thick] (2,4) -- (0,3) -- (0.5,5) -- (3,9);
\draw[blue, thick] (0.5,5) -- (0,6);

\draw[fill] (6,9) circle [radius=0.2];
\draw[fill] (0,6) circle [radius=0.2];
\draw[fill, white] (3,9) circle [radius=0.25];
    \draw (3,9) circle [radius=0.25];
\draw[fill, white] (9,6) circle [radius=0.25];
    \draw (9,6) circle [radius=0.25];
\draw[fill, white] (9,3) circle [radius=0.25];
    \draw (9,3) circle [radius=0.25];
\draw[fill, white] (6,0) circle [radius=0.25];
    \draw (6,0) circle [radius=0.25];
\draw[fill, white] (3,0) circle [radius=0.25];
    \draw (3,0) circle [radius=0.25];
\draw[fill, white] (0,3) circle [radius=0.25];
    \draw (0,3) circle [radius=0.25];
\node[below left] at (3,0) {5};
\node[below right] at (6,0) {6};
\node[below right] at (9,3) {7};
\node[above right] at (9,6) {8};
\node[above right] at (6,9) {1};
\node[above left] at (3,9) {2};
\node[above left] at (0,6) {3};
\node[below left] at (0,3) {4};

\draw[fill] (0.5,5) circle [radius=0.2];
\draw[fill] (2,4) circle [radius=0.2];
\draw[fill] (4,3) circle [radius=0.2];
\draw[fill] (6,4) circle [radius=0.2];
\draw[fill] (7,6) circle [radius=0.2];
\draw[fill] (6,8) circle [radius=0.2];

\begin{scope}[xshift=17cm, yshift=2cm, scale=0.7]
\draw[blue, thick] (3,9) -- (6,9);
\draw[blue, thick] (6,9) -- (9,6) -- (7,6) -- (3,9);
\draw[blue, thick] (7,6) -- (9,3) -- (6,4) -- (3,9);
\draw[blue, thick] (6,4) -- (6,0) -- (4,3) -- (3,9);
\draw[blue, thick] (4,3) -- (3,0) -- (2,4) -- (3,9);
\draw[blue, thick] (2,4) -- (0,3) -- (0,6) -- (3,9);
\draw[blue, thick] (6,9) -- (6,11); 
\draw[blue, thick] (9,6) -- (11,6); 
\draw[blue, thick] (9,3) -- (10.25,1.25); 
\draw[blue, thick] (6,0) -- (6,-2); 
\draw[blue, thick] (3,0) -- (3,-2); 
\draw[blue, thick] (0,3) -- (-1.55,1.75); 
\draw[blue, thick] (0,6) -- (-2,6); 
\draw[blue, thick] (3,9) -- (3,11); 
\draw[fill] (6,9) circle [radius=0.25];
\draw[fill] (0,6) circle [radius=0.25];
\draw[fill, white] (3,9) circle [radius=0.25];
    \draw (3,9) circle [radius=0.25];
\draw[fill, white] (9,6) circle [radius=0.25];
    \draw (9,6) circle [radius=0.25];
\draw[fill, white] (9,3) circle [radius=0.25];
    \draw (9,3) circle [radius=0.25];
\draw[fill, white] (6,0) circle [radius=0.25];
    \draw (6,0) circle [radius=0.25];
\draw[fill, white] (3,0) circle [radius=0.25];
    \draw (3,0) circle [radius=0.25];
\draw[fill, white] (0,3) circle [radius=0.25];
    \draw (0,3) circle [radius=0.25];
\node[below] at (3,-2) {5};
\node[below ] at (6,-2) {6};
\node[below right] at (11,1) {7};
\node[right] at (11,6) {8};
\node[above] at (6,11) {1};
\node[above] at (3,11) {2};
\node[left] at (-2,6) {3};
\node[left] at (-2,1) {4};
\draw[fill] (2,4) circle [radius=0.25];
\draw[fill] (4,3) circle [radius=0.25];
\draw[fill] (6,4) circle [radius=0.25];
\draw[fill] (7,6) circle [radius=0.25];

\draw (4.5,4.5) circle [radius=6.5];
\end{scope}

\begin{scope}[xshift=32cm, yshift=2cm, scale=0.7]
\draw[blue, thick, ->] (3,9) -- (4.5,9);
    \draw[blue, thick] (4.5,9) -- (6,9);
\draw[blue, thick, ->] (6,9) -- (7.5,7.5);
    \draw[blue, thick] (7.5,7.5) -- (9,6);
\draw[blue, thick, ->] (9,6) -- (8,6);
    \draw[blue, thick] (8,6) -- (7,6);
\draw[blue, thick, ->] (3,9) -- (5,7.5);
    \draw[blue, thick] (5,7.5) -- (7,6);
\draw[blue, thick, ->] (7,6) -- (8,4.5);
    \draw[blue, thick] (8,4.5) -- (9,3);
\draw[blue, thick, ->] (9,3) -- (7.5,3.5);
    \draw[blue, thick] (7.5,3.5) -- (6,4);
\draw[blue, thick, ->] (3,9) -- (4.5,6.5);
    \draw[blue, thick] (4.5,6.5) -- (6,4);
\draw[blue, thick, ->] (6,4) -- (6,2);
    \draw[blue, thick] (6,2) -- (6,0);
\draw[blue, thick, ->] (6,0) -- (5,1);
    \draw[blue, thick] (5,1) -- (4,2);
\draw[blue, thick, ->] (3,9) -- (3.5,5.5);
    \draw[blue, thick] (3.5,5.5) -- (4,2);
\draw[blue, thick, ->] (4,2) -- (3.5,1);
    \draw[blue, thick] (3.5,1) -- (3,0);
\draw[blue, thick, ->] (3,0) -- (2.5,2);
    \draw[blue, thick] (2.5,2) -- (2,4);
\draw[blue, thick, ->] (3,9) -- (2.5,6.5);
    \draw[blue, thick] (2.5,6.5) -- (2,4);
\draw[blue, thick, ->] (2,4) -- (1,3.5);
    \draw[blue, thick] (1,3.5) -- (0,3);
\draw[blue, thick, ->] (0,3) -- (0,4.5);
    \draw[blue, thick] (0,4.5) -- (0,6);
\draw[blue, thick, ->] (3,9) -- (1.5,7.5);
    \draw[blue, thick] (1.5,7.5) -- (0,6);

\draw[blue, thick, ->] (6,11) -- (6,10); 
    \draw[blue, thick] (6,10) -- (6,9);
\draw[blue, thick, ->] (9,6) -- (10,6); 
    \draw[blue, thick] (10,6) -- (11,6);
\draw[blue, thick, ->] (9,3) -- (10,2);
    \draw[blue, thick] (10,2) -- (10.25,1.25); 
\draw[blue, thick, ->] (6,0) -- (6,-1);
    \draw[blue, thick] (6,-1) -- (6,-2);
\draw[blue, thick, ->] (3,0) -- (3,-1);
    \draw[blue, thick] (3,-1) -- (3,-2);
\draw[blue, thick, ->] (0,3) -- (-1,2);
    \draw[blue, thick] (-1,2) -- (-1.55,1.75);
\draw[blue, thick, ->] (0,6) -- (-1,6);
    \draw[blue, thick] (-1,6) -- (-2,6);  
\draw[blue, thick, ->] (3,11) -- (3,10);
    \draw[blue, thick] (3,10) -- (3,9); 

\draw[fill] (6,9) circle [radius=0.25];
\draw[fill] (0,6) circle [radius=0.25];
\draw[fill, white] (3,9) circle [radius=0.25];
    \draw (3,9) circle [radius=0.25];
\draw[fill, white] (9,6) circle [radius=0.25];
    \draw (9,6) circle [radius=0.25];
\draw[fill, white] (9,3) circle [radius=0.25];
    \draw (9,3) circle [radius=0.25];
\draw[fill, white] (6,0) circle [radius=0.25];
    \draw (6,0) circle [radius=0.25];
\draw[fill, white] (3,0) circle [radius=0.25];
    \draw (3,0) circle [radius=0.25];
\draw[fill, white] (0,3) circle [radius=0.25];
    \draw (0,3) circle [radius=0.25];
\node[below] at (3,-2) {5};
\node[below] at (6,-2) {6};
\node[below right] at (11,1) {7};
\node[right] at (11,6) {8};
\node[above] at (6,11) {1};
\node[above] at (3,11) {2};
\node[left] at (-2,6) {3};
\node[left] at (-2,1) {4};
\draw[fill] (2,4) circle [radius=0.25];
\draw[fill] (4,2) circle [radius=0.25];
\draw[fill] (6,4) circle [radius=0.25];
\draw[fill] (7,6) circle [radius=0.25];

\draw (4.5,4.5) circle [radius=6.5];
\end{scope}
\end{tikzpicture}
\caption{A plabic $8$-gon for $\Gr(2,\mathbb C^8)$ and the corresponding plabic graph with perfect orientation.}
  \label{fig: palm at vertex 2 and plabic}
\end{center}
\end{figure}

In order to prove Theorem~\ref{Thm:mainDeg} by induction on $n$, we need the following properties of $X$-degrees.

\begin{proposition}\label{x-prop}
For $1\leq i<j<k<l\leq n$,
\begin{enumerate}[(i)]
\item We have $x_{ij}+x_{kl}=x_{ik}+x_{jl}$ if and only if $C_{j,k-1}^{l,i-1}=0$. In this is the case $x_{ij}+x_{kl}<x_{il}+x_{jk}$.
\item We have $x_{il}+x_{jk}=x_{ik}+x_{jl}$ if and only if $C_{i,j-1}^{k,l-1}=0$. In this is the case $x_{ij}+x_{kl}>x_{il}+x_{jk}$.
\end{enumerate}
\end{proposition}

\begin{proof}
First notice that by Lemma \ref{Lem:C-nb}(iii), for $2<i<j\leq n$, $x_{i,j}=C_{i,1}^{i,1}+C_{j,1}^{j,1}+C_{2,i-1}^{j,1}$. The proof is separated into four cases:
\begin{enumerate}[(i)]
\item When $2<i<j<k<l\leq n$, we have:
\begin{equation}\label{Eq:xa}
x_{ij}+x_{kl}=C_{i,1}^{i,1}+C_{j,1}^{j,1}+C_{k,1}^{k,1}+C_{l,1}^{l,1}+C_{2,i-1}^{j,1}+C_{2,k-1}^{l,1},
\end{equation}
\begin{equation}\label{Eq:xb}
x_{ik}+x_{jl}=C_{i,1}^{i,1}+C_{j,1}^{j,1}+C_{k,1}^{k,1}+C_{l,1}^{l,1}+C_{2,i-1}^{k,1}+C_{2,j-1}^{l,1},
\end{equation}
\begin{equation}\label{Eq:xc}
x_{il}+x_{jk}=C_{i,1}^{i,1}+C_{j,1}^{j,1}+C_{k,1}^{k,1}+C_{l,1}^{l,1}+C_{2,i-1}^{l,1}+C_{2,j-1}^{k,1}.
\end{equation}
By Lemma \ref{Lem:C-nb} (i) and (ii), subtracting (\ref{Eq:xa}) from (\ref{Eq:xb}) gives
\[
C_{2,i-1}^{k,1}+C_{2,j-1}^{l,1}-C_{2,i-1}^{j,1}-C_{2,k-1}^{l,1}=-C_{2,i-1}^{j,k-1}-C_{j,k-1}^{l,1}=-C_{j,k-1}^{l,i-1}.
\]
Subtracting \eqref{Eq:xc} from \eqref{Eq:xb} we obtain
\[
C_{2,i-1}^{k,1}+C_{2,j-1}^{l,1}-C_{2,i-1}^{l,1}-C_{2,j-1}^{k,1}=-C_{i,j-1}^{k,1}+C_{i,j-1}^{l,1}=-C_{i,j-1}^{k,l-1}.
\]
These computations prove the proposition in this case.

\item When $i=1<2<j<k<l\leq n$, we have:
\begin{equation}\label{Eq:xaa}
x_{1j}+x_{kl}=C_{j,1}^{j,1}+C_{1,1}^{2,j-1}+C_{k,1}^{k,1}+C_{l,1}^{l,1}+C_{2,k-1}^{l,1},
\end{equation}
\begin{equation}\label{Eq:xab}
x_{1k}+x_{jl}=C_{k,1}^{k,1}+C_{1,1}^{2,k-1}+C_{j,1}^{j,1}+C_{l,1}^{l,1}+C_{2,j-1}^{l,1},
\end{equation}
\begin{equation}\label{Eq:xac}
x_{1l}+x_{jk}=C_{l,1}^{l,1}+C_{1,1}^{2,l-1}+C_{j,1}^{j,1}+C_{k,1}^{k,1}+C_{2,j-1}^{k,1}.
\end{equation}
Again by Lemma \ref{Lem:C-nb}, subtracting (\ref{Eq:xaa}) from (\ref{Eq:xab}) gives
\[
C_{1,1}^{2,k-1}-C_{1,1}^{2,j-1}+C_{2,j-1}^{l,1}-C_{2,k-1}^{l,1}=-C_{2,k-1}^{l,n}+C_{2,j-1}^{l,n}=-C_{j,k-1}^{l,n}.
\]
Subtracting (\ref{Eq:xac}) from (\ref{Eq:xab}) gives
\[
C_{1,1}^{2,k-1}-C_{1,1}^{2,l-1}+C_{2,j-1}^{l,1}-C_{2,j-1}^{k,1}=-C_{1,1}^{k,l-1}-C_{2,j-1}^{k,l-1}=-C_{1,j-1}^{k,l-1}.
\]
\item When $i=2<j<k<l\leq n$, the proof is similar.
\item When $i=1<j=2<k<l\leq n$, we have
\begin{eqnarray*}
x_{12}+x_{kl} &=& C_{k,1}^{k,1}+C_{l,1}^{l,1}+C_{2,k-1}^{l,1},\\
x_{1k}+x_{2l} &=& C_{k,1}^{k,1}+C_{1,1}^{2,k-1}+C_{l,1}^{l,1},\\
x_{1l}+x_{2k} &=& C_{l,1}^{l,1}+C_{1,1}^{2,l-1}+C_{k,1}^{k,1}.
\end{eqnarray*}
It is then easy to deduce the corresponding statement in the proposition.
\end{enumerate}
\end{proof}

\paragraph{Proof of Theorem \ref{Thm:mainDeg}}
Fix a triangulation $\Delta = \Delta_d\cup \Delta_e$ and let $\mathcal{G}_{\Delta}$ be the associated  plabic graph obtained by applying Algorithm~\ref{alg:plabic from triang} to $\Delta$.
The proof of the theorem is executed by induction on $n$. The case $n=4$ contains only two different triangulations and can be verified directly. Suppose $n\geq 5$. First notice that there exists at least two black boundary vertices in the plabic $n$-gon $D_n^\Delta$ and vertices $1$ and $2$ can not be both black vertices. In fact, all neighbors of a black vertex are white vertices. Let $s$ be the black vertex different from $1$ and $2$ such that there is no black vertex in $[ s+1,n ]$. Then $s-1$, $s+1,\ldots,n$ are all white vertices and $(s-1,s+1)\in\Delta_d$.

\begin{lemma}[Sector lemma]\label{Lem:Sector}
If $(s-1,p)\in\Delta_d$ for some $s+1<p\leq n$, then $(s-1,s+2),\ldots,(s-1,p-1)\in\Delta_d$.
\end{lemma}

\begin{proof}
Let $q\in[ s+2,p-1]$ be the smallest integer such that $(s-1,q)\notin\Delta_d$. In this case, there exists an internal arc $(q,r)$ for some $r\in[ q+1,p-1]$. This is not possible since otherwise there must be at least one black vertex in $[ q+1,r-1]$.
\end{proof}

\begin{corollary}
If $s=3$, then $\Delta_d=\{(2,4),(2,5),\ldots,(2,n)\}$.
\end{corollary}

\begin{proof}
When $s=3$, there are only two black vertices $1$ and $3$ in the plabic $n$-gon, which implies that $(2,4),(2,n)\in\Delta_d$. By the Sector Lemma, for any $r\in[ 5,n-1 ]$, $(2,r)\in\Delta_d$.
\end{proof}

According to the corollary, if $s=3$, by Example~\ref{exp:caterpillar2}, Theorem~\ref{Thm:mainDeg} holds. From now on suppose that $s\neq 3$.
The following lemma explains the local orientation on the square containing $s-1,s,s+1$ in the plabic graph $\mathcal{G}_\Delta$, see Figure~9. In the plabic graph $\mathcal{G}_\Delta$, we use $1,2,\ldots,n$ to denote the internal vertices connected to boundary vertices $1,2,\ldots,n$.
First note that, as $s$ is a boundary black vertex, it has already one edge going out in the plabic graph $\mathcal{G}_\Delta$. Hence, the edges in $\mathcal{G}_\Delta$ connecting $s-1$ and $s+1$ to $s$ have orientations pointing towards $s$ (see e.g. Figure~\ref{fig: case (xi)}).
Suppose that the theorem holds for any triangulation of $D_{n-1}$. Let $\overline{D}_n$ be the disk with $n-1$ markes points obtained from $D_n$ by removing the makred point $s$. 
The triangulation $\Delta$ of $D_n$ induces a triangulation $\overline{\Delta}=\overline{\Delta}_d\cup\overline{\Delta}_e$ of $\overline{D}_n$ where
\[
\overline{\Delta}_d=\Delta_d\backslash\{(s-1,s+1)\}\ \ \text{and}\ \ \overline{\Delta}_e=(\Delta_e\backslash\{(s-1,s),(s,s+1)\})\cup\{(s-1,s+1)\}.
\]
We associate to $\overline{D}_n$ and $\overline{\Delta}$ a plabic $(n-1)$-gon $\overline{D}_n^{\overline{\Delta}}$ and a plabic graph $\overline{\mathcal{G}}_{\overline{\Delta}}$.
For $1\leq i<j\leq n$ and $i,j\neq s$, we denote $\ddeg_{\overline{\mathcal{G}}_{\overline{\Delta}}}(\overline{p}_{ij})$ and $\ddeg_X(\overline{p}_{ij})$ the corresponding degrees with respect to $\overline{\mathcal{G}}_{\overline{\Delta}}$ and $\overline{\Delta}$. If one of $i$ and $j$ equals $s$, we set these degrees to be zero. The connection numbers for $\overline{\Delta}$ are denoted by $\overline{C}_{p,q}^{r,s}$. For $1\leq i<j\leq n$, we denote
$$v_{ij}=\ddeg_X(p_{ij})-\ddeg_X(\overline{p}_{ij})\ \ \text{and}\ \ w_{ij}=\ddeg_{\mathcal{G}_{\Delta}}(p_{ij})-\ddeg_{\overline{\mathcal{G}}_{\overline{\Delta}}}(\overline{p}_{ij}).$$
By the induction hypothesis, to prove the theorem, it suffices to show that for any $1\leq i<j\leq n$, $v_{ij}=w_{ij}$.
We start with the following lemma.

\begin{lemma}\label{Lem:count}
Suppose $2<i<s$ and $s<j\leq n$. The face of $\mathcal{G}_\Delta$ corresponding to the internal arc $(s-1,s+1)\in\Delta_d$ is to the left of any directed path from $1$ or $2$ to $i$, and to the right of any directed path from $1$ or $2$ to $j$.
\end{lemma}

\begin{proof}
If there exists a directed path from $1$ or $2$ to $j$ such that this face is to the left of the path, then it passes through the vertex $s$. This is impossible, since all arrows at $s$ not connecting to the boundary go towards $s$.
The proof of the statement on $i$ is similar.
\end{proof}

The rest of this section is dedicated to proving that  $\ddeg_{\mathcal{G}_{\Delta}}(p_{ij})=\ddeg_X(p_{ij})$ for all $1\leq i<j\leq n$. 
We distinguish the following 13 cases:
\begin{enumerate}[(i)]
\item $\mathit{i=1<j<s}.$ By Lemma~\ref{Lem:count}, we have $w_{1j}=1$
and 
\[
v_{1j}=C_{j,1}^{j,1}+C_{1,1}^{2,j-1}-\overline{C}_{j,1}^{j,1}-\overline{C}_{1,1}^{2,j-1}=C_{j,1}^{j,1}-\overline{C}_{j,1}^{j,1}=1.
\]
\item $\mathit{i=1<s<j\leq n}$.  
By Lemma~\ref{Lem:count}, we have $w_{1j}=0$  and a similar argument as above shows that $v_{1j}=0$.
\item $\mathit{1=i<j=s<n}$. 
We need to show that $\ddeg_{\mathcal{G}_{\Delta}}(p_{1s})=\ddeg_X(p_{1s})$. Notice that a directed path from $2$ to $s$ must pass through either $s-1$ or $s+1$. 
Since there always exists a directed path from $2$ to $s+1$, by minimality, we have $\ddeg_{\mathcal{G}_{\Delta}}(p_{1s})=\ddeg_{\mathcal{G}_{\Delta}}(p_{1,s+1})$. On the other hand, since there is no internal arc meeting $s$, we have $C_{s,1}^{s,1} = C_{s+1,1}^{s+1,1}$ and $C_{1,1}^{2,s-1} = C_{1,1}^{2,s}$. 
We compüute
\[
\ddeg_X(p_{1s})=C_{s,1}^{s,1}+C_{1,1}^{2,s-1}=C_{s+1,1}^{s+1,1}+C_{1,1}^{2,s}=\ddeg_X(p_{1,s+1}).
\]
Now the claim follows from the case $s<j=s+1\leq n$.
\item $\mathit{1=i<j=s=n}$. In this case, a directed path from $2$ to $s$ must pass through $s-1=n-1$, since it cannot pass through $1$. 
Therefore  $\ddeg_{\mathcal{G}_{\Delta}}(p_{1n})=\ddeg_{\mathcal{G}_{\Delta}}(p_{1,n-1})$. 
As $(1,n-1)\in \Delta_d$, we have $C_{n,1}^{n,1}=0$ and $C_{1,1}^{n-1,n-1}=C^{n-1,1}_{n-1,1}$. 
We can hence apply Lemma~\ref{Lem:C-nb} (2) and obtain
\[
\ddeg_X(p_{1n})=C_{n,1}^{n,1} + C_{1,1}^{2,n-1}=C_{1,1}^{2,n-2}+C^{n-1,1}_{n-1,1} =\ddeg_X(p_{1,n-1}).
\]
The statement follows from Case (i).
\item $\mathit{i=2<j\leq s}$. This(ese) case(s) can be examined in a similar manner as the corresponding cases for $i=1$.
\item $\mathit{i=2<s+1\leq j}$. The proof of this case is similar to the proof of Case (i).
Nevertheless, we repeat the argument since this case is applied to prove Case (xii).
By Lemma~\ref{Lem:count}, we have $w_{2j}=0$. 
On the other hand $v_{2j}=C_{j,1}^{j,1}-\overline{C}_{j,1}^{j,1}=0$,
since there are no internal arcs of $\Delta_d$ (or $\overline{\Delta}_d$) entirely contained in $[ j,1]$.
\item $\mathit{2<i<s<j}$. By Lemma~\ref{Lem:count}, $w_{ij}=1$. 
By definition, $v_{ij}=(C_{i,1}^{i,1}-\overline{C}_{i,1}^{i,1})+(C_{j,i-1}^{j,1}-\overline{C}_{j,i-1}^{j,1})$. 
Since $i\neq s$, the second bracket gives zero. 
The first bracket gives $1$, as the internal arc $(s-1,s+1)$ is no longer in $\overline{\Delta}$.
\item $\mathit{2<s<i<j}$. By Lemma~\ref{Lem:count}, $w_{ij}=0$. 
A similar argument as above shows $v_{ij}=0$.
\item $\mathit{2<i<j<s}$. By Lemma~\ref{Lem:count}, $w_{ij}=2$.  
A similar argument as above shows $v_{ij}=2$.
\item $\mathit{2<i=s<j=s+1}$. We consider directed paths from $1$ to $s+1$ and from $2$ to $s$. 
Since the vertex $s+1$ is occupied, to reach the vertex $s$, the path from $2$ to $s$ is forced to go through $s-1$, which shows $\ddeg_{\mathcal{G}_{\Delta}}(p_{s,s+1})=\ddeg_{\mathcal{G}_{\Delta}}(p_{s-1,s+1})$. 

By Case (i) we have proved, $\ddeg_{\mathcal{G}_{\Delta}}(p_{s-1,s+1})=\ddeg_X(p_{s-1,s+1})$. It suffices to show that $\ddeg_X(p_{s,s+1})-\ddeg_X(p_{s-1,s+1})=0$. 
Since $s-1>2$, the left hand side equals
\[
C_{s,1}^{s,1}+C_{s+1,s-1}^{s+1,1}-C_{s-1,1}^{s-1,1}-C_{s+1,s-2}^{s+1,1}.
\]
By applying Lemma \ref{Lem:C-nb} several times, we obtain
\begin{eqnarray*}
\ddeg_X(p_{s,s+1})-\ddeg_X(p_{s-1,s+1})&=& C_{s,1}^{s,1}+C_{s+1,s-1}^{s+1,1}-C_{s-1,1}^{s-1,1}-C_{s+1,s-2}^{s+1,1}\\
&=& C_{s-1,1}^{s-1,1}-C^{s,1}_{s-1,s-1}+C_{s+1,s-1}^{s+1,1}-C_{s-1,1}^{s-1,1}-C_{s+1,s-2}^{s+1,1}\\
&=& -C^{s,s}_{s-1,s-1}-C^{s+1,1}_{s-1,s-1}+C_{s+1,s-1}^{s+1,1}-C_{s+1,s-2}^{s+1,1}\\
&=& -C^{s,s}_{s-1,s-1}-C^{s+1,1}_{s+1,s-1}+C_{s+1,s-1}^{s+1,1}\\
&=& -C^{s,s}_{s-1,s-1}.
\end{eqnarray*}
The first two equalities follow from point (4) and (2) of Lemma~\ref{Lem:C-nb} and the third one by combining (1) and (2) of Lemma~\ref{Lem:C-nb}.
Since there is no internal arc touching $s$, the connection number $C_{s-1,s-1}^{s,s}$ is zero and the statement follows.
\item $\mathit{i=s<s+1<j}$.
We claim that $\ddeg_{\mathcal{G}_{\Delta}}(p_{sj}) = \ddeg_{\mathcal{G}_{\Delta}}(p_{s+1,j})$. 
To compute these degrees, we have to consider directed paths from $1$ to $j$ and from $2$ to $s$. Note that the path of minimal degree from $1$ to $j$ is the same for both calculations.

Consider paths from $2$ to $s$ or $s+1$. As all edges in $\mathcal{G}_{\Delta}$ meeting $s+1$ connect to black vertices, there is a unique black vertex $v$ such that the edge connecting $v$ and $s+1$ goes towards $s+1$ (see e.g. Figure~\ref{fig: case (xi)}). 
Let $(p,q,s+1)$ be the triangle in $\Delta$ corresponding to $v$ and assume that $p<q$, then $p\leq s-1$. 
Since $v$ has an outgoing edge to $s+1$, the edge between $p$ and $v$ is directed from $p$ to $v$. 
Hence, the plabic graph has a path $p\to v\to s+1$. As $p$ and $s+1$ are boundary vertices, and $s+1$ can only have one incoming vertex, every path from $2$ to $s$ or $s+1$ has to pass through $p$.
Note that the path of lowest degree must end with $p\to v\to s+1 \to s$, so the claim follows.

By Case (viii), $\ddeg_{\mathcal{G}_{\Delta}}(p_{s+1,j})=\ddeg_X(p_{s+1,j})$, hence it suffices to show that $\ddeg_X(p_{sj})=\ddeg_X(p_{s+1,j})$. 
Their difference is given by
\[
(C_{s,1}^{s,1}-C_{s+1,1}^{s+1,1})+(C_{j,s-1}^{j,1}-C_{j,s}^{j,1}).
\]
As there is no internal arc incident to the vertex $s$, we have $C_{s,1}^{s,1} = 0$ and $C_{j,s-1}^{j,1}=C_{j,s}^{j,1}$. 
It follows from our assumptions on $s$ that $C_{s+1,1}^{s+1,1}=0$.

\begin{figure}
\centering
\begin{tikzpicture}[scale=0.4]
\draw (3,13) -- (6,14) -- (9,13);
\draw (11,10) -- (12,8);
\draw (11,5) -- (9,3.5) -- (6,2) -- (3,3.5) -- (1,5);
\draw (0.5,10) -- (0,8) -- (2,12) -- (9,3.5) -- (0,8);
\draw (9,3.5) -- (3,3.5);

\draw[dashed] (9,13) -- (11,10);
\draw[dashed] (12,8) -- (11,5);
\draw[dashed] (1,5) -- (0,8);
\draw[dashed] (0.5,10) -- (2,12) -- (3,13);

\draw[blue, thick, ->] (2,12) -- (2,10);
\draw[blue, thick] (2,10) -- (2,8);
\draw[blue, thick, ->] (0,8) -- (1,8);
\draw[blue, thick] (1,8) -- (2,8);
\draw[blue, thick, ->] (2,8) -- (5.5,5.75);
\draw[blue, thick] (5.5,5.75) -- (9,3.5);
\draw[blue, thick, ->] (9,3.5) -- (9.5,1.75);
\draw[blue, thick] (9.5,1.75) -- (10,0);
\draw[blue, thick, ->] (6,2) -- (6,1);
\draw[blue, thick] (6,1) -- (6,-1);
\draw[blue, thick, ->] (3,3.5) -- (2.5,1.75);
\draw[blue, thick] (2.5,1.75) -- (2,0);
\draw[blue, thick, ->] (9,3.5) -- (7.5,3);
\draw[blue, thick] (7.5,3) -- (6,2.1);
\draw[blue, thick, ->] (3,3.5) -- (4.5,3);
\draw[blue, thick] (4.5,3) -- (6,2.1);

\draw [thick, red] (6,14) to [out=-90,in=-180] (7,11)
to [out=0,in=180] (9,9) to [out=0,in=-135] (12,8) ;

\draw[fill, white] (2,12) circle [radius=0.25]; 
    \draw (2,12) circle [radius=0.25];
    \node[above left] at (2,12) {$p$};
\draw[fill, white] (0,8) circle [radius=0.25]; 
    \draw (0,8) circle [radius=0.25];
    \node[above left] at (0,8) {$q$};
\draw[fill, white] (3,3.5) circle [radius=0.25];
    \draw (3,3.5) circle [radius=0.25];
    \node[left] at (3,3.5) {$s-1$};
\draw[fill] (6,2) circle [radius=0.25]; 
    \node[right] at (6,2) {$s$};
\draw[fill, white] (9,3.5) circle [radius=0.25]; 
    \draw (9,3.5) circle [radius=0.25];
    \node[right] at (9,3.5) {$s+1$};
\draw[fill] (2,8) circle [radius=0.25]; 
    \node[right] at (2,8) {$v$};
\draw[fill] (6,14) circle [radius=0.1]; 
    \node[above] at (6,14) {$1$};
\draw[fill] (3,13) circle [radius=0.1]; 
    \node[above left] at (3,13) {$2$};
\draw[fill] (12,8) circle [radius=0.1]; 
    \node[right] at (12,8) {$j$};
\draw[fill] (11,5) circle [radius=0.1]; 
    \node[right] at (11,5) {$s+2$};
\draw[fill] (1,5) circle [radius=0.1]; 
    \node[left] at (1,5) {$s-2$};
\end{tikzpicture}
\caption{A picture for Case (xi)}
 \label{fig: case (xi)}
\end{figure}
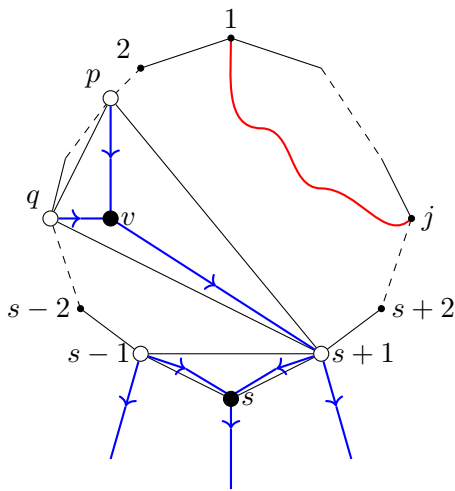

\item$\mathit{2<i<j=s<n}$. From Case (vi), we deduce that $\ddeg_{\mathcal{G}_{\Delta}}(p_{2,s+1})=C^{s+1,1}_{s+1,1}=0$. This implies that we can find a path from $1$ to $s+1$ of plabic degree 0. 
Moreover, since there is an edge between $s+1$ and $s$ oriented towards the latter, we have just shown that there exists a path from 1 to $s$ which does not contribute to the plabic degree of $p_{is}$. 
It follows that $\ddeg_{\mathcal{G}_{\Delta}}(p_{is})=\ddeg_{\mathcal{G}_{\Delta}}(p_{i,s+1})$.
By Case (vii), $\ddeg_{\mathcal{G}_{\Delta}}(p_{i,s+1})=\ddeg_X(p_{i,s+1})$, hence it suffices to show that $\ddeg_X(P_{is})=\ddeg_X(P_{i,s+1})$: this follows from Lemma~\ref{Lem:C-nb} (3) and the fact that no internal arcs end at $s$.

\item  $\mathit{2<i<j=s=n}$.  If $s=n$, then $(1, n-1)\in\Delta_d$ and the edge between $1$ and $n$ has to be oriented towards $n$, since $1$ is a white vertex and has already an edge going in. 
This implies that only the path from 2 to $i$ contributes to the plabic degree of $p_{in}$ and hence $\ddeg_{\mathcal{G}_{\Delta}}(p_{in})=\ddeg_{\mathcal{G}_{\Delta}}(p_{1i})$. 
On the other hand, $\ddeg_X(p_{1i})=C^{i,1}_{i,1}+C^{2,i-1}_{1,1}$ and $\ddeg_X(p_{in})=C^{i,1}_{i,1}+C^{n,1}_{n,i-1}$. 
Since there is no internal arc meeting the vertex $n$,  the connection numbers $C^{2,i-1}_{1,1}$ and $C^{n,1}_{n,i-1}=C_{n,1}^{n,i-1}$ coincide. 
To conclude we hence have to show that $\ddeg_X(p_{1i})=\ddeg_{\mathcal{G}_{\Delta}}(p_{1i})$, but this has been dealt with in Case (i).
\end{enumerate}

\subsubsection*{Main theorem}\label{sec:maintheorem}
Let $\Delta=\Delta_d\cup\Delta_e$ be a triangulation of $D_n$ and let  $T_{\Delta}$ be the labelled tree corresponding to $\Delta$.
Recall that the tree degrees give the weight vector $\mathbf{w}_{T_{\Delta}} = (-\deg_{T_\Delta}(p_{ij}))_{ij}\in\mathbb R^{\binom{n}{2}}$.
Let $\init_{T_{\Delta}}(I_{2,n}) = \init_{\mathbf{w}_{T_{\Delta}}}(I_{2,n})$.
Similarly, consider $\mathbf{w}_P=(\deg_{\mathcal G_\Delta}(p_{ij}))_{ij}\in\mathbb R^{\binom{n}{2}}$ the weight vector associated to a plabic graph $\mathcal G_\Delta $ from Definition~\ref{def: plabic deg} and let 
$\init_{\mathcal{G}_{\Delta}}(I_{2,n}) = \init_{\mathbf{w}_{\mathcal{G}_{\Delta}}}(I_{2,n})$.

We are now prepared to prove the main theorem as stated in the introduction. We restate it here with the notation introduced in the previous paragraphs.

\begin{theorem}\label{Thm:main}
For a given triangulation $\Delta$ of $D_n$, we have  $\init_{T_{\Delta}}(I_{2,n})= \init_{\mathcal{G}_{\Delta}}(I_{2,n})$.
\end{theorem}

\begin{proof}
Recall that from Definition~\ref{def:treedeg} that $\mathbf w_{T_\Delta}\in \trop(\Gr(2,\mathbb C^{n}))$. More precisely, $\mathbf w_{T_\Delta}$ lies in the relative interior of a maximal cone $C_\Delta$ of $\trop(\Gr(2,\mathbb C^{n}))$. To prove the theorem, it suffices to show that $\mathbf w_{\mathcal G_\Delta}$ lies in the relative interior of the same cone.

First note that every arc $(a,b)$ of $\Delta$ connecting $[ l, i-1]$ and $[ j, k-1]$ divides the disk $D_n$ into two parts. 
One of these two parts contains  the marked points in $[i, j-1]$ and has empty intersection with the set of marked points $[k,l-1]$. 
We deduce that every internal arc connecting $[ i,j-1 ]$ and $[ k,l-1]$ intersects $(a,b)$ and therefore $C_{j,k-1}^{l,i-1}\neq 0$ implies $C_{i,j-1}^{k,l-1} = 0$. 
The same argument, applied to $(l,i,j,k)$ instead of $(i,j,k,l)$, shows the opposite implication. 
We conclude that $C_{j,k-1}^{l,i-1}\neq 0$ if and only if $C_{i,j-1}^{k,l-1} = 0$.

From Proposition~\ref{x-prop} it follows that for every $1\leq i<j<k<l\leq n$, the minimum of the numbers $x_{ij}+x_{kl}, x_{ik}+x_{jl}, x_{il}+x_{jk}$ is attained exactly twice.
In particular, by Theorem~\ref{Thm:mainDeg}, this implies that $\init_{\mathcal G_\Delta}(p_{ij}p_{kl}-p_{ik}p_{ji}+p_{il}p_{jk})$ is a binomial.
Further, by Propositions~\ref{Prop:treeA} and \ref{x-prop} we deduce
\[
\init_{\mathcal G_\Delta}(p_{ij}p_{kl}-p_{ik}p_{ji}+p_{il}p_{jk}) = \init_{T_\Delta}(p_{ij}p_{kl}-p_{ik}p_{ji}+p_{il}p_{jk}).
\]
As $\init_{T_{\Delta}}(I_{2,n})$
is generated by the initial terms of the Pl{\"u}cker relations with respect to
$\mathbf{w}_{T_{\Delta}}$ we conclude $\init_{\mathcal G_\Delta}(I_{2,n})=\init_{T_\Delta}(I_{2,n})$.
\end{proof}

\begin{corollary}\label{cor: deg for plabic}
For every plabic graph $\mathcal G$ there exists a maximal prime cone $C\subset \trop(\Gr(2,\mathbb C^{n}))$ with $\mathbf w_{\mathcal{G}}\in C^\circ$. In particular, $\mathcal G$ induces a toric degeneration of $\Gr(2,\mathbb C^{n})$.
\end{corollary}

\begin{proof}
Recall that for $\Gr(2,\mathbb C^{n})$ there is a bijection between seeds of the cluster algebra $\mathbb C[\Gr(2,\mathbb C^{n})]$ and triangulations $\Delta$ of $D_n$ by \cite{FZ02} and \cite{Sco06}. Plabic graphs for general $\Gr(k,\mathbb C^{n})$ encode ($\mathcal A$-)seeds of $\mathbb C[\Gr(k,\mathbb C^{n})]$ given purely in terms of Pl\"ucker coordinates (see e.g. \cite[(4.1)]{RW17}). 
In particular, for $k=2$ there exists a triangulation $\Delta$ of $D_n$ for every such plabic seed with $\mathcal{G}_\Delta=\mathcal G$. 
In fact, for $\Gr(2,\mathbb C^{n})$ there is a bijection between seeds and plabic graphs, as all seeds consist of only Pl\"ucker coordinates.
Applying Algorithm~\ref{alg:tree from triang} we obtain $T_\Delta$ and a corresponding cone $C_{T_\Delta}$. 
By Theorem~\ref{Thm:main} $\mathbf w_{\mathcal G}\in C_{T_\Delta}^\circ$. The toric degeneration is then given by the family described in \eqref{eq: groebner family}.
\end{proof}

\subsection{Mutation and initial ideals}
Recall from \S\ref{subsec: cluster gr2n} cluster mutation in the case of $\mathbb C[\Gr(2,\mathbb C^{n})]$ following \cite{Sco06}.
The aim of this section is to understand the effect cluster mutation has on the inital ideal $\init_{\Delta}(I_{2,n})=\init_{T_\Delta}(I_{2,n})$. 
We translate mutation in terms of flipping arcs to rooted trees and analyze how this changes the associated weight vectors $\mathbf w_{T_\Delta}$.

\begin{center}
\begin{figure}
\centering
\begin{tikzpicture}[scale=0.4]
\node[below] at (3,1) {4};
\node[below] at (6,1) {5};
\node[right, red] at (8,3) {6};
\node[right, red] at (8,5) {7};
\node[above] at (6,7) {8};
\node[above, red] at (3,7) {1};
\node[left] at (1,5) {2};
\node[left, red] at (1,3) {3};

\draw (3,1) --(6,1) -- (8,3) -- (8,5) -- (6,7) -- (3,7) -- (1,5) -- (1,3) -- (3,1);
\draw (6,1) -- (1,3) -- (8,3) -- (3,7) -- (8,5);
\draw[red, ultra thick] (8,3) -- (3,7);
\draw (1,3) -- (3,7);
\node at (10.5,4.5) {$\xrightarrow{\mu_{(1,6)}}$};

\begin{scope}[xshift=12cm, scale=1]
\node[below] at (3,1) {4};
\node[below] at (6,1) {5};
\node[right, red] at (8,3) {6};
\node[right, red] at (8,5) {7};
\node[above] at (6,7) {8};
\node[above, red] at (3,7) {1};
\node[left] at (1,5) {2};
\node[left, red] at (1,3) {3};

\draw (3,1) --(6,1) -- (8,3) -- (8,5) -- (6,7) -- (3,7) -- (1,5) -- (1,3) -- (3,1);
\draw (6,1) -- (1,3);
\draw (3,7) -- (8,5);
\draw (8,3) -- (1,3);
\draw (1,3) -- (3,7);
\draw[red, ultra thick] (1,3) -- (8,5);
\end{scope}
\end{tikzpicture}
\caption{Mutation at the arc $(1,6)$.}
\label{fig:triangulation:mutation}
\end{figure}
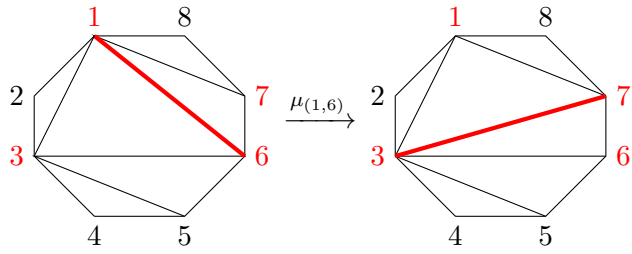
\end{center}

Let $\Delta$ be a triangulation of $D_n$ and $T_{\Delta}$ the associated rooted tree.
Recall that $T_\Delta$ is a rooted tree with root corresponding to the leaf labelled by $1$ and counterclockwise labelling of the leaves $1,2, \ldots, n$.
Then $T_{\Delta}$ can be seen as a directed graph, where an edge $\ba - \bb$ gets an orientation $\ba \longrightarrow \bb$, if the distance of $\ba$ to $1$ is less than the distance of $\bb$ to $1$.

Let $\ba$ be an internal vertex of $T_{\Delta}$ and $\ba \longrightarrow \bb, \ba \longrightarrow \bc$ be the adjacent edges. 
We say $\bc$ is the \textit{left child} of $\ba$ and $\bb$ is the \textit{right child} of $\ba$ if the labels of the leaves reachable by a directed path (with respect to orientation) from $\bb$ are smaller than those reachable from $\bc$, having in mind that leaves are labelled counterclockwise by $1,\cdots,n$ (see Figure~\ref{fig:tree:mutation}).
The following definition formulates on the level of trees how mutation of triangulations deforms the corresponding trees. It coincides with the notation of mutation of phylogenetic trees by Buczynska and Wisniewski in \cite{BW07}.

\begin{definition}\label{def: tree mut}
Let $\ba \longrightarrow \bb$ be an internal edge of $T_{\Delta}$ and $\bb$ be the right child of $\ba$. We further denote by $\bc$ the left child of $\ba$, $\bd$ the right child of $\bb$ and $\be$ the left child of $\bb$ (see Figure~\ref{fig:tree:mutation}). The rooted tree $\mu_{\ba\to \bb}(T_\Delta)$ is the tree obtained from $T_\Delta$ by defining $\bd$ to be the right child of $\ba$, $\bb$ to be the left child of $\ba$, $\bc$ to be the left child of $\bb$ and $\be$ to be the right child of $\bb$. All other edges of the tree remain unchanged.
\end{definition} 

\begin{figure}
\centering
\begin{tikzpicture}[scale=0.3]
\draw[thick, ->] (10,12) -- (10,10); 
\draw[ultra thick, red, ->] (10,10) -- (8,8);
\draw[thick, ->] (8,8) -- (6,6);
\draw[thick, ->] (10,10) -- (12,8);
\draw[thick, ->] (8,8) -- (10,6);

\node[red, left] at (10,10.25) {$\ba$};
\node[red, left] at (8.1,8.25) {$\bb$};
\node[red, below] at (6,6) {$\bd$};
\node[red, below] at (10,6) {$\be$};
\node[red, below] at (12.25,8) {$\bc$};

\node at (16,8) {$\xrightarrow{\mu_{{\bf a}\to{\bf b}}}$};

\begin{scope}[xshift=12cm]
\draw[thick, ->] (10,12) -- (10,10); 
\draw[thick, ->] (10,10) -- (8,8);
\draw[ultra thick, red, ->] (10,10) -- (12,8);
\draw[thick, ->] (12,8) -- (14,6);
\draw[thick, ->] (12,8) -- (10,6);

\node[red, right] at (10,10.25) {$\ba$};
\node[red, right] at (12.1,8.25) {$\bb$};
\node[red, below] at (7.9,8) {$\bd$};
\node[red, below] at (10,6) {$\be$};
\node[red, below] at (14,6) {$\bc$};
\end{scope}
\end{tikzpicture}
    \caption{Mutation on trees}
    \label{fig:tree:mutation}
\end{figure}
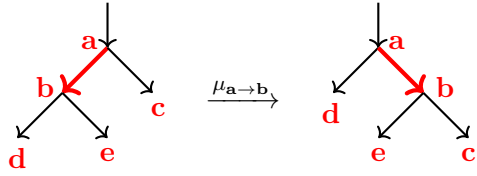

\begin{remark}\label{rem-tree-mut}
Let $(a_r, b_r) \in \Delta_d$ be the arc corresponding to $\ba \longrightarrow \bb$ and $(a_r', b_r')$ be the arc obtained by mutating at $(a_r, b_r)$. Set $\Delta'$ be the triangulation of the $n$-gon obtained through the internal arcs
\[
\Delta'_d = \Delta_d \cup \{(a_r', b_r')\} \setminus \{(a_r, b_r)\}.
\]
Then, by construction,
\[
T_{\Delta'} =  \mu_{\ba \rightarrow \bb}(T_{\Delta}).
\]
\end{remark}

\begin{lemma}
Given two arbitrary trees $T_\Delta$ and $T_{\Delta'}$, then there is a sequence of mutations on inner edges transforming $T_{\Delta}$ into $T_{\Delta'}$.
\end{lemma}
\begin{proof}
This is true for the triangulations $\Delta$ and $\Delta'$ on the $n$-gon $G_n$ and hence by Remark~\ref{rem-tree-mut} for the labelled trees $T_\Delta$ and $T_{\Delta'}$.
\end{proof}

From the direction of edges in $T_\Delta$ we obtain a partial order on the vertices. We set for two vertices $x,y\in T_\Delta$ (internal or leaves)
\[
x\le y, \text{ if } \exists \text{ a directed path } x\to y \text{ in } T_\Delta. 
\]
We further define for a vertex $x$ the subset of leaves $[n]_{\le x}=\{k\in[n]\mid k\le x\}$. Similarly we define $[n]_{\not \le x}$.

Let $\ba \to \bb$ be an internal edge of $T_{\Delta}$, and we keep the same notation as in Figure~\ref{fig:tree:mutation} with $\bc, \bd, \be$. Observe that the set of leaves $[n]$ can be decomposed with respect to $\ba\to\bb$ as follows
\[
[n]=[n]_{\not\le\bb}\cup [n]_{\le \bb}=([n]_{\not \le \ba}\cup [n]_{\le \bc})\cup ([n]_{\le \bd}\cup [n]_{\le \be}).
\]
After mutation the edge $\ba'\to\bb'$ in $T_{\Delta'}=\mu_{\ba\to\bb}(T_\Delta)$ separates $[n]_{\not\le\bb'}=[n]_{\le \bd'}\cup [n]_{\not\le \ba'}$ from $[n]_{\le \bb'}=[n]_{\le \be'}\cup [n]_{\le\bc'}$. Note that $[n]_{\le \mathbf x'}=[n]_{\le \mathbf x}$ for a vertex $\mathbf x\not=\bb$ in $T_\Delta$ resp. $T_{\Delta'}$. An example is shown in Figure~\ref{fig:tree mut exp}.

\begin{figure}
\centering
\begin{tikzpicture}[scale=0.2]
\node[above] at (10,12) {1};
\node[below] at (1,0) {2};
\node[below] at (3,0) {3};
\node[below] at (4,0) {4};
\node[below] at (6,0) {5};
\node[below] at (9,0) {6};
\node[below] at (12,0) {7};
\node[below] at (15,0) {8};

\draw[thick] (1,0) -- (2,2) -- (3,0);
\draw[thick] (2,2) -- (6,6) -- (7,4) -- (5,2) -- (4,0);
\draw[thick] (5,2) -- (6,0);
\draw[thick] (7,4) -- (9,0);
\draw[thick] (6,6) -- (8,8) -- (12,0);
\draw[thick] (8,8) -- (10,10) -- (15,0);
\draw[thick] (10,10) -- (10,12);
\draw[red, ultra thick] (8,8) -- (6,6);

\draw[fill, red] (8,8) circle [radius=0.2];
\node[above, red] at (8,8) {$\bf a$};
\draw[fill, red] (6,6) circle [radius=0.2];
\node[above, red] at (5.75,6) {$\bf b$};
\draw[fill, red] (2,2) circle [radius=0.2];
\node[above, red] at (2,2) {$\bf d$};
\draw[fill, red] (7,4) circle [radius=0.2];
\node[right, red] at (7,4) {$\bf e$};
\draw[fill, red] (12,0) circle [radius=0.2];
\node[right, red] at (12,0) {$\bf c$};

\node at (18,5) {$\xrightarrow{\mu_{\ba \to \bb}}$};

\begin{scope}[xshift=20cm]
\node[above] at (10,12) {1};
\node[below] at (1,0) {2};
\node[below] at (3,0) {3};
\node[below] at (4,0) {4};
\node[below] at (6,0) {5};
\node[below] at (9,0) {6};
\node[below] at (12,0) {7};
\node[below] at (15,0) {8};

\draw[thick] (1,0) -- (2,2) -- (3,0);
\draw[thick] (2,2) -- (8,8) -- (9,6);
\draw[red, ultra thick] (8,8) -- (9,6);
\draw[thick] (4,0) -- (5,2) -- (6,0);
\draw[thick] (5,2) -- (7,4) -- (9,0);
\draw[thick] (7,4) -- (9,6) -- (12,0);
\draw[thick] (8,8) -- (10,10) -- (15,0);
\draw[thick] (10,10) -- (10,12);

\draw[fill, red] (8,8) circle [radius=0.2];
\node[above, red] at (8,8) {$\bf a'$};
\draw[fill, red] (9,6) circle [radius=0.2];
\node[right, red] at (9,6) {$\bf b'$};
\draw[fill, red] (2,2) circle [radius=0.2];
\node[above, red] at (2,2) {$\bf d'$};
\draw[fill, red] (7,4) circle [radius=0.2];
\node[right, red] at (7,4) {$\bf e'$};
\draw[fill, red] (12,0) circle [radius=0.2];
\node[right, red] at (12,0) {$\bf c'$};

\end{scope}
\end{tikzpicture}
\caption{Mutation of the trees corresponding to Figure~\protect{\ref{fig:triangulation:mutation}}. Here we have $n=8$ and $[8]_{\le \bd}=\{2,3\}, [8]_{\le \be}=\{4,5,6\}, [8]_{\le \bc}=\{7\}$ and $[8]_{\not\le\ba}=\{1,8\}$.}
    \label{fig:tree mut exp}
\end{figure}

Consider $i,j,k,l\in[n]$ pairwise distinct and the paths between each of them in $T_\Delta$.
Then there is a unique non-intersecting pair of paths, say $i\to j$ and $k\to l$. Let $\mathbf x$ be the first (in direction of the paths as indictaed) vertex in which the paths $i\to k$ and $j\to k$ intersect. 
Similary, let $\mathbf y$ the last vertex that lies on bath paths $i\to k$ and $i\to l$.

\begin{definition}\label{def: cotracted tree}
We define the \emph{contracted tree} $T_\Delta\vert_{i,j,k,l}$ to be the trivalent tree with four leaves $i,j,k,l$ obtained from $T_\Delta$ by contracting all edges on the paths $i\to\mathbf x$, $j\to\mathbf x$, $\mathbf x\to\mathbf y$, $\mathbf y\to k$, and  $\mathbf y\to l$ to one edge only and deleting all other edges from $T_\Delta$.
With an edge of $T_\Delta\vert_{i,j,k,l}$ we associate the number of internal edges of $T_\Delta$ it contracted (see Figure~\ref{fig: contracted tree}). 
\end{definition}
For example, the edge $i-\mathbf x$ in $T_\Delta\vert_{i,j,k,l}$ obtains the weight $d_{i\mathbf x}=\#\{\text{internal edges on path }i\to\mathbf x \text{ in } T_\Delta \}$.
Note that the sums of edge weights along a path between two leaves in $T_\Delta\vert_{i,j,k,l}$ equals the corresponding $T_\Delta$-degree.
In particular, the tree $T_\Delta\vert_{i,j,k,l}$ encodes all necessary information for computing the initial form $\init_{\Delta}(p_{ij}p_{kl}-p_{ik}p_{jl}+p_{il}p_{jk})$.

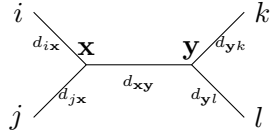
\begin{figure}[ht]
\begin{center}
\begin{tikzpicture}[scale=.7]

\draw (0,0) -- (1,1) -- (0,2);
\draw (1,1) -- (3,1) -- (4,2);
\draw (3,1) -- (4,0);

\node[left] at (0,0) {$j$};
\node[left] at (0,2) {$i$};
\node[above] at (1,1) {$\mathbf x$};
\node[above] at (3,.9) {$\mathbf y$};
\node[right] at (4,2) {$k$};
\node[right] at (4,0) {$l$};

\node[below] at (.75,.75) {\tiny $d_{j\mathbf{x}}$};
\node[below] at (.25,1.75) {\tiny $d_{i\mathbf{x}}$};
\node[below] at (2,1) {\tiny $d_{\mathbf{xy}}$};
\node[below] at (3.75,1.75) {\tiny $d_{\mathbf{y}k}$};
\node[below] at (3.25,.75) {\tiny $d_{\mathbf{y}l}$};

\end{tikzpicture}
\end{center}\caption{The contracted tree $T_{\Delta}\vert_{i,j,k,l}$ with edge weights.}\label{fig: contracted tree}
\end{figure}

We analyze further the tree degrees and their relation to internal edges of $T_\Delta$. For an internal edge $\ba\to\bb$ in $T_\Delta$ consider $r_{\ba\to\bb}\in\mathbb R^{\binom{n}{2}}$ given by
\begin{eqnarray}\label{eq: rays of trop cone}
(r_{\ba\to\bb})_{i,j}=\left\{
    \begin{matrix}
    1, & \text{ if } i\in[n]_{\le \bb},j\in[n]_{\not\le \bb} \text{ or vice versa}\\
    0, & \text{ otherwise.}
    \end{matrix}
\right.
\end{eqnarray}

\begin{remark}
The vector $-r_{\ba\to\bb}$ is in fact a ray generator for the maximal cone $C_T\subset\trop(\Gr(2,\mathbb C^{n}))$ by \cite[Equation~(8)]{SS04} and \cite[Theorem~4.3.5]{M-S}.
\end{remark}

\begin{lemma} \label{lem:rays and tree deg}
For a labelled trivalent tree $T$ we have $\mathbf w_T=-\sum_{\ba\to\bb \text{ internal edge of }T}r_{\ba\to\bb}$.
\end{lemma}

\begin{proof}
Consier $i<j$ in $[n]$. Then $(r_{\ba\to\bb})_{ij}$ is $1$ if $\ba\to\bb$ lies on the path from $i$ to $j$ in $T$ and zero otherwise. In particular, this implies 
\[
\sum_{\ba\to\bb \text{ internal edge of }T}(r_{\ba\to\bb})_{ij}=\deg_{T}(p_{ij}).
\]
The claim follows as $(\mathbf w_T)_{ij}=-\deg_T(p_{ij})$.
\end{proof}

The following lemma follows from Lemma~\ref{lem:rays and tree deg} and the fact that by definition mutation on trees only changes one internal edge and keeps all other edges unchanged.

\begin{lemma}\label{lem:mutation tree weight}
Let $\Delta$ be a triangulation of $D_n$ and $T_\Delta$ the corrresponding tree. Consider an internal edge $\ba\to\bb$ of $T_\Delta$ and let $\mu_{\ba\to\bb}(T_\Delta)=T_{\Delta'}$. Denote the internal edge of $T_{\Delta'}$ obtained from $\ba\to\bb$ by $\ba'\to\bb'$. Then
\[
\mathbf w_{T_{\Delta'}}=\mathbf w_{T_\Delta} +r_{\ba\to\bb} - r_{\ba'\to\bb'}.
\]
\end{lemma}

\begin{theorem}\label{Thm:muta}
Let $\Delta$ be a triangulation of the $D_n$ and $\ba \to \bb$ be an internal edge of $T_\Delta$ and $T_{\Delta'}=\mu_{\ba\to\bb}(T_\Delta)$. Consider the Pl\"ucker relation $R_{i,j,k,l}=p_{ij}p_{kl}-p_{ik}p_{jl}+p_{il}p_{jk}\in I_{2,n}$ for $i,j,k,l$ cyclically ordered. Assume $\init_\Delta(R_{i,j,k,l})=p_{il}p_{jk}-p_{ik}p_{jl}$, then
\[
\init_{\Delta'}(R_{i,j,k,l})=\left\{
    \begin{matrix}
    p_{ij}p_{kl}-p_{il}p_{jk}, & \text{ if } i\in[n]_{\not\le \ba},j\in[n]_{\le\bc},k\in[n]_{\le \bd},l\in[n]_{\le \be}\\
    \init_{\Delta}(R_{i,j,k,l}), & \text{ otherwise.}
    \end{matrix}
\right.
\]
\end{theorem}

\begin{proof}
Assume $i,j,k,l$ are such that $T_\Delta\vert_{i,j,k,l}$ is of shape as in Figure~\ref{fig: contracted tree}, if necessary reorder them. We distinguish two cases, $\ba\to\bb$ contributes to the edge $\mathbf x-\mathbf y$ in $T_\Delta\vert_{i,j,k,l}$ or it does not.

If $\ba\to\bb$ does not contribute to $\mathbf x-\mathbf y$ first observe that we are in the "otherwise" case of the claim. In this case either $\ba\to\bb$ was deleted from $T_\Delta$ by construction $T_\Delta\vert_{i,j,k,l}$ (in which case the claim follows) or it contributes to one of the edges adjacent to a leaf in $T_\Delta\vert_{i,j,k,l}$. Assume without loss of generality that $\ba\to\bb$ contributes to the edge $i-\mathbf x$. Then by Lemma~\ref{lem:mutation tree weight} the $T_\Delta$-degrees of $p_{ij},p_{ik}$ and $p_{il}$ are all changed by $\pm1$ while all others remain the same after mutation. In particular, this implies the claim.

If $\ba\to\bb$ contributes to $\mathbf x-\mathbf y$ we further distinguish depending on $q=\#(\{\mathbf x,\mathbf y\}\cap\{\ba,\bb\})$:
\begin{itemize}
    \item[$q=0$] In this case $\ba'\to\bb'$ does not contribute to the edge $\mathbf x'-\mathbf y'$ of $T_{\Delta'}\vert_{i,j,k,l}$. 
    By the proof of Lemma~\ref{lem:rays and tree deg} the $T_\Delta$-degrees of $p_{ik},p_{il},p_{jk},p_{jl}$ differ by $-1$ from the $T_{\Delta'}$-degrees while the others stay unchanged. 
    Hence, $\init_\Delta(R_{i,j,k,l})=\init_{\Delta'}(R_{i,j,k,l})$. 
    \item[$q=1$] We assume without loss of generality $\mathbf x=\ba$ (otherwise relabel accordingly). 
    After mutation, the edge $\ba'\to\bb'$ contributes to either $i-\mathbf x'$ or $j\mathbf x'$.
    We treat the first case, the argument for the second is the same. In this case the $T_\Delta$-degrees of $p_{ij},p_{ik},p_{il}$ equal the $T_{\Delta'}$ degrees while all others differ by $-1$. 
    In particular, the degree of each monomial in $R_{i,j,k,l}$ decreases by $1$ and so $\init_\Delta(R_{i,j,k,l})=\init_{\Delta'}(R_{i,j,k,l})$.
    \item[$q=0$] Observe that this is (up to relabelling if necessary) the case $i\in[n]_{\not\le \ba},j\in[n]_{\le\bd},k\in[n]_{\le \be},l\in[n]_{\le \bc}$. The tree $T_{\Delta'}\vert_{i,j,k,l}$ has cherries $i,k$ and $j,l$. In particular, $\deg_{\Delta'}(p_{ik}p_{jl})<\deg_{\Delta'}(p_{ij}p_{ik})=\deg_{\Delta'}(p_{il}p_{jk})$ and the claim follows. 
\end{itemize}
\end{proof}

\chapter{Flag and Schubert varieties}\label{chap:flag}

\section{Preliminary notions}\label{sec:pre flag}
\begin{definition}
A \emph{complete flag} in the vector space ${\mathbb C}^n$ is a chain
\[
\mathbb{V}:\ \{0\}= V_0\subset V_1\subset\cdots\subset V_{n-1}\subset V_n=\mathbb C^n
\]
of vector subspaces of $\mathbb C^n$ with ${\rm dim}_{\mathbb C }(V_i) = i$. 
\end{definition}

The set of all complete flags in $\mathbb C^n$ is denoted by $\Flag_n$ and it has the structure of an algebraic variety. More precisely, it is a subvariety of the product of  Grassmannians $\Gr(1,\mathbb C^n)\times \Gr(2,\mathbb C ^n)\times\cdots \times \Gr(n-1,\mathbb C^n)$. 

Composing with the Pl\"ucker embeddings of the Grassmannians, $\Flag_n$ becomes a subvariety of 
$\mathbb{P}^{{\binom{n}{1}}-1}\times\mathbb{P}^{{\binom{n}{2}}-1}\times \cdots\times\mathbb{P}^{{\binom{n}{n-1}}-1}$ 
and so we can ask for its defining ideal $I_n$.  
Each point in the flag variety can be represented by an $n\times n$-matrix $M=(x_{i,j})$ whose first $d$ rows generate $V_d$. Each $V_d$ corresponds to a point in a Grassmannian. Moreover, they satisfy the condition $V_d\subset V_{d+1}$ for $d=1,\ldots,n-1$. 
In order to compute the  ideal $I_n$  defining $\Flag_n$ in   $\mathbb{P}^{{\binom{n}{1}}-1}\times\mathbb{P}^{{\binom{n}{2}}-1}\times \cdots\times\mathbb{P}^{{\binom{n}{n-1}}-1}$ we have to translate the inclusions  $V_d\subset V_{d+1}$ into polynomial equations. We define the map
\begin{equation}\label{eq:def ideal flag}
\varphi_n:  \mathbb C[p_J\mid \varnothing\neq J\subsetneq [n]]\rightarrow \mathbb C[x_{i,j}\mid i,j\in[n]]   
\end{equation}
sending each Pl\"ucker variable
$p_J$ to the determinant of the submatrix of $M$ with row indices $[|J|]$ and column indices $J$. The ideal $I_n$ of $\Flag_n$ is the kernel of $\varphi_n$.

\begin{example}\label{exp: ideal flag4}
Consider $\Flag_4\hookrightarrow \mathbb P^3\times\mathbb P^5\times\mathbb P^3$. 
The Plücker variables are 
\[
p_1,p_2,p_3,p_4,p_{12},p_{13},p_{23},p_{14},p_{24},p_{34},p_{123},p_{124},p_{134},p_{234}.
\]
Then $\ker(\varphi_4)$ contains the Plücker relation defining $\Gr(2,\mathbb C^4)$:
\[
p_{12}p_{34}-p_{13}p_{24}+p_{14}p_{23}.
\]
And further relations between Plücker variables associated to different Grassmannians:
\begin{eqnarray*}
&p_{4}p_{23}-p_{3}p_{24}+p_{2}p_{34}, \ 
p_{4}p_{13}-p_{3}p_{14}+p_{1}p_{34}, &\\
&p_{4}p_{12}-p_{2}p_{14}+p_{1}p_{24}, \ 
p_{3}p_{12}-p_{2}p_{13}+p_{1}p_{23}, &\\
&p_{34}p_{124}-p_{24}p_{134}+p_{14}p_{234},\  
p_{34}p_{123}-p_{23}p_{134}+p_{13}p_{234},&\\
&p_{24}p_{123}-p_{23}p_{124}+p_{12}p_{234}, \ 
p_{14}p_{123}-p_{13}p_{124}+p_{12}p_{134}, &\\
&p_{4}p_{123}-p_{3}p_{124}+p_{2}p_{134}-p_{1}p_{234}.&
\end{eqnarray*}
The above relations form a complete list of generators for the ideal $I_4$.
\end{example}

There is an action of $S_n\rtimes \mathbb Z_2$  on $\Flag_n$. The symmetric group acts by permuting the columns of $M$. The action of $\mathbb Z_2$ maps a complete flag $\mathbb V$ to its complement, which is defined to be
\[
\mathbb{V}^{\bot}:\ \{0\}= V_n^{\bot} \subset V_{n-1}^{\bot}\subset\cdots\subset V_{1}^{\bot}\subset V_0^{\bot}=\mathbb C^n.
\]
We hence do computations in \S\ref{sec:BLMM} up to $S_n \rtimes \mathbb Z_2$-symmetry. 

Recall our notation for $SL_n$ from \S\ref{sec:pre rep theory}. 
Representing flags by matrices corresponds to realizing  the flag variety as the quotient $SL_n/B$. 
We construct line bundles on $SL_n/B$ as follows. Consider a weight $\lambda\in\Lambda^+$, it is a character of $B$ (i.e. a morphism of algebraic groups $\lambda:B\to\mathbb C^*$). 
We have a free action of $B$ on $SL_n\times \mathbb C$, which for $b\in B,g\in SL_n$ and $t\in T$ is given by 
\[
b(g,t):=(gb^{-1},\lambda(b)t).
\]
Let $L_\lambda$ be the fibre product $SL_n\times_B\mathbb C=(SL_n\times \mathbb C)/B$. Then there is a map 
\[
L_\lambda\to SL_n/B, \text{ given by } (g,t)B \mapsto gB. 
\]
It follows that $L_\lambda$ is the total space of a line bundle over $SL_n/B$ called the \emph{homogeneous line bundle associated to the weight $\lambda$}. These line bundles satisfy $L_{m\lambda}=L_\lambda^{\otimes m}$ for $m\ge 1$ and are ample, if $\lambda\in\Lambda^{++}$.
By the \emph{Borel-Weil-Theorem} we have the following correspondence between line bundles $L_\lambda$ for $\lambda\in\Lambda^{++}$ and irreducible highest weight representations
\[
H^0(SL_n/B,L_\lambda)^*\cong V(\lambda).
\]
Recall that the highest weight representation is cyclically generated by a highest weight vector $v_\lambda \in V(\lambda)$. Then the above correspondence induces an embedding
\[
SL_n/B\hookrightarrow \mathbb P(V(\lambda)), \  gB \mapsto g[v_\lambda].
\]
In particular, we can realize the homogeneous coordinate ring of the flag variety as $\mathbb C[SL_n/B]=\bigoplus_{k\ge 0}V(k\lambda)^*$. Similarly, we obtain $\mathbb C[SL_n/U]=\bigoplus_{\lambda\in\Lambda^+} V(\lambda)$ which is a consequence of the \emph{Peter-Weyl-Theorem}. The quasi-affine variety $SL_n/U$ is sometimes also called \emph{base affine space}.

In the next section we consider Schubert varieties. These are subvarieties of $SL_n/B$ indexed by Weyl group elements $w\in S_n=N_{SL_n}(T)/T$. 
We identify $w\in S_n$ with a coset representative in the quotient $N_{SL_n}(T)/T$ and consider the \emph{Bruhat cell} $BwB\subset SL_n$.
The quotient $BwB/B$ is called \emph{Schubert cell}.
\begin{definition}\label{def: Schubert variety}
For $w\in S_n$ the \emph{Schubert variety} $X_w\subset SL_n/B$ is defined as the Zariski closure $X_w:=\overline{BwB/B}$.
\end{definition}
Schubert varieties are normal, not necessarily smooth (but if singular having only rational singularities) subvarieties of the flag variety.
Their dimension equals the length of the associated Weyl group element, i.e. $\dim X_w=\ell(w)$.
The line bundles $L_\lambda$ can be restricted to Schubert varieties and the Borel-Weil Theorem generalizes as follows.
Fix $w\in S_n$ and $\lambda\in\Lambda^+$, then
\[
H^0(X_w,L_\lambda)^*\cong V_w(\lambda),
\]
where $V_w(\lambda)$ is the Demazure module (see Definition~\ref{def: Demazure module} in \S\ref{sec:pre rep theory})
Observe, that the Borel-Weil-Theorem is in fact a special case as we have $X_{w_0}=SL_n/B$ and $V_{w_0}(\lambda)=V(\lambda)$.
Using observation one can generalize many constructions for $SL_n/B$ that rely on Borel-Weil to Schubert varieties. 
An example of this incidence can be found in the following section when studying string polytopes for flag and Schubert varieties.

\newpage
\section{String cones and the Superpotential}\label{sec:BF}

In this section we study the combinatorics of pseudoline arrangements. We associate to each in dual ways two collections of polyhedral objects (consisting of two polyhedral cones and a $\mathbb R^{n-1}$-family of polytopes).
We show they are unimodularly equivalent and relate the cones to the geometry of flag and Schubert varieties.
We prove that one of them is the weighted string cone by Littelmann \cite{Lit98} (see also Berenstein-Zelevinsky \cite{BZ01}). 
It was used by Caldero \cite{Cal02} to degenerate flag and Schubert varieties to toric varieties.
For the other we show that it is closely related to the framework of cluster varieties and mirror symmetry by Gross-Hacking-Keel-Kontsevich \cite{GHKK14}. 
For the flag variety the cone is the tropicalization of their superpotential while for Schbert varieties a restriction of the superpotential is necessary.
In their framework they also give a construction of toric degenerations using the superpotential.
As a corollary of our combinatorial result we realize Caldero's degenerations as GHKK-degenerations using cluster theory.

The section is structured as follows: we recall pseudoline arrangements and define the two collections of polyhedral objects and unimodular equivalences among them in \S\ref{sec:pa and gp}.
In \S\ref{subsec:string} we show that of the cones is the weighted string cone and in \S\ref{subsec:super} we show how the other arises from the superpotential. 
Then in \S\ref{subsec:apply} we apply our combinatorial result and relate to toric degenrations.\footnote{based on joint work with Ghislain Fourier.}

\subsection{Pseudoline arrangements and Gleizer-Postnikov paths}\label{sec:pa and gp}

Recall our notation for the symmetriy group $S_n$ from \S\ref{sec:pre rep theory}. In the following section we associate for $w\in S_n$ a diagram called a pseudoline arrangement to every reduced expression $\w$. 
These diagrams turn out to be closely related to cluster algebras. In fact, to every pseudoline arrangement one can associate a quiver and then using the construction summarized in \S\ref{sec: prep cluster} define a cluster algebra. 
We start by introducing the combinatorial tools: to a pseudoline arrangement we associate two weighted cones and give a unimodular equivalence between them.

\begin{definition}\label{def:pseusoline arr}
A \emph{pseudoline arrangement} $\pa(\w)$ associated to a reduced expression $\w=s_{i_1} \cdots s_{i_{\ell(w)}}$ is a diagram consisting of $n$ horizontal \emph{pseudolines} $l_1,\dots,l_{n}$ (or short \emph{lines}) labelled at the left end from bottom to top, with crossings indicated by the reduced expression. A reflection $s_i$ indicates a crossing at \emph{level} $i$ (see e.g. Figure~\ref{fig:pa 121}).
\end{definition}
For a given reduced expression $\w=s_{i_1} \cdots s_{i_{l(w)}}$, we associate to each $s_{i_j}$ the positive root $\beta_{i_j}:=s_{i_1}\cdots s_{i_{j-1}}(\alpha_{i_j})$.
Then $\beta_{i_j}=\alpha_{k,m-1}$ for $k,m< n$ and $s_{i_j}$ induces the crossing of the lines $l_k$ and $l_m$ in $\pa(\w)$.
The crossing point is a vertex in the diagram and it is labelled ${(k,m)}$.
As two lines $l_k,l_m$ cross at most once, there is at most one position with label $(k,m)$. 
For a given $w$ the pairs appearing as labels for crossing points are exactly those for which $w(\alpha_{k,m-1})<0$.
Further, the right end of a pseudoline $l_i$ is a vertex labelled $L_i$. 
Let $\pa(\w)_0$ be the set of all vertices in $\pa(\w)$.

\begin{definition}\cite[Definition~2.2]{BFZ05}\label{def:quiver pa}
Let $w\in S_n$ with reduced expression $\w$. 
Then the quiver $Q_{\w}$ associated to $\pa(\w)$ has \emph{vertices} $w_F$ associated to faces $F$ of $\pa(\w)$ and \emph{arrows}:

(1) if two faces are at the same level separated by a crossing then there is an arrow from

left to right (see Figure~\ref{fig:arrow}a);

(2) if two faces are on consecutive levels separated by two crossings then there is an arrow 

from right to left (either upwards or downwards, see Figure~\ref{fig:arrow}b, \ref{fig:arrow}c).

\noindent
Vertices corresponding to unbounded faces are \emph{frozen} and we disregard arrows between them. All the other vertices are called \emph{mutable}.
\end{definition}

\begin{center}
\begin{figure}[ht]
\centering
\begin{tikzpicture}[scale=.8]
\draw[rounded corners] (0,.5) -- (1,.5) -- (2,1.5) -- (3,1.5);
\draw[rounded corners] (0,1.5) -- (1,1.5) -- (2,.5) -- (3,.5);
\draw[thick, ->] (.5,1) -- (2.5,1);
\node at (1.5,-.8) {$a$};

\begin{scope}[xshift=4cm]
\draw[rounded corners] (0,2) -- (1,2) -- (2,1) -- (3,1) -- (4,0) -- (5,0);
\draw[rounded corners] (0,1) -- (1,1) -- (2,2) -- (3,2);
\draw[rounded corners] (2,0) -- (3,0) -- (4,1) -- (5,1);
\draw[thick, ->] (3.5,1.5) -- (1.5,0.5);
\node at (2.5,-.8) {$b$};
\end{scope}

\begin{scope}[xshift=10cm]
\draw[rounded corners] (0,0) -- (1,0) -- (2,1) -- (3,1) -- (4,2) -- (5,2);
\draw[rounded corners] (0,1) -- (1,1) -- (2,0) -- (3,0);
\draw[rounded corners] (2,2) -- (3,2) -- (4,1) -- (5,1);
\draw[thick,->] (3.5,0.5) -- (1.5,1.5);
\node at (2.5,-.8) {$c$};
\end{scope}

\end{tikzpicture}
\caption{Arrows of the quiver arising from the pseudoline arrangement.}\label{fig:arrow}
\end{figure}
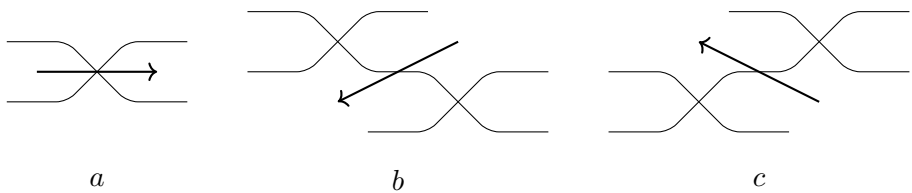
\end{center}

Recall the notion of of quiver mutation from Definition~\ref{def:quiver mutation} in \S\ref{sec: prep cluster}. 

\begin{definition}\label{defn:mutation pa}
Let $w\in S_n$ with reduced expression $\w$. 
A \emph{mutation} of $\pa(\w)$ (resp. of $\w$) is a change of consecutive $s_r s_{r + 1} s_r$ in $\w$ to $s_{r+1} s_r s_{r+1}$ (or vice versa) (see Figure \ref{fig:pseudo.mut}). 
We call a face $F$ of $\pa(\w)$ \emph{mutable} if it corresponds to  $s_rs_{r+1}s_r$ (or $s_{r+1}s_r s_{r+1}$) and denote the corresponding mutation by $\mu_F$. 
The resulting pseudoline arrangement is associated to the reduced expression $\mu_F(\w)$ of $w$ and denoted by $\pa(\mu_F(\w))$.
\end{definition}

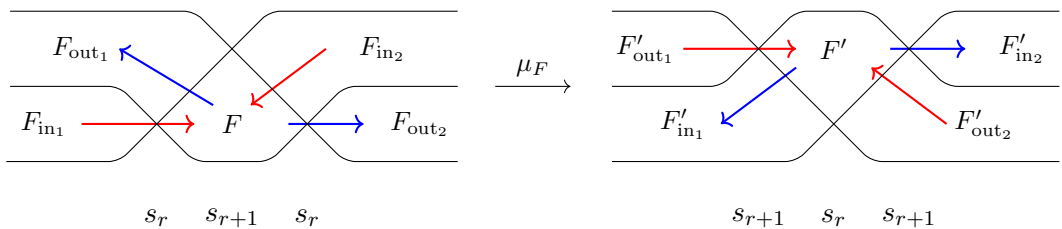
\begin{figure}[h]
    \centering
\begin{tikzpicture}

\node at (5.5,1.5) {\small $F_{\init_2}$};
\node at (6,.5) {\small $F_{\text{out}_2}$};
\node at (3.5,.5) {\small $F$};
\node at (1.5,1.5) {\small $F_{\text{out}_1}$};
\node at (1,.5) {\small $F_{\init_1}$};

\draw[thick, red, ->] (1.5,.5) -- (3,.5);
\draw[thick, blue, ->] (4.25,.5) -- (5.25,.5);
\draw[thick, red, ->] (4.75,1.5) -- (3.75,.75);
\draw[thick, blue, <-] (2,1.5) -- (3.25,.75);

\draw[rounded corners] (0.5,0) --(1.25,0) -- (2,0)-- (3,1) --(4,2) -- (6.5,2);
\draw[rounded corners] (.55,1) -- (2,1) -- (3,0) -- (3.5,0)-- (4,0) -- (5,1) -- (5.57,1)-- (6.5,1);
\draw[rounded corners] (.55,2) -- (3,2) -- (4,1) -- (5,0) -- (5.75,0) -- (6.5,0);

\node at (3.5,-.75) {$s_{r+1}$};
\node at (2.5,-.75) {$s_r$};
\node at (4.5,-.75) {$s_r$};

\draw[->] (7,1) -- (8,1);
\node[above] at (7.5,1) {\small $\mu_F$};
    
\begin{scope}[xshift=8cm]

\node at (6,1.5) {\small $F'_{\init_2}$};
\node at (5.5,.5) {\small $F'_{\text{out}_2}$};
\node at (3.5,1.5) {\small$F'$};
\node at (1,1.5) {\small $F'_{\text{out}_1}$};
\node at (1.5,.5) {\small $F'_{\init_1}$};

\draw[thick, red, ->] (1.5,1.5) -- (3,1.5);
\draw[thick, blue, ->] (4.25,1.5) -- (5.25,1.5);
\draw[thick, blue, ->] (3,1.25) -- (2,.5);
\draw[thick, red, ->] (5,.5) -- (4,1.25);

\draw[rounded corners] (0.55,0) --(1.25,0) -- (3,0)-- (4,1) -- (5,2) -- (6.5,2);
\draw[rounded corners] (0.55,1) -- (2,1) -- (3,2) -- (4,2) -- (5,1) -- (6.5,1);
\draw[rounded corners] (.55,2) -- (2,2) -- (3,1)-- (4,0) -- (5.75,0) -- (6.5,0);

\node at (3.5,-.75) {$s_{r}$};
\node at (2.5,-.75) {$s_{r+1}$};
\node at (4.5,-.75) {$s_{r+1}$};
\end{scope}    
\end{tikzpicture}
\caption{Mutation of pseudoline arrangements.}
\label{fig:pseudo.mut}
\end{figure}

Note, that the quivers $Q_{\w}$ and $Q_{\mu_{F(\w)}}$ are related by quiver mutation at the vertex $w_F$. However, $Q_{\w}$ has more mutable vertices than $\pa(\w)$ has mutable faces. 
When mutating $Q_{\w}$ at a vertex $w_{F'}$ with $F'$ not mutable in $\pa(\w)$, then for $\mu_{F'}(Q_{\w})$ there is no reduced expression of $w$ that would give rise to this quiver via a pseudoline arrangement.

Consider $\w_0\in S_n$ with reduced expression $\hat\w_0:=s_1s_2s_1s_3s_2s_1\dots s_{n-1}s_{n-2}\dots s_3s_2s_1$ and the quiver $Q_{\hat\w_0}$. 
We label the vertices for faces $F_{(i,j)}$ bounded to the left by the crossing of lines $l_i$ and $l_j$ by $w_{(i,j)}$. 
In particular, the frozen vertices at the right boundary are labelled $w_{(n-1,n)},\dots,w_{(1,n)}$ from bottom to top.
Referring to their level, the frozen vertices on the left boundary are labelled by $w_{1},\dots,w_{n-1}$ from bottom to top. 
In the following example we describe the quiver corresponding to this \emph{initial} reduced expression $\hat\w_0$ for $n=5$.

\begin{example}\label{exp:initial seed S_5}
Consider $\hat \w_0=s_1s_2s_1s_3s_2s_1s_4s_3s_2s_1\in S_5$. The pseudoline arrangement and the corresponding quiver are depicted in Figure~\ref{fig:initial}. 

\begin{center}
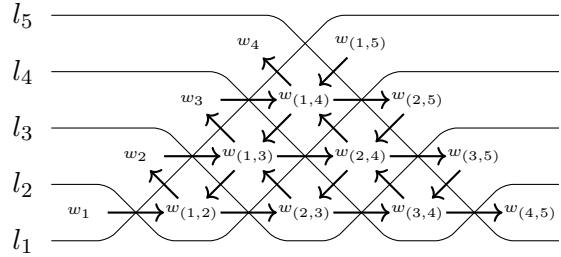
\begin{figure}[ht]
\centering
\begin{tikzpicture}[scale=.75]


\node at (-.5,0) {$l_1$};
\node at (-.5,1) {$l_2$};
\node at (-.5,2) {$l_3$};
\node at (-.5,3) {$l_4$};
\node at (-.5,4) {$l_5$};

\node at (.5,.5) {\tiny $w_{1}$};
\node at (1.5,1.5) {\tiny $w_{2}$};
\node at (2.5,2.5) {\tiny $w_{3}$};
\node at (3.5,3.5) {\tiny $w_{4}$};
\node at (5.5,3.5) {\tiny $w_{(1,5)}$};
\node at (6.5,2.5) {\tiny $w_{(2,5)}$};
\node at (7.5,1.5) {\tiny $w_{(3,5)}$};
\node at (8.5,.5) {\tiny $w_{(4,5)}$};
\node at (4.5,2.5) {\tiny $w_{(1,4)}$};
\node at (3.5,1.5) {\tiny $w_{(1,3)}$};
\node at (5.5,1.5) {\tiny $w_{(2,4)}$};
\node at (2.5,.5) {\tiny $w_{(1,2)}$};
\node at (4.5,.5) {\tiny $w_{(2,3)}$};
\node at (6.5,.5) {\tiny $w_{(3,4)}$};


\draw[rounded corners] (0,0) --(1,0) -- (5,4) -- (9,4);
\draw[rounded corners] (0,1) -- (1,1) -- (2,0) -- (3,0) -- (6,3) -- (9,3);
\draw[rounded corners] (0,2) -- (2,2) -- (4,0) -- (5,0) -- (7,2) -- (9,2);
\draw[rounded corners](0,3) -- (3,3) -- (6,0) -- (7,0) -- (8,1) -- (9,1);
\draw[rounded corners] (0,4) -- (4,4) -- (8,0) -- (9,0);

\draw[thick,->] (5.25,3.25) -- (4.75,2.75);
\draw[thick,->] (4.25,2.25) -- (3.75,1.75);
\draw[thick,->] (3.25,1.25) -- (2.75,.75);

\draw[thick,<-] (3.75,3.25) -- (4.25,2.75);
\draw[thick,<-] (4.75,2.25) -- (5.25,1.75);
\draw[thick,<-] (5.75,1.25) -- (6.25,.75);

\draw[thick,<-] (2.75,2.25) -- (3.25,1.75);
\draw[thick,<-] (3.75,1.25) -- (4.25,.75);

\draw[thick,<-] (1.75,1.25) -- (2.25,.75);

\draw[thick,->] (6.25,2.25) -- (5.75,1.75);
\draw[thick,->] (5.25,1.25) -- (4.75,0.75);

\draw[thick,->] (7.25,1.25) -- (6.75,0.75);

\draw[thick,->] (3,2.5) -- (4,2.5);
\draw[thick,->] (5,2.5) -- (6,2.5);

\draw[thick,->] (2,1.5) -- (3,1.5);
\draw[thick,->] (4,1.5) -- (5,1.5);
\draw[thick,->] (6,1.5) -- (7,1.5);

\draw[thick,->] (1,0.5) -- (2,0.5);
\draw[thick,->] (3,0.5) -- (4,0.5);
\draw[thick,->] (5,0.5) -- (6,0.5);
\draw[thick,->] (7,0.5) -- (8,0.5);
\end{tikzpicture}
\caption{$\pa(\hat\w_0)$ and $Q_{\hat\w_0}$ with $\hat\w_0=s_1s_2s_1s_3s_2s_1s_4s_3s_2s_1\in S_5$.}
\label{fig:initial}
\end{figure}
\end{center}
\end{example}

\subsubsection*{Orientation and paths.}
For every pair $(l_i,l_{i+1})$ with $1\le i\le n-1$ we give an orientation to a pseudoline arrangement by orienting lines $l_1,\dots,l_i$ from right to left and lines $l_{i+1},\dots,l_{n}$ from left to right, see Figure~\ref{fig:pa 121}. 
Consider an oriented path with three consecutive crossings $v_{k-1}\to v_{k}\to v_{k+1}$ belonging to the same pseudoline $l_i$. Then $v_k$ is the intersection of $l_i$ with some line $l_j$, i.e. $v_k=v_{(i,j)}$.
If either $i<j$ and both lines are oriented to the left, or $i>j$ and both lines are oriented to the right, the path is called \emph{non-rigorous}. Figure~\ref{fig:rigorous} shows these two situations. A path is called \emph{rigorous} if it is not non-rigorous. 
\begin{figure}[h]
\centering
\begin{center}
\begin{tikzpicture}[scale=.7]

\node at (0,2) {$l_3$};
\node at (0,1) {$l_2$};
\node at (0,0) {$l_1$};

\node at (7,0) {$L_3$};
    \draw [fill] (6.5,0) circle [radius=0.05];
\node at (7,1) {$L_2$};
    \draw [fill] (6.5,1) circle [radius=0.05];
\node at (7,2) {$L_1$};
    \draw [fill] (6.5,2) circle [radius=0.05];

\node[below] at (2.5,-.5) {$s_{1}$};
\node[below] at (3.5,-.5) {$s_{2}$};
\node[below] at (4.5,-.5) {$s_{1}$};

\node[above] at (3.5,1.65) {$v_{(1,3)}$};
    \draw [fill] (3.5,1.5) circle [radius=0.05];
\node[right] at (4.6,0.5) {$v_{(2,3)}$};
    \draw [fill] (4.5,0.5) circle [radius=0.05];
\node[left] at (2.4,0.5) {$v_{(1,2)}$};
    \draw [fill] (2.5,0.5) circle [radius=0.05];

\draw[rounded corners] (0.5,0) --(1.25,0) -- (2,0)-- (3,1) --(4,2) -- (5.25,2);
    \draw[->, rounded corners] (1.5,0) -- (1.25,0);
    \draw[->, rounded corners] (3.5,1.5) -- (3,1);
    \draw[->, rounded corners] (6.5,2) -- (5.25,2);
\draw[rounded corners] (1.25,1) -- (2,1) -- (3,0) -- (3.5,0)-- (4,0) -- (5,1) -- (5.57,1)-- (6.5,1);
    \draw[->] (0.5,1) -- (1.25,1);
    \draw[->] (3.25,0) -- (3.5,0);
    \draw[->] (5.5,1) -- (5.75,1);
\draw[rounded corners] (1.5,2) -- (3,2) -- (4,1)-- (5,0) -- (5.75,0) -- (6.5,0);
    \draw[->] (0.5,2) -- (1.5,2);
    \draw[->] (3.5,1.5) -- (4,1);
    \draw[->] (5.5,0) -- (5.75,0);

\end{tikzpicture}
\end{center}
\caption{$\pa(\w_0)$ for $\underline w_0=s_1s_2s_1\in S_3$ with orientation for $(l_1,l_2)$.}\label{fig:pa 121} 
\end{figure}

\begin{definition} Let $\underline{w}$ be a fixed reduced expression of $w \in S_{n}$. 

A \emph{Gleizer-Postnikov path} (or short \emph{GP-path}) is a rigorous path $\p$ in $\pa(\w)$ endowed with some orientation $(l_i,l_{i+1})$ for $i\in[n-1]$. 
It has source $L_p$ and sink $L_q$ for $p\le i$ and $ q\ge i+1$. 
Further, $w(i+1)\le w(p)\le w(i)$ and $w(i+1)\le w(q)\le w(i)$. 
The set of all GP-paths for all orientations in the pseudoline arrangement associated to $\underline{w}$ is denoted by $\mathcal P_{\w}$.
\end{definition}

\begin{figure}[h]
\centering
\begin{center}
\begin{tikzpicture}

\draw[rounded corners] (3,0) -- (2,0) -- (1,1) -- (0,1);
        \draw[->] (3,0) -- (2.5,0);
        \draw[->] (0.8,1) -- (.5,1);
\draw[rounded corners] (3,1) -- (2,1) -- (1,0) -- (0,0);
    \draw[->] (3,1) -- (2.5,1);
    \draw[->] (0.8,0) -- (.5,0);
\draw[->, ultra thick, red] (1.9,0.9) -- (1.1,0.1);

\begin{scope}[xshift=5cm]
  \draw[rounded corners] (3,0) -- (2,0) -- (1,1) -- (0,1);
        \draw[->] (.5,1) -- (.8,1);
        \draw[->] (2.2,0) -- (2.5,0);
\draw[rounded corners] (3,1) -- (2,1) -- (1,0) -- (0,0);
    \draw[->] (.5,0) -- (.8,0);
    \draw[->] (2.2,1)-- (2.5,1);
\draw[<-, ultra thick, red] (1.9,0.1) -- (1.1,0.9);
\end{scope}
\end{tikzpicture}
\end{center}
\caption{The two red arrows are forbidden in rigorous paths.}\label{fig:rigorous}
\end{figure}
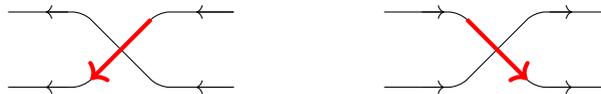

Note that if $w(i)<w(i+1)$ there are no GP-paths of shape $(l_i,l_{i+1})$ and in case $w(p)\le w(q)$ there are no GP-paths with source $L_p$ and sink $L_q$.

\begin{proposition}\label{prop:in stream} 
Let $w\in S_n$ with reduced expression $\w$.
Consider $\mathbf{p}\in \mathcal P_{\w}$ of shape $(l_i,l_{i+1})$.
Then $\mathbf{p}$ is either the empty path or does not cross the lines $l_{i+1}$ and $l_i$.
In particular, $\p$ does not leave the area in $\pa(\w)$ bounded by $l_i$ and $l_{i+1}$ to the left. 
\end{proposition}

\begin{proof}
Without loss of generality we assume $w(i)<w(i+1)$, otherwise $\mathbf{p}$ is empty and we are done. 
Further, let $L_p$ be the source of $\mathbf p$ and $L_q$ the sink.
We assume $w(p)\le w(q)$, otherwise, again, $\mathbf p$ is empty. 
We focus on the part of $\pa(\w)$ to the \emph{right} of the crossing of $l_i$ and $l_{i+1}$ (which exists as $w(i)<w(i+1)$).
Observe the following:

$\bullet$ all lines crossing $l_i$ do so oriented from top to bottom.

$\bullet$ all lines crossing $l_{i+1}$ do so oriented from bottom to top.

\noindent
As $L_p$ and $L_q$ lie in between the lines $l_i$ and $l_{i+1}$ this observation implies that $\mathbf p$ can not cross $l_i$ and if it was to cross $l_{i+1}$ it could not return to $L_q$, a contradiction. 
The only possibility that is left, is if $\mathbf p$ was to follow $l_i$ through the crossing with $l_{i+1}$, but then again, it could not return to $L_q$.
\end{proof}

\subsubsection*{Cones and polytopes arising from pseudoline arrangements}
We define two weighted cones, two cones, and two families of polytopes that arise from $\mathcal P_{\w}$ for $\w$ reduced expression of $w\in S_n$. 
We relate the two cones in the forthcoming sections, one to the weighted string cone (introduced by Littelmann \cite{Lit98} and Berenstein-Zelevinsky \cite{BZ01}), the other to the tropicalization of the (restriction of the) superpotential for a double Bruhat cell (see Magee \cite{Mag15}).

\paragraph{The (weighted) GP-cone} For $\w=s_{i_1}\dots s_{i_{\ell(w)}}$ we label the standard basis of $\mathbb R^{\ell(w)}$ by crossing points in $\pa(\w)$, i.e. $\{c_{(k,m)}\mid w(\alpha_{k,m-1})<0\}$. 
Sometimes it is also convenient to use the notation $c_{i_j}:=c_{(k,m)}$, when $s_{i_j}$ induces the crossing of $l_k$ and $l_m$ in $\pa(\w)$.
Consider $\mathbf{p}\in\mathcal P_{\underline{w}}$. It is uniquely determined by those vertices in $\pa(\w)_0$ where $\p$ changes from one line to another. 
For some $1\le  p \le i < q \le n$ we can therefore write $\p$ as
\[ 
\mathbf{p}=L_{p}\to v_{(p,j_1)}\to v_{(j_1,j_2)}\to\dots\to v_{(j_{k},q)}\to L_{q}.
\]
Set $j_0:=p$ and $j_{k+1}:=q$, then we associate to $\p$ the vector
\begin{eqnarray}\label{eq:def c_p}
c_{\p} := \sum_{s=0}^{k} c_{(j_s, j_{s+1})} \in \mathbb R^{\ell(w)},
\end{eqnarray}
where we set $c_{(i,j)} := - c_{(j,i)}$ if $i > j$ and $c_{(i,i)}:=0$.

\begin{definition}\label{def:gp-cone}
The following polyhedral cone is called \textit{GP-cone} (due to Gleizer-Postnikov \cite{GP00} who call it \emph{principal cone}):
\begin{eqnarray}\label{eq:def GP cone}
C_{\w} = \{ {\mathbf x} \in \mathbb{R}^{\ell(w)} \mid (c_{\p})^t({\mathbf x}) \geq 0, \forall \p \in \mathcal P_{\w} \}.
\end{eqnarray}
\end{definition}

\begin{example}
Consider the reduced expression $\w_0=s_1s_2s_1\in S_3$. 
We endow $\pa(s_1s_2s_1)$ with the orientation for $(l_1,l_2)$, i.e. $l_1$ is oriented to the left and $l_2,l_3$ are oriented to the right (see Figure~\ref{fig:pa 121}). There are two paths in $\mathcal P_{s_1s_2s_1}$ from $L_1$ to $L_2$,
\[
\mathbf p_1=L_1\to v_{(1,3)}\to v_{(1,2)}\to v_{(2,3)}\to L_2 \text{ and } \mathbf p_2=L_1\to v_{(1,3)}\to v_{(2,3)} \to L_2.
\]
They yield $c_{\mathbf p_1}= c_{(1,2)}$ and $c_{\mathbf p_2}=c_{(1,3)}-c_{(2,3)}$. Similarly for the orientation $(l_2,l_3)$ we find a path $\mathbf p_3=L_2\to v_{(2,3)}\to L_3$ with $c_{\mathbf p_3}=c_{(2,3)}$. 
Then
\[
C_{s_1s_2s_1}=\{(x_{(1,2)},x_{(1,3)},x_{(2,3)})\in \mathbb R^3\mid x_{(1,2)}\ge 0, x_{(1,3)}\ge x_{(2,3)}\ge 0\}.
\]
\end{example}

We are interested in a weighted version of this cone to relate it to string polytopes in the next section.
The weighted cone lives in $\mathbb{R}^{\ell(w)+n-1}$, where the additional basis elements are indexed $c_1, \ldots, c_{n-1}$. 
By some abuse of notation we denote by $c_{\p}$ also the vector $(c_{\p},0\dots,0)\in \mathbb R^{\ell(w)}\times \{0\}^{n-1}\subset \mathbb R^{\ell(w)+n-1}$.

For every $i\in[n-1]$ we define the following subset of $[\ell(w)]$
\begin{eqnarray}\label{eq:def J(i) and n_i}
J(i):=\{k\in[\ell(w)]\mid s_{i_k}=s_i\} \text{ with } n_i:=\#J(i).
\end{eqnarray}
Let $J(i)=\{j_1,\dots,j_{n_i}\}$, then we set $c_{[i:0]}:=c_i$ and for $1\le k\le n_i$ we define 
\begin{eqnarray}\label{eq: def wt ineq GP}
c_{[i:k]} := c_i - c_{{i_{j_k}}} - 2 \sum_{j\in J(i),j>j_k} c_{{i_j}} + \sum_{l\in J(i-1)\cup J(i+1), l>j_k} c_{{i_l}}.
\end{eqnarray}
These vectors are normal vectors to the faces of the following weighted cone.

\begin{definition}\label{def:wgp-cone}
The \emph{weighted Gleizer-Postnikov cone} $\mathcal{C}_{\w} \subset \mathbb{R}^{\ell(w)+n-1}$ is defined as
\begin{eqnarray}\label{eq:def weighted GP cone}
\mathcal{C}_{\w}  := \left\{ 
{\mathbf x} \in \mathbb{R}^{\ell(w)+n-1} \left|
\begin{matrix}
(c_\mathbf{p})^t({\mathbf x}) \ge 0 \; , \; &\forall \; \p \in \mathcal P_{\w},&\\
(c_{[i:k]})^t ({\mathbf x}) \ge 0, \; & \forall  \; i\in[n-1], 0\le k \le n_i&
\end{matrix}
\right.\right\}.
\end{eqnarray} 
\end{definition}

\begin{example}\label{exp:pathGT}
Consider $w_0\in S_n$ and consider the reduced expression $\hat\w_0$ defined above Example~\ref{exp:initial seed S_5}. 
For $ i\in [n-1]$ all GP-paths in $\pa(\w_0)$ with orientation $(l_i,l_{i+1})$ are of form
\[
\p_{i,j}:=L_i\to v_{(i,n)} \to v_{(i,n-1)}\to \dots \to v_{(i,j)}\to v_{(i+1,j)}\to \dots \to v_{(i+1,n)}\to L_{i+1}.
\]
In particular, the GP-cone $C_{\w_0}$ is described by inequalities defined by the normal vectors $c_{(i,j+1)}- c_{(i+1,j+1)}$ and $c_{(i,i+1)}$ for $i\in[n-1]$ and $j \in[i+1,n-1]$.
The vectors defining weight inequalities are (for all $i < j$):
\[
c_{j-i} - c_{(i,j)} - 2 \sum_{k = 1}^{n-j} c_{(i+k, j+k)} + \sum_{k = 0}^{n-j-1} c_{(i+k, j+1+k)} + \sum_{k = 0}^{n-j} c_{(i+1+k, j+k)}.
\]
\end{example}

\paragraph{The (weighted) area cone}
We associate to the set of all GP-paths $\mathcal P_{\w}$ a second cone. In this setup, the standard basis of $\mathbb{R}^{\ell(w)+n-1}$ is indexed by the faces of the pseudoline arrangement $\{e_F\mid F \text{ face of }\pa(\w)\}$. 
Namely, there are basis vectors associated to faces $F_{(i,j)}$ bounded to the left by a crossing $(i,j)$, and to faces $F_l$ unbounded to the left for every $l\in[n-1]$. 
Let $\p \in \mathcal P_{\w}$.
We denote by $\area_{\p}$ the area to the left of $\p$ (with respect to the orientation), i.e. the area enclosed by $\p$. 
Note that for non-trivial $\p$, $\area_{\p}$ is a non-empty union of faces $F$ in the pseudoline arrangement. We associate to $\p$ the vector
\begin{eqnarray}\label{eq:def area ineq}
e_{\p} := - \sum_{F \subset \area_{\p}} e_F \in \mathbb{R}^{\ell(w)+n-1}.
\end{eqnarray}
With a little abuse of notation we denote by $e_{\p}$ also the vector in $\mathbb R^{\ell(w)}$ obtained by projecting onto the first $\ell(w)$ coordinates (forgetting the coordinates belonging to the faces that are unbounded to the left, which equal $0$ in $e_{\p}$).

\begin{definition}\label{def:s-cone} For a reduced expression $\underline{w} \in S_{n}$, we define the \emph{area cone}
\begin{eqnarray}\label{eq:def area cone}
S_{\w} := \{ {\mathbf x} \in \mathbb{R}^{\ell(w)} \mid (e_{\p})^t({\mathbf x}) \geq 0, \forall\; \p \in \mathcal P_{\w} \}.
\end{eqnarray}
\end{definition}

Again, we are interested in a weighted extension of this cone.
For this, we associate to every level $ i\in[n-1]$ a union of faces. 
Consider $F_{i}$, the face of $\pa(\w)$ that is unbounded to the left at level $i$.
As before for crossings we set $F_{i_j}:=F_{(k,m)}$ if $s_{i_j}$ in $\w$ induces the crossing of $l_k$ and $l_m$ in $\pa(\w)$.
We define $\area_{i}:=F_{i}\cup \bigcup_{k=1}^{n_i}F_{i_k}$, then $\area_i \cap \area_{i'} = \varnothing$ if $i \neq i'$. 
It is called the \emph{weight area} associated to the level $i$. 
For each $k$ with $0\le k\le n_i$, we define a vector
\begin{eqnarray}\label{eq:def area wt ineq}
e_{[i:k]} := -e_{F_i} - \sum_{j\in J(i),j\le j_k}  e_{F_{i_j}} \in \mathbb{R}^{\ell(w)+n}.
\end{eqnarray}
Note that $e_{[i:0]}=-e_{F_i}$ and $e_{[i:n_i]}=-\sum_{F\subset \area_i} e_F$.

\begin{definition}\label{def:ws-cone}
The \emph{weighted area cone} $\mathcal S_{\w}\subset \mathbb R^{\ell(w)+n-1}$ associated to the reduced expression $\w$ of $w\in S_n$ is defined as 
\begin{eqnarray}\label{eq: def wt area cone}
\mathcal S_{\w} := 
\left\{  {\mathbf x} \in \mathbb{R}^{\ell(w)+n-1} \left|
\begin{matrix} (e_{\p})^t({\mathbf x}) \ge 0 \; , \; &\forall \;  \p \in \mathcal P_{\w},& \\ 
(e_{[i:k]})^t({\mathbf x}) \ge 0 \; , \; &\forall \; i\in[n-1], 0\le k \le n_i&   
\end{matrix} 
\right. \right\}. 
\end{eqnarray}
The additional inequalities induced by the $e_{[i:k]}$ are called \textit{weight inequalities}.
\end{definition}

\begin{remark}
In all four cases, $C_{\w}, \mathcal{C}_{\w}, S_{\w}$ and $\mathcal{S}_{\w}$, some of the inequalities might be redundant and these cones are far from being simplical in general. The vectors $e_{\mathbf p},c_{\mathbf p},e_{[i:k]}$ and $c_{[i:k]}$ are normal vectors to the defining hyperplanes of the cones $S_{\w},C_{\w},\mathcal S_{\w}$ and $\mathcal C_{\w}$ respectively. Not all of them are normal vectors to facets of these cones in general.
\end{remark}

\begin{example}\label{exp:area}
Consider the reduced expression $\hat\w_0\in S_5$. We have seen all GP-paths in $\pa(\w)$ in Example~\ref{exp:pathGT}. 
Take the path ${\p}= L_1\to v_{(1,5)}\to v_{(1,4)}\to v_{(1,3)} \to v_{(1,2)}\to v_{(2,3)}\to v_{(2,4)}\to v_{(2,5)}\to L_2$. 
The area $\area_{\p}$ associated to this path is shaded blue in Figure~\ref{fig:GTstreams}. 
The weight area $\area_2$ corresponding to level $2$ is also shown in Figure~\ref{fig:GTstreams} dotted in red.
\begin{center}
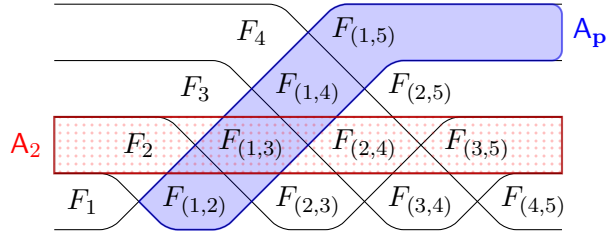
\begin{figure}[ht]
\centering
\begin{tikzpicture}[scale=.75]


\draw[red] (0,1) -- (9,1);
\draw[red] (0,2) -- (9,2);

\draw[rounded corners] (0,0) --(1,0) -- (5,4) -- (9,4);
\draw[rounded corners] (0,1) -- (1,1) -- (2,0) -- (3,0) -- (6,3) -- (9,3);
\draw[rounded corners] (0,2) -- (2,2) -- (4,0) -- (5,0) -- (7,2) -- (9,2);
\draw[rounded corners] (0,3) -- (3,3) -- (6,0) -- (7,0) -- (8,1) -- (9,1);
\draw[rounded corners] (0,4) -- (4,4) -- (8,0) -- (9,0);

\draw [fill=blue, rounded corners, opacity=.2] (1.5,0.5) --  (2,0) -- (3,0)-- (6,3) -- (9,3) -- (9,4)-- (5,4) -- (1.5,0.5);
\draw [blue, opacity=.5, thick, rounded corners] (1.5,0.5) --  (2,0) -- (3,0)-- (6,3) -- (9,3) -- (9,4)-- (5,4) -- (1.5,0.5);
\draw[pattern=dots, pattern color=red, opacity=.5, thick] (0,1) rectangle (9,2);
\draw[red, opacity=.5, thick] (0,1) rectangle (9,2);

\node[blue] at (9.5,3.5) {$\area_{\p}$};
\node at (.5,.5) {$F_{1}$};
\node at (1.5,1.5) {$F_{2}$};
\node at (2.5,2.5) {$F_{3}$};
\node at (3.5,3.5) {$F_{4}$};
\node at (5.5,3.5) {$F_{(1,5)}$};
\node at (6.5,2.5) {$F_{(2,5)}$};
\node at (7.5,1.5) {$F_{(3,5)}$};
\node at (8.5,.5) {$F_{(4,5)}$};
\node at (4.5,2.5) {$F_{(1,4)}$};
\node at (3.5,1.5) {$F_{(1,3)}$};
\node at (5.5,1.5) {$F_{(2,4)}$};
\node at (2.5,.5) {$F_{(1,2)}$};
\node at (4.5,.5) {$F_{(2,3)}$};
\node at (6.5,.5) {$F_{(3,4)}$};
\node[red] at (-.5,1.5) {$\area_2$};

\end{tikzpicture}
\caption{The area $\area_{\p}$ for ${\p}$ as in Example~\protect{\ref{exp:area}} shaded in blue and the weight area $\area_2$ dotted in red.}
\label{fig:GTstreams}
\end{figure}
\end{center}
\end{example}

\begin{example}\label{exp:areaGT}
Consider as in Example~\ref{exp:pathGT} $\hat\w_0\in S_n$ and recall $\p_{i,j}\in\mathcal P_{\w_0}$ with $i\in[n-1]$ and $j\in[i+1,n-1]$.
The assigned area is $\area_{\p_{i,j}}=F_{(i,j)}\cup F_{(i,j+1)}\cup \dots\cup F_{(i,n)}$ for $F_{(i,k)}$ the area bounded by $v_{(i,k)}$ to the left. Hence, the cone $S_{\w_0}$ is given by inequalities defined by
\begin{eqnarray}\label{eq: e_p in GT w_0}
e_{\p_{i,j}}=-e_{F_{(i,j)}}-e_{F_{(i,j+1)}}-\dots -e_{F_{(i,n)}}-e_{F_{(i,n+1)}}.
\end{eqnarray}
The additional weight inequalities defining the cone $\mathcal S_{\w_0}$ are given by the normal vectors
\begin{eqnarray}\label{eq: f_i,k in GT w_0}
e_{[i:k]}=-e_{F_i}-e_{F_{(1,i+1)}}-e_{F_{(2,i+2)}}-\dots -e_{F_{(k,i+k)}},
\end{eqnarray}
for $i\in[n-1]$ and $0\le k\le n-i$.
\end{example}

\paragraph{The polytopes} Let $\pi: \mathbb{R}^{\ell(w) + n-1} \to \mathbb{R}^{n-1}$ be the projection onto the last $n-1$ coordinates, also called \emph{weight coordinates}. 
We are interested in the preimage $\pi^{-1}(\lambda)$ for $\lambda\in\mathbb R^{n-1}$.
It is the intersection of the following hyperplanes for each $i\in[n-1]$ defined by
\begin{eqnarray}\label{eq: wt hyperplanes gp}
(c_{[i:0]})^t(\mathbf x) =\lambda_i,  \ \forall \ \mathbf x\in\mathbb R^{\ell(w)+n-1}.
\end{eqnarray}
Fix $w\in S_n$ with reduced expression $\w$. We define a second map $\tau_{\w}:\mathbb R^{\ell(w)+n-1}\to \mathbb R^{n-1}$ by
$\tau_{\w}(\mathbf x)=((e_{[i:n_i]})^t(\mathbf x))_{i=1,\dots,n-1}$.
The preimage of $\lambda\in \mathbb R^{n-1}$ with respect to $\tau_{\w}$ is also an intersection of hyperplanes in $\mathbb R^{\ell(w)+n-1}$. For each $i\in[n-1]$ they are defined by
\begin{eqnarray}\label{eq:wt hyperplanes area}
(e_{[i:n_i]})^t(\mathbf x)=\lambda_i, \  \forall \  \mathbf x\in\mathbb R^{\ell(w)+n-1}.
\end{eqnarray}

\begin{definition}\label{def: polytopes from cones}
For $w\in S_n$ with reduced expression $\w$ and for $\lambda\in\mathbb R^{n-1}$ we define the following polytopes in $\mathbb R^{\ell(w)+n-1}$
\begin{eqnarray}\label{eq:def polytopes}
\mathcal S_{\w}(\lambda):=\mathcal{S}_{\w}\cap \tau_{\w}^{-1}(\lambda)
\text{   and   } \mathcal C_{\w}(\lambda):=\mathcal C_{\w}\cap \pi^{-1}(\lambda).
\end{eqnarray}
\end{definition}

Note that by \eqref{eq:wt hyperplanes area} (resp. \eqref{eq: wt hyperplanes gp}) we obtain a description of $\mathcal S_{\w}(\lambda)$ (resp. $\mathcal C_{\w}(\lambda)$) in terms of defining equalities and inequalities
by replacing the weight inequalities $e_{[i:n_i]}^t({\mathbf x})\ge 0$ in \eqref{eq: def wt area cone} (resp. $(c_i)^t(\mathbf x)\ge 0$ in \eqref{eq:def weighted GP cone}) by $(e_{[i:n_i]})^t(\mathbf x)=\lambda_i$ (resp. $(c_i)^t(\mathbf x) =\lambda_i$).
In particular, the defining normal vectors for $\mathcal S_{\w}$ (resp. $\mathcal C_{\w}$) coincide with those for $\mathcal S_{\w}(\lambda)$ (resp. $\mathcal C_{\w}(\lambda)$).
This observation is important in the proof of Theorem~\ref{thm:unimod}.

\subsubsection*{A unimodular equivalence} The above pairs of cones (resp. polytopes) $(S_{\w},C_{\w})$ and $(\mathcal S_{\w},\mathcal C_{\w})$ ( resp. $(\mathcal S_{\w}(\lambda),\mathcal C_{\w}(\lambda))$) have in fact more in common than the combinatorics defining them. 
To make this statement precise we need to introduce the notion of unimodular equivalence (see e.g. \cite[\S2]{HL16}).

\begin{definition}\label{def:unimod equiv}
Two polytopes $P,Q\subset \mathbb R^d$ (resp. polyhedral cones $C,D\subset \mathbb R^{d}$) are called \emph{unimodularly equivalent} if there exists matrix $M\in GL_d(\mathbb Z)$ and $w\in \mathbb Z^d$
\[
Q=f_M(P)+w \ (\text{resp. } D=f_M(C)+w),
\]
where $f_M(x)=xM$ for $x\in \mathbb R^d$. We denote this by $Q\cong P$ (resp. $C\cong D$).
\end{definition}

This notion of equivalence is of particular interest to us because of its implication on the associated toric varieties. 
Recall the construction of a projective toric variety $X_P\subset \mathbb P^{d-1}$ associated with a polytope $P\subset \mathbb R^d$ in \cite[\S2.1 and \S2.3]{CLS11}. Then
\begin{eqnarray}\label{eq: unimod for toric}
Q\cong P \text{ implies } X_Q\cong X_P.
\end{eqnarray}

We want to construct a unimodular equivalence between $\mathcal C_{\w}$ and $\mathcal S_{\w}$ for all reduced expression $\w$ of $w\in S_n$.
The following definition is the affine lattice transformation ($f_M$ in Definition~\ref{def:unimod equiv}) that defines the unimodular equivalence.
We give it in terms of the bases $\{e_F\mid F \text{ face of }\pa(\w)\}$ and $\{c_{(k,m)},c_i\mid v_{(k,m)}\in\pa(\w)_0 ,i\in[n-1]\}$.
Morally, we send a face $F$ bounded to the left by a crossing to a linear combination of its adjacent crossings (see \eqref{eq: def psi wt}).
A face unbounded to the left is sent to the sum of all crossings at its level.

\begin{definition}\label{def: psi_w}
For $w\in S_n$ and $\w$ a reduced expression we define the linear map $\Psi_{\w}:\mathbb R^{\ell(w)+n-1}\to \mathbb R^{\ell(w)+n-1}$ on the basis $\{-e_{F}\}$ associated to faces $F$ of $\pa(\w)$.
Let $F=F_{i_{j_k}}$ be the face bounded to the left by the crossing induced from $s_{i_{j_k}}=s_i$ and $J(i)=\{j_1,\dots,j_{n_i}\}$ (see \eqref{eq:def J(i) and n_i}). Then \begin{eqnarray}\label{eq: def psi}
\Psi_{\w}(-e_{F_{i_{j_k}}}):=  c_{i_{j_k}}+c_{i_{j_{k+1}}} - \sum_{\begin{smallmatrix}j\in J(i-1)\cup J(i+1),\\ j_k<j<j_{k+1}\end{smallmatrix}} c_{i_j}.
\end{eqnarray}
For every level $i\in[n-1]$, we define
\begin{eqnarray}\label{eq: def psi wt}
\Psi_{\w}(-e_{F_i}) := c_{[i:1]}.
\end{eqnarray}
\end{definition}

\begin{example}\label{exp: psi lattice}
Consider $\pa(\w)$ for $\w=s_1s_2s_1\in S_3$ as in Figure~\ref{fig:pa 121}. The two bases for $\mathbb R^{5}$ are
\[
\mathcal B_{e}=\{-e_{F_1},-e_{F_2},-e_{F_{(1,2)}},-e_{F_{(1,3)}},-e_{F_{(2,3)}}\} \text{ and } \mathcal B_c=\{c_1,c_2,c_{(1,2)},c_{(1,3)},c_{(2,3)}\}.
\]
We compute the images of elements in $\mathcal B_e$ and express them in $\mathcal B_c$. The coefficients form the columns of the following matrix with the order of the bases as given above.
\begin{eqnarray*}
\Bigg(\begin{smallmatrix}
1 & 0 & 0 & 0 & 0 \\
0 & 1 & 0 & 0 & 0 \\
-1& 0 & 1 & 0 & 0 \\
1 & -1& -1& 1 & 0 \\
-2& 1 & 1 & -1 & 1
\end{smallmatrix}\Bigg) \in GL_5(\mathbb Z).
\end{eqnarray*}
\end{example}

The observation in the example above is true in general. We obtain the following Lemma as a straightforward consequence of the definition of $\Psi_{\w}$.

\begin{lemma}\label{lem: psi lattice}
Let $w\in S_n$ with reduced expression $\w$. Order the bases induced by the faces of $\pa(\w)$ resp. by the crossing points in $\pa(\w)$ as
\[
\mathcal B_e=\{-e_{F_1},\dots,-e_{F_{n-1}},-e_{F_{i_1}},\dots-e_{F_{i_{\ell(w)}}}\},
\text{ resp. } \mathcal B_c=\{c_1,\dots,c_{n-1},c_{i_1},\dots,c_{i_{\ell(w)}}\}.
\]
Then $\Psi_{\w}$ can be represented by a lower triangular matrix ${M}_{\w}^{e,c}$ with all diagonal entries being 1. In particular, ${M}_{\w}^{e,c}\in GL_{\ell(w)+n-1}(\mathbb Z)$.
\end{lemma}

\begin{corollary}\label{cor: res psi lattice}
With assumptions as in Lemma~\ref{lem: psi lattice} consider $\Psi_{\w}\vert_{\mathbb R^{\ell(w)}}:\mathbb R^{\ell(w)}\to\mathbb R^{\ell(w)}$. We order as before the bases for $\mathbb R^{\ell(w)}$ induced by the faces resp. crossing points in $\pa(\w)$ by
\[
\overline{\mathcal B}_e=\{-e_{F_{i_1}},\dots-e_{F_{i_{\ell(w)}}}\},
\text{ resp. }  \overline{\mathcal B}_c=\{c_{i_1},\dots,c_{i_{\ell(w)}}\}.
\]
Then $\Psi_{\w}\vert_{\mathbb R^{\ell(w)}}$ can be represented by a lower triangular matrix $\overline{M}_{\w}^{e,c}$ with all diagonal entries 1. In particular, $\overline{M}_{\w}^{e,c}\in GL_{\ell(w)}(\mathbb Z)$.
\end{corollary}

\begin{remark}
The map $\Psi_{\w}$ restricted to $\mathbb R^{\ell(w)}$ is related to the Chamber Ansatz due to Berenstein-Fomin-Zelevinsky in \cite{BFZ96} (see also \cite{GKS}).
\end{remark}

\begin{figure}
\centering
\begin{tikzpicture}
  \node at (-.5,.5) {1a};
  \node at (3.25,1) {$l_i$};
  \node at (3.25,0) {$l_{j}$};
  \draw [fill] (1.5,.5) circle [radius=0.05];
  \node at (.9,0.5) {$c_{(i,j)}$};

\draw [fill=red, semitransparent, red, rounded corners] (3,1) --(2,1) --(1.5,.5)-- (2,0) -- (3,0) -- (3,1);
\draw[rounded corners] (0,1) -- (1,1) -- (2,0) -- (3,0);
\draw[rounded corners] (0,0) -- (1,0) -- (2,1) -- (3,1);
    \draw[->] (0.25,1) -- (0.5,1);
    \draw[->] (2.25,0) -- (2.5,0);
\draw[->] (2.5,1) -- (2.25,1);
\draw[->] (0.75,0) -- (0.5,0);

 \begin{scope}[xshift=4.5cm]
   \node at (-.5,.5) {2a};
     \node at (3.25,1) {$l_i$};
  \node at (3.25,0) {$l_{j}$};
    \draw [fill] (1.5,.5) circle [radius=0.05];
   \node at (.8,0.5) {$-c_{(i,j)}$};

\draw [fill=red, semitransparent, red, rounded corners] (0,1) --(1,1) --(1.5,.5)-- (2,1) -- (3,1) -- (3,1.5) -- (0,1.5) -- (0,1);
\draw[rounded corners] (0,1) -- (1,1) -- (2,0) -- (3,0);
\draw[rounded corners] (0,0) -- (1,0) -- (2,1) -- (3,1);
    \draw[->] (0.25,1) -- (0.5,1);
    \draw[->] (2.25,0) -- (2.5,0);
\draw[->] (0.25,0) -- (0.5,0);
\draw[->] (2.25,1) -- (2.5,1);

\end{scope}

\begin{scope}[xshift=9cm]
  \node at (-.5,.5) {3a};
    \node at (3.25,1) {$l_i$};
  \node at (3.25,0) {$l_{j}$};
    \draw [fill] (1.5,.5) circle [radius=0.05];
   \node at (.8,0.5) {$-c_{(i,j)}$};

\draw [fill=red, semitransparent, red, rounded corners] (0,0) --(1,0) --(1.5,.5)-- (2,0) -- (3,0) -- (3,-.5) -- (0,-.5) -- (0,0);

\draw[rounded corners] (0,1) -- (1,1) -- (2,0) -- (3,0);
\draw[rounded corners] (0,0) -- (1,0) -- (2,1) -- (3,1);
    \draw[->] (0.75,1) -- (0.5,1);
    \draw[->] (2.5,0) -- (2.25,0);
\draw[->] (2.5,1) -- (2.25,1);
\draw[->] (0.75,0) -- (0.5,0);

\end{scope}

\begin{scope}[yshift=-2.5cm]
  \node at (-.5,.5) {1b};
    \node at (3.25,1) {$l_i$};
  \node at (3.25,0) {$l_{j}$};
    \draw [fill] (1.5,.5) circle [radius=0.05];
  \node at (.9,0.5) {$-c_{(i,j)}$};

\draw [fill=red, semitransparent, red, rounded corners] (0,1) --(1,1) --(1.5,.5)-- (1,0) -- (0,0) -- (0,-.5) -- (3,-.5) -- (3,1.5) -- (0,1.5) -- (0,1);
\draw[rounded corners] (0,1) -- (1,1) -- (2,0) -- (3,0);
\draw[rounded corners] (0,0) -- (1,0) -- (2,1) -- (3,1);
    \draw[->] (0.25,1) -- (0.5,1);
    \draw[->] (2.25,0) -- (2.5,0);
\draw[->] (2.5,1) -- (2.25,1);
\draw[->] (0.75,0) -- (0.5,0);
\end{scope}

\begin{scope}[xshift=4.5cm, yshift=-2.5cm]
  \node at (-.5,.5) {2b};
    \node at (3.25,1) {$l_i$};
  \node at (3.25,0) {$l_{j}$};
    \draw [fill, rounded corners] (1.5,.5) circle [radius=0.05];
  \node at (1.5,1) {$c_{(i,j)}$};

\draw [fill=red, semitransparent, red, rounded corners] (0,1) --(1,1) --(1.5,.5)-- (2,1) -- (3,1) -- (3,-.5) -- (0,-.5) -- (0,1);
\draw[rounded corners] (0,1) -- (1,1) -- (2,0) -- (3,0);
\draw[rounded corners] (0,0) -- (1,0) -- (2,1) -- (3,1);
    \draw[->] (0.75,1) -- (0.5,1);
    \draw[->] (2.5,0) -- (2.25,0);
\draw[->] (2.5,1) -- (2.25,1);
\draw[->] (0.75,0) -- (0.5,0);
\end{scope}

\begin{scope}[xshift=9cm, yshift=-2.5cm]
  \node at (-.5,.5) {3b};
    \node at (3.25,1) {$l_i$};
  \node at (3.25,0) {$l_{j}$};
    \draw [fill] (1.5,.5) circle [radius=0.05];
  \node at (1.5,0) {$c_{(i,j)}$};

\draw [fill=red, semitransparent, red, rounded corners] (0,0) --(1,0) --(1.5,.5)-- (2,0) -- (3,0) -- (3,1.5) -- (0,1.5) -- (0,0);
\draw[rounded corners] (0,1) -- (1,1) -- (2,0) -- (3,0);
\draw[rounded corners] (0,0) -- (1,0) -- (2,1) -- (3,1);
    \draw[->] (0.25,1) -- (0.5,1);
    \draw[->] (2.25,0) -- (2.5,0);
\draw[->] (0.25,0) -- (0.5,0);
\draw[->] (2.25,1) -- (2.5,1);
\end{scope}

\end{tikzpicture}
\caption{A path $\p$ changing the line at a crossing $(i,j)$ and the corresponding area $\area_{\p}$.}\label{fig:path.cross}
\end{figure}
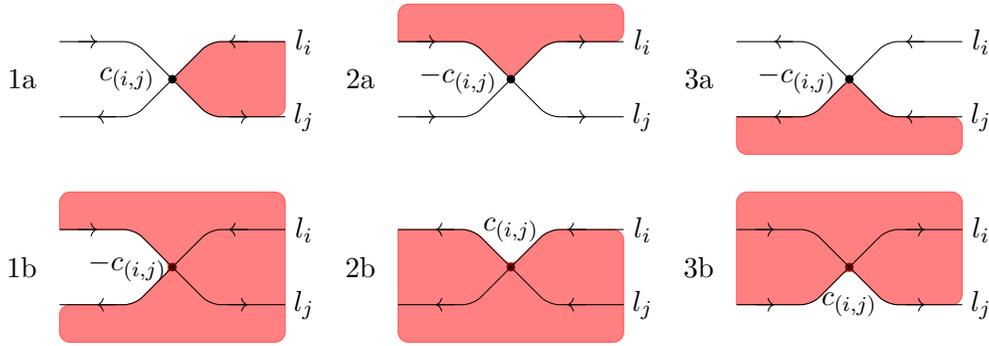

\begin{proposition}\label{prop:unimod}
Let $w\in S_n$ with reduced expression $\w$. 
For every $\p\in\mathcal P_{\w}$ we have
\[
\Psi_{\w}(e_\p)=c_\p.
\]
In particular, $\Psi_{\w}$ sends the normal vector of a defining hyperplane of $\mathcal S_{\w}$ to the normal vector of a defining hyperplane of $\mathcal C_{\w}$.
\end{proposition}

\begin{proof}
We show that for every crossing point $(i,j)$ the coefficient of $c_{(i,j)}$ coincides in $\Psi_{\w}(e_\p)$ and $c_\p$.
Recall that $\area_\p$ is the union of all faces to the left of $\p$ with respect to the given orientation. 
We distinguish three cases:

\begin{itemize}
\item If $(i,j)$ lies in the interior of $\area_{\p}$, then four faces $F^r\subset \area_\p$ with $r\in[4]$ are adjacent to $(i,j)$. 
For two of them in $\Psi_{\w}(-e_{F^r})$ the coefficient of $c_{(i,j)}$ is $+1$, for the other two it is $-1$.
Hence, they cancel each other and in $\Psi_{\w}(-e_\p)$ it is zero as it is in $c_{\p}$.

\item If $\p$ contains $(i,j)$ but does not change the line at $(i,j)$, then $c_{(i,j)}$ has coefficient zero in $c_{\p}$. 
For $\area_{\p}$, this means that two faces, $F^1$ and $F^2$, are adjacent to ${(i,j)}$. 
One of the two, say $F^1$, is bounded by ${(i,j)}$ to the left where for $F^2$, ${(i,j)}$ is part of the upper or lower boundary. 
In particular, $\Psi_{\w}(e_{\p})$ contains $c_{(i,j)}$ once with positive and once with negative sign, hence with the coefficient is zero.

\item Assume $\p$ changes the line at the crossing ${(i,j)}$. Figure \ref{fig:path.cross} shows the three possible orientations of $l_i$ and $l_j$.
Each yields two possibilities for the path. 
If in situation 1a, there is one face $F$ in $\area_{\p}$ bounded by ${(i,j)}$ to the left.
So $c_{(i,j)}$ has coefficient $1$ in $\Psi_{\w}(e_{\p})$. 
As $\p$ changes from $l_i$ to $l_j$ and $i<j$, also $c_{\p}$ contains $c_{(i,j)}$ with coefficient $1$.

In cases 2a and 3a, $\area_{\p}$ contains only one face bounded by ${(i,j)}$ below resp. above. 
Hence $c_{(i,j)}$ appears with coefficient $-1$ in $\Psi_{\w}(e_{\p})$. 
The same is true for $c_{\p}$: in both cases $\p$ changes from line $l_j$ to $l_i$ but $i<j$. 

Three cases remain to be checked, 1b, 2b and 3b in Figure~\ref{fig:path.cross}. 
In all of them $\area_{\p}$ contains three faces $F^1,F^2$ and $F^3$ adjacent to ${(i,j)}$. 
In case 1b, ${(i,j)}$ bounds one face to the left and the other two from above, resp. below. 
This implies that $c_{(i,j)}$ appears with coefficient $-1$ in $\Psi_{\w}(e_{\p})$. 
As $\p$ changes from $l_j$ to $l_i$ the same is true for $c_{\p}$. 
For 2b and 3b we are in the opposite case: two faces in $\area_{\p}$ are bounded to the left, resp. right, by ${(i,j)}$ and only one from above, resp. below. 
Hence, $\Psi_{\w}(e_{\p})$ contains $c_{(i,j)}$ with coefficient $1$ and the same is true for $c_{\p}$, as $\p$ changes from line $l_i$ to line $l_j$.
\end{itemize}
\end{proof}

For our application later, it remains to show that the normal vectors defining the weight inequalities are mapped onto each other by $\Psi_{\w}$. 
Recall the weight area $\area_{i}=F_{i}\cup \bigcup_{r=1}^{n_i}F_{(i_r,j_r)}$ of level $i\in[n-1]$, with $n_i$ the number of faces bounded to the left of level $i$.

\begin{proposition}\label{prop:wunimod}
Let $w\in S_n$ with reduced expression $\w$. 
Consider $i\in [n-1]$ with $J(i)=\{j_1,\dots,j_{n_i}\}$.
Then for $ k\in [0, n_i]$ we have
\[
\Psi_{\w}(e_{[i:k]})=c_{[i:k+1]} \text{ and } \Psi_{\w}(e_{[i:n_i]})=c_{[i:0]}.
\]
In particular, $\Psi_{\w}$ sends normal vectors of defining (weight) hyperplanes of $\mathcal S_{\w}$ to normal vectors of defining (weight) hyperplanes of $\mathcal C_{\w}$.
\end{proposition}

\begin{proof}
We prove the claim by induction on $k$.
By definition we have $\Psi_{\w}(e_{[i:0]})=\Psi_{\w}(-e_{F_i}) = c_{[i:1]}$.
Let $1\le k< n_i-1$, then using induction for the third equation, we obtain
\begin{eqnarray*}
\Psi_{\w}(e_{[i:k+1]}) & \stackrel{\tiny\eqref{eq:def area wt ineq}}{=} & \Psi_{\w}(e_{[i:k]}-e_{F_{i_{j_{k+1}}}}) \\ 
& \stackrel{\tiny\eqref{eq: def psi}}{=} & c_{[i:k+1]} + c_{i_{j_{k+1}}} + c_{i_{j_{k+2}}} -\sum_{\begin{smallmatrix} j\in J(i-1)\cup J(i+1),\\ j_{k+1}<j<j_{k+2}\end{smallmatrix}} c_{i_j}\\
& \stackrel{\tiny\eqref{eq: def wt ineq GP}}{=} & c_i -c_{i_{j_{k+2}}}-2\sum_{j\in J(i),j>j_{k+2}}c_{i_j} +\sum_{\begin{smallmatrix} j\in J(i-1)\cup J(i+1),\\ j>j_{k+2}\end{smallmatrix}} c_{i_j} \ = \ c_{[i:k+2]}.
\end{eqnarray*}
Now consider $e_{[i:n_i]}=e_{[i:n_i-1]}-e_{F_{i_{j_{n_i}}}}$.  
We apply $\Psi_{\w}$ and obtain the following by induction. 
\begin{eqnarray*}
\Psi(e_{[i:n_i]})& \stackrel{\tiny\eqref{eq:def area wt ineq}}{=} & \Psi_{\w}(e_{[i:n_i-1]}-e_{F_{j_{n_i}}}) \\
& \stackrel{\tiny\eqref{eq: def psi}}{=} & c_{[i:n_i]}+ c_{i_{j_{n_i}}} - \sum_{\begin{smallmatrix} j\in J(i-1)\cup J(i+1),\\ j_{n_i}<j \end{smallmatrix}} c_{i_j} \ \stackrel{\tiny\eqref{eq: def wt ineq GP}}{=} \   c_i.
\end{eqnarray*}

\end{proof}

We can now prove the first Theorem of this section. It is a more precise formulation of Theorem~\ref{thm:dual cones intro} as stated in the Introduction.

\vbox{
\begin{theorem}\label{thm:unimod}
Let $w\in S_n$ and $\w$ a reduced expression. The following polyhedral objects are unimodularly equivalent
\begin{enumerate}[(i)]
    \item $\mathcal S_{\w}\cong \mathcal C_{\w}$ via $\Psi_{\w}$,
    \item $S_{\w}\cong C_{\w}$ via $\Psi_{\w}\vert_{\mathbb R^{\ell(w)}}$,
    \item $\mathcal S_{\w}(\lambda)\cong \mathcal C_{\w}(\lambda)$ for all $\lambda\in\mathbb R^{n-1}$ via $\Psi_{\w}$.
\end{enumerate}
\end{theorem}}

\begin{proof}
We begin by proving (i). From Proposition~\ref{prop:unimod} we know that normal vectors $e_{\p}$ of $\mathcal S_{\w}$ for $\p\in\mathcal P_{\w}$ are mapped to the normal vectors $c_{\p}$ of $\mathcal C_{\w}$, i.e. $\Psi_{\w}(e_{\p})=c_{\p}$.
Further, from Proposition~\ref{prop:wunimod} we know the same is true for the normal vectors $e_{[i:k]}$ of $\mathcal S_{\w}$ for $i\in[n-1]$: 
we have $\Psi_{\w}(e_{[i:k]})=c_{[i:k+1]}$ for $k<n_i$  and  $\Psi_{\w}(e_{[i:n_i]})=c_{[i:0]}$.
As the right hand side of all defining inequalities is zero, we deduce that $\Psi_{\w}(\mathcal S_{\w})=\mathcal C_{\w}$.
By Lemma~\ref{lem: psi lattice}, $\Psi_{\w}$ is given by a matrix in $GL_{\ell(w)+n-1}(\mathbb Z)$ and hence, $\mathcal S_{\w}\cong \mathcal C_{\w}$.

To show (ii), note that by the same argument as for (i) we have $\Psi_{\w}\vert_{\mathbb R^{\ell(w)}}(S_{\w})=C_{\w}$. By Corollary~\ref{cor: res psi lattice} we deduce $S_{\w}\cong C_{\w}$.

For (iii) recall that for $\lambda\in\mathbb R^{n-1}$ the polytopes $\mathcal S_{\w}(\lambda)$ resp. $\mathcal C_{\w}(\lambda)$ are defined by the same normal vectors as $\mathcal S_{\w}$ resp. $\mathcal C_{\w}$. 
As above, it is also true that the right hand sides of the defining (in-)equalities coincides for all normal vectors being mapped onto each other. 
We therefore deduce $\mathcal{S}_{\w}(\lambda) =\mathcal{C}_{\w}(\lambda)$.
\end{proof}


\subsection{String cones, polytopes and toric degenerations}\label{subsec:string}

Recall from \S\ref{sec:pre rep theory} our notation for the representation theoretic background regarding $SL_n$.
In this section we recall \emph{string polytopes} and \emph{string cones} introduced by Littelmann in \cite{Lit98} and Berenstein-Zelevinsky in\cite{BZ01} as well as the \emph{weighted string cones} defined in \cite{Lit98}. 
We prove using a result from Gleizer-Postnikov in \cite{GP00} that these are exactly $\mathcal C_{\w}(\lambda)$ resp. $C_{\w}$ and $\mathcal C_{\w}$.

Littelmann \cite{Lit98} introduced in the context of quantum groups and crystal bases the so called (weighted) string cones and string polytopes $Q_{\w}(\lambda)$.
The motivation is to find monomial bases for the Demazure modules $V_{\w}(\lambda)$ for $w\in S_n$ and $\lambda\in\Lambda^+$.
Recall that by \eqref{eq: PBW for demazure} $\{ f_{\alpha_{i_1}}^{m_{i_1}} \cdots f_{\alpha_{i_{\ell(w)}}}^{m_{i_{\ell(w)}}} \cdot v_\lambda \in V(\lambda) \mid m_{i_j} \geq 0 \}$
is a spanning set for $V_{\w}(\lambda)$ depending on a reduced expression $\w=s_{i_1}\cdots s_{i_{\ell(w)}}$.
Littelmann identifies a linearly independent subset of this spanning set by introducing the notion of \emph{adapted string} (see \cite[p.~4]{Lit98}) referring to a tuple $(a_1,\dots,a_{\ell(w)})\in\mathbb Z_{\ge 0}^{\ell(w)}$.
His basis for $V_{\w}(\lambda)$ consists of those elements $f_{i_1}^{a_1}\cdots f_{i_{\ell(w)}}^{a_{\ell(w)}}\cdot v_\lambda$ for which $(a_1,\dots,a_{\ell(w)})$ is adapted.

For a fixed reduced expression $\w$ of $w\in S_n$ and $\lambda\in\Lambda^+$ he gives a recursive definition of the the \emph{string polytope} $Q_{\w}(\lambda)\subset\mathbb R^{\ell(w)}$ (\cite[p.~5]{Lit98}, see also \cite{BZ01}). 
The lattice points $Q_{\w}(\lambda)\cap \mathbb Z^{\ell(w)}$ are the adapted strings for $\w$ and $\lambda$.
The \emph{string cone} $Q_{\w}\subset \mathbb R^{\ell(w)}$ is the convex hull of all $Q_{\w}(\lambda)$ for $\lambda\in\Lambda^+$. The \emph{weighted string cone} $\mathcal Q_{\w}\subset\mathbb R^{\ell(w)+n-1}$ is defined as
\[
\mathcal Q_{\w} := \conv \bigg( \bigcup_{\lambda\in\Lambda^+}Q_{\w}(\lambda)\times\{\lambda\}\bigg)\subset \mathbb R^{\ell(w)+n-1}.
\]
By definition, one obtains the string polytope from the weighted string cone by intersecting it with the hyperplanes given by $\pi^{-1}(\lambda)$ as in \eqref{eq: wt hyperplanes gp}.
The lattice points in the weighted string cone for $w=w_0$ parametrize a basis of $\mathbb{C}[SL_{n}/U]\cong \bigoplus_{\lambda\in\Lambda^+}V(\lambda)$.

\medskip

String polytopes are of great interest to us because of Caldero's work \cite{Cal02} in 2002.
He defines for a Schubert variety $X_w$ a flat family over $\mathbb A^1$ with generic fibre $X_w$ and special fibre a toric variety. 
The family is of form \eqref{eq: def Rees} given by a construction using Rees algebras (see \S\ref{sec:pre val}). Although not defined using valuations initially, it was realized this way in \cite{Kav15} and \cite{FFL15}.
His main tools are Lusztig's dual canonical basis and the string parametrization due to \cite{BZ01} and \cite{Lit98}.
We summarize his results (restricted to the case of $SL_n$) below.

Let $\w = s_{i_1} \cdots s_{i_{\ell(w)}}$ be a reduced expression of $w \in S_{n}$. 
We extend $\w$ to the \emph{right} to a reduced expression $\w_0 =\w s_{i_{\ell(w)+1}} \cdots s_{i_N}$ of $w_0$. 
This extension is not unique but the results are independent of the extension.
Caldero realizes the string cone $Q_{\w}$ for the Demazure module $V_w(\lambda)$ as a face of the string cone $Q_{\w_0}$. He deduces the following Lemma as a consequence of \cite[\S1]{Lit98}.

\begin{lemma*}(see \cite{Lit98},\cite[Lemma~3.3]{Cal02})\label{lem: restr. string cone}
Let $w\in S_n$ with reduced expression $\w=s_{i_1}\cdots s_{i_{\ell(w)}}$ and choose a reduced expression $\w_0=\w s_{i_{\ell(w)+1}}\cdots s_{i_N}$. Then the weighted string cone $\mathcal Q_{\w}$ is obtained from the weighted string cone $\mathcal Q_{\w_0}$ by setting the variables corresponding to $s_{i_{\ell(w)+1}},\dots,s_{i_N}$ equal to zero.
\end{lemma*}

Caldero defines a filtration $(\mathcal F_{\le m})_{m\ge 1}$ on $\mathbb C[SL_n/U]$ with associated graded algebra the semi-group algebra $\mathbb C[\mathcal Q_{\w_0}\cap\mathbb Z^{N+n-1}]$. Using the Lemma, he defines a quotient filtration $(\overline{\mathcal F}_{\le m})_{m\ge 1}$ on $\mathbb C[SL_n/U]/I_{w}$, where $I_w=\bigoplus_{\lambda\in \Lambda^+} V_w(\lambda)^\perp$ (recall \S\ref{sec:pre rep theory}), i.e. (from what we have seen in \S\ref{sec:pre flag})
\[
\mathbb C[SL_n/U]/I_{w}=\bigoplus_{\lambda\in\Lambda^+}V(\lambda)^*/\bigoplus_{\lambda\in\Lambda^+}V_w(\lambda)^{\perp}.
\]
The semi-group algebra $\mathbb C[\mathcal Q_{\w}\cap \mathbb Z^{\ell(\w)+n-1}]$ is the associated graded algebra of the quotient filtration.
In particular, he degenerates $X_w$ into a toric variety $Y$, whose normalization is the toric variety $X_{\mathcal Q_{\w}(\lambda)}$ associated to the string polytope $\mathcal Q_{\w}(\lambda)$ for $\lambda\in\Lambda^{++}$.

\paragraph{Relation to the GP cones} Gleizer and Postnikov develop in \cite{GP00} a combinatorial model to describe string cones $Q_{\w_0}$ \emph{non-recursively} for every reduced expression $\w_0$ of $w_0\in S_n$. They use pseudoline arrangements and GP-paths to obtain the following.

\begin{corollary*}\cite[Corollary~5.8]{GP00}
Let $\w_0$ be a reduced expression for $w_0\in S_n$. Then $C_{\w_0}=Q_{\w_0}$.
\end{corollary*}

On our way to showing that a toric variety isomorphic to $X_{\mathcal Q_{\w}(\lambda)}$ arises in the context of cluster varieties and mirror symmetry, we first generalize Gleizer-Postnikov's result as follows.

\vbox{
\begin{theorem}\label{thm:wt GP is wt string}
For every $w \in S_{n}$ with reduced expression $\w$ and every extension $\w_0=\w s_{i_{\ell(w)}+1}\cdots s_{i_N}$ the following polyhedral objects coincide
\begin{enumerate}[(i)]
    \item $\mathcal C_{\w}=\mathcal Q_{\w}$,
    \item $C_{\w}=Q_{\w}$,
    \item $\mathcal{C}_{\w}(\lambda)=\mathcal Q_{\w}(\lambda)$ for $\lambda\in\mathbb R^{n-1}$.
\end{enumerate}
\end{theorem}}

In order to prove Theorem~\ref{thm:wt GP is wt string} we show how to obtain $\mathcal C_{\w}$ from restricting $\mathcal C_{\w_0}$ for appropriate $\w_0$. The next subsection is dedicated to introducing restricted paths and concludes with the proof of Theorem~\ref{thm:wt GP is wt string}.

\subsubsection*{Restriction of paths} 
We show that for $\w_0=\w s_{i_{\ell(w)+1}}\cdots s_{i_N}$ we obtain $\mathcal C_{\w}$ from $\mathcal C_{\w_0}$ by setting to zero the coordinates corresponding to crossing points $c_{i_{\ell(w)+1}},\dots,c_{i_N}$ in $\pa(\w_0)$.

\begin{definition}\label{def: res path}
Let $\w$ be a reduced expression of $w\in S_n$ and fix $\w_0=\w s_{i_{\ell(w)+1}}\cdots s_{i_N}$. 
Consider $\mathbf{p}_{\w_0}\in\mathcal P_{\w_0}$ and draw it in $\pa(\w_0)$.
Then cut $\pa(\w_0)$ in two pieces along a vertical line, such that all crossing points $v_{i_p}$ corresponding to $s_{i_p}$ with $1\le p\le \ell(w)$ are on the left of the cut and all $v_{i_q}$ corresponding to $s_{i_q}, \ell(w)<q\le N$ are on the right (see Figure~\ref{fig: res path}).
We define the \emph{restriction} $\res_{\w}(\mathbf{p}_{\w_0})$ of ${\mathbf p}_{\w_0}$ to $\pa(\w)$ as the part of $\p_{\w_0}$ that is to the left of the cut.
\end{definition}

We label the intersection points of the lines $l_i$ with the cutting line by $\hat L_i$.
An alternative way of describing $\res_{\w}(\p_{\w_0})$ is by removing all vertices $v_{(i,j)}$ from it for which $w(\alpha_{i,j-1})>0$. 
Denote by $\res_{\w}(\mathcal P_{\w_0})$ the set of all paths in $\pa(\w)$ that appear in a restriction of a path in $\mathcal P_{\w_0}$ (counting each path only once).

\begin{center}
\begin{figure}[]
\centering
\begin{tikzpicture}[scale=.8]
\draw[rounded corners] (0,0) -- (2,0) -- (3,1) -- (5,1) -- (7,3) -- (9,3);
\draw[rounded corners] (0,1) -- (1,1) -- (2,2) -- (3,2) -- (4,3) -- (6,3) -- (7,2) -- (9,2);
\draw[rounded corners] (0,2) -- (1,2) -- (3,0) -- (7,0) -- (8,1) -- (9,1);
\draw[rounded corners] (0,3) -- (3,3) -- (4,2) -- (5,2) -- (6,1) -- (7,1) -- (8,0) -- (9,0);

\draw[->] (9,3) --  (8,3);
\draw[<-] (8.5,0) -- (8.1,0);
\draw[<-] (2,3) -- (1.5,3);
\draw[<-] (1,0) -- (1.9,0);
\draw[->] (0,2) -- (.5,2);
\draw[->] (.8,1) -- (.5,1);

\node at (1.5,-.5) {$s_2$};
\node at (2.5,-.5) {$s_1$};
\node at (3.5,-.5) {$s_3$};
\node at (5.5,-.5) {$s_2$};
\node at (6.5,-.5) {$s_3$};
\node at (7.5,-.5) {$s_1$};

\node at (-.5,0) {$l_1$};
\node at (-.5,1) {$l_2$};
\node at (-.5,2) {$l_3$};
\node at (-.5,3) {$l_4$};

\draw[rounded corners, blue, ultra thick] (4.5,3) -- (4,3) -- (3.5,2.5);
    \draw[rounded corners, blue, ultra thick] (3.5,2.5) -- (4,2) -- (4.5,2);
    \draw[blue, ultra thick, ->] (4.5,3) -- (4.1,3);
    \draw[blue, ultra thick, ->] (4.1,2) -- (4.2,2);
\draw[rounded corners, blue, ultra thick] (4.5,1) -- (3,1) -- (2.5,.5);
    \draw[rounded corners, blue, ultra thick] (2.5,0.5) -- (3,0) -- (4.5,0);
    \draw[blue, ultra thick, ->] (3.8,1) -- (3.5,1);
    \draw[blue, ultra thick, ->] (3.1,0) -- (3.5,0);
\draw[rounded corners, ultra thick, red] (9,2) -- (7,2) -- (6,3) -- (4.5,3);
    \draw[rounded corners, ultra thick, red] (9,1) -- (8,1) -- (7,0) -- (4.5,0);
    \draw[red, ultra thick, ->] (8,2) -- (7.8,2);
    \draw[red, ultra thick, ->] (6,0) -- (6.5,0);
\draw[rounded corners, ultra thick, red]    (4.5,2) -- (5,2) -- (5.5,1.5);
    \draw[rounded corners, ultra thick, red] (4.5,1) -- (5,1) -- (5.5,1.5);
    \draw[red, ultra thick,->] (5.1,1.9) -- (5.3,1.7);
    \draw[red, ultra thick, ->] (4.8,1) -- (4.6,1);
    
\draw[dashed, thick] (4.5,-.8) -- (4.5,3.8);    
\end{tikzpicture}
\caption{A path ${\p}_{\w_0}\in\mathcal P_{\w_0}$ for $\w_0=s_2s_1s_3s_2s_3s_1$ that restricts to two paths ${\p}_{\w},{\p}'_{\w}\in\mathcal{P}_{\w}$ (in blue to the left of the dashed cut) for $\w=s_2s_1s_3$. 
}\label{fig: res path}
\end{figure}
\end{center}

\begin{example}\label{exp: res path}
Consider $\w=s_2s_1s_3$ and extend it to $\w_0=s_2s_1s_3s_2s_3s_1\in S_4$. We draw $\pa(\w)$ and endow it with the orientation for $(l_2,l_3)$. 
Figure~\ref{fig: res path} shows a GP-path ${\p}_{\w_0}$.
Its restriction $\res_{\w}({\p}_{\w_0})$ consists of two GP-paths for $\w$ shown in blue to the left of the cut. 
\end{example}

\begin{proposition}\label{prop:respath}
Let $\w$ be a reduced expression of $w\in S_n$ and fix $\w_0=\w s_{i_{\ell(w)+1}}\cdots s_{i_N}$. 
Consider $\p_{\w_0}\in\mathcal P_{\w_0}$, then $\res_{\w}(\p_{\w_0})$ is either empty or a union of paths in $\mathcal P_{\w}$.
In particular, $\res_{\w}(\mathcal P_{\w_0})\subset \mathcal P_{\w}$.
\end{proposition}

\begin{proof}
Let $\p_{\w_0}$ be a path for orientation $(l_i,l_{i+1})$, i.e. of form $\p_{\w_0}=L_{i}\to v_{(i,j_1)}\to v_{(j_1,j_2)}\to\dots\to v_{(j_k,i+1)}\to L_{i+1}$.
To simplify notation we set $i=j_0$ and $i+1=j_{k+1}$.
First note that if $w(\alpha_{j_r,j_{r+1}-1})>0$ for all $0\le r\le k$ then $\res_{\w}(\p_{\w_0})=\varnothing$.
Otherwise $\res_{\w}(\p_{\w_0})$ is a union of paths
\[
\p_r=\hat L_{j_r}\to v_{(j_r,j_{r+1})}\to\dots\to v_{(j_{r+s},j_{r+s+1})}\to \hat L_{j_{r+s+1}}
\]
such that $w(\alpha_{j_{r+p},j_{r+p+1}-1})<0$ for all $0\le p\le s$, $0\le r\le k$ and $ 0\le s\le k-r$.
By definition, each $\p_r$ is rigorous and hence, in $\mathcal P_{\w}$.
\end{proof}

\begin{algorithm}
\SetAlgorithmName{Algorithm}{} 
\KwIn{\medskip {\bf Input:\ }  A path in $\mathcal P_{\w} \ni {\p}_{\w}=\hat{L}_{i-l}\to v_{r_1}\to \dots \to v_{r_m}\to \hat{L}_{i+m}$ for orientation $(l_i,l_{i+1})$.}
\BlankLine
{\bf Initialization:} extend $\w$ to $\w_0=\w s_{i_{\ell(w)+1}}\cdots s_{i_N}$;\\ complete $\pa(\w)$ to $\pa(\w_0)$ with orientation for $(l_i,l_{i+1})$;\\
set $p=q=0$ and $\hat\p_{\w}=\p_{\w}$.\\
\For{$p<m-1$}{follow $l_{i+m-p}$ with respect to the orientation to the next crossing with a line $l_{i+m-p-p'}$ with $p'\in[m-p-1]$,\\
    \If{$p'=m-p-1$}{
        {\bf Output:} $\hat\p_{\w}\to v_{(i+m-p,i+1)}\to L_{i+1}$.
        }
    \Else{replace $p$ by $p+p'$ and $\hat\p_{\w}$ by $\hat\p_{\w}\to v_{(i+m-p,i+m-p-p')}$ and start over.}
}
\For{$q<l$}{
follow $l_{i-l+q}$ against the orientation to the next crossing with a line $l_{i-l+q+q'}$ with $q'\in[l-q]$,\\
    \If{$q'=l-q$}{
        {\bf Output:} $L_i\to v_{(i,i-l+q+q')}\to\hat\p_{\w}$.
        }
    \Else{replace $q$ by $q+q'$ and $\hat\p_{\w}$ by $v_{(i-l+q+q',i-l+q)}\to\hat\p_{\w}$ and start over.}
}
\BlankLine
{\bf Output:\ } A path $\ind_{\w_0}(\p_{\w}):= L_{i}\to v_{(i,i-l+q)}\to \dots\to\p_{\w}\to \dots\to v_{(i+m-p,i+1)}\to L_{i+1}$.
\label{alg:induced path}
\caption{Constructing the induced path $\ind_{\w_0}(\p_{\w})$ from $\p_{\w}\in\mathcal P_{ \w}$.}
\end{algorithm}

We want to show the other implication, $\mathcal P_{\w}\subset \res_{\w}(\mathcal P_{\w_0})$. 
In Algorithm~\ref{alg:induced path} we give a construction to obtain a path in $\pa(\w_0)$ from a given path in $\mathcal P_{\w}$.
The following proposition shows that the algorithm always terminates and that the output is in fact a path in $\mathcal P_{\w_0}$.

\begin{proposition}\label{prop: algo induced path}
Algorithm~\ref{alg:induced path} terminates for all $\p_{\w}\in\mathcal P_{\w}$ and $\ind_{\w_0}(\p_{\w})\in\mathcal P_{\w_0}$.
\end{proposition}
\begin{proof}
By Proposition~\ref{prop:in stream} $\p_{\w}$ lies in the region of $\pa(\w)$ in between the lines $l_i$ and $l_{i+1}$. In particular, at some point there is a $p'$ with $l_{i+m-p-p'}=l_{i+1}$ terminating the first loop and a $q'$ with $l_{i-l+q+q'}=l_i$ terminating the second loop.

To see that $\ind_{\w_0}(\p_{\w})\in\mathcal P_{\w_0}$ observe that changing the lines as indicated by the algorithm avoids exactly the two situations from Figure~\ref{fig:rigorous} forbidden in rigorous paths. Hence, $\ind_{\w_0}(\p_{\w})$ is rigorous. 
\end{proof}

By Proposition~\ref{prop: algo induced path} we can define the following.

\begin{definition}\label{def: ind path}
Let $\w$ be a reduced expression of $w\in S_n$ and fix $\w_0=\w s_{i_{\ell(w)+1}}\cdots s_{i_N}$. 
For $\p_{\w}\in\mathcal P_{\w}$ we define the \emph{induced path} $\ind_{\w_0}(\p_{\w})\in\mathcal P_{\w_0}$ as the output of Algorithm~\ref{alg:induced path}.
\end{definition}

\begin{example}\label{exp: ind path}
Consider $\w=s_2s_1s_3$ and extend it to $\w_0=s_2s_1s_3s_2s_3s_1\in S_4$. We draw $\pa(\w)$ and endow it with the orientation for $(l_2,l_3)$. 
Figure~\ref{fig: ind path} shows a GP-path ${\p}_{\w}$ in blue to the left of the cut. The extension of $\p_{\w}$ in red to the right of the cut completes $\p_{\w}$ to the induced path $\ind_{\w_0}(\p_{\w})$ that is the output of Algorithm~\ref{alg:induced path}.
\end{example}

\begin{center}
\begin{figure}[]
\centering
\begin{tikzpicture}[scale=.8]
\draw[rounded corners] (0,0) -- (2,0) -- (3,1) -- (5,1) -- (7,3) -- (9,3);
\draw[rounded corners] (0,1) -- (1,1) -- (2,2) -- (3,2) -- (4,3) -- (6,3) -- (7,2) -- (9,2);
\draw[rounded corners] (0,2) -- (1,2) -- (3,0) -- (7,0) -- (8,1) -- (9,1);
\draw[rounded corners] (0,3) -- (3,3) -- (4,2) -- (5,2) -- (6,1) -- (7,1) -- (8,0) -- (9,0);

\draw[->] (9,3) --  (8,3);
\draw[<-] (8.5,0) -- (8.1,0);
\draw[<-] (2,3) -- (1.5,3);
\draw[<-] (1,0) -- (1.9,0);
\draw[->] (0,2) -- (.5,2);
\draw[->] (.8,1) -- (.5,1);

\node at (1.5,-.5) {$s_2$};
\node at (2.5,-.5) {$s_1$};
\node at (3.5,-.5) {$s_3$};
\node at (5.5,-.5) {$s_2$};
\node at (6.5,-.5) {$s_3$};
\node at (7.5,-.5) {$s_1$};

\node at (-.5,0) {$l_1$};
\node at (-.5,1) {$l_2$};
\node at (-.5,2) {$l_3$};
\node at (-.5,3) {$l_4$};

\draw[rounded corners, blue, ultra thick] (4.5,1) -- (3,1) -- (2.5,.5);
    \draw[rounded corners, blue, ultra thick] (2.5,0.5) -- (3,0) -- (4.5,0);
    \draw[blue, ultra thick, ->] (3.8,1) -- (3.5,1);
    \draw[blue, ultra thick, ->] (3.1,0) -- (3.5,0);
    \draw[rounded corners, ultra thick, red] (9,1) -- (8,1) -- (7,0) -- (4.5,0);   \draw[red, ultra thick, ->] (6,0) -- (6.5,0); 
    \draw[red, ultra thick, rounded corners] (4.5,1) -- (5,1)-- (6.5,2.5);
    \draw[red, ultra thick, rounded corners] (6.5,2.5) -- (7,2) -- (9,2);
    \draw[red, ultra thick, ->] (8,2) -- (7.8,2);
    
\draw[dashed, thick] (4.5,-.8) -- (4.5,3.8);    
\end{tikzpicture}
\caption{A path ${\p}_{\w}\in\mathcal P_{\w}$ for $\w_0=s_2s_1s_3$ and the induced path $\ind_{\w_0}(\p_{\w})\in\mathcal P_{\w_0}$ with $\w_0=\w s_2s_3s_1$.
}\label{fig: ind path}
\end{figure}
\end{center}

\begin{proposition}\label{prop:restrict}
Let $\w$ be a reduced expression of $w\in S_n$ and fix $\w_0=\w s_{i_{\ell(w)}+1}\cdots s_{i_N}$.
For every $\mathbf{p}_{\w}\in\mathcal P_{\w}$ there exists $\mathbf{p}_{\w_0}\in\mathcal P_{\w_0}$ such that, $\res_{\w}(\mathbf{p}_{\underline w_0})=\mathbf{p}_{\w}$.  
In particular, we have $\mathcal P_{\w}\subset\res_{\w}(\mathcal P_{\w_0})$.
\end{proposition}

\begin{proof}
By construction $\ind_{\w_0}({\p}_{\w})\in\mathcal P_{\w_0}$ satisfies $\res_{\w}(\ind_{\w_0}({\p}_{\w}))={\p}_{\w}$.
\end{proof}

Recall for $i\in[n-1]$ the definition $J(i)$ and $n_i$ from \eqref{eq:def J(i) and n_i}. To distinguish between the sets for $\w$ and $\w_0$, we use the notation $J(i)^{\w}$ (resp. $J(i)^{\w_0}$) and $n_i^{\w}$ (resp. $n_i^{\w_0}$). We define the following polyhedral objects from restricted paths and show they equal the (weighted) GP-cone, respectively polytope, in the subsequent key proposition for proving Theorem~\ref{thm:wt GP is wt string}. 

\begin{definition}\label{def: res GP cone}
Let $\w$ be a reduced expression of $w\in S_n$ and fix $\w_0=\w s_{i_{\ell(w)}+1}\cdots s_{i_N}$.
We define the  \emph{restricted weighted GP-cone} as
\begin{eqnarray}\label{eq:def res wt GPcone}
\res_{\w}(\mathcal C_{\w_0}):=\left\{ 
\mathbf x\in \mathbb R^{\ell(w)+n-1} \left\vert
\begin{matrix}
(c_{\res_{\w}(\p_{\w_0})})^t(\mathbf x)\ge 0, & \forall \ \p_{\w_0}\in\mathcal P_{\w_0},\\
(c_{[i:k]})^t(\mathbf x)\ge 0, & \forall i\in[n-1], 0\le k\le n_i^{\w}
\end{matrix}
\right.\right\}.
\end{eqnarray}
Similarly, we define $\res_{\w}(C_{\w_0}):=\{ \mathbf x\in \mathbb R^{\ell(w)}\mid  
(c_{\res_{\w}(\p_{\w_0})})^t(\mathbf x)\ge 0,  \forall \ \p_{\w_0}\in\mathcal P_{\w_0}\}$ the \emph{restricted GP-cone} and the polytope $\res_{\w}(\mathcal C_{\w_0}(\lambda)):=\res_{\w}(\mathcal C_{\w_0})\cap \pi^{-1}(\lambda)$ (see \eqref{eq: wt hyperplanes gp}) for $\lambda\in\mathbb R^{n-1}$.
\end{definition}

\vbox{
\begin{proposition}\label{prop:wt GP res}
For every $w \in S_{n}$ with reduced expression $\w$ and every extension $\w_0=\w s_{i_{\ell(w)}+1}\cdots s_{i_N}$ the following polyhedral objects coincide
\begin{enumerate}[(i)]
    \item $\mathcal C_{\w}=\res_{\w}(\mathcal C_{\w_0})$,
    \item $C_{\w}=\res_{\w}(C_{\w_0})$,
    \item $\mathcal{C}_{\w}(\lambda)=\res_{\w}(\mathcal C_{\w_0}(\lambda))$ for $\lambda\in\mathbb R^{n-1}$.
\end{enumerate}
\end{proposition}}

\begin{proof}
We start by showing (i), then (ii) and (iii) are direct implications.
Note that only the inequalities induced by GP-paths differ in the definition of $\mathcal C_{\w}$ \eqref{eq:def weighted GP cone}, resp. $\res_{\w}(\mathcal C_{\w_0})$ \eqref{eq:def res wt GPcone}.
By Proposition~\ref{prop:respath} we have $\mathcal C_{\w}\subseteq \res_{\w}(\mathcal C_{\w_0})$. By Proposition~\ref{prop:restrict} we deduce $\res_{\w}(\mathcal C_{\w_0})\subseteq\mathcal C_{\w}$ and hence, equality follows.
\end{proof}

We have now collected all ingredients necessary to provide the proof Theorem~\ref{thm:wt GP is wt string}.

\begin{proof}[Proof of Theorem~\ref{thm:wt GP is wt string}]
We show $\mathcal Q_{\w}=\res_{\w}(\mathcal C_{\w_0})$ for every extension $\w_0=\w s_{i_{\ell(w)+1}}\cdots s_{i_N}$ and then apply Proposition~\ref{prop:wt GP res}.
By \cite[Lemma~3.3]{Cal02} (restated above) we know that
\[
\mathcal Q_{\w}=\mathcal Q_{{\w}_0}\cap\bigcap_{(i,k):\ w(\alpha_{i,k-1})>0}\{x_{(i,k)}=0\},
\]
as the $x_{(i,k)}$ appearing in the intersection of hyperplanes on the right correspond to the coordinates $x_{s_p}$ with $\ell(w)<p\le N$ in the extension of $\w$ to $\w_0$.
Further, we observe that if $c_{\p_{\w_0}}=\sum_{k}c_{(i_k,j_k)}$ then $c_{\res_{\w}(\p_{\w_0})}=\sum_{k: w(\alpha_{i_k,j_k-1})>0} c_{(i_k,j_k)}$.
Regarding the  normal vectors for weight inequalities $c_{[i:k]}$ (see \eqref{eq: def wt ineq GP}), observe that for $k>n_i^{\w}$ we obtain $c_i$ from $c_{[i:k]}$ when setting those $c_{(i_k,j_k)}$ to zero with $w(\alpha_{i_k,j_k-1})>0$.
Hence, $\mathcal Q_{\w}=\res_{\w}(\mathcal C_{\w_0})=\mathcal C_{\w}$ by Propositon~\ref{prop:wt GP res}.
Then $Q_{\w}=C_{\w}$ is a direct consequence and identifying $\Lambda^+$ with $\mathbb R^{n-1}$ using the fundamental weights, (iii) follows.
\end{proof}

\subsection{Double Bruhat cells and the superpotential}\label{subsec:super}

Recall the introduction to cluster varieties given in \S\ref{sec: prep cluster} as well as sections \S\ref{sec:pre rep theory} and \S\ref{sec:pre flag}.
In this section we explain the $\A$-cluster variety that can be associated to the quiver from a pseudoline arrangement.
Based on results of Berenstein-Fomin-Zelevinsky this variety is a double Bruhat cell (see Definition~\ref{def: double bruhat cell}).
We apply the construction of \cite{GHKK14} (see \S\ref{sec: prep cluster}) and recall results of Magee in \cite{Mag15} regarding the superpotential.

Recall that $SL_n$ has two cell decompositions (\emph{the Bruhat decompositions}) in terms of Bruhat cells indexed by elements of the symmetric group
\[
SL_n=\bigcup_{u\in S_n}BuB=\bigcup_{v\in S_n} B^-vB^-.
\]

\begin{definition}\label{def: double bruhat cell}
The \emph{double Bruhat cell} associated to $e$ and $w$ in $S_n$ is
\[
G^{e,w}:=B\cap B^-wB^-\subset SL_n.
\]
\end{definition}

The cluster structure of $G^{e,w}$ can be established as follows. Choose a reduced expression $\w$ and consider $\pa(\w)$.
Recall from Definition~\ref{def:quiver pa} that every face of $\pa(\w)$ corresponds to a vertex of $Q_{\w}$.
We therefore associate cluster variables to faces of $\pa(\w)$.
Let $F$ be such a face and assume the lines $l_{j_1},\dots,l_{j_k}$ pass below $F$.
In particular, $F$ is of level $k$.
Then associate the Pl\"ucker coordinate $\bar p_{j_1,\dots,j_k}\in\mathbb C[SL_n]$ to $F$, i.e. the minor of the columns $[k]$ and rows $\{{j_1},\dots,{j_k}\}$.
To remember it was associated with $F$, we set $A_F:=\bar p_{j_1,\dots,j_k}$.

\begin{definition}
Let $w\in S_n$ with reduced expression $\w$. Then the quiver $Q_{\w}$ together with the set of cluster variables $\mathbf A_{\w}:=\{A_F\mid F \text{ a face of }\pa(\w)\}$ forms the seed $s_{\w}:=(\mathbf A_{\w},Q_{\w})$.
\end{definition}

\begin{example}\label{exp: initial cluster var. S_5}
Recall from Example~\ref{exp:initial seed S_5} and Figure~\ref{fig:initial} the pseudoline arrangement $\pa(\hat \w_0)$ and the quiver $Q_{\hat \w_0}$ for $\hat\w_0\in S_5$.
To a face $F_{(i,j)}$ with $i\in[n-1]$ and $j\in[i+1,n]$ we associate following the above recipe the cluster variable $A_{(i,j)}:=p_{i+1,\dots,j}$.
To the faces unbounded $F_i, i\in[4]$ to the left, we associate the variables $A_i:=p_{5-i+1,\dots,5}$.
Note that the variables associated to the frozen vertices on the left (from bottom to top) are $\bar p_5,\bar p_{45},\bar p_{345},\bar p_{2345}$ and those associated to the frozen vertices on the right are $\bar p_{1},\bar p_{12},\bar p_{123},\bar p_{1234}$.
These Pl\"ucker coordinates are called \emph{consecutive minors}.
The collection of all cluster variables associated to this initial seed is
\[
\mathbf A_{\hat \w_0}=\{\bar p_{1},\bar p_{2},\bar p_{3},\bar p_4,\bar p_5,\bar p_{12},\bar p_{23},\bar p_{34},\bar p_{45},\bar p_{123},\bar  p_{234},\bar p_{345},\bar p_{1234},\bar
p_{2345}\}.
\]
\end{example}

\begin{example}\label{exp:initial seed w_0}
Consider $\hat\w_0\in S_n$ as in Examples~\ref{exp:pathGT} and \ref{exp:areaGT}. Then the collection $\mathbf A_{\hat \w_0}$ of associated cluster variables is
\[
\mathbf A_{\hat\w_0}=\{\bar p_{i,\dots, j}\mid i\in[n-1],j\in[i,n] \}, 
\]
where $\bar p_{i,\dots,j}$ is a frozen variable if either $i=1$ or $j=n$. Note that $\bar p_{1,\dots,n}=\det$, which is constant on $SL_n$, hence we disregard it. From now on we denote by $s_0$ the seed $s_0:=s_{\hat \w_0}=(\mathbf A_{\hat \w_0},Q_{\hat\w_0})$.
\end{example}

Berenstein-Fomin-Zelevinsky show
\begin{theorem*}(\cite[Theorem~2.10]{BFZ05})
Let $w\in S_n$ with reduced expression $\w$. Then for the upper cluster algebra $\overline{\mathcal{Y}}(s_{\w})$ we haven an isomorphism of algebras
\begin{eqnarray}\label{eq: cluster alg G^e,w}
\overline{\mathcal{Y}}(s_{\w})\otimes_{\mathbb Z} \mathbb C\cong \mathbb C[G^{e,w}].
\end{eqnarray}
\end{theorem*}

In particular, the Theorem implies the following: 
if $\w_1$ and $\w_2$ are two reduced expressions of $w\in S_n$ related by mutation in the sense of Definition~\ref{defn:mutation pa}, then the associated seeds $s_{\w_1}$ and $s_{\w_2}$ are related by cluster ($\A$-)mutation in the sense of Definition~\ref{def: A mutation}.
This explains our abuse of notation using the same letter $\mu$ for both types of mutation.

We now focus on the $\A$-cluster variety $G^{e,w_0}$ and the natural partial compactification using the frozen variables to study the superpotential as in \S\ref{sec: prep cluster}.
We partially compactify $G^{e,w_0}$ to $\bar G^{e,w_0}$ by allowing the frozen variables $\bar p_{[i]}$ and $\bar p_{[n-i,n]}$ for $i\in[n-1]$ to vanish.
Denote the resulting boundary divisor $D\subset \bar G^{e,w_0}$ and its irreducible components by
\[
D_{i}:=\{\bar p_{[i]}=0\}, \ \text{ resp. } \ D_{i,n}:=\{\bar p_{{[n-i,n]}}=0\}.
\]
There is an open embedding
$G^{e,w_0}\hookrightarrow SL_n/U$ given by  $g\mapsto g^tU$ and
up to codimension 2 the variety $\bar G^{e,w_0}$ agrees with $SL_n/U$ (this follows, for example, from \cite[Proposition~23]{Mag15}).
Hence, we have an isomorphism of rings $\mathbb C[\bar G^{e,w_0}]\cong \mathbb C[SL_n/U]$.
One of Magee's main results in \cite{Mag15} is the following.
\begin{theorem*}(\cite[Corollary~3]{Mag15})
The full Fock-Goncharov conjecture holds for $SL_n/U$.
\end{theorem*}

Moreover, Magee shows that there exists an optimized seed for every frozen vertex and therefore we can apply Algorithm~\ref{alg:superpot via opt seeds} stated in \S\ref{sec: prep cluster} to compute the superpotential. 
This is indeed what Magee did for the intial seed $s_0$ (see Example~\ref{exp:initial seed w_0}). 
Let $\X$ denote the Fock-Goncharov dual to the $\A$-cluster variety $G^{e,w_0}$ (see Definition~\ref{def:cluster variety}).
Recall that in the initial seed $s_0$ we have $N_{s_0}\cong \mathbb Z^{N+n-1}$ with basis $\{e_F\mid F\text{ face of }\pa(\w_0)\}$.
As before we set $e_{F_{(i,j)}}=:e_{(i,j)}$ and $e_{F_k}=:e_k$.
Further, recall that the superpotential $W:\X\to\mathbb C$ is given by the sum of $\vartheta$-functions associated to frozen variables. 
We denote by $\vartheta_{i}$ (resp. $\vartheta_{(i,n)}$) the $\vartheta$-functions associated to the frozen vertex $w_i$ (resp. $w_{(i,n)}$) in the inital quiver $Q_{s_0}$ (see Figure~\ref{fig:initial}) for $i\in[n-1]$.

\begin{proposition*}(\cite[Corollary~24]{Mag15})
Let $W:\X\to\mathbb C$ denote the superpotential. Then we have
$W\vert_{\X_{s_0}}=\sum_{i=1}^{n-1} \vartheta_{i}\vert_{\X_{s_0}}+\vartheta_{(n-i,n)}\vert_{\X_{s_0}}$, where 
\[
\vartheta_{i}\vert_{\X_{s_0}}=\sum_{k=0}^{n-1-i}z^{-e_i-\sum_{j=1}^k e_{(j,i+j)}}, \ \text{ and } \
\vartheta_{(i,n)}\vert_{\X_{s_0}}=\sum_{k=0}^{n-1-i}z^{-\sum_{j=0}^k e_{(i,n-j)}}, \ \text{ for } \ i\in[n-1].
\]
\end{proposition*}

\begin{example}
Consider $S_3$ and the initial seed with quiver $Q_{s_1s_2s_1}$. Then 
\begin{eqnarray*}
W\vert_{\X_{s_0}}&=&\vartheta_{(1,3)}+\vartheta_{(2,3)}+\vartheta_{1}+\vartheta_{2}\\
&=&z^{-e_{(1,3)}}+z^{-e_{(1,3)}-e_{(1,2)}}+z^{-e_{(2,3)}}+z^{-e_{1}}+z^{-e_{1}-e_{(1,2)}}+z^{-e_{2}}. \end{eqnarray*} 
\end{example}

\vbox{
\begin{definition}\label{def:super cone}
For $\w_0$ a reduced expression of $w_0\in S_n$ we define the following polyhedral objects by tropicalizing a sum of $\vartheta$-functions resp. the superpotential:
\begin{eqnarray*}
\Xi_{\w_0}&:=&\{\mathbf x\in \mathbb R^{N+n-1}\mid W\vert_{\X_{\w_0}}^{\trop}(\mathbf x)\ge0\},\\
\mathsf \Xi_{\w_0}&:=& \{\mathbf x\in \mathbb R^N\mid (\sum_{i=1}^{n-1}\vartheta_{(i,n)}\vert_{\X_{\w_0}})^{\trop}(\mathbf x)\ge 0\},\\
\Xi_{\w_0}(\lambda)&:=& \Xi_{\w_0}\cap \tau^{-1}_{\w_0}(\lambda) \text{ for } \lambda\in \mathbb R^{n-1}.
\end{eqnarray*}
\end{definition}}

The $\A_{\text{prin}}$-construction in \cite{GHKK14} applied to our setting defines a flat family over $\mathbb A^{N-2n+2}$ for every choice of seed, in particular for every ${\w}_0$.
The central fibre is by \cite[Theorem~8.39]{GHKK14} the toric variety associated to $\Xi_{\w_0}(\lambda)$ for $\lambda\in\mathbb Z_{>0}^{n-1}$.
One generic fibre is $SL_n/B$, hence we have a toric degeneration of the flag variety.
We do not go into the details on this construction but refer the reader to \cite[\S8]{GHKK14}.

\paragraph{Relating to the area cones} Let $\w_0$ be an arbitrary reduced expression of $w_0\in S_n$. In what follows we show how to obtain an expression of the superpotential in any seed $s_{\w_0}$ associated to $\w_0$ by ``detropicalizing" the weighted cone $\mathcal S_{\w_0}$. 
We define it more generally for $\w$ a reduced expression of $w\in S_n$.
Denote by $\X_{\w}$ the cluter torus associated to the seed $s_{\w}$.

\begin{definition}\label{def:detrop area cone}
Let $\w$ be an arbitrary reduced expression of $w\in S_n$. Then the \emph{detropicalization} of the cone $\mathcal S_{\w}$ is defined as the function $W_{\mathcal S_{\w}}:\X_{\w}\to \mathbb C$ with
\begin{eqnarray}
W_{\mathcal S_{\w}}:=\sum_{\p\in\mathcal P_{\w}}z^{e_{\p}}+\sum_{i\in[n-1],0\le k\le n_i}z^{e_{[i:k]}}.
\end{eqnarray}
\end{definition}

The name is self-explanatory, observe that by definition we have
\[
\{\mathbf x\in \mathbb R^{\ell(w)+n-1}\mid W_{\mathcal S_{\w}}^{\trop}(\mathbf x)\ge 0\}=\mathcal S_{\w}.
\]

\begin{proposition}\label{prop: superpot area GT}
Let $\w_0=s_1s_2s_1\cdots s_{n-1}s_{n-2}\cdots s_2s_1$ be the reduced expression associated to the initial seed $s_0$ as above. Then $W_{\mathcal S_{\w_0}}=W\vert_{\X_{s_0}}$.
\end{proposition}

\begin{proof}
Recall from Example~\ref{exp:areaGT} the expressions $e_{\p_{i,j}}$ \eqref{eq: e_p in GT w_0} and $e_{[i:k]}$ \eqref{eq: f_i,k in GT w_0} for $i\in[n-1]$ and $j,k\in[i+1,n]$.
In comparison with \cite[Corollary~24]{Mag15} (restated above) we obtain
\begin{eqnarray*}
\vartheta_{(i,n)}\vert_{\X_{s_0}}= \sum_{j=i+1}^n z^{e_{\p_{i,j}}}, \ \text{ and } \
\vartheta_{i}\vert_{\X_{s_0}}= \sum_{k=0}^{n-1-i} z^{e_{[i:k]}}.
\end{eqnarray*}
As from Example~\ref{exp:pathGT} we know $\mathcal P_{\w_0}=\{\p_{i,j}\mid i\in[n-1],j\in[i+1,n]\}$, the claim follows.
\end{proof}


\subsubsection*{Mutation of $\mathcal S_{\w}$}\label{subsubsec: path and area mut}

Our aim is to generalize Proposition~\ref{prop: superpot area GT} for arbitrary reduced expressions $\w_0$. 
We achieve this by showing that the detropicalization of $\mathcal S_{\w_0}$ behaves as the superpotential does when applying $\mathcal X$-mutation. 
Further, we show that if $\mu(\w)$ and $\w$ are reduced expressions of $w\in S_n$, then $\mu^*(W_{\mathcal S_{\mu(\w)}})=W_{\mathcal S_{\w}}$, where $\mu^*: \mathbb C[\X_{\mu(\w)}]\to\mathbb C[\X_{\w}]$ is the pull-back of the cluster mutation as in \eqref{eq: def pullback X-mut}.
This follows from Lemma~\ref{lem: mutation paths} and Lemma~\ref{lem:mutation wt area}.
Recall from Definition~\ref{defn:mutation pa} the mutation of pseudoline arrangements.
The core of this subsection is the case-by-case analysis of how mutation effects GP-paths.

In Figure~\ref{fig:loc.orient} we display locally around the mutable face $F=F_{(i,j)}$ (resp. $F'$) the orientations of $\pa(\w)$ (resp. $\pa(\mu_F(\w))$). 
The red arrows indicate which passages are forbidden in GP-paths. 
In Tables~\ref{tab:(lll) cases of path mut} to \ref{tab:(rrr) cases of path mut} we list in the second column all possibilities how a GP-path ${\p}$ locally looks around the face $F$. 
In the third column of each table is a complete list of how GP-paths look locally around the face $F'$ obtained from $F$ by mutation $\mu_F$.

\begin{figure}
\centering
\begin{center}
\begin{tikzpicture}[scale=0.75]

\begin{scope}[yshift=-6cm]
\node at (-2,1.25) {\small $i\le r$};
\node at (-2,.75) {\small $r+1\le j$};

\draw [fill] (.55,2) circle [radius=0.05];
\draw [fill] (.55,1) circle [radius=0.05];
\draw [fill] (.55,0) circle [radius=0.05];
\draw [fill] (6.5,2) circle [radius=0.05];
\draw [fill] (6.5,1) circle [radius=0.05];
\draw [fill] (6.5,0) circle [radius=0.05];
\node at (0,2) {$a_{k}$};
\node at (0,1) {$a_{j}$};
\node at (0,0) {$a_{i}$};
\node at (7,2) {$b_{i}$};
\node at (7,1) {$b_{j}$};
\node at (7,0) {$b_{k}$};

\draw[fill] (3.5,1.5) circle [radius=.05];
\draw[fill] (2.5,.5) circle [radius=.05];
\draw[fill] (4.5,.5) circle [radius=.05];

\node[right] at (3.5,1.5) {$v_{(i,k)}$};
\node[right] at (2.5,.5) {$v_{(i,j)}$};
\node[right] at (4.5,.5) {$v_{(j,k)}$};

    \draw[rounded corners] (0.5,0) --(1.25,0) -- (2,0)-- (3,1) --(4,2) -- (5.25,2);
    \draw[->, rounded corners] (1.5,0) -- (1.25,0);
    \draw[->, rounded corners] (3.5,1.5) -- (3,1);
    \draw[->, rounded corners] (6.5,2) -- (5.25,2);
\draw[rounded corners] (1.25,1) -- (2,1) -- (3,0) -- (3.5,0)-- (4,0) -- (5,1) -- (5.57,1)-- (6.5,1);
    \draw[->] (0.5,1) -- (1.25,1);
    \draw[->] (3.25,0) -- (3.5,0);
    \draw[->] (5.5,1) -- (5.75,1);
\draw[rounded corners] (1.5,2) -- (3,2) -- (4,1) -- (5,0) -- (5.75,0) -- (6.5,0);
    \draw[->] (0.5,2) -- (1.5,2);
    \draw[->] (3.5,1.5) -- (4,1);
    \draw[->] (5.5,0) -- (5.75,0);
    \draw[red, ->, ultra thick] (4.1,0.9) -- (4.9,0.1);
    
\begin{scope}[xshift=9cm]
\draw [fill] (.55,2) circle [radius=0.05];
\draw [fill] (.55,1) circle [radius=0.05];
\draw [fill] (.55,0) circle [radius=0.05];
\draw [fill] (6.5,2) circle [radius=0.05];
\draw [fill] (6.5,1) circle [radius=0.05];
\draw [fill] (6.5,0) circle [radius=0.05];
\node at (0,2) {$a'_{k}$};
\node at (0,1) {$a'_{j}$};
\node at (0,0) {$a'_{i}$};
\node at (7,2) {$b'_{i}$};
\node at (7,1) {$b'_{j}$};
\node at (7,0) {$b'_{k}$};

\draw[fill] (3.5,.5) circle [radius=.05];
\draw[fill] (2.5,1.5) circle [radius=.05];
\draw[fill] (4.5,1.5) circle [radius=.05];
\node[right] at (3.5,.5) {$v'_{(i,k)}$};
\node[right] at (2.5,1.5) {$v'_{(j,k)}$};
\node[right] at (4.5,1.5) {$v'_{(i,j)}$};

\draw[rounded corners] (0.5,0) --(1.25,0) -- (3,0)-- (4,1) -- (5,2) -- (5.25,2);
\draw[rounded corners] (0.5,1) -- (2,1) -- (3,2) -- (4,2) -- (5,1) -- (6.5,1);
\draw[rounded corners] (1.5,2) -- (2,2) -- (3,1)-- (4,0) -- (5.75,0) -- (6.5,0);
 \draw[->] (0.5,1) -- (1.25,1);
\draw[->, red , ultra thick]  (2.25,1.75) -- (3,1);
    \draw[->] (2.75,0) -- (1.25,0);
   \draw[->] (4.85,1.85) -- (4,1);
    \draw[->] (6.5,2) -- (5.25,2);
\draw[->, rounded corners] (1.25,1) -- (2,1) -- (2.95,1.95);
\draw[->, rounded corners] (3.5,2) -- (4,2) -- (4.75,1.25);
\draw[->, rounded corners] (5.75,1)-- (6,1);
    \draw[->] (0.5,2) -- (1.5,2);
    \draw[->] (4.25,0) -- (5.75,0);      

\end{scope}
\end{scope}
\begin{scope}[yshift=-9cm]

\node at (-2,1) {\small $r+1\le i$};

\draw [fill] (.55,2) circle [radius=0.05];
\draw [fill] (.55,1) circle [radius=0.05];
\draw [fill] (.55,0) circle [radius=0.05];
\draw [fill] (6.5,2) circle [radius=0.05];
\draw [fill] (6.5,1) circle [radius=0.05];
\draw [fill] (6.5,0) circle [radius=0.05];
\node at (0,2) {$a_{k}$};
\node at (0,1) {$a_{j}$};
\node at (0,0) {$a_{i}$};
\node at (7,2) {$b_{i}$};
\node at (7,1) {$b_{j}$};
\node at (7,0) {$b_{k}$};

\draw[fill] (3.5,1.5) circle [radius=.05];
\draw[fill] (2.5,.5) circle [radius=.05];
\draw[fill] (4.5,.5) circle [radius=.05];

\node[right] at (3.5,1.5) {$v_{(i,k)}$};
\node[right] at (2.5,.5) {$v_{(i,j)}$};
\node[right] at (4.5,.5) {$v_{(j,k)}$};

    \draw[rounded corners] (0.5,0) --(1.25,0) -- (2,0)-- (3,1) --(4,2) -- (5.25,2) -- (6.5,2);
    \draw[<-, rounded corners] (1.5,0) -- (1.25,0);
    \draw[<-, rounded corners] (3.1,1.1) -- (3,1);
    \draw[<-, rounded corners] (5.75,2) -- (5.25,2);
\draw[rounded corners] (1.25,1) -- (2,1) -- (3,0) -- (3.5,0)-- (4,0) -- (5,1) -- (5.57,1)-- (6.5,1);
    \draw[->] (0.5,1) -- (1.25,1);
    \draw[->] (3.25,0) -- (3.5,0);
    \draw[->] (5.5,1) -- (5.75,1);
\draw[rounded corners] (1.5,2) -- (3,2) -- (4,1) -- (5,0) -- (5.75,0) -- (6.5,0);
    \draw[->] (0.5,2) -- (1.5,2);
    \draw[->] (5.5,0) -- (5.75,0);    
\draw[->,red, ultra thick] (2.1,0.9) -- (2.9,0.1);  
\draw[->,red, ultra thick] (3.1,1.9) -- (4,1);
\draw[red, ultra thick, ->] (4,1) -- (4.9,0.1);

\begin{scope}[xshift=9cm]

\draw [fill] (.55,2) circle [radius=0.05];
\draw [fill] (.55,1) circle [radius=0.05];
\draw [fill] (.55,0) circle [radius=0.05];
\draw [fill] (6.5,2) circle [radius=0.05];
\draw [fill] (6.5,1) circle [radius=0.05];
\draw [fill] (6.5,0) circle [radius=0.05];
\node at (0,2) {$a'_{k}$};
\node at (0,1) {$a'_{j}$};
\node at (0,0) {$a'_{i}$};
\node at (7,2) {$b'_{i}$};
\node at (7,1) {$b'_{j}$};
\node at (7,0) {$b'_{k}$};

\draw[fill] (3.5,.5) circle [radius=.05];
\draw[fill] (2.5,1.5) circle [radius=.05];
\draw[fill] (4.5,1.5) circle [radius=.05];
\node[right] at (3.5,.5) {$v'_{(i,k)}$};
\node[right] at (2.5,1.5) {$v'_{(j,k)}$};
\node[right] at (4.5,1.5) {$v'_{(i,j)}$};

\draw[rounded corners] (0.5,0) --(1.25,0) -- (3,0)-- (4,1) -- (5,2) -- (5.25,2);
\draw[rounded corners] (0.5,1) -- (2,1) -- (3,2) -- (4,2) -- (5,1) -- (6.5,1);
\draw[rounded corners] (1.5,2) -- (2,2) -- (3,1)-- (4,0) -- (5.75,0) -- (6.5,0);
    \draw[<-] (2.75,0) -- (1.25,0);
    \draw[<-] (4.85,1.85) -- (4,1);
    \draw (6.5,2) -- (5.25,2);
    \draw[->] (5.25,2) -- (6,2);
\draw[->, rounded corners] (1.25,1) -- (2,1) -- (2.95,1.95);
\draw[->, rounded corners] (5.75,1)-- (6,1);
\draw[->] (4.25,0) -- (5.75,0);   
    \draw[->] (0.5,2) -- (1.5,2); 
     \draw[->] (0.5,1) -- (1.25,1);
    \draw[->, red, ultra thick] (4.1,1.9) -- (4.9,1.1);
    \draw[->, red, ultra thick] (2.1,1.9) -- (3,1);
    \draw[red, ultra thick, ->] (3,1) -- (3.9,0.1);

\end{scope}
\end{scope}

\begin{scope}[yshift=0cm]
\node at (3.5,3) {$\pa(\w)$};
\node at (8,3) {$\xrightarrow{\mu_{F}}$};

\node at (-2,1) {\small$k\le r$};

\draw [fill] (.55,2) circle [radius=0.05];
\draw [fill] (.55,1) circle [radius=0.05];
\draw [fill] (.55,0) circle [radius=0.05];
\draw [fill] (6.5,2) circle [radius=0.05];
\draw [fill] (6.5,1) circle [radius=0.05];
\draw [fill] (6.5,0) circle [radius=0.05];
\node at (0,2) {$a_{k}$};
\node at (0,1) {$a_{j}$};
\node at (0,0) {$a_{i}$};
\node at (7,2) {$b_{i}$};
\node at (7,1) {$b_{j}$};
\node at (7,0) {$b_{k}$};

\draw[fill] (3.5,1.5) circle [radius=.05];
\draw[fill] (2.5,.5) circle [radius=.05];
\draw[fill] (4.5,.5) circle [radius=.05];

\node[right] at (3.5,1.5) {$v_{(i,k)}$};
\node[right] at (2.5,.5) {$v_{(i,j)}$};
\node[right] at (4.5,.5) {$v_{(j,k)}$};

\draw[rounded corners] (0.5,0) --(1.25,0) -- (2,0)-- (3,1) --(4,2) -- (5.25,2);
    \draw[->, rounded corners] (1.5,0) -- (1.25,0);
    \draw[->, rounded corners] (3.5,1.5) -- (3,1);
    \draw[->, rounded corners] (6.5,2) -- (5.25,2);
\draw[rounded corners] (.5,1) -- (2,1) -- (3,0) -- (3.5,0)-- (4,0) -- (5,1) -- (5.57,1)-- (6.5,1);
    \draw[<-] (1,1) -- (1.25,1);
    \draw[<-] (3.25,0) -- (3.5,0);
    \draw[<-] (5.5,1) -- (5.75,1);
\draw[rounded corners] (.5,2) -- (3,2) -- (4,1) -- (5,0) -- (5.75,0) -- (6.5,0);
    \draw[<-] (1,2) -- (1.5,2);
    \draw[<-] (3.5,1.5) -- (4,1);
    \draw[<-] (5.5,0) -- (5.75,0);
    \draw[<-] (4.1,0.9) -- (4.9,0.1);
   
\draw[->,red, ultra thick] (3,1) -- (2.1,0.1);
\draw[->,red, ultra thick] (3.9,1.9) -- (3,1);
\draw[<-,red, ultra thick] (4.1,0.1) -- (4.9,.9);
    
\begin{scope}[xshift=9cm]
\node at (3.5,3) {$\pa(\mu_F(\w))$};

\draw [fill] (.55,2) circle [radius=0.05];
\draw [fill] (.55,1) circle [radius=0.05];
\draw [fill] (.55,0) circle [radius=0.05];
\draw [fill] (6.5,2) circle [radius=0.05];
\draw [fill] (6.5,1) circle [radius=0.05];
\draw [fill] (6.5,0) circle [radius=0.05];
\node at (0,2) {$a'_{k}$};
\node at (0,1) {$a'_{j}$};
\node at (0,0) {$a'_{i}$};
\node at (7,2) {$b'_{i}$};
\node at (7,1) {$b'_{j}$};
\node at (7,0) {$b'_{k}$};

\draw[fill] (3.5,.5) circle [radius=.05];
\draw[fill] (2.5,1.5) circle [radius=.05];
\draw[fill] (4.5,1.5) circle [radius=.05];
\node[right] at (3.5,.5) {$v'_{(i,k)}$};
\node[right] at (2.5,1.5) {$v'_{(j,k)}$};
\node[right] at (4.5,1.5) {$v'_{(i,j)}$};

\draw[rounded corners] (0.5,0) --(1.25,0) -- (3,0)-- (4,1) -- (5,2) -- (5.25,2);
\draw[rounded corners] (0.5,1) -- (2,1) -- (3,2) -- (4,2) -- (5,1) -- (5.75,1);
\draw[rounded corners] (1.5,2) -- (2,2) -- (3,1)-- (4,0) -- (5.75,0) -- (6.5,0);
    \draw[->] (2.75,0) -- (1.25,0);
    \draw[->,red, ultra thick] (4.9,1.9) -- (4,1);
    \draw[->] (6.5,2) -- (5.25,2);
\draw[<-, rounded corners] (1.25,1) -- (2,1) -- (2.75,1.75);
\draw[<-,red, ultra thick] (2,1) -- (2.9,1.9);
\draw[<-, rounded corners] (3.5,2) -- (4,2) -- (4.75,1.25);
\draw[<-, rounded corners] (5.57,1)-- (6.5,1);
    \draw[<-] (0.5,2) -- (1.5,2);
    \draw[<-] (2.25,1.75) -- (3,1);
    \draw[<-] (4.25,0) -- (5.75,0);

\end{scope}
\end{scope}

\begin{scope}[yshift=-3cm]
\node at (-2,1.25) {\small$j\le r$};
\node at (-2,.75) {\small$r+1\le k$};
\draw [fill] (.55,2) circle [radius=0.05];
\draw [fill] (.55,1) circle [radius=0.05];
\draw [fill] (.55,0) circle [radius=0.05];
\draw [fill] (6.5,2) circle [radius=0.05];
\draw [fill] (6.5,1) circle [radius=0.05];
\draw [fill] (6.5,0) circle [radius=0.05];
\node at (0,2) {$a_{k}$};
\node at (0,1) {$a_{j}$};
\node at (0,0) {$a_{i}$};
\node at (7,2) {$b_{i}$};
\node at (7,1) {$b_{j}$};
\node at (7,0) {$b_{k}$};

\draw[fill] (3.5,1.5) circle [radius=.05];
\draw[fill] (2.5,.5) circle [radius=.05];
\draw[fill] (4.5,.5) circle [radius=.05];

\node[right] at (3.5,1.5) {$v_{(i,k)}$};
\node[right] at (2.5,.5) {$v_{(i,j)}$};
\node[right] at (4.5,.5) {$v_{(j,k)}$};

\draw[rounded corners] (0.5,0) --(1.25,0) -- (2,0)-- (3,1) --(4,2) -- (5.25,2);
    \draw[->, rounded corners] (1.5,0) -- (1.25,0);
    \draw[->, rounded corners] (3.5,1.5) -- (3,1);
    \draw[->, rounded corners] (6.5,2) -- (5.25,2);
\draw[rounded corners] (.5,1) -- (2,1) -- (3,0) -- (3.5,0)-- (4,0) -- (5,1) -- (5.57,1)-- (6.5,1);
    \draw[<-] (1,1) -- (1.25,1);
    \draw[<-] (3.25,0) -- (3.5,0);
    \draw[<-] (5.5,1) -- (5.75,1);
\draw[rounded corners] (1.5,2) -- (3,2) -- (4,1) -- (5,0) -- (5.75,0) -- (6.5,0);
    \draw[->] (0.5,2) -- (1.5,2);
    \draw[->] (3.5,1.5) -- (4,1);
    \draw[->] (5.5,0) -- (5.75,0);
    \draw[->] (4.1,0.9) -- (4.9,0.1);
\draw[<-,red, ultra thick] (2.1,0.1)-- (2.9,0.9);

\begin{scope}[xshift=9cm]
\draw [fill] (.55,2) circle [radius=0.05];
\draw [fill] (.55,1) circle [radius=0.05];
\draw [fill] (.55,0) circle [radius=0.05];
\draw [fill] (6.5,2) circle [radius=0.05];
\draw [fill] (6.5,1) circle [radius=0.05];
\draw [fill] (6.5,0) circle [radius=0.05];

\node at (0,2) {$a'_{k}$};
\node at (0,1) {$a'_{j}$};
\node at (0,0) {$a'_{i}$};
\node at (7,2) {$b'_{i}$};
\node at (7,1) {$b'_{j}$};
\node at (7,0) {$b'_{k}$};

\draw[fill] (3.5,.5) circle [radius=.05];
\draw[fill] (2.5,1.5) circle [radius=.05];
\draw[fill] (4.5,1.5) circle [radius=.05];
\node[right] at (3.5,.5) {$v'_{(i,k)}$};
\node[right] at (2.5,1.5) {$v'_{(j,k)}$};
\node[right] at (4.5,1.5) {$v'_{(i,j)}$};

\draw[rounded corners] (0.5,0) --(1.25,0) -- (3,0)-- (4,1) -- (5,2) -- (5.25,2);
\draw[rounded corners] (0.5,1) -- (2,1) -- (3,2) -- (4,2) -- (5,1) -- (5.75,1);
\draw[rounded corners] (1.5,2) -- (2,2) -- (3,1)-- (4,0) -- (5.75,0) -- (6.5,0);
    \draw[->] (2.75,0) -- (1.25,0);
    \draw[->,red, ultra thick] (4.9,1.9) -- (4,1);
    \draw[->] (6.5,2) -- (5.25,2);
\draw[<-, rounded corners] (1.25,1) -- (2,1) -- (2.75,1.75);
\draw[<-, rounded corners] (3.5,2) -- (4,2) -- (4.75,1.25);
\draw[<-, rounded corners] (5.57,1)-- (6.5,1);
    \draw[->] (0.5,2) -- (1.5,2);
    \draw[->] (2.25,1.75) -- (3,1);
    \draw[->] (4.25,0) -- (5.75,0);
\end{scope}
\end{scope}

\end{tikzpicture}
\end{center}
\caption{The pseudoline arrangement $\pa(\w)$ (resp. $\pa(\mu_F(\w))$) locally around the face $F=F_{(i,j)}$ (resp. $F'=F'_{(j,k)}$) bounded by lines $l_i,l_j,l_k$ with $i<j<k$ and orientations $(l_r,l_{r+1})$ for all possible $r$.
The red arrows are those forbidden in GP-paths.}\label{fig:loc.orient}
\end{figure}
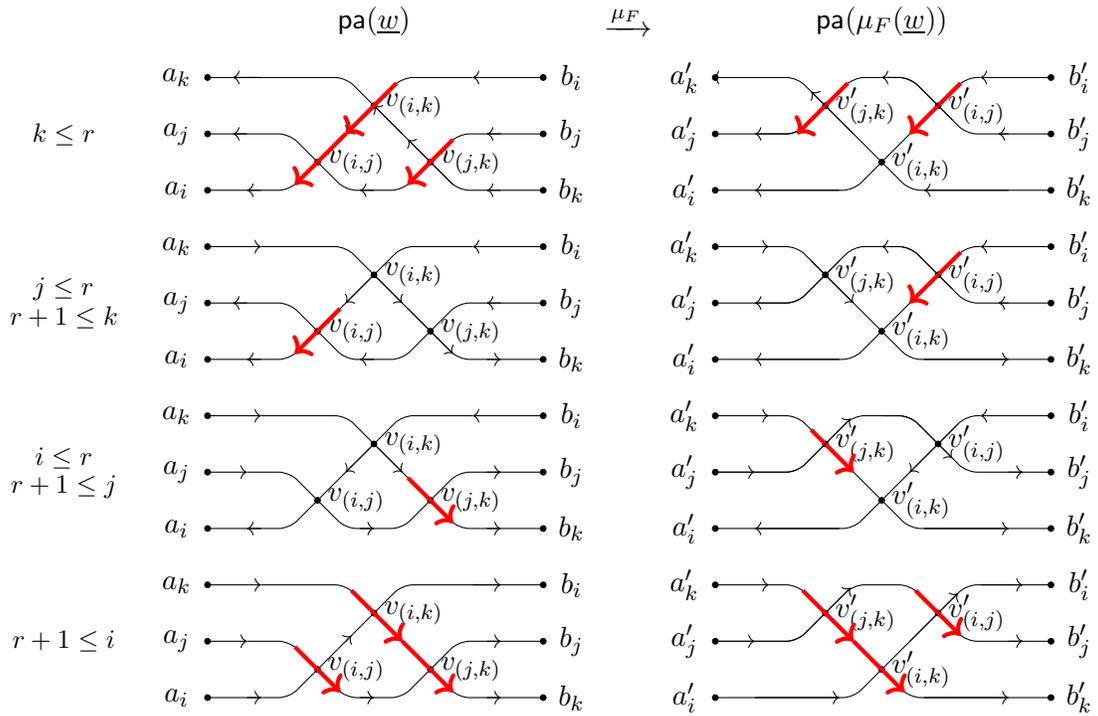

Recall the arrows for the quiver corresponding to $\pa(\w)$ and $\pa(\mu_F(\w))$ from Figure~\ref{fig:pseudo.mut}.
We call a face $E$ \emph{incoming} (resp. \emph{outgoing}) with respect to $F$ in $\pa(\w)$, if there is an arrow in the quiver $Q_{\w}$ from (resp. to) the vertex corresponding to $E$ to (resp. from) the vertex corresponding to $F$. 
We denote by $\text{In}_{F}$ the union of all incoming faces and by $\text{Out}_{F}$ the union of all outgoing faces.
See for example, Figure~\ref{fig:pseudo.mut}.

\vbox{
\begin{definition}\label{def:local type}
Let $\p\in\mathcal P_{\w}$ for $\w\in S_n$ and consider a mutable face $F$ of $\pa(\w)$. 
Set $\delta_{F\subset \area_{\p}}:=1$ if $F\in \area_{\p}$ and zero otherwise.
Then we define the \emph{$F$-local type} of $\p$ as the triple
\[
F(\p):=(i_{F,{\p}},x_{F,{\p}},o_{F,{\p}}):=(\#\{\text{In}_{F}\cap \area_{\p}\}, 
\delta_{F\in \area_{\p}},
\#\{\text{Out}_{F}\cap \area_{\p}\}).
\]

\end{definition}}
For example, if $\area_{\p}$ in Figure~\ref{fig:pseudo.mut} contains the faces $F,F_{\init_1}$ and $F_{\text{out}_2}$ but not $F_{\init_1}$ and $F_{\text{out}_1}$, then the $F$-local type of $\p$ is $(1,1,1)$.
The following lemma is a crucial observation on the $F$-local type of GP-paths.

    
    \begin{table}[]
    \centering 
    \scalebox{.8}{
    \begin{tabular}{|c|c|c|c|} \hline
$F$-local type of $\p$ & ${\p}$ in $\pa(\w)$ & 
    $\p'=\mut_{F}({\p})$ in $\pa(\mu_F(\w))$ & $F'$-local type of $\p'$\\ \hline\hline
$(2,1,2)$ &
    $b_i\to v_{(i,k)}\to a_k$ &
    $b'_i\to v'_{(i,j)}\to v'_{(j,k)} \to a'_k$ &
     $(2,1,2)$ \\ \hline

$(1,1,2)$ &
    $b_{j}\to v_{(j,k)}\to v_{(i,k)}\to a_k$ &
    \begin{tabular}{c} 
        $b'_j\to v'_{(i,j)}\to v'_{(j,k)}\to a'_k$ \\ \hline
        $b'_j\to v'_{(i,j)}\to v'_{(i,k)}\to v'_{(j,k)}\to a'_k$ \end{tabular} & 
    \begin{tabular}{c}
        $(2,1,1)$ \\ \hline
        $(2,0,1)$
    \end{tabular}   \\ \hline 

 $(1,1,1)$&        
    $ b_j\to v_{(j,k)}\to v_{(i,k)}\to v_{(i,j)}\to a_j$ &
    $ b'_j\to v'_{(i,j)}\to v'_{(i,k)}\to v'_{(j,k)}\to a'_j$ &
     $(1,0,1)$ \\ \hline
    
 $(1,1,1)$ &
        $b_k \to v_{(j,k)}\to v_{(i,k)}\to a_k $   &
        $b'_k\to v'_{(i,k)}\to v'_{(j,k)}\to a'_k$     &
          $(1,0,1)$   \\ \hline
    
\begin{tabular}{c}
      $(1,1,0)$ \\ \hline
      $(1,0,0)$ \end{tabular} &    
    \begin{tabular}{c}
         $b_k\to v_{(j,k)}\to v_{(i,k)}\to v_{(i,j)}\to a_j$  \\ \hline
          $b_k\to v_{(j,k)}\to v_{(i,j)}\to a_j$
    \end{tabular} &
    $ b'_k\to v'_{(i,k)}\to v'_{(j,k)}\to a'_j$ &
          $(0,0,1)$ \\ \hline
  
 $(0,0,0)$ &
    $b_k\to v_{(j,k)}\to v_{(i,j)}\to a_i$ &
    $b'_k\to v'_{(i,k)}\to a'_i$ &
          $(0,0,0)$ \\ \hline 
    
        \end{tabular}}
    \caption{Shapes of paths locally around $F$ (resp. $F'$) in $\mathcal P_{\w}$ (resp. $\mathcal P_{\mu_F(\w)}$) for orientation $(l_r,l_{r+1})$ with $i<j<k\le r$ (see Figure~\protect{\ref{fig:loc.orient}}) and how they are mapped onto each other by $\mut_F$.}
    \label{tab:(lll) cases of path mut}
\end{table}

\vbox{
\begin{lemma}\label{lem:possible F-local types}
Let $\p\in\mathcal P_{\w}$ for $\w\in S_n$ and consider a mutable face $F$ of $\pa(\w)$. Then the following are all possible $F$-local types $\p$ can have:
\begin{itemize}
    \item[$i_{F,\p}=o_{F,\p}$:] then $F(\p) \in \{(0,0,0),(1,0,1),(1,1,1),(2,1,2)\}$;
    \item[$i_{F,\p}<o_{F,\p}$:] then $F(\p) \in \{(1,1,2),(0,0,1)\}$;
    \item[$i_{F,\p}>o_{F,\p}$:] then $F(\p) \in \{(1,0,0),(1,1,0),(2,0,1),(2,1,1)\}$.
\end{itemize}
Moreover, the $F$-local types of $\p$ with $i_{F,\p}>o_{F,\p}$ come in pairs as $((1,0,0),(1,1,0))$ or $((2,0,1),(2,1,1))$. Meaning that if a path of one type exists for a fixed orientation then so does a path of the corresponding other type for the same orientation.
\end{lemma}}

\begin{proof}
The lemma follows from case-by-case consideration of all possible shapes of $\p\in\mathcal P_{\w}$ around a mutable face $F$ of $\pa(\w)$. First observe, that $F$ can have two different shapes, depending on whether it is defined by simple reflections $s_ms_{m+1}s_m$ (as on the left in Figure~\ref{fig:pseudo.mut}) or by $s_{m+1}s_ms_{m+1}$ (as on the right in Figure~\ref{fig:pseudo.mut}).
We endow $\pa(\w)$ for either case of $F$ with all possible orientations $(l_r,l_{r+1})$. Then locally at $F$, there are four cases of orientation depending on $r$ and $r+1$ in relation to $i,j,k$ (see Figure~\ref{fig:loc.orient}).
We consider all possibilities for the path $\p$ to pass $F$ for each case of orientation and shape of $F$.
These are listed in Tables~\ref{tab:(lll) cases of path mut} to \ref{tab:(rrr) cases of path mut}, in the second column for $F$ as on the left of Figure~\ref{fig:loc.orient} and in the third for $F$ as on the right of Figure~\ref{fig:loc.orient}.
In the first and last columns of these tables we indicate the corresponding $F$-local type.
Observe, that the list in the claim of the lemma covers all occurring $F$-local types.

Regarding the second part of the claim, this also follows as an observation from Tables~\ref{tab:(lll) cases of path mut} to \ref{tab:(rrr) cases of path mut}.
\end{proof}

\begin{table}[]
    \centering 
    \scalebox{.8}{
    \begin{tabular}{|c|c|c|c|} \hline
		$F$-local type of $\p$ & ${\p}$ in $\pa(\w)$ & $\p'=\mut_{F}({\p})$ in $\pa(\mu_F(\w))$ & $F'$-local type of $\p'$\\ \hline\hline	
		$(1,0,1)$ &   $b_i\to v_{(i,k)}\to v_{(j,k)}\to v_{(i,j)}\to a_i $ &  $b'_i\to v'_{(i,j)} \to v'_{(j,k)} \to v'_{(i,k)}\to a'_i$ & $(1,1,1)$ \\ \hline

		\begin{tabular}{c}
			$(2,1,1)$ \\ \hline
			$(2,0,1)$
		\end{tabular} &
		\begin{tabular}{c}
			$b_i\to v_{(i,k)}\to v_{(i,j)}\to a_j$   \\ \hline 
			$b_i\to v_{(i,k)}\to v_{(j,k)}\to v_{(i,j)}\to a_j$
		\end{tabular} &
		$ b'_i\to v'_{(i,j)}\to v'_{(j,k)}\to a'_j$ & $(1,1,2)$ \\ \hline

		$(1,0,1)$  & $b_i\to v_{(i,k)}\to v_{(j,k)}\to b_k$ & $ b'_i\to v'_{(i,j)}\to v'_{(j,k)}\to v'_{(i,k)}\to b'_k$ & $(1,1,1)$  \\ \hline
    
		$(0,0,1)$  & $b_j\to v_{(j,k)}\to v_{(i,j)}\to a_i$ &
		\begin{tabular}{c}
			$b'_j \to v'_{(i,j)} \to v'_{(j,k)} \to v'_{(i,k)} \to a'_i$ \\
			$b'_i \to v'_{(i,j)} \to v'_{(i,k)} \to a'_i$
		\end{tabular} &
		\begin{tabular}{c}
			$(1,1,0)$  \\ \hline
			$(1,0,0)$
		\end{tabular} \\ \hline    

		$(1,0,1)$  & $b_j\to v_{(j,k)}\to v_{(i,j)}\to a_j $ & $b'_j\to v'_{(i,j)}\to v'_{(j,k)}\to a'_j $ &  $(1,1,1)$  \\ \hline    
    
		$(0,0,1)$  &	$b_{j}\to v_{(j,k)}\to b_k  $ &
		\begin{tabular}{c}
			$b'_j\to v'_{(i,j)} \to v'_{(j,k)}\to v'_{(i,k)}\to b'_k $  \\ \hline 
			$b'_j\to v'_{(i,j)}\to v'_{(i,k)}\to b'_k$
		\end{tabular} &
		\begin{tabular}{c}
			$(1,1,0)$ \\ \hline
			$(1,0,0)$ 
		\end{tabular} \\ \hline      

		$(1,0,1)$   & $a_k\to v_{(i,k)}\to v_{(j,k)}\to v_{(i,j)}\to a_i $ & $ a'_k\to v'_{(j,k)}\to v'_{(i,k)}\to a'_i $ &  $(1,1,1)$ \\ \hline  

		\begin{tabular}{c}
			$(2,1,1)$  \\ \hline
			$(2,0,1)$ 
		\end{tabular}&
		\begin{tabular}{c}
			$a_k\to v_{(i,k)}\to v_{(i,j)}\to a_j $ \\ \hline 
			$a_k\to v_{(i,k)}\to v_{(j,k)}\to v_{(i,j)}\to a_j $
		\end{tabular} &
		$ a'_k\to v'_{(j,k)}\to a'_j $ & $(1,1,2)$ \\ \hline      
    
		$(1,0,1)$ & $a_k\to v_{(i,k)}\to v_{(j,k)}\to b_k$ & $ a'_k\to v'_{(j,k)}\to v'_{(i,k)}\to b'_k$ &  $(1,1,1)$  \\ \hline  

   \end{tabular}}
    \caption{Shapes of paths locally around $F$ (resp. $F'$) in $\mathcal P_{\w}$ (resp. $\mathcal P_{\mu_F(\w)}$) for orientation $(l_r,l_{r+1})$ with $i<j\le r$ and $r+1\le k$ (see Figure~\protect{\ref{fig:loc.orient}}) and how they are mapped onto each other by $\mut_F$.}
    \label{tab:(llr) cases of path mut}
\end{table}

With notation as in the lemma, if $\p_1,\p_2$ are paths with $i_{F,\p_j}>o_{F,\p_j}, j=1,2$ such that $((i_{F,\p_1},x_{F,\p_1},o_{F,\p_1}),(i_{F,\p_2},x_{F,\p_2},o_{F,\p_2}))$ is one of the pairs, then we denote by $\p_1\oplus\p_2$ their formal sum.
If $\p_1$ and $\p_2$ are equal away from $F$, we denote this by $\p_1/F=\p_2/F$.
Observe, that this is the case here.
With this notation we define the following set of paths, respectively formal sums of paths.
\begin{eqnarray}\label{eq: def P_w,F}
\widehat{\mathcal{P}}_{\w,F} := \left\{ \begin{matrix} \p,\\ \p_1\oplus\p_2\end{matrix} \right\vert 
\left.\begin{matrix}
\p\in\mathcal{P}_{\w} \text{ with }  i_{F,\p}=o_{F,\p} \text{ or } i_{F,\p}<o_{F,\p},\\
 \p_1,\p_2\in\mathcal P_{\w} \text{ with } i_{F,\p_j}>o_{F,\p_j}\text{ for } j=1,2
\end{matrix}\right\}.
\end{eqnarray}

Note that for every mutable face $F$ of $\pa(\w)$ every path in $\mathcal{P}_{\w}$ appears in $\widehat{\mathcal{P}}_{\w,F}$ either on its own or as a formal summand. This additional structure on $\mathcal{P}_{\w}$ allows us to define \emph{mutation} on it.

\begin{definition}\label{def:mut path}
Let $w\in S_n$ with reduced expressions $\w$ and $\mu_F(\w)$, where $F$ is a mutable face in $\pa(\w)$. Denote by $F'$ the corresponding face in $\pa(\mu_F(\w))$. We define  $\mut_F:\widehat{\mathcal{P}}_{\w,F}\to \widehat{\mathcal{P}}_{\mu_F(\w),F'}$ depending on the $F$-local type by 
\begin{itemize}
    \item[$i_{F,\p}=o_{F,\p}$:] $\mut_F(\p)=\p'$ with $\p/ F=\p'/ F'$, where for $F(\p)\in\{(0,0,0),(2,1,2)\}$ we have $F(\p)=F'(\p')$, and for $F(\p)\in\{(1,0,1),(1,1,1)\}$ we have $F'(\p')=(i_{F,\p},\vert x_{F,\p}-1\vert,o_{F,\p})$;
    \item[$i_{F,\p}<o_{F,\p}$:] $\mut_F(\p)=\p_1'\oplus\p_2'$ with $\p/F=\p'_1/F'=\p'_2/F'$, for $F(\p)\in\{(0,0,1),(1,1,2)\}$ with $F'(\p_1')=(o_{F,\p},x_{F,\p},i_{F,\p})$ and $F'(\p'_2)=(o_{F,\p},\vert x_{F,\p}-1\vert,i_{F,\p})$;
    \item[$i_{F,\p}>o_{F,\p}$:] $\mut_F(\p_1\oplus \p_2)=\p'$ with $\p_1/F=\p_2/F=\p'/F'$, for $(F(\p_1),F(\p_2))$ either $((1,0,0),(1,1,0))$ or $((2,1,1),(2,0,1))\}$ with $F'(\p')=(o_{F,\p_1},x_{F,\p_1},i_{F,\p_1})$.
\end{itemize}
\end{definition}

Consider the torus $\mathcal X_{\w}$ corresponding to the seed (associated with) $\pa(\w)$. 
For the lattice $N_{\w}$ we have the basis $\{e_E\}_{E \text{ face of }\pa(\w)}$. 
Then $e_{\p}\in N$ for ${\p}\in \mathcal{P}_{\w,F}$ is an expression in this basis and  $z^{e_{\p}}$ a function on $\mathcal X_{\w}$. 
To extend our definition of $e_{\p}$ in \eqref{eq:def area ineq} for $\p\in \mathcal P_{\w}$ to $\p\in\widehat{\mathcal{P}}_{\w,F}$, we set $z^{e_{{\p}_1\oplus {\p}_2}}:=z^{e_{{\p}_1}}+z^{e_{{\p}_2}}$.
Then for every mutable face $F$ of $\pa(\w)$ we have
\[
\left\{\mathbf x\in \mathbb R^{\ell(w)}\left\vert (\sum_{\p\in\widehat{\mathcal{P}}_{\w,F}}z^{e_{\p}})^{\trop}(\mathbf x)\ge 0\right.\right\}=S_{\w}.
\]
The following is the key lemma of this section.

\begin{table}[]
    \centering
    \scalebox{.8}{
    \begin{tabular}{|c|c|c|c|} \hline
$F$-local type of $\p$ & ${\p}$ in $\pa(\w)$
    & $\p'=\mut_{F}({\p})$ in $\pa(\mu_F(\w))$ & $F'$-local type of $\p'$\\ \hline\hline
 $(1,1,1)$ &
    $b_i\to v_{(i,k)}\to v_{(i,j)}\to a_i $ &
    $ b'_i\to v'_{(i,j)}\to v'_{(i,k)}\to a'_i $ &
     $(1,0,1)$ \\ \hline
    
\begin{tabular}{c} 
     $(1,0,0)$ \\ \hline
     $(1,1,0)$ 
    \end{tabular} &
    \begin{tabular}{c}
         $b_i\to v_{(i,k)}\to v_{(j,k)}\to b_{j} $  \\ \hline
         $ b_i\to v_{(i,k)}\to v_{(i,j)}\to v_{(j,k)}\to b_{j} $
    \end{tabular} &
    $b'_i\to v'_{(i,j)}\to b_{k} $ &
     $(0,0,1)$ \\ \hline
    
 $(1,1,1)$ &
    $b_i\to v_{(i,k)}\to v_{(i,j)}\to v_{(j,k)}\to b_k $ &
    $b'_i\to v'_{(i,j)}\to v'_{(i,k)}\to b'_k $ &
     $(1,0,1)$ \\ \hline    
    
 $(1,1,2)$ &
    $ a_j\to v_{(i,j)}\to  a_i $ &
    \begin{tabular}{c}
         $ a'_j\to v'_{(j,k)}\to v'_{(i,k)}\to a'_i $  \\
         $a'_j\to v'_{(j,k)}\to v'_{(i,j)}\to v'_{(i,k)}\to a'_i $ 
    \end{tabular} &
    \begin{tabular}{c} 
     $(2,1,1)$ \\ \hline
     $(2,0,1)$ 
    \end{tabular} \\ \hline

 $(1,1,1)$ &
    $ a_j\to v_{(i,j)}\to v_{(j,k)}\to b_j  $ &
    $ a'_j\to v'_{(j,k)}\to v'_{(i,j)}\to b'_j $ &
     $(1,0,1)$ \\ \hline 

 $(1,1,2)$ &
    $ a_j\to v_{(i,j)}\to v_{(j,k)}\to b_k $ &
    \begin{tabular}{c}
         $ a'_j\to v'_{(j,k)}\to v'_{(i,k)}\to b'_k $  \\
         $ a'_j\to v'_{(j,k)}\to v'_{(i,j)}\to v'_{(i,k)}\to b'_k $ 
    \end{tabular}  &
    \begin{tabular}{c} 
     $(2,1,1)$ \\ \hline
     $(2,0,1)$ 
    \end{tabular}\\ \hline 

 $(1,1,1)$  &
     $ a_k\to v_{(i,k)}\to v_{(i,j)}\to a_i $   &
    $  a'_k\to v'_{(j,k)}\to v'_{(i,j)}\to v'_{(i,k)}\to a'_i $ &
     $(1,0,1)$ \\ \hline  
    
    \begin{tabular}{c} 
     $(1,0,0)$ \\ \hline
     $(1,1,0)$ 
    \end{tabular} &
    \begin{tabular}{c}
         $a_k\to v_{(i,k)}\to v_{(j,k)}\to b_j $  \\
         $a_k\to v_{(i,k)}\to v_{(i,j)}\to v_{(j,k)}\to b_j $ 
    \end{tabular}  &
    $ a'_k\to v'_{(j,k)}\to v'_{(i,j)}\to b'_j $ &
     $(0,0,1)$ \\ \hline  

 $(1,1,1)$  &
    $ a_k\to v_{(i,k)}\to v_{(i,j)}\to  v_{(j,k)}\to b_k $ &
    $  a'_k\to v'_{(j,k)}\to v'_{(i,j)}\to v'_{(i,k)}\to b'_k $ &
     $(1,0,1)$ \\ \hline  
    \end{tabular}}
    \caption{Shapes of paths locally around $F$ (resp. $F'$) in $\mathcal P_{\w}$ (resp. $\mathcal P_{\mu_F(\w)}$) for orientation $(l_r,l_{r+1})$ with $i\le r$ and $r+1\le j<k$ (see Figure~\protect{\ref{fig:loc.orient}}) and how they are mapped onto each other by $\mut_F$.}
    \label{tab:(lrr) cases of path mut}
\end{table}

\vbox{
\begin{lemma}\label{lem: mutation paths}
Let $w\in S_n$ with reduced expressions $\w$ and $\mu_F(\w)$, where $F$ is a mutable face of $\pa(\w)$ and $F'$ the corresponding face of $\pa(\mu_F(\w))$ (i.e. $\mu_{F'}(\mu_F(\w))=\w$).
Let $\{e_E\}_E$ denote the basis for $N_{\w}$ and $\{e'_E\}_E$ the basis for $N_{\mu_F(\w)}$. Then for ${\p}\in \widehat{\mathcal P}_{\w,F}$ we have
\[
\mu_{F'}^*(z^{e_{\p}})=z^{e'_{\mut_F(\p)}}.
\]
\end{lemma}}

\begin{proof}
We prove the claim case-by-case depending on the $F$-local type of $\p$ as in Lemma~\ref{lem:possible F-local types}.
As notation we use $n\in N_{\w}$ (resp. $n'\in N_{\mu_F(\w)}$) referring to an expression of $n$ is the basis $\{e_E\}_{E\text{ face of }\pa(\w)}$ (resp. $\{e'_E\}_{E\text{ face of }\pa(\mu_F(\w))}$).
Consider ${\p}\in \mathcal P_{\w}$, then 
\[
e_{\p}=-\sum_{E\subset \area_{\p}}e_E=-\sum_{E\subset ( \text{In}_{F}\cup \text{Out}_{F}) \cap \area_{\p}} e_E - \sum_{E\not \subset ( \text{In}_{F}\cup \text{Out}_{F}) \cap \area_{\p}} e_E=: n_{\p}+m_{\p}
\]
As by definition $\mut_F$ effects a path only locally around $F$, we have $m_{\p}=m_{\mut_F(\p)}$ (resp. $m_{\p}=m_{\p'_1}=m_{\p'_2}$ if $\mut_F(\p)=\p'_1\oplus\p'_2\in\widehat{\mathcal P}_{\mu_F(\w),F'}$).
Both have the same expressions in bases $\{e_E\}_E$ and $\{e'_E\}_E$ as the corresponding basis elements are not effected by mutation: 
only basis elements corresponding to vertices (i.e. faces of $\pa(\w)$) adjacent to $F$ (i.e. in $\text{In}_F\cup \text{Out}_F$) are changed by mutation in \eqref{eq: def mutation lattice basis}.
We use this fact throughout the proof.
Denote basis elements associated with faces $F_{\init},F_{\init_1},F_{\init_2}\in \text{In}_F$ by $e_{\init},e_{\init_1},e_{\init_2}$ and similarly for $e_{\text{out}}$. After mutation, $e'_{\init}$ is associated with the face $F'_{\init}\in \text{Out}_{F'}$ in $\pa(\mu_F(\w))$. 

\begin{table}[]
    \centering
    \scalebox{.8}{
    \begin{tabular}{|c|c|c|c|} \hline
$F$-local type of $\p$ & ${\p}$ in $\pa(\w)$
    & $\p'=\mut_{F}({\p})$ in $\pa(\mu_F(\w))$ & $F'$-local type of $\p'$\\ \hline\hline
 $(1,0,1)$ &
    $ a_i\to v_{(i,j)}\to v_{(i,k)}\to b_i $ &
    $ a'_i\to v'_{(i,k)}\to v'_{(i,j)}\to b'_i $ &
     $(1,1,1)$  \\ \hline      

    \begin{tabular}{c} 
     $(2,0,1)$   \\ \hline
     $(2,1,1)$ \end{tabular}&
    \begin{tabular}{c}
         $ a_i\to v_{(i,j)}\to v_{(i,k)}\to v_{(j,k)}\to b_j  $  \\
         $ a_i\to v_{(i,j)}\to v_{(j,k)}\to b_j $ 
    \end{tabular}   &
    $ a'_i\to v'_{(i,k)}\to v'_{(i,j)}\to b'_j $ &
     $(1,1,2)$ \\ \hline 

 $(2,1,2)$  &
    $ a_i\to v_{(i,j)}\to v_{(j,k)}\to b_k $ &
    $ a'_i\to v'_{(i,k)}\to b'_k $ &
     $(2,1,2)$ \\ \hline 

 $(1,0,1)$ &
    $ a_j\to v_{(i,j)}\to v_{(i,k)}\to v_{(j,k)} \to b_j  $ &
    $ a'_j\to v'_{(j,k)}\to v'_{(i,k)}\to v'_{(i,j)}\to b'_j $ &
     $(1,1,1)$ \\ \hline  
    
 $(0,0,0)$  &
    $ a_k\to v_{(i,k)}\to b_i $ &
    $ a'_k\to v'_{(j,k)}\to v'_{(i,j)}\to b'_i $ &
     $(0,0,0)$ \\ \hline  

 $(0,0,1)$  &
    $ a_j\to v_{(i,j)}\to v_{(i,k)}\to b_i $ &
    \begin{tabular}{c}
         $ a'_j\to v'_{(j,k)}\to v'_{(i,j)}\to b'_i $  \\
         $ a'_j\to v'_{(j,k)}\to v'_{(i,k)}\to v'_{(i,j)}\to b'_i $ 
    \end{tabular}   &
    \begin{tabular}{c} 
     $(1,0,0)$ \\ \hline
     $(1,1,0)$ \end{tabular} \\ \hline    

    \end{tabular}}
    \caption{Shapes of paths locally around $F$ (resp. $F'$) in $\mathcal P_{\w}$ (resp. $\mathcal P_{\mu_F(\w)}$) for orientation $(l_r,l_{r+1})$ with $r+1\le i<j<k$ (see Figure~\protect{\ref{fig:loc.orient}}) and how they are mapped onto each other by $\mut_F$.}
    \label{tab:(rrr) cases of path mut}
\end{table}

We distinguish the cases as in Lemma~\ref{lem:possible F-local types}.
\begin{itemize}
    \item[$i_{F,\p}<o_{F,\p}$] From Lemma~\ref{lem:possible F-local types} we know that in this case $n_{\p}=-e_{\init}-e_F-e_{\text{out}_1}-e_{\text{out}_2}$ (resp. $n_{\p}=-e_{\text{out}}$) and $\mut_F({\p})={\p'}_1\oplus {\p'}_2$ with $\p_1',\p_2'$ as in Definition~\ref{def:mut path}.
    Then $n'_{\p'_1}=-e'_{\init}-e'_F-e'_{\text{out}_1}-e'_{\text{out}_2}$ (resp. $n'_{\p'_1}=-e'_{\text{out}}-e_F$) and $n'_{\p'_2}=-e'_{\init}-e'_{\text{out}_1}-e'_{\text{out}_2}$ (resp. $n'_{\p'_1}=-e'_{\text{out}}$). 
    We compute using formulas \eqref{eq: def mutation lattice basis}, \eqref{eq: def pullback X-mut} and the observation that $m'_{\p'_1}=m'_{\p'_2}$:
    \begin{eqnarray*}
    \mu_{F'}^*(z^{n_{\p}+m_{\p}}) &=& z^{-e_{\init}-e_F-e_{\text{out}_1}-e_{\text{out}_2}+m_{\p}}(1+z^{e_F}) \\
    &=& z^{-e'_{\init}-e'_{\text{out}_1}-e'_{\text{out}_2}+m'_{\p'}}(1+z^{-e'_F})\\
    &=& z^{n'_{\p'_1}+m'_{\p'_1}}+z^{n'_{\p'_2}+m'_{\p'_2}} 
    =z^{e'_{{\p'}_1}}+z^{e'_{{\p'}_2}}\\
    &\stackrel{\text{(by def.)}}{=}& z^{e'_{{\p'}_1\oplus {\p'}_2}}=
    z^{e'_{\mut_F(\p)}}\\
    (\text{resp. }  \mu_{F'}^*(z^{n_{\p}+m_{\p}}) &=& z^{-e_{\text{out}}+m_{\p}}(1+z^{e_F})
    = z^{-e'_{\text{out}}+m'_{\p'}}(1+z^{-e'_F})\\
    &=& z^{n'_{\p'_1}+m'_{\p'_1}}+z^{n'_{\p'_2}+m'_{\p'_2}} = z^{e'_{\mut_F(\p)}}).
    \end{eqnarray*}
    \item[${i_{F,{\p}}=o_{F,{\p}}}$] In this case $\mut_F({\p})={\p'}\in\widehat{\mathcal P}_{\mu_F(\w),F'}$ as in Definition~\ref{def:mut path}.
    We divide into three cases: $i_{F,{\p}}\in\{0,1,2\}$. 
    If $i_{F,{\p}}=0$, consider $\area_{\p}=F_1\cup\dots \cup F_r$ then $\area_{{\p}'}=F'_1\cup \dots \cup F'_r$. Further, 
    \begin{eqnarray*}
    \mu_{F'}^*(z^{e_{\p}})=\mu_{F'}^*(z^{m_{\p}})=z^{m_{\p}}=z^{m_{{\p}'}}=z^{e'_{{\p}'}}=z^{e'_{\mut_F({\p})}}.
    \end{eqnarray*}
    If $i_{F,{\p}}=1$ we have $n_{\p}= -e_F-e_{\init}-e_{\text{out}}$ (resp. $n_{\p}= -e_{\init}-e_{\text{out}}$).
    We have $n'_{\p'}= -e'_{\init}-e'_{\text{out}}$ (resp. $n'_{\p'}= -e'_F-e'_{\init}-e'_{\text{out}}$) and compute
    \begin{eqnarray*}
    \mu_{F'}^*(z^{n_{\p}})&=& z^{-e_F-e_{\init}-e_{\text{out}}} = z^{e'_F-(e'_{\init}+e'_F)-e'_{\text{out}}} = z^{-e'_{\init}-e'_{\text{out}}}=z^{n'_{\p'}}\\
    (\text{resp. }    \mu_{F'}^*(z^{n_{\p}})&=& z^{-e_{\init}-e_{\text{out}}} = z^{-(e'_{\init}+e'_F)-e'_{\text{out}}} = z^{-e'_F-e'_{\init}-e'_{\text{out}}}=z^{n'_{\p'}}).
    \end{eqnarray*}
    If $i_{F,{\p}}=2$ we have $n_{\p}=-e_{\init_1}-e_{\init_2}-e_F-e_{\text{out}_1}-e_{\text{out}_2}$. Now $n'_{{\p}'}=-e'_{\init_1}-e'_{\init_2}-e'_F-e'_{\text{out}_1}-e'_{\text{out}_2}$ and we compute
    \begin{eqnarray*}
    \mu_{F'} ^*(z^{n_{\p}})&=& z^{-e_{\init_1}-e_{\init_2}-e_F-e_{\text{out}_1}-e_{\text{out}_2}} = z^{-(e'_{\init_1}+e'_F) - (e'_{\init_2}+e'_F)-(-e'_F)-e'_{\text{out}_1}-e'_{\text{out}_2}}\\
    &=& z^{-e'_{\init_1}-e'_{\init_2}-e'_F-e'_{\text{out}_1}-e'_{\text{out}_2}} = z^{n'_{\p'}}.
    \end{eqnarray*}
    In all three cases the claim follows from the computation.
    \item[$i_{F,{\p}}>o_{F,{\p}}$] In this case by Lemma~\ref{lem:possible F-local types} there are paths $\p_1,\p_2\in\mathcal P_{\w}$ with $\p_1\oplus \p_2\in\widehat{\mathcal{P}}_{\w,F}$ and $\mut_{F}(\p_1\oplus\p_2)=\p'\in\widehat{\mathcal P}_{\mu_F({\w}),F'}$ as in Definition~\ref{def:mut path}. 
    We have $n_{\p_1}=-e_{\init}$ and $n_{\p_2}=-e_{\init}-e_F$ (resp. $n_{\p_1}=-e_{\init_1}-e_{\init_2}-e_{\text{out}}$ and $n_{\p_2}=-e_{\init_1}-e_{\init_2}-e_F-e_{\text{out}}$). For $\p'$ we have $n'_{\p'}=-e'_{\init}$ (resp. $n'_{\p'}=-e'_{\init_1}-e'_{\init_2}-e'_F-e'_{\text{out}}$). We compute
    \begin{eqnarray*}
    \mu_{F'}^*(z^{n_{\p_1}}+z^{n_{\p_2}}) &=&
    z^{-e_{\init}}(1+z^{e_F})^{-1}+z^{-e_{\init}-e_F}(1+z^{e_F})^{-1} \\
    &=& (z^{-e'_{\init}-e'_F}+z^{-e'_{\init}})(1+z^{-e'_F})^{-1} = z^{-e'_{\init}} = z^{n'_{\p'}}\\
    (\text{resp. } \mu_{F'}^*(z^{n_{\p_1}}+z^{n_{\p_2}}) &=&
    z^{-e_{\init_1}-e_{\init_2}-e_{\text{out}}}(1+z^{e_F})^{-1}+z^{-e_{\init_1}-e_{\init_2}-e_F-e_{\text{out}}} (1+z^{e_F})^{-1} \\
    &=& (z^{-e'_{\init_1}-e'_{\init_2}-2e'_F-e'_{\text{out}}} + z^{-e'_{\init_1}-e'_{\init_2}-e'_F-e'_{\text{out}}})(1+z^{-e'_F})^{-1} \\
    &=& z^{-e'_{\init_1}-e'_{\init_2}-e'_F-e'_{\text{out}}} = z^{n'_{\p'}}).
    \end{eqnarray*}
    In both cases the claim follows.
\end{itemize}
\end{proof}

Before proving a generalization of Proposition~\ref{prop: superpot area GT} we have to show that also the normal vectors associated to the weight inequalities $e_{[i:k]}$ \eqref{eq:def area wt ineq} mutate as expected.
We use the notation as in Lemma~\ref{lem: mutation paths} and its proof.
Recall the normal vectors of the weight inequalities for $\mathcal S_{\w}$ from \eqref{eq:def area wt ineq}.
For $i\in[n-1]$ let $e_{[i:0]},\dots,e_{[i:n_i]}$ be those for $\w$ as expressions in $\{e_E\}_E$ and $e'_{[i:0]},\dots,e'_{[i,n'_i]}$ those for $\mu_F(\w)$ as expressions in $\{e'_E\}_E$.

\begin{lemma}\label{lem:mutation wt area}
With notation as above we have for every $i\in[n-1]$
\[
\mu_{F'}^*\left(\sum_{k=0}^{n_i}z^{e_{[i:k]}}\right)=\sum_{k'=0}^{n'_i}z^{e'_{[i:k']}}.
\]
\end{lemma}
\begin{proof}
We treat the case where $F$ is of level $l$ and $F'$ of level $l+1,$ with $ l\in[n-2]$ (the proof of the other case is similar).
Recall that $e_{[i:k]}=-e_{F_i}-e_{F_{j_1}}-\dots-e_{F_{j_k}}$, where $k\in[0,n_i]$, $s_{j_1},\dots,s_{j_{n_i}}=s_i$ in $\w$, and $F_{j_k}$ is bounded to the left by the crossing in $\pa(\w)$ induced by $s_{j_k}$.
Let $F'_i,F'_{j_1},\dots,F'_{j_{n'_i}}$ be the corresponding faces in $\pa(\mu_F(\w))$.
In particular, if $i\not\in\{l,l+1\}$ we have 
\[
\mu_{F'}^*(z^{-e_{F_i}-e_{F_{j_1}}-\dots-e_{F_{j_k}}})=z^{-e'_{F'_i}-e'_{F'_{j_1}}-\dots-e'_{F'_{j_k}}}.
\]
We therefore focus on the cases $i\in\{l,l+1\}$.
\begin{itemize}
    \item[$i=l$] As $F$ is of level $l$ we have $F=F_{j_k}$ for one $k\in[n_l],s_{j_k}=s_l$. 
    By \eqref{eq: def mutation lattice basis} we have $\mu_F(e_{F_l})=e'_{F'_l}$ and $\mu_{F}(e_{F_{j_r}})=e'_{F'_{j_r}}$ for $r\in[k-1]$, hence
    \[
    \mu_{F'}^*(z^{-e_{F_l}-e_{F_{j_1}}-\dots-e_{F_{j_r}}})=z^{-e'_{F'_l}-e'_{F'_{j_1}}-\dots-e'_{F'_{j_r}}}.
    \]
    Still by \eqref{eq: def mutation lattice basis} we have $\mu_F(e_{F_{k-1}})=e'_{F'_{j_{k-1}}}+e'_{F'_{j_k}}, \mu_F(e_{F_{j_k}})=-e'_{F'_{j_k}}$ and $\mu_F(e_{F_{j_s}})=e'_{F'_{j_s}}$ for $s\in[k+1,n_l]$. Plugging in to \eqref{eq: def pullback X-mut} we obtain
    \begin{eqnarray*}
    \mu_{F'}^*(z^{e_{[l:k-1]}}+z^{e_{[l:k]}}+z^{e_{[l:k+1]}})&= & z^{-e_{F_l}-\dots-e_{F_{j_{k-1}}}}(1+z^{e_{F_{j_k}}})^{-1}\\ &+& z^{-e_{F_l}-\dots-e_{F_{j_k}}}(1+z^{e_{F_{j_k}}})^{-1} + z^{-e_{F_l}-\dots-e_{F_{j_{k+1}}}}\\
    &=& z^{-e'_{F'_l}-\dots-e'_{F'_{j_{k-1}}}} + z^{-e'_{F'_l}-\dots-e'_{F'_{j_{k+1}}}}\\
    &=& z^{e'_{[l:k-1]}}+z^{e_{[l:k]}}.
    \end{eqnarray*}
    Note that the index shift in the last equality comes from the fact that $\pa(\mu_F(\w))$ has one less face of level $l$ than $\pa(\w)$ as $F'$ is of level $l+1$. So the claim follows for level $l$.
    \item[$i=l+1$] Let $F_{j_r}$ be the face of level $l+1$ in $\text{Out}_F$ and $F_{j_{r+1}}$ the one in $\text{In}_F$. Then we compute with notation as above
    \begin{eqnarray*}
    \mu_{F'}^*(z^{e_{[l:r]}}+z^{e_{[l:r+1]}})&=& z^{-e_{F_{l+1}}-\dots-e_{F_{j_r}}}(1+z^{e_F})+z^{-e_{F_{l+1}}-\dots-e_{F_{j_{r+1}}}}\\
    &=& z^{-e'_{F_{l+1}}-\dots-e'_{F'_{j_r}}} + z^{-e'_{F'_{l+1}}-\dots-e'_{F'_{j_r}}-e'_{F'}} + z^{-e'_{F'_{l+1}}-\dots-e'_{F'_{j_r}}-e'_{F'}-e'_{F'_{j_{r+1}}}}\\
    &=& z^{e'_{[l+1:r]}}+z^{e'_{[l+1:r+1]}}+z^{e_{[l+1:r+2]}}.
    \end{eqnarray*}
    As before the index shift occurs because $\pa(\mu_F(\w))$ has additionally the face $F'$ of level $l+1$ in comparison to $\pa(\w)$.
\end{itemize}
\end{proof}

We can now prove the following theorem.

\vbox{
\begin{theorem}\label{thm: superpot and area potential}
Let $\w_0$ be an arbitrary reduced expression of $w_0\in S_n$. Then the superpotential expressed in the seed given by $\pa(\w_0)$ satisfies $W\vert_{\X_{\w_0}}=W_{\mathcal S_{\w_0}}$. In particular,
\[
W\vert_{\X_{\w_0}}=\sum_{\p\in\mathcal P_{\w_0}}z^{e_{\p}}+\sum_{i\in[n-1],0\le k\le n_i}z^{e_{[i:k]}}.
\]
\end{theorem}}

\begin{proof}
By Proposition~\ref{prop: superpot area GT} the claim is true for the seed $s_0$ with $\w_0=s_1s_2s_1\cdots s_{n-1}\cdots s_2s_1$. 
Now Lemmata~\ref{lem: mutation paths} and \ref{lem:mutation wt area} imply that the claim holds for all seeds that are related to $s_0$ by a finite sequence of mutations.
As there are only finitely many reduced expressions for $w_0$ and they are all related by mutation as defined in Definition~\ref{defn:mutation pa} the claim is true for all $\w_0$.
\end{proof}

\vbox{
\begin{corollary}\label{cor:area is trop super}
For every reduced expression $\w_0\in S_n$ the following polyhedral objects coincide
\begin{enumerate}[(i)]
    \item $\mathcal S_{\w_0}=\Xi_{\w_0}$,
    \item $S_{\w_0}=\mathsf \Xi_{\w_0}$,
    \item $\mathcal S_{\w_0}(\lambda)=\Xi_{\w_0}(\lambda)$ for $\lambda\in\mathbb R^{n-1}$.
\end{enumerate}
\end{corollary}}

\begin{proof}
The claim in (i) follows immediately from Theorem~\ref{thm: superpot and area potential} by tropicalizing. 
Then (iii) follows by definition as we intersect both cones with the same collection of hyperplanes.
To see (ii), recall from the proof of Proposition~\ref{prop: superpot area GT} that for the initial seed $s_0$ the $\vartheta$-functions $\vartheta_{(i,n)}$ correspond to GP-paths.
Then the claim follows by Lemma~\ref{lem: mutation paths} and the proof of Theorem~\ref{thm: superpot and area potential}.
\end{proof}


\subsection{Applications of Theorem~\ref{thm:unimod}}\label{subsec:apply}

We have seen in the last two subsections how the cones and polytopes defined in \S\ref{sec:pa and gp} arise from a representation theoretic point of view and in the context of cluster varieties.
The following theorem is the main combinatorial result of this section. We obtain it as an application of the unimodular equivalences in Theorem~\ref{thm:unimod}.

\vbox{
\begin{theorem}\label{thm:application unimod w_0}
Let $\w_0$ be an arbitrary reduced expression of $w_0\in S_n$. Then the following polyhedral objects are unimodularly equivalent
\begin{enumerate}[(i)]
    \item $\mathcal Q_{\w_0} \cong \Xi_{\w_0}$ via $\Psi_{\w_0}$,
    \item $Q_{\w_0}\cong \mathsf \Xi_{\w_0}$ via $\Psi_{\w_0}\vert_{\mathbb R^N}$,
    \item $\mathcal Q_{\w_0}(\lambda)\cong \Xi_{\w_0}(\lambda)$ for $\lambda\in\mathbb R^{n-1}$ via $\Psi_{\w_0}$.
\end{enumerate}
\end{theorem}
}

\begin{proof}
Combine Theorem~\ref{thm:unimod} with Theorem~\ref{thm:wt GP is wt string} and Corollary~\ref{cor:area is trop super}.
\end{proof}

\begin{remark}
For the special case of the initial seed $s_0$ the theorem can also be proved by combining results of Magee and Littelmann. 
In \cite{Lit98} Littelmann shows that the string polytope $\mathcal Q_{\w_0}(\lambda)$ for $\w_0=s_1s_2s_1\cdots s_{n-1}s_{n-2}\cdots s_2s_1$ is unimodularly equivalent to the \emph{Gelfand-Tsetlin polytope} defined in \cite{GT50}.
Magee shows in \cite{Mag15} that $\Xi_{s_0}$ (resp. $\Xi_{s_0}(\lambda)$) is unimodularly equivalent to the Gelfand-Tsetlin cone (resp. polytope). 
Combining both, one obtains Theorem~\ref{thm:application unimod w_0} for $s_0$.
In fact, to understand Magee's result was driving motivation behind this project.
\end{remark}

By the construction of toric varieties associated to polytopes as in \cite[\S2.1 and \S2.3]{CLS11} and the toric degenerations of Caldero \cite{Cal02} and Gross-Hacking-Keel-Kontsevich \cite{GHKK14} we obtain the following corollary from Theorem~\ref{thm:application unimod w_0} relating these toric varieties.
It is the main result regarding toric degenerations of flag varieties in this section and an answer to Question~\ref{Q:cluster} in the introduction.

\begin{corollary}
Let $\w_0$ be an arbitrary reduced expression of $w_0\in S_n$ and $\lambda\in\mathbb Z_{>0}^{n-1}$. We have an induced isomorphism of the following toric varieties that are degenerations (resp. normalizations of such) of $SL_n/B$
\[
X_{\mathcal Q_{\w_0}}(\lambda)\cong X_{\Xi_{\w_0}}(\lambda).
\]
\end{corollary}

In order to achieve a similar result for Schubert varieties, we study the restriction of the superpotential in the following subsection.

\subsubsection*{Restricted Superpotential and Schubert varieties}
Caldero's degeneration works more generally for Schubert varieties. 
As we have seen above, he uses the degeneration for the flag variety and by a quotient construction on the level of rings he obtains a family for the Schubert variety.
For the cones, taking this quotient corresponds to setting certain variables to zero, or equivalently, restricting the defining GP-paths as in Definition~\ref{def: res path}.
In a similar fashion we want to proceed with the superpotential.
We show how the polytopes defining toric degenerations of Schubert varieties arise in the setting of \cite{GHKK14}.

Consider $w\in S_n$ with reduced expression $\w$ and extension $\w_0=\w s_{i_{\ell(w)+1}}\cdots s_{i_N}$. 
Recall that for a seed corresponding to $\w_0$ we have a basis $\{e_F\mid F \text{ face of }\pa(\w_0)\}$ for $N_{\w_0}$ and further $\mathbb C[\X_{\w_0}]=\mathbb C[z^{\pm e_F}\mid F\text{ face of }\pa(\w_0)]$.
Then $\{e_F\mid F\text{ face of }\pa(\w)\}$ generates a sublattice in $N_{\w_0}$, which we denote by $N_{\w}$ with dual lattice $M_{\w}$ a quotient of $M_{\w_0}$.
We have the torus $\X_{\w}=T_{M_{\w}}=\Spec(\mathbb C[N_{\w}])$ associated with $M_{\w}$ as in \eqref{def: seed tori}.
In particular, $\mathbb C[\X_{\w}]=\mathbb C[z^{\pm e_{F}}\mid F\text{ face of }\pa(\w)]$ and we have a restriction morphism between the Laurent polynomial rings
\[
\res_{\w}:\mathbb C[\X_{\w_0}]\to\mathbb C[\X_{\w}], \ \ f\mapsto f\vert_{\X_{\w}}.
\]

We are interested in the restrictions to $\X_{\w}$ of the superpotential $W\vert_{\X_{\w_0}}$ and the detropicalization $W_{\mathcal S_{\w_0}}$ of $\mathcal S_{\w_0}$ (they are equal by Theorem~\ref{thm: superpot and area potential}). We want to show that they coincide with the detropicalization of $\mathcal S_{\w}$.
In analogy with Definiton~\ref{def:super cone} for $w_0$ we consider for arbitrary $w$ the following polyhedral objects.

\begin{definition}\label{def: res super cone}
For $w\in S_n$ with reduced expression $\w$ and an extension $\w_0=\w s_{i_{\ell(w)+1}}\cdots s_{i_N}$  polyhedral objects by tropicalizing the restriction of a sum of $\vartheta$-functions resp. the superpotential: 
\begin{eqnarray*}
\res_{\w}(\Xi_{\w_0}) &:=& \{ \mathbf x\in \mathbb R^{\ell(w)+n-1}\mid \res_{\w}(W\vert_{\X_{\w_0}})^{\trop}(\mathbf x)\ge 0 \},\\
\res_{\w}(\mathsf{\Xi}_{\w_0}) &:=& \{ \mathbf x\in\mathbb R^{\ell(w)}\mid \res_{\w}(\sum_{i=1}^{n-1}\vartheta_{(i,n)}\vert_{\X_{\w_0}})^{\trop}(\mathbf x)\ge 0 \},\\
\res_{\w}(\Xi_{\w_0}(\lambda)) &:=&  \res_{\w}(\Xi_{\w_0})  \cap\tau_{\w}^{-1}(\lambda) \text{ for } \lambda\in\mathbb R^{n-1}.
\end{eqnarray*}    
\end{definition}

\begin{center}
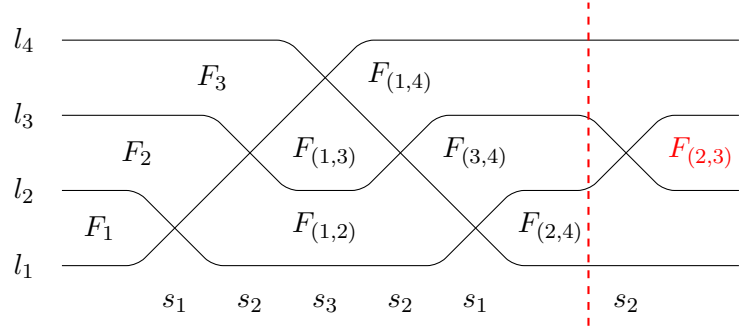
\begin{figure}
\centering
\begin{tikzpicture}{scale=.5}
\draw[rounded corners] (0,0) -- (1,0) -- (4,3) -- (9,3);
\draw[rounded corners]  (0,1) -- (1,1) -- (2,0) -- (5,0) -- (6,1) -- (7,1) -- (8,2) -- (9,2);
\draw[rounded corners]  (0,2) -- (2,2) -- (3,1) -- (4,1) -- (5,2) -- (7,2) -- (8,1) -- (9,1);
\draw[rounded corners]  (0,3) -- (3,3) -- (6,0) -- (9,0);

\node at (-.5,0) {$l_1$};
\node at (-.5,1) {$l_2$};
\node at (-.5,2) {$l_3$};
\node at (-.5,3) {$l_4$};

\node at (.5,.5) {$F_{1}$};
\node at (1,1.5) {$F_{2}$};
\node at (2,2.5) {$F_{3}$};
\node at (3.5,0.5) {$F_{(1,2)}$};
\node at (3.5,1.5) {$F_{(1,3)}$};
\node at (5.5,1.5) {$F_{(3,4)}$};
\node at (6.5,.5) {$F_{(2,4)}$};
\node at (4.5,2.5) {$F_{(1,4)}$};
\node[red] at (8.5,1.5) {$F_{(2,3)}$};

\draw[red, dashed, thick] (7,-.8) -- (7,3.5);

\node at (1.5,-.5) {$s_1$};
\node at (2.5,-.5) {$s_2$};
\node at (3.5,-.5) {$s_3$};
\node at (4.5,-.5) {$s_2$};
\node at (5.5,-.5) {$s_1$};
\node at (7.5,-.5) {$s_2$};
\end{tikzpicture}
\caption{Restriction/Extension of a pseudoline arrangement.}\label{fig: restr.pseudo}
\end{figure}
\end{center}

\begin{example}\label{exp: res superpot}
Consider $\w=s_1s_2s_3s_2s_1\in S_4$ with extension $\w_0=\w s_2$. 
We compute the superpotential in $W{\vert}_{\X_{\w_0}}\in\mathbb C[\X_{\w_0}]$.
\begin{eqnarray*}
W\vert_{\X_{\w_0}}&=& (z^{-e_3}+z^{-e_3-e_{(1,4)}})+ (z^{-e_2} + z^{-e_2-e_{(1,3)}} + z^{-e_2-e_{(1,3)}-e_{(3,4)}})\\
&+& (z^{-e_1}+z^{-e_1-e_{(1,2)}}+z^{-e_1-e_{(1,2)}-e_{(2,4)}}) + (z^{-e_{(2,4)}}+z^{-e_{(2,4)}-e_{(3,4)}}) + (z^{-e_{2,3}})\\
&+& (z^{-e_{(1,4)}}+z^{-e_{(1,4)}-e_{(1,3)}}+z^{-e_{(1,4)}-e_{(1,3)}-e_{(3,4)}}+z^{-e_{(1,4)}-e_{(1,3)}-e_{(3,4)}-e_{(1,2)}}).
\end{eqnarray*}
From Figure~\ref{fig: restr.pseudo} we see that $F_{(2,3)}$ is a face of $\pa(\w_0)$, but not of $\pa(\w)$. Hence,
\begin{eqnarray*}
{\res}_{\w}(W\vert_{\X_{\w_0}}) &=&  (z^{-e_3}+z^{-e_3-e_{(1,4)}})+ (z^{-e_2} + z^{-e_2-e_{(1,3)}} + z^{-e_2-e_{(1,3)}-e_{(3,4)}})\\
&+& (z^{-e_1}+z^{-e_1-e_{(1,2)}}+z^{-e_1-e_{(1,2)}-e_{(2,4)}}) + (z^{-e_{(2,4)}}+z^{-e_{(2,4)}-e_{(3,4)}}) \\
&+& (z^{-e_{(1,4)}}+z^{-e_{(1,4)}-e_{(1,3)}}+z^{-e_{(1,4)}-e_{(1,3)}-e_{(3,4)}}+z^{-e_{(1,4)}-e_{(1,3)}-e_{(3,4)}-e_{(1,2)}}).
\end{eqnarray*}
\end{example}

\begin{proposition}\label{prop: res superpt}
Let $w\in S_n$ and consider a reduced expression $\w$ with an extension to $\w_0=\w s_{i_{\ell(w)+1}}\cdots s_{i_N}$. Then
\[
\res_{\w}(W\vert_{\X_{\w_0}})=W_{\mathcal S_{\w}}.
\]
\end{proposition}

\begin{proof}
Recall the restriction of GP-paths defined in Definition~\ref{def: res path}. 
By Propositions~\ref{prop:respath} and \ref{prop:restrict} we have seen $\res_{\w}(\mathcal P_{\w_0})=\mathcal P_{\w}$. 
To avoid confusion we denote as before for $i\in[n-1]$ by $n_i^{\w}:=\#\{j\mid s_{i_j}=s_i \text{ in }\w\}$ and $n_i^{\w_0}:=\#\{j\mid s_{i_j}=s_i\text{ in }\w_0\}$.
Using Theorem~\ref{thm: superpot and area potential} we compute
\begin{eqnarray*}
\res_{\w}(W\vert_{\X_{\w_0}}) &=& \sum_{\p\in\mathcal P_{\w_0}}z^{e_{\p}}\vert_{\X_{\w}}+\sum_{i\in[n-1],0\le k\le n_i^{\w_0}}z^{e_{[i:k]}}\vert_{\X_{\w}}\\
&=& \sum_{\p\in\res_{\w}(\mathcal P_{\w_0})}z^{e_{\p}}+\sum_{i\in[n-1],0\le k\le n_i^{\w}}z^{e_{[i:k]}}\\
&=& W_{\mathcal S_{\w}}.
\end{eqnarray*}
\end{proof}

The last proposition enables us to formulate a theorem similar to Theorem~\ref{thm:application unimod w_0} for arbitrary $w\in S_n$.

\vbox{
\begin{theorem}\label{thm: application unimod w}
Let $w\in S_n$ and consider a reduced expression $\w$ with an extension to $\w_0=\w s_{i_{\ell(w)+1}}\cdots s_{i_N}$. Then the following polyhedral objects are unimodularly equivalent
\begin{enumerate}[(i)]
    \item $\mathcal Q_{\w}\cong \res_{\w}(\Xi_{\w_0})$ via $\Psi_{\w}$,
    \item $Q_{\w}\cong \res_{\w}(\mathsf\Xi_{\w_0})$ via $\Psi_{\w}\vert_{\mathbb R^{\ell(w)}}$,
    \item $\mathcal Q_{\w}(\lambda)\cong  \res_{\w}(\Xi_{\w_0}(\lambda))$ for $\lambda\in\mathbb R^{n-1}$ via $\Psi_{\w}$. 
\end{enumerate}
\end{theorem}}

\begin{proof}
For (i) combine Proposition~\ref{prop: res superpt} with Theorem~\ref{thm:unimod} and Theorem~\ref{thm:wt GP is wt string}, which directly implies (iii).
To see (ii), recall that by Lemma~\ref{lem: mutation paths} and the proof of Proposition~\ref{prop: superpot area GT} we have
\[
\sum_{\p\in\mathcal P_{\w_0}} z^{e_{\p}} = \sum_{i\in[n-1]}\vartheta_{(i,n)}\vert_{\X_{\w_0}}.
\]
By the proof of Proposition~\ref{prop: res superpt} the same equality when replacing $\w_0$ by $\w$. Then the claim follows by Theorem~\ref{thm:unimod} and Theorem~\ref{thm:wt GP is wt string}.
\end{proof}

For the following corollary relating the toric degenerations of Schubert varieties by Caldero \cite{Cal02} to the toric degenerations of flag varieties by Gross-Hacking-Keel-Kontsevich \cite{GHKK14} we briefly remind you about the \emph{Orbit-Cone-Correspondence} for toric varieties (see \cite[\S3.2]{CLS11}).

For a (full-dimensional) polytope $P\subset \mathbb R^n$ denote by $\Sigma_P\subset\mathbb R^n$ its \emph{normal fan} (see \cite[Remark~2.3.3]{CLS11}).
Every cone $\sigma\in\Sigma_P$ corresponds to a torus orbits in $X_{\Sigma_{P}}$ of dimension $n-\dim \sigma$ (\cite[Theorem~3.2.6]{CLS11}).
The closure of each torus orbit is a toric variety.
For a face $Q$ of $P$ let $\sigma_Q\in\Sigma_P$ be the cone in $\Sigma_P$ spanned by the normal vectors of all facets of $P$ containing $Q$.
Then by \cite[Proposition~3.2.9]{CLS11} the toric variety $X_Q$ is isomorphic to the closure of the torus orbit corresponding to the cone $\sigma_Q\in\Sigma_P$.

Consider an arbitrary $w\in S_n$ with a reduced expression $\w$ and an extension $\w_0=\w s_{i_{\ell(w)+1}}\cdots s_{i_N}$. 
For every $\lambda\in\Lambda^{++}$ recall that the toric variety $X_{\mathcal Q_{\w}(\lambda)}$ is (the normalization of) a toric degeneration of $X_w$ by \cite{Cal02}. Similarly, $X_{\Xi_{\w_0}(\lambda)}$ is a flat degeneration of $SL_n/B$ by \cite{GHKK14}. 
We can now formulate the geometric version of our main result on toric degenrations of Schubert varieties.

\begin{corollary}
The toric variety $X_{\mathcal Q_{\w}(\lambda)}$ is isomorphic to a subvariety of $X_{\Xi_{\w_0}(\lambda)}$.
More precisely, we have
\[
X_{\mathcal Q_{\w}(\lambda)}\cong X_{\res_{\w}(\Xi_{\w_0}(\lambda))},
\]
where $ X_{\res_{\w}(\Xi_{\w_0}(\lambda))}$ is the closure of the torus orbit corresponding to the cone $\sigma_{\res_{\w}(\Xi_{\w_0}(\lambda))}\in \Sigma_{\Xi_{\w_0}(\lambda)}$.
\end{corollary}

\begin{proof}
By definiton $\res_{\w}(\Xi_{\w_0}(\lambda))$ is a union of faces of $\Xi_{\w_0}(\lambda)$. Theorem~\ref{thm: application unimod w}(iii) implies in particular, that $\res_{\w}(\Xi_{\w_0}(\lambda))$ is a polytope itself, hence a face of $\Xi_{\w_0}(\lambda)$.
Further, the unimodular equivalence $Q_{\w}(\lambda)\cong \res_{\w}(\Xi_{\w_0}(\lambda))$ induces an isomorphism of toric varieties $X_{Q_{\w}(\lambda)}\cong X_{\res_{\w}(\Xi_{\w_0}(\lambda))}$.
Then the Corollary follows by \cite[Proposition~3.2.9]{CLS11}.
\end{proof}

\paragraph{Restriction vs. superpotential for $G^{e,w}$} We conclude with an example that shows how $\res_{\w}(W\vert_{\X_{\w_0}})$ is essentially different from a function one would obtain from applying Algorithm~\ref{alg:superpot via opt seeds} to the quiver $Q_{\w}$

\begin{example}
Let $s = s_{\w}$ be the seed of the reduced expression $\w=s_1s_2s_3s_2s_1 \in S_4$ as in Figure~\ref{fig: restr.pseudo}. The corresponding quiver is pictured in Figure~\ref{restr.quiver}. We apply Algorithm~\ref{alg:superpot via opt seeds} and compute optimized seeds for all frozen vertices in $Q_{\w}$.
As $w_{3}$ and $w_{(2,4)}$ are sinks in $Q_{\w}$ we set $\vartheta_{3}\vert_{\X_s}=z^{-e_{3}}$ and $\vartheta_{(2,4)}\vert_{\X_s}=z^{-e_{(2,4)}}$, where $\{e_1,e_2,e_3,e_{(1,2)},e_{(1,3)},e_{(1,4)},e_{(2,4)},e_{(3,4)}\}$ is the lattice basis associated to $s$.

\begin{center}
\begin{figure}[h]
\centering
\begin{tikzpicture}

\node at (0,0) {\tiny$w_{1}$};
    \draw (-.35,-.15) rectangle (.35,0.2);
\node at (3,0) {\tiny $w_{(2,4)}$};
    \draw (2.55,-.15) rectangle (3.45,0.25);
\node at (1.5,0.5) {\tiny $w_{(1,2)}$};
\node at (0,1) {\tiny $w_{2}$};
    \draw (-.35,.85) rectangle (.35,1.2);
\node at (3,1) {\tiny $w_{(3,4)}$};
        \draw (2.55,.85) rectangle (3.45,1.25);
\node at (1.5,1.5) {\tiny $w_{(1,3)}$};
\node at (0,2) {\tiny $w_{3}$};
    \draw (-.35,1.85) rectangle (.35,2.2);
\node at (3,2) { \tiny $w_{(1,4)}$};
        \draw (2.55,1.85) rectangle (3.45,2.25);

\draw[->, thick] (1.15,1.65) -- (0.45,2);
\draw[->, thick] (2.55,2) -- (1.85,1.65);
\draw[->, thick] (0.45,1.1) -- (1.2,1.45);
\draw[->, thick] (1.85,1.35) -- (2.55,1.1);

\draw[->, thick] (1.15,0.65) -- (0.45,1);
\draw[->, thick] (2.55,1) -- (1.85,.65);
\draw[->, thick] (0.45,.1) -- (1.2,.45);
\draw[->, thick] (1.85,.35) -- (2.55,.1);
\node at (1.5,-.5) {$Q_{\w}$};
\begin{scope}[xshift=5cm]

\node at (0,0) {\tiny $w_{1}$};
    \draw (-.35,-.15) rectangle (.35,0.2);
\node at (3,0) {\tiny $w_{(2,4)}$};
    \draw (2.55,-.15) rectangle (3.45,0.25);
\node at (1.5,0.5) {\tiny $w_{(1,2)}$};
\node at (0,1) {\tiny $w_{2}$};
    \draw (-.35,.85) rectangle (.35,1.2);
\node at (3,1) {\tiny $w_{(3,4)}$};
        \draw (2.55,.85) rectangle (3.45,1.25);
\node at (1.5,1.5) {\tiny $w_{(1,3)}$};
\node at (0,2) {\tiny $w_{3}$};
    \draw (-.35,1.85) rectangle (.35,2.2);
\node at (3,2) {\tiny $w_{(1,4)}$};
        \draw (2.55,1.85) rectangle (3.45,2.25);

\draw[<-, thick] (1.15,1.65) -- (0.45,2);
\draw[<-, thick] (2.55,2) -- (1.85,1.65);
\draw[<-, thick] (0.45,1.1) -- (1.2,1.45);
\draw[<-, thick] (1.85,1.35) -- (2.55,1.1);

\draw[->, thick] (1.15,0.65) -- (0.45,1);
\draw[->, thick] (2.55,1) -- (1.85,.65);
\draw[->, thick] (0.45,.1) -- (1.2,.45);
\draw[->, thick] (1.85,.35) -- (2.55,.1);

\node at (1.5,-.5) {$\mu_{(1,3)}(Q_{\w})$};

\begin{scope}[xshift=5cm]
\node at (0,0) {\tiny $w_{1}$};
    \draw (-.35,-.15) rectangle (.35,0.2);
\node at (3,0) {\tiny $w_{(2,4)}$};
    \draw (2.55,-.15) rectangle (3.45,0.25);
\node at (1.5,0.5) {\tiny $w_{(1,2)}$};
\node at (0,1) {\tiny $w_{2}$};
    \draw (-.35,.85) rectangle (.35,1.2);
\node at (3,1) {\tiny $w_{(3,4)}$};
        \draw (2.55,.85) rectangle (3.45,1.25);
\node at (1.5,1.5) {\tiny $w_{(1,3)}$};
\node at (0,2) {\tiny $w_{3}$};
    \draw (-.35,1.85) rectangle (.35,2.2);
\node at (3,2) {\tiny $w_{(1,4)}$};
        \draw (2.55,1.85) rectangle (3.45,2.25);
        
\draw[->, thick] (1.15,1.65) -- (0.45,2);
\draw[->, thick] (2.55,2) -- (1.85,1.65);
\draw[->, thick] (0.45,1.1) -- (1.2,1.45);
\draw[->, thick] (1.85,1.35) -- (2.55,1.1);

\draw[<-, thick] (1.15,0.65) -- (0.45,1);
\draw[<-, thick] (2.55,1) -- (1.85,.65);
\draw[<-, thick] (0.45,.1) -- (1.2,.45);
\draw[<-, thick] (1.85,.35) -- (2.55,.1);

\node at (1.5,-.5) {$\mu_{(1,2)}(Q_{\w})$};
\end{scope}
\end{scope}

\end{tikzpicture}
\caption{The quivers $Q_{\w}$, $\mu_{(1,3)}(Q_{\w})$ and $\mu_{(1,2)}(Q_{\w})$ for $\w=s_1s_2s_3s_2s_1$. The boxes denote frozen variables.}\label{restr.quiver}
\end{figure}
\end{center}
\vspace{-.5cm}

For the other variables we have to find a mutation sequence to an optimized seed. Mutation at $w_{(1,3)}$ (resp. $w_{(1,2)}$) yields the quiver $\mu_{(1,3)}(Q_{\w})$ (resp. $\mu_{(1,2)}(Q_{\w})$) in Figure~\ref{restr.quiver}. 
The seed $\mu_{(1,3)}(s)$ is optimized for $w_{(1,4)}$ and $w_{2}$, so $\vartheta_{(1,4)}\vert_{\X_{\mu_{(1,3)}(\w)}}=z^{-e'_{(1,4)}}$ and $\vartheta_{2}\vert_{\X_{\mu_{(1,3)}(\w)}}=z^{-e'_{2}}$. 
In $\X_{\w}$ we obtain $\vartheta_{(1,4)}\vert_{\X_s}=z^{-e_{(1,4)}}+z^{-e_{(1,4)}-e_{(1,3)}}$ and $\vartheta_{2}\vert_{\X_{s}}=z^{-e_{2}}+z^{-e_{2}-e_{(1,3)}}$. 
Proceeding analogously with $\mu_{(1,2)}(s)$, optimized for $w_{(3,4)}$ and $w_{1}$, we obtain a function on $\X_{\w}$
\begin{eqnarray*}
F &:=& (z^{-e_3}) + (z^{-e_{2}}+z^{-e_2-e_{(1,3)}})
+ (z^{-e_1}+ z^{-e_1-e_{(1,2)}}) + (z^{-e_{(2,4)}})\\
&+& (z^{-e_{(3,4)}}+z^{-e_{(3,4)}-e_{(1,2)}}) + (z^{-e_{(1,4)}}+z^{-e_{(1,4)}-e_{(1,3)}}).
\end{eqnarray*}

Comparing to Example~\ref{exp: res superpot} where $\w_0=\w s_2$ we observe that $F\not = \res_{\w}(W\vert_{\X_{\w_0}})$. 
Tropicalizing $\res_{\w}(W\vert_{\X_{\w_0}})$ we get the following set of inequalities defining the cone $\mathcal S_{\w}\subset \mathbb R^8$
\begin{eqnarray*}
&-x_3\ge 0, -x_3-x_{(1,4)}\ge 0, &\\
&-x_2\ge 0,-x_{2}-x_{(1,3)}\ge 0, -x_2-x_{(1,3)}-x_{(3,4)}\ge 0&\\
& -x_1\ge 0, -x_{1}-x_{(1,2)}\ge 0, -x_{1}-x_{(1,2)}-x_{(2,4)}\ge 0, &\\
&-x_{(2,4)}\ge 0, -x_{(2,4)}-x_{(3,4)}\ge 0,&\\
&-x_{(1,4)}\ge 0,-x_{(1,4)}-x_{(1,3)}\ge 0,-x_{(1,4)}-x_{(1,3)}-x_{(3,4)}\ge 0,&\\
&-x_{(1,4)}-x_{(1,3)}-x_{(3,4)}-x_{(1,2)}\ge0.&
\end{eqnarray*}
From $F^{\trop}$ we get inequalities defining a cone $\mathcal D_{F}\subset\mathbb R^{8}$:
\begin{eqnarray*}
&-x_3\ge 0,&\\
&-x_2\ge 0, -x_2-x_{(1,3)}\ge 0,&\\
&-x_1\ge 0, -x_1-x_{(1,2)}\ge 0,&\\
&-x_{(2,4)}\ge 0,&\\
&-x_{(3,4)}\ge 0,-x_{(3,4)}-x_{(1,2)}\ge 0,&\\
&-x_{(1,4)}\ge 0, -x_{(1,4)}-x_{(1,3)}\ge 0.&
\end{eqnarray*}

Observe that $\mathcal D_{F}\subset \mathcal S_{\w}$. We compute the polytopes $\mathcal S_{\w}(\lambda)$ and $\mathcal D_{F}\cap \tau_{\w}^{-1}(\lambda)$ for $\lambda=(1,1,1)$ and their lattice points using \emph{polymake}\cite{GJ00}. The outcome is
\[
\vert \mathcal S_{\w}(\lambda)\cap \mathbb Z^8\vert = 49 = \dim_{\mathbb C}H^0(X_w,L_\lambda) > \vert \mathcal D_{F}\cap \tau_{\w}^{-1}(\lambda)\cap\mathbb Z^{8}\vert =30.
\]

In particular, the toric variety $X_{\mathcal D_{F}\cap \tau_{\w}^{-1}(\lambda)}$ can not be a flat degeneration of the Schubert variety $X_w$.
However, this observation is not too surprising from a geometric point of view, as the restricted superpotential and the function $F$ correspond to different partial compactifications of the $\A$-cluster variety $G^{e,w}$ associated with $\mathcal Y(s_{\w})$.

When considering the restricted superpotential, the cluster variety we are dealing with is $G^{e,w_0}$ and its compactification $\bar G^{e,w_0}$ with boundary divisors
\[
\{\bar p_1=0\},\{\bar p_{12}=0\},\{\bar p_{123}=0\},\{\bar p_{4}=0\},\{\bar p_{34}=0\},\{\bar p_{234}=0\}.
\]
Recall that $G^{e,w_0}$ is $SL_4/U$ up to codimension 2.
The Schubert variety of our interest is $X_{w}$ with $s_1s_2s_3s_2s_1=w$. 
It is given by $\{\bar p_{34}=0\}$ as a subvariety $SL_4/B$. 
Note that in fact, whenever we have a reduced expression $\w$ and an extension $\w_0=\w s_{i_{\ell(w)+1}}\cdots s_{i_N}$, then the Plücker coordinates that appear as $\A$-cluster variables for faces of $\pa(\w_0)$ that are not faces of $\pa(\w)$ vanish identically on $X_w$.
When restricting the superpotential, we consider the divisor of $\bar G^{e,w_0}$ (resp. $SL_4/U$) given by $\{\bar p_{34}=0\}$, which is closely related to $X_w$.

The function $F$ on the other hand corresponds to the $\A-$cluster variety $G^{e,w}$ and its partial compactification $\bar G^{e,w}$ with boundary divisors
\[
\{\bar p_{1}=0\},\{\bar p_{12}=0\},\{\bar p_{123}=0\},\{\bar p_4=0\},\{\bar p_{24}=0\},\{\bar p_{234}=0\}.
\]
In this case, the defining equation for $X_w$ in $SL_4/B$ is not part of the boundary, so there is no reason to expect information for the Schubert variety from the potential $F$ encoding this boundary.
\end{example}

\newpage
\section{Computing toric degenerations of flag varieties}\label{sec:BLMM}

\noindent
In this section we compute toric degenerations arising from the tropicalization of the full flag varieties $\Flag_4$ and $\Flag_5$ embedded in a product of Grassmannians.
For $\Flag_4$ and $\Flag_5$ we compare toric degenerations arising from string polytopes and the FFLV polytope with those obtained from the tropicalization of the flag varieties.
We also present a general procedure to find toric degenerations  in the cases where the initial ideal arising from a cone of the tropicalization of a variety is not prime.\footnote{Based on joint work with Sara Lamboglia, Kalina Mincheva, and Fatemeh Mohammadi in \cite{BLMM}.}

\bigskip

This project was initialized during the Apprenticeship Program at the
Fields Institute, held 21 August–3 September 2016.
The solutions to the following questions posed during the program can be found in Theorem~\ref{flag4} (see \cite[Problem~5\&6 on Grassmannians]{Stu16}).
\begin{enumerate}
    \item[$5$.] \emph{The complete flag variety for $SL_4$ is a six-dimensional subvariety of $\mathbb P^3\times \mathbb P^4\times \mathbb P^{3}$. Compute its ideal and determine its tropicalization.}
    \item[$6$.] \emph{Classify all toric ideals that arises as initial ideals for the flag variety above. For each such toric degeneration, compute the Newton-Okounkov polytope.}
\end{enumerate}
This section is structured as follows.
We study the tropicalization of the flag varieties $\Flag_n$ for $n=4,5$ and the induced toric degenerations in \S\ref{sec:3}.

In \S\ref{sec:string&FFLV} we recall the definition of the FFLV polytope for regular dominant integral weights. We compute for $\Flag_4$ and $\Flag_5$ all string polytopes and the FFLV polytope for the weight $\rho\in\Lambda^{++}$, the sum of all fundamental weights. 
Moreover, in \S\ref{string:weight} for every string cone we construct a weight vector ${\bf w}_{\underline w_0}$ contained in the tropicalization of the flag variety  in order to further explore the connection between these two different approaches. The construction is inspired by Caldero \cite{Cal02}.
Our work is closely related to \cite{KM16}. We were particularly curious about \cite[Problem~1]{KM16}:
\begin{itemize}
    \item[] \emph{Given a projective variety $Y$, find an embedding of this variety into a projective toric variety so that the resulting tropicalization contains a prime cone of maximal dimension.}
\end{itemize}
In \S\ref{Algorithmic approach} we give an algorithmic approach (see Procedure~\ref{alg:Kh_Basis}) to solving this problem for a subvariety $X$ of a toric variety $Y$ when each cone in $\trop(X)$ has multiplicity one. 
Procedure~\ref{alg:Kh_Basis} aims at computing a new embedding $X'$ of $X$ in case $\trop(X)$ has some non-prime cones. 
Once we have such an embedding, we explain how to get new toric degenerations of $X$. We apply the procedure to $\Flag_4$. 
Furthermore, we explain how to interpret the procedure in terms of finding valuations with finite  Khovanskii basis on the algebra given by the homogeneous coordinate ring of $X$.

\subsection{Tropicalizing \texorpdfstring{$\Flag_n$}{}}\label{sec:3}

In this section we study the tropicalization of $\Flag_4 $ and $\Flag_5$. We analyze the  Gr\"obner toric degenerations arising from $\trop(\Flag_4)$ and $\trop(\Flag_5)$, and we compute the polytopes associated to their normalizations. In Proposition~\ref{6config} we describe the \textit{tropical configurations} arising from the maximal cones of $\trop(\Flag_4)$. These are configurations of a point on a tropical line in a tropical plane corresponding to the points in the relative interior of a maximal cone.

We are interested in finding distinct polytopes up to unimodular equivalence (recall Definition~\ref{def:unimod equiv}) as they give rise to non-isomorphic toric varieties. 
Often it is only possible to determine combinatorial equivalence (see \cite[\S 2.2]{Cox}). 

\begin{definition}\label{def: comb equiv}
Consider two polytopes $P$ and $Q$ in $\mathbb R^n$. Then $P$ and $Q$ are \emph{combinatorially equivalent}, if there exists a bijection
\[
\{\text{faces of }P\}\leftrightarrow \{\text{faces of }Q\}.
\]
\end{definition}
Note that in particular, when $P$ and $Q$ are unimodularly equivalent (see Definition~\ref{def:unimod equiv}) then $P$ and $Q$ are combinatorially equivalent. Hence, if they are not combinatorially equivalent $P$ and $Q$ yield non-isomorphic toric varieties.
We use this fact throughout the section.
 
\begin{theorem}\label{flag4}
The tropical variety $\trop(\Flag_4)$ is a $6$-dimensional rational fan in $\mathbb{R}^{14}/\mathbb R^3$ with a $3$-dimensional lineality space. It consists of 78 maximal cones, 72 of which are prime. They are organized in five $S_4\rtimes \mathbb Z_2$-orbits, four of which contain  prime cones. The prime cones give rise to four 
non-isomorphic toric degenerations.
\end{theorem}
\begin{proof}
The theorem is proved by explicit computations. We developed a   $\emph{Macaulay2}$ package called $\mathtt{ToricDegenerations}$ containing all the functions we use. The package and the  data needed for this proof are available at 
\[
\mathtt{https://github.com/ToricDegenerations}.
\]
The flag variety $\Flag_4$ is a subvariety of $\Gr(1,4)\times \Gr(2,4)\times \Gr(3,4)$, which makes it using the Plücker embedding of Grassmannians a $6$-dimensional subvarity of
$\mathbb{P}^3\times \mathbb{P}^5 \times \mathbb{P}^3$.
The ideal $I_4$ is the kernel of the map $\varphi_4$ defined in \eqref{eq:def ideal flag}. 
It is contained in the total coordinate ring $R$ of $\mathbb{P}^3\times \mathbb{P}^5 \times \mathbb{P}^3$, a $\mathbb C$-polynomial ring in the Plücker variables
\[
p_1,p_2,p_3,p_4,p_{12},p_{13},p_{14},p_{23},p_{24},p_{34},p_{123},p_{124},p_{134},p_{234}.
\]
The (multi-)grading on $R$ is given by the matrix \begin{equation}\label{degree}
\setcounter{MaxMatrixCols}{14}
D:=\bigg(\begin{smallmatrix}
 1 & 1 & 1 & 1 & 0 & 0 & 0 & 0 & 0 & 0 & 0 & 0 & 0 & 0 \\
 0 & 0& 0 &0 &1 &1&1&1&1&1&0&0&0&0\\
 0&0&0&0&0&0&0&0&0&0&1&1&1&1
\end{smallmatrix}\bigg).
\end{equation}

The explicit form of $I_4$ can be found in \cite[page 276]{MS05}. 
As we have seen in \S\ref{sec:pre_trop} the tropicalization of $\Flag_4$ is contained in $\mathbb R^{14}/H$, where $H$ is the subspace of $\mathbb R^{14}$ spanned by the rows of $D$.

We use the $\emph{Macaulay2}$ \cite{M2} interface to $\emph{Gfan}$ \cite{Gfan} to compute $\trop(\Flag_4)$. The given input is the ideal $I_4$ and the $S_4\rtimes \mathbb Z_2$-action (see \cite[\S3.1.1]{GfanM}). The output is a subfan $F$ of the Gr\"obner fan of dimension  $9$. We quotient it by $H$ to get $\trop(\Flag_4)$ as a $6$-dimensional fan contained in $\mathbb R^{14}/H\cong \mathbb R^{14}/\mathbb {R}^3$. 

Firstly, the function \texttt{computeWeightVectors} computes a list of vectors.
There is one for every maximal cone of $\trop(\Flag_4)$ and it is contained in the relative interior of the corresponding cone. 
Then  \texttt{groebnerToricDegenerations} computes all the initial ideals and  checks if they are binomial and prime over $\mathbb Q$. 
These are organized in a hash table, which is the output of the function.

All 78 initial ideals are binomial and all maximal cones have multiplicity one. 
In order to check primeness over $\mathbb C$, we consider for every cone $C$ with $\init_C(I_4)$ prime over $\mathbb Q$ the ideal $I(W_C)$ as defined in \eqref{eq: def I(W_C)}.
We verify $\init_C(I_4)=I(W_C)$ computationally using \emph{Macaulay2}\cite{M2}.

We consider the orbits of the $S_4\ltimes \mathbb Z_2$-action on the set of initial ideals. These correspond to the orbits of maximal cones of $F$ and hence of $\trop(\Flag_4)$. There is one orbit of non-prime initial ideals and four orbits of prime initial ideals. 
The varieties corresponding to initial ideals contained in the same orbit are isomorphic. 
Therefore, for each orbit we choose a representative of the form $\init_{C}(I_4)=I(W_C)$ for some cone $C$ in the orbit. 

We now compute for each of the four prime orbits, the polytope of the normalization of the associated toric varieties. We use the \emph{Macaulay2}-package \emph{Polyhedra}~\cite{Poly2} for the following computations. 
 
The lattice $M$ associated to $S/I(W_C)$ is generated over $\ZZ$ by the columns of $W_C$. 
To use \emph{Polyhedra} we want to have a lattice with index $1$ in $\mathbb Z^{9}$.
If the index of $M$ in $\mathbb Z^9$ is different from $1$, we consider $M$ as the lattice generated by the columns of the matrix  $(\ker( (\ker( W_C))^t)^{t}$.
Here, for a matrix $A$ we consider $\ker(A)$ to be the matrix whose columns minimally generate the kernel of the map $\mathbb Z^{14}\to\mathbb Z^9$ defined by $A$.
We denote the set of generators of $M$ by $\mathcal B_C=\{{\bf b}_1,\ldots ,{\bf b}_{14}\}$ so that 
$M=\mathbb Z\mathcal B_C$.

The toric variety $\mathbb P^3\times \mathbb P^5 \times \mathbb P^3$ can be seen as $\Proj(\oplus_{\ell} R_{\ell(1,1,1)})$ and $I(W_C)$
as an ideal in $\oplus_{\ell} R_{\ell(1,1,1)}$ (see \cite[Chapter 10]{MS05}). The associated toric variety is $\Proj(\oplus_{\ell} \mathbb C[\mathbb Z_{\ge 0} \mathcal B_C]_{\ell(1,1,1)})$.
The polytope $P$ of the normalization is given as the convex hull of those lattice points in $\mathbb Z_{\ge 0} \mathcal B_C$ corresponding to degree $(1,1,1)$-monomials in $\mathbb C[\mathbb Z_{\ge 0} \mathcal B_C]$.

These can be found in the following way. We order the rows of the matrix $({\bf b}_1,\ldots,{\bf b}_{14})$ associated to $\mathcal B_C$ so that the first three rows give the matrix $D$ from \eqref{degree}. Now the matrix $({\bf b}_1,\ldots,{\bf b}_{14})$ represents a map $\ZZ^{14}\to\ZZ^3\oplus\ZZ^6$, where $\ZZ^3\oplus\ZZ^6$ is the lattice $M$ and the $\ZZ^3$ part gives the degree of the monomials associated to each lattice point on $M$.
The lattice points, whose convex hull give the polytope $P$, are those with the first three coordinates being $1$. 
In other words, we have obtained $P$ by applying the reverse procedure of constructing a toric variety from a polytope (see \cite[\S 2.1-\S 2.2]{CLS11}). 
Note that the difference from the procedure given in \cite[\S 2.1-\S 2.2]{CLS11} is the $\mathbb Z^3$-grading and because of that we do not consider  the convex hull of $\mathcal B_C$, but the intersection of $\mathbb Z_{\ge 0} \mathcal B_C$ with these hyperplanes.

In Table~\ref{fig:num} there are the numerical invariants of the initial ideals and their corresponding  polytopes.
Using \emph{polymake} \cite{GJ00} we first obtain that there is no combinatorial equivalence between each pair of polytopes. 
This means that there is no unimodular equivalence between the corresponding normal fans, 
hence the normalization of the toric varieties associated to these toric degenerations are not isomorphic. 
This implies that we obtain four non-isomorphic  toric degenerations.  
\end{proof}

\begin{table}
\begin{center}\scalebox{.9}{
\begin{tabular}{|l|l|l|l|l|l|l|}
 \hline
 Orbit&Size&Cohen-Macaulay&Prime&$\#$Generators&F-vector of associated polytope \\
 \hline
1&  24 &  Yes& Yes &10&(42, 141, 202, 153, 63, 13)\\
2&  12 &  Yes& Yes &10&(40, 132, 186, 139, 57, 12)\\
3&  12 &  Yes & Yes &10&(42, 141, 202, 153, 63, 13)\\
4&    24 & Yes & Yes &10&(43, 146, 212, 163, 68, 14)\\
5&    6 & Yes & No &10& Not applicable\\
\hline
\end{tabular}}
\end{center}
\caption{The tropical variety $\trop(\Flag_4)$ has 78 maximal cones organized in five $S_4\rtimes \mathbb Z_2$-orbits. The algebraic invariants of the initial ideals associated to these cones and the F-vectors of their associated polytopes are listed here.}\label{fig:num} 
\end{table}

\begin{proposition}\label{6config}
There are six tropical configurations up to symmetry (depicted in Figure~\ref{figure:1}) arising from the maximal cones of $\trop(\Flag_4)$. They are further organized in five $S_4\rtimes \mathbb Z_2$-orbits.
\end{proposition}

\begin{proof}
The tropical variety $\trop(\Flag_4)$ is contained in 
\[
\trop(\Gr(1,\mathbb C^4))\times\trop(\Gr(2,\mathbb C^4))\times\trop(\Gr(3,\mathbb C^4)).
\]
Each tropical Grassmannian parametrizes tropicalized linear spaces (see \cite[Theorem 4.3.17]{M-S}). This implies that  every point $p$ in $\trop(\Flag_4)$ corresponds to a  chain of tropical linear subspaces given by a point on a tropical line contained in a tropical plane. All  tropical chains are \textit{realizable}, meaning that they are the tropicalization of the classical chains of linear spaces of $\Bbbk^4$ corresponding to a point $q$ in $\Flag_4(\Bbbk)$ such that $\val(q)=p$, where $\Bbbk=\mathbb C\!\{\!\{t\!\}\!\}$ is the field of Piusseux series and $\val$ is the natural valuation on it (see \cite[Part (3) of Theorem 3.2.3]{M-S}).

In this case, there is only one combinatorial type for the tropical plane and four possible types  for the lines up to symmetry (see \cite[Example 4.4.9]{M-S}). The plane consists of six $2$-dimensional cones positively spanned by all possible pairs of vectors  $(1,0,0)^{t},(0,1,0)^{t},(0,0,1)^{t}$, and $(-1,-1,-1)^{t}$. The combinatorial types of the tropical lines are shown in Figure \ref{Comb types}. The leaves of these graphs represent the rays of the tropical line labeled $1$ up to $4$ corresponding to the positive hull of each of the vectors 
$(1,0,0)^{t},(0,1,0)^{t},(0,0,1)^{t}$, and $(-1,-1,-1)^{t}$.

\begin{figure}
    \begin{center}
    \begin{tikzpicture}[scale=.4]
\draw (-2,0) -- (-1,1) -- (1,1) -- (2,0);
\draw (-2,2) -- (-1,1);
\draw (1,1) -- (2,2);
\node at (-2.2,0) {2};
\node at (-2.2,2) {1};
\node at (2.2,0) {4};
\node at (2.2,2) {3};

\draw (3.5,0) -- (4.5,1) -- (6.5,1) -- (7.5,0);
\draw (3.5,2) -- (4.5,1);
\draw (6.5,1) -- (7.5,2);
\node at (3.3,0) {3};
\node at (3.3,2) {1};
\node at (7.7,0) {4};
\node at (7.7,2) {2};

\draw (9,0) -- (10,1) -- (12,1) -- (13,0);
\draw (9,2) -- (10,1);
\draw (12,1) -- (13,2);
\node at (8.7,0) {4};
\node at (8.7,2) {1};
\node at (13.2,0) {3};
\node at (13.2,2) {2};

\draw (15.3,0) -- (16.3,1);
\draw (15.3,2) -- (16.3,1);
\draw (16.3,1) -- (17.3,2);
\draw (16.3,1) -- (17.3,0);
\node at (15,0) {4};
\node at (15,2) {1};
\node at (17.6,0) {3};
\node at (17.6,2) {2};

\end{tikzpicture}
\caption{Combinatorial types of tropical lines in $\mathbb R^4/ \mathbb R\bf{1}$.}\label{Comb types}
\end{center}
\end{figure}
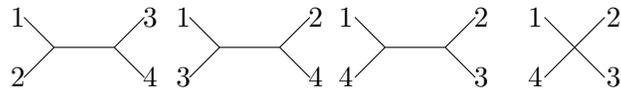

Consider the $S_4\rtimes \mathbb Z_2$-orbits of maximal cones of $\trop(\Flag_4)$. If we compute the chain of tropical linear spaces corresponding to an element in each orbit, we get the configurations in  Figure~\ref{figure:1}. Note that we do not include the labeling since up to symmetry we can get all possibilities.
The point on the line is the black dot. In case the intersection of the line with the rays of the plane is the vertex of the plane then we denote this with a  hollow dot. A vertex of the line is colored in gray if it   lies on a ray of the plane.
For example in orbit 2,  label the rays $1$ to $4$ anti-clockwise starting from the top left edge. We have rays $1$ and $2$ 
in the $2$-dimensional positive hull of $(1,0,0)^{t}$ and $(0,1,0)^{t}$. The vector associated to the internal edge is $(1,1,0)^{t}$. 
The gray point is the origin and the black point has coordinates $(a,1,0)^{t}$ for $a>1$.

Orbits 1 and 4 in Figure~\ref{figure:1} have size $24$, orbits 2 and 3 have  size $12$ and orbit 5 has size $6$. Note that orbit 5 corresponds to non-prime initial ideals. Orbit 1 contains two combinatorial types of tropical configurations and one is sent to the other by the $\mathbb Z_2$-action on the tropical variety.
The orbits $2$ and $3$ differ from the fact that for each combinatorial type of line the gray dot can lie on one of the four rays of the tropical plane. These possibilities are grouped in two pairs, one is in orbit $2$ and the other in orbit $3$. 
\end{proof}

\begin{figure}
    \begin{center}
    \begin{tikzpicture}[scale=.7]
\begin{scope}[xshift=-.8cm]
\node at (-1.5,1) { Orbit 1};
\draw (0,0) -- (1,1) -- (3,1) -- (4,0);
\draw (0,2) -- (1,1);
\draw (3,1) -- (4,2);

\draw [fill] (3.5,1.5) circle [radius=0.15];
\draw [fill, white] (1.5,1) circle [radius=0.15];
    \draw [gray, ultra thick] (1.5,1) circle [radius=0.
    15];
   
\begin{scope}[xshift=5.5cm]
\draw (0,0) -- (1,1) -- (3,1) -- (4,0);
\draw (0,2) -- (1,1);
\draw (3,1) -- (4,2);

\draw [fill, gray] (3,1) circle [radius=0.15];
\draw [fill] (1.5,1) circle [radius=0.15];    
\end{scope}   
\end{scope}
\begin{scope}[yshift=-3cm, xshift=-.75cm]
\node at (-1.5,1) {Orbit 2};
\draw (0,0) -- (1,1) -- (3,1) -- (4,0);
\draw (0,2) -- (1,1);
\draw (3,1) -- (4,2);

\draw [fill] (.5,1.5) circle [radius=0.15];
\draw [fill, gray] (3,1) circle [radius=0.15];

\begin{scope}[xshift=8cm]
\node at (-1.5,1) { Orbit 3};
\draw (0,0) -- (1,1) -- (3,1) -- (4,0);
\draw (0,2) -- (1,1);
\draw (3,1) -- (4,2);

\draw [fill] (3.5,1.5) circle [radius=0.15];
\draw [fill, gray] (1,1) circle [radius=0.15];
\end{scope} 

\begin{scope}[yshift=-3cm]
\node at (-1.5,1) {Orbit 4};
\draw (0,0) -- (1,1) -- (3,1) -- (4,0);
\draw (0,2) -- (1,1);
\draw (3,1) -- (4,2);

\draw [fill] (3.5,1.5) circle [radius=0.15];
\draw [fill, gray] (3,1) circle [radius=0.15];

\begin{scope}[xshift=8cm]
\node at (-1.5,1) {Orbit 5};
\node at (-1.5,0.5) {(non-prime)};
\draw (0,0) -- (1,1) -- (3,1) -- (4,0);
\draw (0,2) -- (1,1);
\draw (3,1) -- (4,2);

\draw [fill] (2.5,1) circle [radius=0.15];
\draw [fill, white] (1.5,1) circle [radius=0.15];
    \draw [gray, ultra thick] (1.5,1) circle [radius=0.15];    
\end{scope}
\end{scope}   
\end{scope}
\begin{scope}[yshift=-8.5cm]
\draw [fill] (3,0) circle [radius=0.15];
\draw [fill, gray] (3,.5) circle [radius=0.15];
\draw [fill, white] (3,1) circle [radius=0.15];
    \draw [gray, ultra thick] (3,1) circle [radius=0.15]; 
\node at (4.65,0) {the point};  
\node at (7,.5) { a point on a ray of the plane};
\node at (6.3,1) {the vertex of the plane};
\end{scope}
\end{tikzpicture}
    \end{center}
    \caption{The list of all tropical configurations up to symmetry that arise in $\Flag_4$. The hollow and the full gray dot denote whether that vertex of the line is the vertex of the plane  or it is contained in a ray of the plane. The black dot is the position of the point on the line.}
   \label{figure:1} 
\end{figure}

\begin{theorem}\label{flag5}
The tropical variety $\trop(\Flag_5)$ is a $10$-dimensional fan in $\mathbb{R}^{30}/\mathbb {R}^4$ with a $4$-dimensional lineality space. It consists of $69780$ maximal cones which are grouped in $536$ $S_5\rtimes \mathbb Z_2$-orbits.
These give rise to $531$ orbits of binomial initial ideals and among these $180$ are prime. They correspond to $180$ non-isomorphic toric degenerations.  
\end{theorem}

\begin{proof}
The flag variety $\Flag_5$ is a $10$-dimensional variety defined by 66
quadratic polynomials in the total coordinate ring of $\mathbb P^4\times \mathbb{P}^{9}\times \mathbb{P}^{9}\times \mathbb P^4$. These are of the form $\sum_{j\in J\backslash I}(-1)^{l_j}p_{I\cup\{j\}}p_{J\backslash\{j\}}$, where $J,I\subset \{1,\ldots,5\}$ and $l_j=\#\{k\in J\mid j<k\}+\#\{i\in I\mid i<j\}$. 

The proof is similar to the proof of Theorem~\ref{flag4}.
The only difference is that  the action of $S_5\rtimes \mathbb Z_2$  on $\Flag_5$ is crucial for the computations. In fact, without exploiting the symmetries the calculations to get the tropicalization would not terminate. Moreover, we only verify primeness  of the initial ideals over $\mathbb Q$ using the \emph{primdec} library \cite{Primedec} in \emph{Singular}
\cite{DGPS}. We compute the polytopes associated to the normalization of the $180$  toric varieties in the same way as Theorem~\ref{flag4}, only changing the matrix of the grading, which is now given by
\begin{equation}\label{degree2}
\setcounter{MaxMatrixCols}{30}
D:=\bigg(\begin{smallmatrix}
 1 & 1 & 1 & 1 &1& 0 & 0 & 0 & 0 & 0 & 0 & 0 & 0 & 0 & 0 &0 & 0 & 0 & 0 & 0 & 0 & 0 & 0 & 0 & 0 &0&0&0&0&0\\
 0 & 0& 0 &0 &0&1 &1&1&1&1&1&1&1&1&1&0 & 0 & 0 & 0 & 0 & 0 & 0 & 0 & 0 & 0&0&0&0&0&0\\
 0 & 0& 0 &0 &0 &0 & 0 & 0 & 0 & 0 & 0 & 0 & 0 & 0 & 0 &1 &1&1&1&1&1&1&1&1&1&0&0&0&0&0\\
  0&0 & 0& 0 &0 &0 &0 & 0 & 0 & 0 & 0 & 0 & 0 & 0 & 0 & 0 & 0 & 0 & 0 & 0 & 0 & 0 & 0 & 0 & 0 &1 & 1 & 1 & 1 &1
\end{smallmatrix}\bigg).
\end{equation}

Since there are no combinatorial equivalences among the normal fans to these polytopes, we deduce that the obtained toric degenerations are pairwise non-isomorphic. More information on the non-prime initial ideals is available in Table~\ref{nonprime} in the appendix.
\end{proof}

\subsection{String\&FFLV polytopes and the tropical flag variety}\label{sec:string&FFLV}
This section provides an introduction to FFLV polytope and explicit computations of the FFLV polytope and the string polytopes for $\Flag_4$ and $\Flag_5$. 
FFLV stands for Feigin, Fourier, and Littelmann, 
who defined this polytope in \cite{FFL11}, and Vinberg who conjectured its existence in a special case. 
We have already seen how string polytopes can be used to construct toric degenerations of the flag variety following Caldero \cite{Cal02}. 
The same is true for the FFLV polytope.
Recall the definition of string polytopes from \S\ref{subsec:string} and the parametrization given in \S\ref{sec:pa and gp}.    

Consider the weight $\rho\in\Lambda^{++}$. 
The string polytope $\mathcal Q_{\underline w_0}(\rho)$ is in general \emph{not} the Minkowski sum of string polytopes $\mathcal Q_{\underline w_0}(\omega_1),\dots,\mathcal Q_{\underline w_0}(\omega_{n-1})$, which motivates the following definition. 

\begin{definition}\label{def:mp}
A string cone has the \emph{weak Minkowski property} (MP), if for every lattice point $p\in \mathcal Q_{\underline w_0} (\rho)$ there exist lattice points $p_{\omega_i}\in \mathcal Q_{\underline w_0}(\omega_i)$ such that
\[
p=p_{\omega_1}+p_{\omega_2}+\dots + p_{\omega_{n-1}}.
\]
\end{definition}

\begin{remark}
Note that the (non-weak) Minkowski property would require the above condition on lattice points to be true for arbitrary weights $\lambda\in \Lambda^{++}$. Further, note that if $\mathcal Q_{\underline w_0}(\rho)$ is the Minkowski sum of the fundamental string polytopes $\mathcal Q_{\underline w_0}(\omega_i)$, then MP is satisfied.
\end{remark}

\begin{center}
\begin{table}
\scalebox{.75}{
\begin{tabular}{| l|  l|  l|  l|  l|  l|}

\hline
$\underline w_0$ & Normal & MP & Weight vector $-{\bf w}_{\underline w_0}$& Prime  & Tropical cone \\
\hline

\begin{tabular}{ l}
String 1: \\
121321 \\
212321 \\
232123 \\
323123
\end{tabular} &
\begin{tabular}{ l}
 \\
yes \\
yes \\
yes \\
yes
\end{tabular}
&
\begin{tabular}{ l}
 \\
yes \\
yes \\
yes \\
yes
\end{tabular}
&
\begin{tabular}{ l}
 \\
$(0, 32, 24, 7, 0, 16, 6, 48, 38, 30, 0, 4, 20, 52)$ \\
$(0, 16, 48, 7, 0, 32, 6, 24, 22, 54, 0, 4, 36, 28)$ \\
$(0, 4, 36, 28, 0, 32, 24, 6, 22, 54, 0, 16, 48, 7)$ \\
$(0, 4, 20, 52, 0, 16, 48, 6, 38, 30, 0, 32, 24, 7)$
\end{tabular}
&
\begin{tabular}{ l}
 \\
yes \\
yes \\
yes \\
yes
\end{tabular} &
\begin{tabular}{ l}
 \\
rays $\{10, 18, 19\}$, cone 71 \\
rays $\{6, 10, 19\}$, cone 44 \\
rays $\{0, 3, 6\}$, cone 3 \\
rays $\{0, 1, 3\}$, cone 1
\end{tabular} \\
\hline

\begin{tabular}{ l}
String 2: \\
123212 \\
321232
\end{tabular} &
\begin{tabular}{ l}
  \\
yes \\
yes
\end{tabular}&
\begin{tabular}{ l}
  \\
yes \\
yes
\end{tabular}&
\begin{tabular}{ l}
  \\
$(0, 32, 18, 14, 0, 16, 12, 48, 44, 27, 0, 8, 24, 56)$ \\
$(0, 8, 24, 56, 0, 16, 48, 12, 44, 27, 0, 32, 18, 14)$
\end{tabular}
&
\begin{tabular}{ l}
  \\
yes \\
yes
\end{tabular}
&
\begin{tabular}{ l}
  \\
rays $\{2, 10, 18\}$, cone 36 \\
rays $\{0, 1, 2\}$, cone 0
\end{tabular} \\
\hline

\begin{tabular}{ l}
String 3: \\
213231
\end{tabular}
&\begin{tabular}{ l}
  \\
yes
\end{tabular}
&
\begin{tabular}{ l}
  \\
yes
\end{tabular}
&
\begin{tabular}{ l}
  \\
$(0, 16, 48, 13, 0, 32, 12, 20, 28, 60, 0, 8, 40, 22)$
\end{tabular}&
\begin{tabular}{ l}
  \\
yes
\end{tabular}
&
\begin{tabular}{ l}
  \\
rays $\{3, 6, 19\}$, cone 24
\end{tabular}\\
\hline

\begin{tabular}{ l}
String 4: \\
132312 
\end{tabular}&
\begin{tabular}{ l}
  \\
yes
\end{tabular}
&
\begin{tabular}{ l}
  \\
no
\end{tabular}
&
\begin{tabular}{ l}
  \\
$(0, 16, 12, 44, 0, 8, 40, 24, 56, 15, 0, 32, 10, 26)$
\end{tabular}&
\begin{tabular}{ l}
  \\
no
\end{tabular}
&
\begin{tabular}{ l}
  \\
rays  $\{1, 2, 17\}$, cone 17
\end{tabular}\\
\hline
\ FFLV & \ yes & \ yes & 

\begin{tabular}{l }
$\mathbf w^{\text{min}}=(0,2,2,1,0,1,1,2,1,2,0,1,1,1)$\\
$\mathbf w^{\text{reg}}=(0,3,4,3,0,2,2,4,3,5,0,1,2,3)$
\end{tabular}
 & \ yes & 
\begin{tabular}{l}
rays $\{9, 11, 12\}$, cone 56\\
rays $\{9, 11, 12\}$, cone 56
\end{tabular}
 \\
\hline
\end{tabular}}
\caption{Isomorphism classes of string polytopes for $n=4$ and $\rho$ depending on $\underline w_0$, normality, the weak Minkowsky property, the weight vectors ${\bf w}_{\underline w_0}$ constructed in \S\ref{string:weight}, primeness of the binomial initial ideals $\init_{{\bf w}_{\underline w_0}}(I_4)$, and the corresponding tropical cones with their spanning rays as they appear 
at \protect{\url{http://www.mi.uni-koeln.de/~lbossing/tropflag/tropflag4.html
}} .}  \label{tab:string4}
\end{table}
\end{center}

\begin{proposition}\label{thm:stringpoly}
For $\Flag_4$ there are four string polytopes in $\mathbb R^{10}$ up to unimodular equivalence 
and three of them satisfy MP. For $\Flag_5$ there are $28$ string polytopes in $\mathbb R^{14}$ up to unimodular equivalence and $14$ of them satisfy MP.
\end{proposition}

\begin{proof}
We first consider $\Flag_4$. There are 16 reduced expressions for $w_0$. Simple transpositions $s_i$ and $s_j$ with $1\le i<i+1<j< n$ commute and are also called \emph{orthogonal}. We consider reduced expressions up to changing those, so there are eight symmetry classes. 
We fix the weight $\rho=\omega_1+\omega_2+\omega_3$ in $\Lambda^{++}$. The string polytopes are organized in four classes up to unimodular equivalence. 
See Table~\ref{tab:string4}, in which $121321$ denotes the reduced expression $\underline w_0=s_1s_2s_1s_3s_2s_1$. 
Hence, they give four different toric degenerations for the embedding $\Flag_4\hookrightarrow \PP(V(\rho))$. 

In order to verify whether the weak Minkowski property holds or not, we proceed as follows. 
We fix $\underline w_0$ to compute the string polytope $\mathcal Q_{\underline w_0}(\rho)$ using \emph{polymake}. 
The number of lattice points in $\mathcal Q_{\underline w_0}(\rho)$ is $\dim(V(\rho))=64$.
Then we compute the polytopes $\mathcal Q_{\underline w_0}(\omega_1),\mathcal Q_{\underline w_0}(\omega_2),\mathcal Q_{\underline w_0}(\omega_3)$ and set $P_{\w_0}:=\mathcal Q_{\w_0}(\omega_1)+\mathcal Q_{\w_0}(\omega_2)+\mathcal Q_{\w_0}(\omega_3)\subset \mathbb R^{9}$. 
If $\vert P_{\w_0}\cap \mathbb Z^{9}\vert <64$, then there exists a lattice point in $\mathcal Q_{\w_0}(\rho)$, that can not be expressed as $p_1+p_2+p_3$ for $p_i\in \mathcal Q_{\w_0}(\omega_i)$. 
For $\underline w_0=s_1s_3s_2s_3s_1s_2$, we compute
\[
\vert P_{\w_0}\cap \mathbb Z^9 \vert =62<64.
\]
Hence, polytopes in the class String 4 do not satisfy MP. 
For polytopes in the classes String 1, 2, and 3 equality holds and MP is satisfied. 

Now consider $\Flag_5$. There are 62 reduced expressions $\underline w_0$ up to changing orthogonal transpositions. 
The map $L:S_5\to S_5$ given on simple reflections by $L(s_i)=s_{4-i+1}$ induces a symmetry on the pseudoline arrangements. 
Further, for a fixed $\lambda \in P^{++}$ is induces a unimodular equivalence between $\mathcal Q_{\w_0}(\lambda)$ and $\mathcal Q_{L(\w_0)}(\lambda)$. 
Exploiting this symmetry, we compute 31 string polytopes for $\rho$. 
They are organized in 28 unimodular equivalence classes, that arise from further symmetries of the underlying pseudoline arrangements. 
Table~\ref{tab:stringweight5} shows which reduced expressions belong to string polytopes within one unimodular equivalence class, and which string cones satisfy MP. Proceeding as for $\Flag_4$, we observe that 14 out of 28 classes satisfy MP.
\end{proof}

We now turn to the FFLV polytope. It is defined in \cite{FFL11} by Feigin, Fourier, and Littelmann to describe bases of irreducible highest weight representations $V(\lambda)$. 
In \cite{FeFL16} they give a construction of a flat degeneration of the flag variety into the toric variety associated to the FFLV polytope. 
It is also an example of the more general setup of birational sequences presented in \cite{FFL15}. 
We recall the definition and compute the FFLV polytopes for $\Flag_4$ and $\Flag_5$ for $\rho$. 
Recall, that $\alpha_i=\epsilon_i-\epsilon_{i+1}\in\mathbb R^n$ for $1\le i< n$ are the simple roots of $\mathfrak{sl}_n$, and  $\alpha_{p,q}$ is the positive root $\alpha_p+\alpha_{p+1}+\dots+\alpha_{q}$ for $1\le p\le q<n$.

\begin{definition}\label{def:dyck}
A \emph{Dyck path} for $\mathfrak{sl}_n$ is a sequence of positive roots 
${\bf d}=(\beta_0,\ldots,\beta_k)$ with $ k\ge0$
satisfying the following conditions
\begin{enumerate}
\item if $k=0$ then ${\bf d}=(\alpha_i)$ for $1\le i\le n-1$,
\item if $k\ge 1$ then \begin{enumerate}
    \item the first and the last roots are simple, i.e. $\beta_0=\alpha_i$, $\beta_k=\alpha_j$ for $1\le i<j\le n-1$, 
    \item if $\beta_s=\alpha_{p,q}$ then $\beta_{s+1}$ is either $\alpha_{p,q+1}$ or $\alpha_{p+1,q}$.
 \end{enumerate}
\end{enumerate}
Denote by $\mathcal D$ the set of all Dyck paths. We choose the positive roots $\alpha>0$ as an indexing set for a basis of $\mathbb R^N$.
\end{definition}

\begin{definition}\label{def:ffl}
The \emph{FFLV polytope} $P(\lambda)\subset \mathbb R^{N}_{\ge 0}$ for a weight $\lambda=\sum_{i=1}^{n-1}m_i\omega_i\in \Lambda^{++}$ is defined as
\begin{eqnarray}\label{eq: def FFLV poly}
P(\lambda)=\left\{ (r_{\alpha})_{\alpha>0}\in \mathbb R^N_{\ge 0}\left|  \begin{matrix} \forall {\bf d}\in \mathcal D:\text{ if }\beta_0=\alpha_{i} \text{ and } \beta_k=\alpha_{j} \\
r_{\beta_0}+ \dots +r_{\beta_k}\le m_{i}+\dots+m_{j}\end{matrix}\right.\right\}.
\end{eqnarray}
\end{definition}

\begin{example}
Consider $\Flag_4$. Then the Dyck paths are
\begin{eqnarray*}
&(\alpha_1),(\alpha_2),(\alpha_3),&\\
&(\alpha_1,\alpha_{1,2},\alpha_2),(\alpha_2,\alpha_{2,3},\alpha_3),&\\
&(\alpha_1,\alpha_{1,2},\alpha_{2},\alpha_{2,3},\alpha_3)
\text{ and } (\alpha_1,\alpha_{1,2},\alpha_{1,3},\alpha_{2,3},\alpha_3)& 
\end{eqnarray*} For our favorite choice of weight $\lambda=\rho=\omega_1+\omega_2+\omega_3$ we obtain the FFLV polytope
\begin{eqnarray*}
P(\rho)=\left\{ (r_{\alpha})_{\alpha>0}\left| \begin{matrix} 
r_{\alpha_1}\le 1,r_{\alpha_2}\le 1,r_{\alpha_3}\le 1,\\
r_{\alpha_1}+r_{\alpha_{1,2}}+r_{\alpha_2}\le 2, r_{\alpha_2}+r_{\alpha_{2,3}}+r_{\alpha_3}\le 2,\\
r_{\alpha_1}+r_{\alpha_{1,2}}+r_{\alpha_2}+r_{\alpha_{2,3}}+r_{\alpha_3}\le 3,\\
r_{\alpha_1}+r_{\alpha_{1,2}}+r_{\alpha_{1,3}}+r_{\alpha_{2,3}}+r_{\alpha_3}\le 3
\end{matrix}\right.\right\}\subset \mathbb R^{6}_{\ge 0}.
\end{eqnarray*}
\end{example}

The following is a corollary of \cite[Proposition~11.6]{FFL11}, which says that a strong version of the Minkowski property is satisfied by the FFLV polytope for $\Flag_n$. It can alternatively be shown for $n=4,5$ using the methods in the proof of Proposition~\ref{thm:stringpoly}.

\begin{corollary}
The FFLV polytope $P(\rho)$ satisfies the weak Minkowski property.
\end{corollary}

\begin{remark}
The FFLV polytope is in general not a string polytope. A computation in \emph{polymake} 
shows that $P(\rho)$ for $\Flag_5$ is not combinatorially equivalent to 
any string polytope for $\rho$.
\end{remark}

\subsubsection*{String cones and points in $\trop(\Flag_n)$}\label{string:weight}

We have seen in \S\ref{sec:pre trop} how to obtain toric degenerations from  maximal prime cones of the tropicalization of a variety. 
We compare the degenerations arising from $\trop(\Flag_n)$ with those from string polytopes and the FFLV polytope. 
Moreover, applying \cite[Lemma~3.2]{Cal02} (see \S\ref{sec:pre val}) we construct a weight vector from a string cone, which allows us to apply Theorem~\ref{thm: val and quasi val with wt matrix} from \S\ref{sec:val and quasival}. 
Computational evidence for $\Flag_4$ and $\Flag_5$ shows that each constructed weight vector lies in the relative interior of a maximal cone in $\trop(\Flag_n)$. 
A similar idea for a more general case is carried out in \cite[\S7]{KM16}. 
For the FFLV polytope we compute weight vectors for $\Flag_n$ with $n=4,5$ (see Example~\ref{exp.ffl}) following a construction given in \cite{FFR15}. 

\medskip
We now prove the result in Theorem~\ref{ts-comparison} by analyzing the  polytopes associated to the different  toric degenerations of  $\Flag_n$ for $n=4,5$.

\begin{table}
\begin{center}
\begin{tabular}{ c |l}

Orbit  & Combinatorially equivalent polytopes\\
\hline
1  &  String 2 \\[-1.5ex]
2 &  String 1 (Gelfand-Tsetlin)\\[-1.5ex]
3  & String 3 and FFLV \\[-1.5ex]
4 & - \\

\end{tabular}
\end{center}
\caption{Combinatorial equivalences among the polytopes obtained from prime cones in $\trop(\Flag_4)$ and string polytopes resp. the FFLV polytope.}\label{tab:f-vector4}
\end{table}

\begin{proof}[Proof of Theorem~\ref{ts-comparison}]
In order to distinguish the different toric degenerations, we consider the normalizations of the toric varieties associated to their special fibers. 
Two projective normal toric varieties are isomorphic, if their corresponding polytopes are unimodularly equivalent. 
For this reason we first look for combinatorial equivalences between the polytopes.
If they are not combinatorially equivalent then they can not be unimodularly equivalent, hence they define non-isomorphic toric varieties. 
We use \emph{polymake}~\cite{GJ00} for computations with polytopes.

From Table~\ref{tab:f-vector4} one can see that for $\Flag_4$ there is one toric degeneration, whose associated polytope is not combinatorially 
equivalent to any string polytope or the FFLV polytope for $\rho$.
Hence, its corresponding normal toric variety is not isomorphic to any toric variety associated to these polytopes. 
For the toric varieties associated to the other polytopes we can not exclude isomorphisms since there might be a unimodular equivalences.

For $\Flag_5$, Table~\ref{tab:f-vector5} in the appendix shows that there are 168 polytopes obtained from prime cones of $\trop(\Flag_5)$ that are not combinatorially equivalent to any string polytope or the FFLV polytope for $\rho$. 
\end{proof}

\begin{remark}
There are also string polytopes, which are not combinatorially equivalent to any polytope from prime cones in $\trop(\Flag_n)$ for $n=4,5$. These are exactly those not satisfying MP, i.e. one string polytope for $\Flag_4$ and 14 for $\Flag_5$. See also Table~\ref{tab:stringweight5}.
\end{remark}

From now on, we fix a reduced expression $\underline w_0=s_{i_1}\ldots s_{i_N}$ and we consider the birational sequence of simple roots $S:=(\alpha_{i_1},\dots, \alpha_{i_N})$.
As we have seen in Example~\ref{exp: birat seq, PBW and string}.2 for Grassmannains, the same is true here: $S$ is a birational sequence. 
In \cite{FFL15} (see also \cite{Kav15}) they realize string polytopes as Newton-Okounkov polytopes associated to the valuation from this birational sequence.
Another necessary ingredient to obtain such a valuation on $\mathbb C[\Flag_n]$ was the choice of total order on $\mathbb Z^{N}$.
Recall therefore the definition of the $\Psi$-weighted reverse lexicographic order $\prec_\Psi$ from \eqref{eq: def psi wt order} and the definition of $\val_S$ from \eqref{eq:valseq}.
For our choice of $S$ from a reduced expression $\w_0$ we denote $\val_{\w_0}:=\val_S$.

Then $\val_{\w_0}$ can be computed explicitly on Plücker coordinates. We have seen this for $\Gr(2,n)$ in \eqref{eq: val seq on plucker}, but more generally we have by \cite[Proposition~2]{FFL15} for $\{j_1,\dots,j_k\}\subset [n]$
\[
\val_S(\bar p_{j_1,\dots,j_k})=\min{}_{\prec_\Psi}\{\bm \in\mathbb Z^{N}_{\ge 0}\mid {\bf f^m}(e_{1}\wedge\dots\wedge e_{k})=e_{j_1}\wedge\dots\wedge e_{j_k}\},
\]
where ${\bf f^m}=f_{\alpha_{i_1}}^{m_1}\cdots f_{\alpha_{i_N}}^{m_N}\in U(\mathfrak{n}^-)$.

\begin{example}\label{ex:exterioralg}
For $\Flag_4$ three root vectors in $\mathfrak n^-$ are 
\[
f_{\alpha_1}=\bigg(\begin{smallmatrix} 0 & 0 & 0 & 0\\ 1 & 0 & 0 & 0\\ 0 & 0 & 0 & 0\\ 0 & 0 & 0 & 0\end{smallmatrix}\bigg),\quad 
f_{\alpha_2}=\bigg(\begin{smallmatrix} 0 & 0 & 0 & 0\\ 0 & 0 & 0 & 0\\ 0 & 1 & 0 & 0\\ 0 & 0 & 0 & 0\end{smallmatrix}\bigg),
\quad\text{ and }\quad f_{\alpha_3}=\bigg(\begin{smallmatrix} 0 & 0 & 0 & 0\\ 0 & 0 & 0 & 0\\ 0 & 0 & 0 & 0\\ 0 & 0 & 1 & 0\end{smallmatrix}\bigg).
\]
Consider $V=\bigwedge^2\mathbb C^4$. As we have seen above the action of $\mathfrak n^-$ on $\mathbb C^4$ is given by $f_{\alpha_i}(e_i)=e_{i+1}$ and $f_{\alpha_i}(e_j)=0$ for $j\not = i$. 
On $V$ the $\mathfrak n^-$-action is given by
\[
f_{\alpha_i}(e_j\wedge e_k)=f_{\alpha_i}(e_j)\wedge e_k + e_{j}\wedge f_{\alpha_i}(e_k).
\]
The Pl\"ucker coordinate $\bar p_{13}\in \mathbb C[\Flag_4]$ is of degree $(0,1,0)$, i.e. contained in $\mathbb C[\Flag_4]_{(0,1,0)}\cong (\bigwedge^2\mathbb C^4)^*$.
As we have seen before $\bar p_{13}=(e_1\wedge e_3)^*$, and so we consider $e_1\wedge e_3\in V$. 
Then $f_{\alpha_2}(e_1\wedge e_2)=e_1\wedge e_3$.  
We fix $\hat\w_0=s_1s_2s_1s_3s_2s_1\in S_4$, then as seen in \eqref{eq: PBW for demazure} we have
$U(\mathfrak n^-)\cdot (e_1\wedge e_2)=\langle f_{\alpha_1}^{m_1}f_{\alpha_2}^{m_2}f_{\alpha_1}^{m_3}f_{\alpha_3}^{m_4}f_{\alpha_2}^{m_5}f_{\alpha_1}^{m_6}\cdots (e_1\wedge e_2)\mid m_i\in \mathbb Z_{\ge 0}\rangle$. Hence,
\[
{\bf f}^{(0,1,0,0,0,0)}(e_1\wedge e_2)={\bf f}^{(0,0,0,0,1,0)}(e_1\wedge e_2)=e_1\wedge e_3.
\]
The minimal $\bm\in(\mathbb Z^{6},\prec_{\Psi})$ satisfying ${\bf f^m}(e_1\wedge e_2)=e_1\wedge e_3$ is ${(0,1,0,0,0,0)}$, so we have $\val_{\hat\w_0}(\bar p_{13})=(0,1,0,0,0,0)$.
\end{example}

We want to apply the results from \S\ref{sec:val and quasival} to the given valuations of form $\val_{\w_0}:\mathbb C[\Flag_n]\setminus\{0\}\to (\mathbb Z^N,\prec_\Psi)$.
Let therefore $M_{\w_0}:=(\val_{\w_0}(\bar p_J))_{0\not=J\subsetneq [n]}\in\mathbb Z^{N\times\binom{n}{1}+\dots+\binom{n}{n-1}}$ be the matrix whose columns are given by the images of Plücker coordinates under $\val_{\w_0}$.
We define a linear form $e:\mathbb Z^N\to\mathbb Z$ by
\[
-e({\bf m}):=2^{N-1}m_1+2^{N-2}m_2+\ldots +2m_{N-1}+m_N.
\]
By \cite[Proof of Lemma 3.2]{Cal02} and our choice of total order $\prec_\Psi$, it satisfies $\val_{\w_0}(\bar p_I)\prec_\Psi \val_{\w_0}(\bar p_J)$ implies $e(\val_{\w_0}(\bar p_I))<e(\val_{\w_0}(\bar p_J))$ for $0\not=I,J\subsetneq [n]$.

\begin{definition}\label{str.wt.vec}
For a fixed reduced expression ${\underline w_0}$ the \emph{weight} of the Pl\"ucker variable $p_J$ is $e(\val_{\w_0}(\bar p_J))$. We define the \emph{weight vector} ${\bf w}_{\underline w_0}$ in $\mathbb R^{{\binom{n}{1}}+{\binom{n}{2}}+\cdots+{\binom{n}{n-1}}}$ by 
\[
{\bf w}_{\underline w_0}:=e(M_{\w_0})=(e(\val_{\w_0}(\bar p_J)))_{0\not =J\subsetneq [n]},
\]
where we order the subsets $0\not =J\subsetneq [n]$ lexicographically from $[1]$ to $[2,n]$.
\end{definition}

\begin{example}
We continue as in Example~\ref{ex:exterioralg} with the reduced expression $\hat\w_0\in S_4$.
As $\val_{\hat\w_0}(\bar p_{13})=(0,1,0,0,0,0)$ the weight of $\bar p_{13}$ is $e(0,1,0,0,0,0)=-(1\cdot 2^4)=-16$. 
Similarly, we obtain weights for all Pl\"ucker coordinates and
\[
-{\bf w}_{\underline w_0}=(0,32,24,7,0,16,6,48,38,30,0,4,20,52).
\]
Table~\ref{tab:string4} contains all weight vectors (up to sign) for $\Flag_4$ constructed in the way just described.
\end{example}

\begin{proposition}\label{thm:stringweight}
Consider $\Flag_n$ with $n=4,5$. The above construction produces a weight vector ${\bf w}_{\underline w_0}$ for every string cone. This weight vector lies in the relative interior of a maximal cone of $\trop(\Flag_n)$.
If further the string cone satisfies MP, then ${\bf w}_{\underline w_0}$ lies in the relative interior of a prime cone whose associated polytope is combinatorially equivalent to $\mathcal Q_{\w_0}(\rho)$.
\end{proposition}

\begin{proof}
The constructed weight vectors ${\bf w}_{\underline w_0}$ can be found in Table~\ref{tab:string4} for $\Flag_4$ and Table~\ref{tab:stringweight5} in the appendix for $\Flag_5$. A computation in \emph{Macaulay2} shows that all initial ideals $\init_{{\bf w}_{\underline w_0}}(I_n)$ for $n=4,5$ are binomial, hence in the relative interiors of maximal cones of $\trop(\Flag_n)$.

Moreover, if MP is satisfied  we check using \emph{polymake} that  the polytope constructed from the maximal prime cone $C^{\underline w_0}\subset \trop(\Flag_n)$ with ${\textbf w}_{\underline w_0}$ in its relative interior is combinatorially equivalent to the string polytope $\mathcal Q_{\w_0}(\rho)$. 
See Table~\ref{tab:string4} and Table~\ref{tab:stringweight5}. 
\end{proof}

This computational outcome can actually be explained by Theorem~\ref{thm: val and quasi val with wt matrix}. In this context in can be stated as follows.

\begin{theorem}\label{thm: quasival for string}
Let $\w_0$ be a reduced expression of $w_0\in S_n$ and consider $\bw_{\w_0}\in\mathbb R^{\binom{n}{1}+\dots+\binom{n}{n-1}}$. 
If $\init_{\bw_{\w_0}}(I_n)$ is prime, then $S(A_n,\val_{\w_0})$ is generated by $\{\val_{\w_0}(\bar p_J)\mid 0\not=J\subsetneq [n]\}$. In particular, 
\[
\mathcal Q_{\w_0}(\rho)=\conv(\val_{\w_0}(\bar p_J)\mid 0\not=J\subsetneq [n]),
\]
and the Plücker coordinates form a Khovanskii basis for $\val_{\w_0}$.
\end{theorem}
\begin{proof}
First note, that by \cite[Theorem~14.6]{MS05} the ideal $I_n$ is generated by elements $f$ satisfying $\deg f>\varepsilon_i$ for all $i\in[n-1]$.
Further, by Lemma~\ref{lem: init wM vs init M} $\init_{\bw_{\w_0}}(I_n)=\init_{M_{\w_0}}(I_n)$ and so $\init_{M_{\w_0}}(I_n)$ is prime by assumption.
As $\val_{\w_0}$ is of full rank, we can apply Theorem~\ref{thm: val and quasi val with wt matrix}.
If follows that $S(A_n,\val_{\w_0})$ is generated by $\{\val_{\w_0}(\bar p_J)\mid 0\not=J\subsetneq [n]\}$ and that
\[
\Delta(A_n,\val_{\w_0})=\conv(\val_{\w_0}(\bar p_J)\mid 0\not=J\subsetneq [n]).
\]
As by \cite[\S11]{FFL15} $\mathcal Q_{\w_0}(\rho)$ is the Newton-Okounkov body of the valuation $\val_{\w_0}$ the claim follows. 
\end{proof}

\begin{corollary}\label{cor: prime implies MP}
With assumptions as in Theorem~\ref{thm: quasival for string}, we have: 
\[
\init_{\bw_{\w_0}}(I_n)\ \text{ is prime } \Rightarrow \ \mathcal Q_{\w_0} \ \text{ has the (strong) Minkowski property}.
\]
\end{corollary}
\begin{proof}
By Theorem~\ref{thm: quasival for string} the value semigroup $S(A_n,\val_{\w_0})$ is generated by $\{\val_{\w_0}(\bar p_J)\mid 0\not=J\subsetneq[n]\}$.

Recall that $\mathbb C[SL_n/U]\cong \bigoplus_{\lambda\in\Lambda^{+}} V(\lambda)$ from \S\ref{sec:pre flag}.
The algebra therefore has a natural multigrading given by $\lambda^+$.
Consider the valuation $\hat \val_{\w_0}:\mathbb C[SL_n/U]\setminus\{0\}\to \Lambda^{+}\times \mathbb Z_{\ge 0}^{N}$ given by $ \hat \val_{\w_0}(f)=(\deg f,\val_{\w_0}(f))$ as in \cite{FFL15}.
Then for every $\lambda\in\Lambda^{++}$ we have
\[
\mathcal Q_{\w_0}(\lambda)=C(A_n,\hat\val_{\w_0})\cap \{\lambda\}\times\mathbb R^N.
\]
If $\lambda=\sum_{i=1}^{n-1}a_i\omega_i$ we know (by an argument similar to \cite[Proposition~2]{FFL15}) that the Minkowski sum of the fundamental string polytopes satisfies
\[
a_1\mathcal{Q}_{\w_0}(\omega_1)+\dots+a_{n-1}\mathcal{Q}_{\w_0}(\omega_{n-1})\subseteq \mathcal{Q}_{\w_0}(\lambda).
\]
As the generators of $S(A_n,\val_{w_0})$ are the union of lattice point of all fundamental string polytopes, we count
\[
\vert (a_1\mathcal{Q}_{\w_0}(\omega_1)+\dots+a_{n-1}\mathcal{Q}_{\w_0}(\omega_{n-1}))\cap \mathbb Z^N\vert = \dim_{\mathbb C}V(\lambda)=\vert  \mathcal{Q}_{\w_0}(\lambda)\vert.
\]
Hence, the two polytopes are equal.
\end{proof}

Regarding the opposite implication to Corollary~\ref{cor: prime implies MP}, we know that if $\mathcal Q_{\w_0}$ satisfies the strong Minkowski property, then $S(A_n,\val_{\w_0})$ is generated by $\{\val_{\w_0}(p_J)\mid 0\not=J\subsetneq[n]\}$. 
Hence, $\gr_{\val_{\w_0}}(A_n)$ is generated by $\overline{p_J}$ for $ 0\not=J\subsetneq[n]$ and we have a surjective morphism $\pi:\mathbb C[p_J]_J\to \gr_{\val_{\w_0}}(A_n)$.
Then $\ker(\pi)\subset \mathbb C[p_J]_J$ is a prime ideal with $\gr_{\val_{\w_0}}(A_n)\cong \mathbb C[p_J]_J/\ker(\pi)$.
So far, we didn't manage to prove that $\ker(\pi)=\init_{\bw_{\w_0}}(I_n)$.

Computational evidence for $\Flag_4$ and $\Flag_5$ leads us to the following conjecture.

\begin{conjecture}\label{conjecture}
Let $n\geq 3$ be an arbitrary integer. For every reduced expression $\underline w_0$, the weight vector ${\bf w}_{\underline w_0}$ lies 
in the relative interior of a maximal cone in $\trop(\Flag_n)$. 

Moreover, if the string cone satisfies MP this vector lies in the relative interior of the prime cone $C\subset \trop(\Flag_n)$, whose associated polytope is combinatorially equivalent to the string polytope $\mathcal Q_{\w_0}(\rho)$.
\end{conjecture}

The following example discusses a similar construction of weight vectors for the FFLV polytope.

\begin{example}\label{exp.ffl}
Consider for $\Flag_4$ the birational PBW-sequence with good ordering 
$S:=(\alpha_1+\alpha_2+\alpha_3,\alpha_1+\alpha_2,\alpha_2+\alpha_3,\alpha_1,\alpha_2,\alpha_3)$ (similar to Example~\ref{exp: birat seq, PBW and string}.1 for Grassmannians).
We choose as total order on $\mathbb Z^{N}$ the \emph{homogeneous right lexicographic order} (see \cite[Example~9]{FFL15}), i.e.
\[
\bm \succ_{\text{rlex}} \bn :\Leftrightarrow \sum_{i=1}^N m_i>  \sum_{i=1}^N n_i, \text{ or } \sum_i m_i=\sum_i n_i
 \text{ and } \bm>_{\text{rlex}} \bn,
\]
where $<_{\text{rlex}}$ denotes the right lexicographic order on $\mathbb Z^N$.
With these choices the associated Newton-Okounkov polytope to the valuation $\val_S$ is the FFLV polytope (see \cite[\S13]{FFL15} and also \cite{Kir15}).
Similar to the above we define (according to the degrees defined in \cite{FFR15}) linear forms $e^{\min},e^{\text{reg}}:\mathbb Z^{N}\to \mathbb Z$ by
\begin{eqnarray*}
e^{\min}({{\bf m}}):=m_1+2m_2+m_3+2m_4+m_5+m_6,\\
e^{\rm reg}({\bf m}):=3m_1+4m_2+2m_3+3m_4+2m_5+m_6.
\end{eqnarray*}
We obtain in analogy to Definition~\ref{str.wt.vec} the corresponding weight vectors in $\mathbb R^{\binom{n}{1}+\dots +\binom{n}{n-1}}$
\begin{eqnarray*}
{\bf w}^{\min}=(0,2,2,1,0,1,1,2,1,2,0,1,1,1),\\
{\bf w}^{\rm reg}=(0,3,4,3,0,2,2,4,3,5,0,1,2,3).
\end{eqnarray*}
A computation in \emph{Macaulay2} shows that $\init_{{\bf w}^{\min}}(I_4)=\init_{{\bf w}^{\rm reg}}(I_4)$ is a binomial prime ideal. 
Hence, ${\bf w}^{\min}$ and ${\bf w}^{\rm reg}$ lie in the relative interior of the same prime cone $C\subset \trop(\Flag_4)$. 
Using \emph{polymake}~\cite{GJ00} we verify that the polytope associated to $C$ is combinatorially equivalent to the FFLV polytope $P(\rho)$. 
We did the analogue of this computation for $\Flag_5$ and the outcome is the same, $\init_{{\bf w}^{\min}}(I_5)=\init_{{\bf w}^{\rm reg}}(I_5)=\init_C(I_5)$ with the polytope associated to $C$ being combinatorially equivalent to $P(\rho)$. 
The weight vectors ${\bf w}^{\min}$ and ${\bf w}^{\text{reg}}$ for $\Flag_5$ can be found in Table~\ref{tab:stringweight5} in the appendix.
In fact, in \cite{FFFM} Fang-Feigin-Fourier-Makhlin show that for arbitrary $n$ the vectors $\bw^{\min}$ and $\bw^{\text{reg}}$ lie in the relative interior of a maximal prime cone of $\trop(\Flag_n)$ and they give explicit inequalities to describe the cone.
\end{example}

\subsection{Toric degenerations from non-prime cones}\label{Algorithmic approach}

As we have seen in \S\ref{sec:3}, not all maximal cones in the tropicalization of a variety give rise to prime initial ideals and hence to toric degenerations. In fact, there may also be tropicalizations without prime cones (see Example \ref{ex:badCone}).
Let $X=V(I)$ be a subvariety of a toric variety $Y$. 
In this section, we give a recursive procedure (Procedure~\ref{alg:Kh_Basis}) to compute a new embedding $V(I')\subset Y'$ of $X$ in case $\trop(X)$ has non-prime cones. Let $C$ be a non-prime cone. If the procedure terminates, the new tropical variety $\trop(X')$ has more prime cones than $\trop(X)$ and at least one of them is projecting onto $C$. 
We apply this procedure to $\Flag_4$ and compare the new toric degenerations with those obtained so far (see Proposition~\ref{prop:output}). The procedure terminates for $\Flag_4$, but we are still investigating the conditions for which this is true in general.

\medskip

\begin{algorithm}[h]
\SetAlgorithmName{Procedure}{} 
\KwIn{\medskip {\bf Input:\ }  $A = \mathbb C[x_1,\ldots,x_n]/I$, where $\mathbb C[x_1,\ldots,x_n]$ is the total coordinate ring of the toric variety $Y$ and $I $ defines the subvariety $V(I)\subset Y$,
$C$ a non-prime cone of $\trop(V(I))$.
}
\BlankLine

{\bf Initialization:}\\
Compute the primary decomposition of $\init_C(I)$;\\
$I(W_C)=$ unique prime toric component in the decomposition;\\
$G=$ minimal generating set of $I(W_C)$.\\
Compute a list of binomials $L_C=\{f_{1},\ldots,f_{s}\}$ in $G$, which are not in $\init_C(I)$;\\
$A'=\mathbb C[x_1, \dots, x_n,y_{1},\ldots,y_{s}]/I'$ with $I'=I+\langle y_{1}-f_{1},\ldots,y_{s}-f_{s}\rangle $;\\
$V(I')$ subvariety of $Y'$ whose total coordinate ring is $\mathbb C[Y]:=\mathbb C[x_1, \dots, x_n,y_{1},\ldots,y_{s}]$.\\
Compute $\trop(V(I'))$;\\
\For {all  prime cones $C'\in\trop(V(I'))$} {

\If{    $\pi(C')$ is contained in the relative interior of $C$ }{{\bf Output:} The algebra $A'$   
and the ideal $\init_{C'}(I')$ of a  toric degeneration of $V(I')$.}\Else{
Apply the procedure again  to $A'$  and $C'$.}
}
\label{alg:Kh_Basis}
\caption{Computing  new embeddings of  the variety $X$ in case $\trop(X)$ contains non-prime cones}
\end{algorithm}

We explain Procedure~\ref{alg:Kh_Basis}: consider a toric variety $Y$ whose total coordinate ring with associated $\mathbb Z^k$-degree $\deg:\mathbb Z^{n}\to\mathbb Z^k$ is $\mathbb C[x_1, \dots, x_n]$. 
Let $X$ be the subvariety of $Y$ associated to an ideal $I\subset \mathbb C[x_1, \dots, x_n]$, where the Krull dimension of $A=\mathbb C[x_1, \dots, x_n]/I$ is $d$.
Denote by $\trop(V(I))$ the tropicalization of $X$ intersected with the torus of $Y$. 
Suppose there is a non-prime cone $C\subset \trop(V(I))$ with multiplicity one. 
By  Lemma~\ref{toric:lemma}, we have that $I(W_C)$ is the unique toric ideal in the primary decomposition of $\init_{C}(I)$, hence $\init_{C}(I)\subset I(W_C)$. 
We compute $I(W_C)$ using the function $\mathtt{primaryDecomposition}$ in \emph{Macaulay2}.
Fix a minimal binomial generating set $G$ of $I(W_C)$, and let $L_C=\{f_1, \ldots, f_{s}\}$ be the set consisting of binomials in $G$ that are not in $\init_{C}(I)$. 
By \emph{Hilbert's Basis Theorem} ${s}$ is a finite number. 
The absence of these binomials in $\init_C(I)$ is the reason why the initial ideal is not equal to $I(W_C)$, hence not prime.
We introduce new variables $\{y_{1},\ldots,y_{{s}}\}$  and consider the algebra $A' = \mathbb C[x_1, \ldots, x_n,y_{1},\ldots,y_{s}]/I'$, where 
\[
I'=I+\langle y_{1}-f_{1},\ldots,y_{s}-f_{s}\rangle.
\]  
The ideal $I'$  is a homogeneous ideal in $\mathbb{C}[x_1, \ldots, x_n, y_{1},\ldots,y_{s}]$ with respect to the grading 
\[
(\deg(x_1), \ldots, \deg(x_n),\deg(f_1),\ldots,\deg(f_s)).
\]
The new variety $V(I')$ is a subvariety of the toric variety $Y'$ with total coordinate  $\mathbb C[Y']:=\mathbb C[x_0, \ldots, x_n,y_{1},\ldots,y_{s}]$.
For example, if $V(I)$ is a subvariety of a projective space then $V(I')$ is contained in a weighted projective space.

Since the algebras $A$ and $A'$ are isomorphic as graded algebras, the varieties $V(I)$ and $V(I')$ are isomorphic. 
We have a monomial map 
\[
\pi: \mathbb C[x_1,\ldots,x_n]/I\to \mathbb C[x_0, \ldots, x_n,y_{1},\ldots,y_{s}]/I'
\]
inducing a surjective map $\trop(\pi):\trop(V(I'))\to \trop(V(I))$ (see \cite[Corollary 3.2.13]{M-S}). 
The map $\trop(\pi)$ is the projection onto the first $n$ coordinates.
Suppose there exists a prime cone $C'\subset \trop(V(I'))$, whose projection has a non-empty intersection with the relative interior of $C$. Then by construction we have $\init_C(I)\subset \init_{C'}(I') \cap \mathbb{C}[x_0,\ldots,x_n]$ and the procedure terminates. In this way we obtain a new initial ideal $\init_{C'}(I')$ which is toric and hence gives a new toric degeneration of the variety $V(I')\cong V(I)$.
If only non-prime cones are projecting to $C$ then run this procedure again with $A'$ and $ C'$, where the latter is a maximal cone of $\trop(V(I'))$, which projects to $C$.
We can then repeat the procedure starting from a different non-prime cone.

The  function to apply Procedure \ref{alg:Kh_Basis} is  \texttt{findNewToricDegenerations} and it is part of the package $\mathtt{ToricDegenerations}$. This computes only one re-embedding for each non-prime cone. It is possible to use \texttt{mapMaximalCones} to obtain the image of $\trop(V(I'))$ under the map $\pi$.

\begin{remark}
If $f_i$ is a polynomial in $\mathbb C[x_1,\ldots,x_n]$ with the standard grading and  $ \deg(f_i) > 1$, then we need to
homogenize the ideal $I'$ before computing the tropicalization with \emph{Gfan}. This is done by adding a new variable $h$. The homogenization of $I'$ with respect to $h$ is denoted by $I'_{proj}\subseteq \mathbb C[x_1,\ldots,x_n,y_1,\ldots,y_s,h]$. Then by \cite[Proposition 2.6.1]{M-S} for every $\bf {w}$ in $\mathbb {R}^{n+s+2}$ the ideal $\init_{\bf w}(I')$ is obtained from $\init_{(\bf{w},\text{0})}(I'_{proj})$ by setting $h=1$.
\end{remark}

If the cone $C$ is prime, we can construct a valuation $\val_C$ on $\Bbbk[x_1,\ldots,x_n]/I$ in the following way. 
Consider the matrix $W_C$ in Equation~(\ref{def:W}). 
For monomials $m_i = c {\bf x} ^{{\bf {\alpha}}_i}\in \mathbb C[x_1,\ldots x_n] $ define 
\begin{equation}
\val(m_i) = W_C {\bf \alpha}_i\quad \text{and}\quad \val (\sum_i m_i) = \min_i \{\val(m_i)\},
\end{equation} 
where the minimum on the right side is taken with respect to the lexicographic order on $(\ZZ^d, +)$. 
This is a valuation on $\mathbb C[x_1,\dots,x_n]$ of rank equal to the Krull dimension of $A$ for every cone $C$. 
Composing $\val $ with the quotient morphism $p:\mathbb C[x_1,\ldots,x_n]\to \mathbb C[x_1,\ldots,x_n]/I$ we obtain a map $\val_C$, which is a valuation if and only if the cone $C$ is prime. 
Moreover, in \cite{KM16} Kaveh and Manon prove that a cone $C$ in $\trop(V(I))$ is prime if and only if $A=\Bbbk[x_1,\ldots,x_n]/I$ has a  finite \textit{Khovanskii basis} for the valuation $\val_C$ constructed from the cone $C$. 

Procedure \ref{alg:Kh_Basis} can be interpreted as finding an extension $\val_{C'}$ of $\val_C$  so that $A'$ has finite Khovanskii basis with respect to $\val_{C'}$. The  Khovanskii basis is given by the images of $x_1, \dots, x_n,y_1, \dots ,y_s$ in $A'$. 
We illustrate the procedure in  the following  example. 

\begin{example}
\label{ex:badCone}
Consider the algebra $A = \mathbb C[x, y, z]/ \langle xy+xz+yz\rangle $. The tropicalization of $V(\langle xy+xz+yz\rangle)\subset \mathbb P^2$ has three maximal cones. The corresponding initial ideals are $\langle xz+yz\rangle,\langle xy+yz\rangle$ and $\langle xy+xz\rangle$, none of which is prime. Hence they do not give rise to toric degenerations. The matrices associated to each cone are
\[
{W_{C_1}} = \begin{pmatrix} 
0 & 0 & -1 \\
1 & 1 & 1  
\end{pmatrix},\quad
{W_{C_2}} = \begin{pmatrix} 
0 &-1 &  0 \\
1 & 1 & 1  
\end{pmatrix}\quad\text{and}\quad
{W_{C_3}} = \begin{pmatrix} 
-1 & 0 & 0 \\
 1 &  1 &  1  
\end{pmatrix}.
\]
We now apply Procedure \ref{alg:Kh_Basis} to the cone $C_1$. The initial ideal associated to $C_1$ is generated by $xz+yz$. In this case $\init_{C_1}(I)=\langle z \rangle \cdot  \langle x+y\rangle$ hence for the missing binomial $x+y$ we adjoin a new variable $u$ to $ \mathbb C[x, y, z]$  and the new relation $u-x-y$ to $I$. 
We have  
\[
I'=\langle xy+xz+yz,u-x-y\rangle\ \text{and}\ A' = \mathbb C[x, y,z,u]/I'
\]
with $V(I')$ a subvariety of $\mathbb P^3$.
After computing the tropicalization of $V(I')$ we see that there exists a prime cone $C'$ such that  $\pi(C')=C$. The matrix $W_{C'}$ associated to the cone $C'$ is
\[
{W'} = \begin{pmatrix} 
0 & 0 & -1 &  1\\
1 & 1 & 1 & 1
\end{pmatrix}.
\]
The initial ideal $\init_{C'}(I')$ gives a toric degeneration of $V(I')$. The image of the set $\{ x, y, z,u\}$ in $A'$ is a Khovanskii basis for $S(A', \val_{C'})$.
Repeating this process for the cones $C_2$ and $C_3$ of $\trop(V(xy+xz+yz))$, we get prime cones $C'_2$ and $C'_3$ whose projections are $C_2$ and $C_3$ respectively. 
Hence, there is a valuation with finite Khovanskii basis and a corresponding toric degeneration for every maximal cone.
\end{example}

\begin{figure}[h]
\begin{center}
\begin{tikzpicture}[scale=.7]
\node at (-4,4.5) {$\trop(V(I'))\supset C_1,C_2,C_3$};

\draw[-][dashed] (1.7,5.7) -- (4,6);
\draw[-][dashed] (3,5.7) -- (4,6);
\draw[-][dashed] (5,5) -- (4,6);
\draw[-] (3,3.5) -- (4,6);
\draw[fill= green,opacity=0.6] (3,3.5) -- (4,6) -- (5,5) --cycle;

\draw[fill= yellow,opacity=0.5] (1.7,5.7) -- (4,6) -- (3,3.5) --cycle;
\draw[fill= red,opacity=0.5] (3,5.7) -- (4,6) -- (3,3.5) --cycle;
\draw[->] (3,3) -- (3,2);
\node[right] at (3,2.5) {$\trop(\pi)$};
\node at (-4,1.1) {$\trop(\Flag_4)\supset C$};
\draw[fill=yellow] (1,0) -- (3,1.7) -- (3,0) --cycle;
\draw[fill= green] (5,0) -- (3,1.7) -- (3,0) --cycle;
\draw[-] (1,0) -- (3,1.7);
 \draw[-] (5,0) -- (3,1.7);
 \draw[-][ultra thick ,red] (3,0) -- (3,1.7);
\draw[-] (1,0) -- (5,0);
    \end{tikzpicture}
\caption{The three triangles above represent  the three cones in $\trop(V(I'))$ which project down to the non-prime cone $C$ in $\trop(\Flag_4)$.}\label{figure:2}
\end{center}
\end{figure}
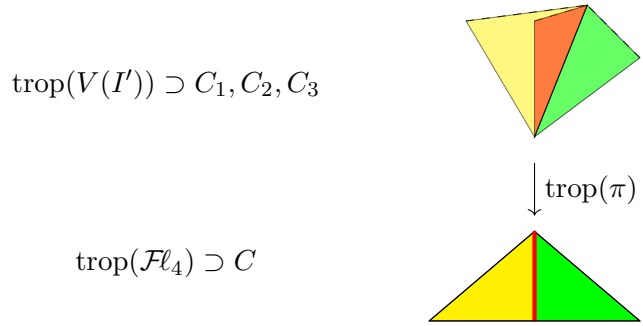

We now apply Procedure~\ref{alg:Kh_Basis}  to $\trop(\Flag_4)$.

\begin{proposition}\label{prop:output}
Each of the non-prime cones of $\trop(\Flag_4)$ gives rise to three toric degenerations, which are not isomorphic to any degeneration coming from the prime cones of $\trop(\Flag_4)$. Moreover, two of the three new polytopes are combinatorially equivalent to the previously missing string polytopes for $\rho$ in the class String 4.
\end{proposition}

\begin{proof}
By Theorem \ref{flag4} we know that $\trop(\Flag_4)$ has six non-prime cones forming one $S_4\rtimes\mathbb Z_2$-orbit. Hence, we apply Procedure~\ref{alg:Kh_Basis} to only one non-prime cone. The result for the other non-prime cones is the same up to symmetry. In particular, the obtained toric degenerations from one cone is isomorphic to those coming from another cone. 
We describe the results for the maximal cone $C$ associated to the  initial ideal $\init_C(I_4)$ defined by the following binomials: 
\begin{center}
\begin{tabular}{llll}

$p_4p_{123}-p_{3}p_{124},$ &$p_{24}p_{134}-p_{14}p_{234}$, &$p_{23}p_{134}-p_{13}p_{234}$,
&$p_{2}p_{14}-p_{1}p_{24}$, \\ $p_{2}p_{13}-p_{1}p_{23}$, &$p_{24}p_{123}-p_{23}p_{124}$,& $p_{14}p_{123}-p_{13}p_{124}$, &$p_{4}p_{23}-p_{3}p_{24}$\\ 
&$p_{4}p_{13}-p_{3}p_{14}$, \text{ and }&$p_{14}p_{23}-p_{13}p_{24}.$
\end{tabular}
\end{center}
We define the ideal $I'=I_4+\langle w-p_{2}p_{134}+p_{1}p_{234}\rangle$.  
The grading on the variables $p_1,\ldots,p_{234}$ and $w$  
is given by the matrix 
\setcounter{MaxMatrixCols}{15}
\[
D':=\bigg(\begin{smallmatrix}
 1 & 1 & 1 & 1 & 0 & 0 & 0 & 0 & 0 & 0 & 0 & 0 & 0 & 0 &1\\
 0 & 0& 0 &0 &1 &1&1&1&1&1&0&0&0&0&0\\
 0&0&0&0&0&0&0&0&0&0&1&1&1&1&1
\end{smallmatrix}\bigg).
\]
It extends the grading on the variables $p_1,\ldots,p_{234}$ given by the matrix $D$ in (\ref{degree}).
The tropical variety $\trop(V(I'))$ has $105$ maximal cones, $99$ of which are prime. Among them we can find three maximal prime cones, which are mapped to $C$ by $\trop(\pi)$ (see Figure \ref{figure:2}). We compute the polytopes associated to the normalization of  these three toric degenerations by applying the same methods as in Theorem \ref{flag4}. Using \emph{polymake} we check that two of them are combinatorially equivalent to the string polytopes for $\rho$ in the class String 4. Moreover, none of them is combinatorially equivalent to any polytope coming from prime cones of $\trop(\Flag_4)$, hence they define different toric degenerations.
\end{proof}

Proposition~\ref{prop:output} suggests that for $\w_0=s_1s_3s_2s_3s_1s_2\in$ String~4 the weighted string cone $\mathcal Q_{\w_0}$ does not satisfy MP because the element 
\[
\val_{\w_0}(\bar p_{2}\bar p_{134}+\bar p_{1}\bar p_{234})\succeq_{\Psi} \min{}_{\prec_\Psi}\{\val_{\w_0}(\bar p_2\bar p_{134}),\val_{\w_0}(\bar p_1\bar p_{234})\}
\]
is missing as a generator for $S(A_4,\val_{\w_0})$. As $\val_{\w_0}(\bar p_2\bar p_{134})=\val_{\w_0}(\bar p_1\bar p_{234})=(1,0,1,1,0,0)$, we deduce $\val_{\w_0}(\bar p_{2}\bar p_{134}+\bar p_{1}\bar p_{234})\succ_{\Psi} (1,0,1,1,0)$.
Hence, this element can not be obtained from the images of Plücker coordinates under $\val_{\w_0}$ and therefore $\val_{\w_0}(\bar p_{2}\bar p_{134}+\bar p_{1}\bar p_{234})$ has to be added as a generator for $S(A_4,\val_{\w_0})$.

\begin{remark}
Procedure~\ref{alg:Kh_Basis} could be applied also to $\Flag_5$ but we have not been able to do so. In fact, the tropicalization for $\trop(V(I_5'))$ did not terminate since the computation can not be simplified by symmetries.  \end{remark}

\begin{appendices}
\appendixpage
\noappendicestocpagenum
\addappheadtotoc

\chapter{Grassmannians}
\section{Plabic weight vectors for \texorpdfstring{$\Gr(3,\mathbb C^6)$}{}}\label{sec:app Gr}

Here are the computational findings on plabic weight vectors $\bw_\G$ defined in \S\ref{sec:BFFHL} in more detail for $\Gr(3,\mathbb C^6)$ (based on the joint work \cite{BFFHL}). The code can be found in \cite{M2c}.

There are 34 reduced plabic graphs for $\Gr(3,\mathbb C^{6})$, and they give rise to the weight vectors in Table~\ref{tab:matching}.
The first column indicates the unfrozen variables corresponding to a cluster
and determining a plabic graph;
the second column gives the corresponding weight vector
in the basis indexed by
\begin{multline*}
\{123, 124, 134, 234, 125, 135, 235,
145, 245, 345, 126, 136, 236,
146, 246, 346, 156, 236, 356, 456\}.
\end{multline*}
The third column gives the corresponding isomorphism class of a cone in the tropical Grassmannian as described  in \cite[after Lemma 5.3]{SS04} and, for $GG$ \cite[after Lemma 5.1]{SS04}; the fourth column gives the permutation
($\sigma = [a_1a_2a_3a_4a_5a_6]$ where $\sigma(i) = a_i$) that moves the initial ideal
of the weight vector in column 2 to the initial ideal of the corresponding cone in the tropical Grassmannian using the sample vectors given in \cite{SS04}.
The permutations were obtained using Macaulay 2 \cite{M2}, see \cite{M2c} for the code. The last column refers to the enumeration of cluster seeds from \cite{BCL}, where a combinatorial model for cluster algebras of type $D_4$ is studied. The 50 seeds are given by centrally symmetric pseudo-triangulations of a once punctured dirk with $8$ marked points. 
In the paper they analyze symmetries among the cluster seeds and associate each seed to a isomorphism class of maximal cones in $\trop(\Gr(3,\mathbb C^{6}))$. Although they consider all 50 cluster seeds, the outcome is similar to ours: they recover only six of the seven types of maximal cones, missing the cone of type $EEEE$.

\begin{center}
\begin{table}
\centering
\scalebox{0.76}{
\begin{tabular}{|c|c|c|c|c|}
\hline
mutable $\A$-cluster variables in seed $s_{\mathcal G}$ & plabic weight vector $\mathbf w_{\mathcal G}\in\mathbb R^{\binom{6}{3}}$ & type in \cite{SS04} & $\sigma$ & $\#$ of $s_{\mathcal G}$ in \cite{BCL}\\ 
\hline
 $p_{135}, p_{235}, p_{145}, p_{136}$&(0,0,1,1,1,1,1,1,1,4,1,1,1,1,1,4,4,4,5,5) & GG & 123456 & 49\\
\hline
$p_{124}, p_{246}, p_{346}, p_{256}$&(0,0,0,3,0,0,3,3,4,4,3,3,4,4,4,4,4,4,4,7) & GG & 134562& 23\\
\hline
$p_{125}, p_{235}, p_{245}, p_{256}$&(0,0,1,1,0,1,1,2,2,5,3,3,3,3,3,5,3,3,5,6) & EEFF1 & 124563&17\\
\hline
$p_{235}, p_{136}, p_{236}, p_{356}$&(0,0,0,0,0,0,0,1,1,2,0,0,0,1,1,2,2,2,3,5) & EEFF1 & 123456& 11\\
\hline
$p_{124}, p_{125}, p_{145}, p_{245}$&(0,0,2,3,1,2,3,2,3,6,4,4,4,4,4,6,5,5,6,6) &  EEFF1 & 135642& 15\\
\hline
$p_{124}, p_{125}, p_{245}, p_{256}$&(0,0,1,2,0,1,2,2,3,5,4,4,4,4,4,5,4,4,5,6) & EEFF1 & 134562& 27\\
\hline
$p_{235}, p_{236}, p_{256}, p_{356}$&(0,0,0,0,0,0,0,2,2,3,1,1,1,2,2,3,2,2,3,6) & EEFF1 & 123564& 29\\
\hline
$p_{136}, p_{236}, p_{346}, p_{356}$&(0,0,0,1,0,0,1,2,2,2,0,0,1,2,2,2,3,3,3,6) & EEFF1 & 312456&12\\
\hline
$p_{236}, p_{346}, p_{256}, p_{356}$&(0,0,0,1,0,0,1,3,3,3,1,1,2,3,3,3,3,3,3,7) & EEFF1 & 312564& 28\\
\hline
$p_{134}, p_{136}, p_{146}, p_{346}$&(0,0,0,3,1,1,3,2,3,3,1,1,3,2,3,3,5,5,5,6) & EEFF1 & 356421&9\\
\hline
$p_{134}, p_{145}, p_{136}, p_{146}$&(0,0,1,3,2,2,3,2,3,4,2,2,3,2,3,4,6,6,6,6) & EEFF1 & 345612&8\\
\hline
$p_{124}, p_{134}, p_{145}, p_{146}$&(0,0,1,4,2,2,4,2,4,5,3,3,4,3,4,5,6,6,6,6) & EEFF1 & 145632&32\\
\hline
$p_{125}, p_{235}, p_{145}, p_{245}$&(0,0,2,2,1,2,2,2,2,6,3,3,3,3,3,6,4,4,6,6) & EEFF1 & 125643&14\\
\hline
$p_{124}, p_{134}, p_{146}, p_{346}$&(0,0,0,4,1,1,4,2,4,4,2,2,4,3,4,4,5,5,5,6) & EEFF1 & 156432&31\\
\hline
$p_{125}, p_{235}, p_{256}, p_{356}$&(0,0,0,0,0,0,0,2,2,4,2,2,2,3,3,4,3,3,4,6) & EEFF2 & 125346&30\\
\hline
$p_{124}, p_{134}, p_{125}, p_{145}$&(0,0,1,3,1,1,3,1,3,5,3,3,4,3,4,5,5,5,5,5) & EEFF2 & 163452&33\\
\hline
$p_{134}, p_{136}, p_{346}, p_{356}$&(0,0,0,2,0,0,2,2,3,3,0,0,2,2,3,3,4,4,4,6) & EEFF2 & 512634&10\\
\hline
$p_{136}, p_{236}, p_{146}, p_{346}$&(0,0,0,2,1,1,2,2,2,2,1,1,2,2,2,2,4,4,4,6) & EEFF2 & 612534&13\\
\hline
$p_{124}, p_{145}, p_{245}, p_{146}$&(0,0,2,4,2,3,4,3,4,6,4,4,4,4,4,6,6,6,7,7) & EEFF2 & 153462&16\\
\hline
$p_{235}, p_{245}, p_{236}, p_{256}$ &(0,0,1,1,0,1,1,2,2,4,2,2,2,2,2,4,2,2,4,6) & EEFF2 & 126345&18\\
\hline
$p_{125}, p_{135}, p_{235}, p_{145}$&(0,0,1,1,1,1,1,1,1,5,2,2,2,2,2,5,4,4,5,5) & EFFG & 123456&43\\
\hline
$p_{135}, p_{235}, p_{136}, p_{356}$&(0,0,0,0,0,0,0,1,1,3,0,0,0,1,1,3,3,3,4,5) & EFFG & 345612 &45\\
\hline
$p_{236}, p_{246}, p_{346}, p_{256}$&(0,0,0,2,0,0,2,3,3,3,2,2,3,3,3,3,3,3,3,7) & EFFG & 612345&26\\
\hline
$p_{124}, p_{146}, p_{246}, p_{346}$&(0,0,0,4,1,1,4,3,4,4,3,3,4,4,4,4,5,5,5,7) & EFFG & 134562&21\\
\hline
$p_{134}, p_{135}, p_{145}, p_{136}$ &(0,0,1,2,1,1,2,1,2,4,1,1,2,1,2,4,5,5,5,5) & EFFG & 561234&47\\
\hline
$p_{124}, p_{245}, p_{246}, p_{256}$&(0,0,1,3,0,1,3,3,4,5,4,4,4,4,4,5,4,4,5,7) & EFFG & 356124 &24\\
\hline
$p_{245}, p_{236}, p_{146}, p_{246}$&(0,0,1,3,1,2,3,3,3,4,3,3,3,3,3,4,4,4,5,7) & EEEG & 265341&20\\
\hline
$p_{134}, p_{125}, p_{135}, p_{356}$&(0,0,0,1,0,0,1,1,2,4,1,1,2,2,3,4,4,4,4,5) & EEEG & 126534&48\\
\hline
$p_{125}, p_{135}, p_{235}, p_{356}$&(0,0,0,0,0,0,0,1,1,4,1,1,1,2,2,4,3,3,4,5) & EEFG & 342156 &42\\
\hline
$p_{134}, p_{125}, p_{135}, p_{145}$&(0,0,1,2,1,1,2,1,2,5,2,2,3,2,3,5,5,5,5,5) & EEFG & 563421&44\\
\hline
$p_{134}, p_{135}, p_{136}, p_{356}$&(0,0,0,1,0,0,1,1,2,3,0,0,1,1,2,3,4,4,4,5) & EEFG & 215634&46\\
\hline
$p_{245}, p_{236}, p_{246}, p_{256}$&(0,0,1,2,0,1,2,3,3,4,3,3,3,3,3,4,3,3,4,7) & EEFG & 156342&25\\
\hline
$p_{236}, p_{146}, p_{246}, p_{346}$&(0,0,0,3,1,1,3,3,3,3,2,2,3,3,3,3,4,4,4,7) & EEFG & 634215&19\\
\hline
$p_{124}, p_{245}, p_{146}, p_{246}$&(0,0,1,4,1,2,4,3,4,5,4,4,4,4,4,5,5,5,6,7) & EEFG & 321564&22\\
\hline
\end{tabular}}
\caption{Dictionary for the 34 plabic graphs.}\label{tab:matching}
\end{table}
\end{center}

\chapter{Flag varieties}
In this Appendix we provide numerical evidence of our computations in \S\ref{sec:BLMM}.  Table~\ref{nonprime}  contains data on the non-prime maximal cones of $\trop(\Flag_5)$. In Table~\ref{tab:f-vector5} there is information on the polytopes obtained from maximal prime cones of $\trop(\Flag_5)$. This includes the F-vectors, combinatiral equivalences among the polytopes, and between those and the string polytopes, resp. FFLV polytope, for $\rho$.
Lastly Table~\ref{tab:stringweight5} contains information on the string polytopes and FFLV polytope for $\Flag_5$, such as the weight vectors constructed in \S\ref{string:weight}, primeness of the initial ideals with respect to these vectors, and the MP property.  

\section{Algebraic and combinatorial invariants of \texorpdfstring{$\trop(\Flag_5)$}{}}\label{polyFlag5}

Below we collect in a table all the information about the non-prime initial ideals of $\Flag_5$ up to symmetry.

\begin{table}[h]
\begin{center}
\begin{tabular}{|l|l|l|l|l|l|}
 \hline
Number of Orbits & $\#$Generators\\
 \hline
30&   69\\[-1ex]
267  &66\\[-1ex]
37  &68\\[-1ex]
11  &70\\[-1ex]
10  &71\\[-1ex]
2&  73\\
\hline
\end{tabular}
\caption{Data for non-prime initial ideals of $\Flag_5$.}\label{nonprime}
\end{center}
\end{table}

The following table shows the F-vectors of the polytopes associated to maximal prime cones of $\trop(\Flag_5)$ for one representative in each orbit. The last column contains information on the existence of a combinatorial equivalence between these polytopes and the string polytopes resp. FFLV polytope for $\rho$. The initial ideals are all \textit{Cohen-Macaulay}.

\footnotesize
\begin{center}
\begin{longtable}{| l | l | l |}
\hline
Orbit & F-vector & Combinatorial equivalences \\
\hline \endhead
0  &  475 2956 8417 14241 15690 11643 5820 1899 374 37 & \\ 
\hline
1  &  456 2799 7843 13023 14038 10159 4938 1565 301 30 & \\ 
\hline
2  &  425 2573 7108 11626 12333 8779 4201 1316 253 26 & \\ 
\hline
3  &  393 2313 6200 9833 10125 7021 3297 1027 201 22 & \\ 
\hline
4  &  433 2621 7230 11796 12473 8847 4219 1318 253 26 & \\ 
\hline
5  &  435 2630 7246 11810 12479 8848 4219 1318 253 26 & \\ 
\hline
6  &  425 2553 6988 11317 11888 8388 3987 1245 240 25 & \\ 
\hline
7  &  450 2751 7677 12699 13648 9863 4800 1529 297 30 & \\ 
\hline
8  &  435 2630 7246 11810 12479 8848 4219 1318 253 26 & \\ 
\hline
9  &  419 2522 6922 11243 11842 8373 3985 1245 240 25 & \\ 
\hline
10  &  453 2785 7817 12999 14027 10157 4938 1565 301 30 & \\ 
\hline
11  &  463 2885 8237 13987 15474 11532 5788 1895 374 37 & \\ 
\hline
12  &  463 2852 8020 13365 14459 10501 5121 1627 313 31 & \\ 
\hline
13  &  457 2840 8078 13638 14954 10996 5413 1726 330 32 & \\ 
\hline
14  &  454 2819 8016 13540 14870 10968 5427 1744 337 33 & \\ 
\hline
15  &  445 2748 7770 13050 14254 10464 5161 1658 322 32 & \\ 
\hline
16  &  441 2681 7438 12228 13056 9369 4525 1430 276 28 & \\ 
\hline
17  &  440 2704 7602 12684 13752 10014 4897 1560 301 30 & \\ 
\hline
18  &  471 2923 8298 13995 15369 11369 5667 1845 363 36 & \\ 
\hline
19  &  464 2883 8200 13861 15258 11313 5651 1843 363 36 & \\ 
\hline
20  &  467 2911 8309 14097 15574 11586 5804 1897 374 37 & \\ 
\hline
21  &  461 2876 8225 13993 15509 11575 5814 1903 375 37 & \\ 
\hline
22  &  397 2363 6416 10313 10755 7536 3561 1109 215 23 & \\ 
\hline
23  &  437 2669 7447 12319 13236 9556 4642 1475 286 29 & \\ 
\hline
24  &  425 2553 6988 11317 11888 8388 3987 1245 240 25 & \\ 
\hline
25  &  415 2498 6861 11158 11772 8339 3976 1244 240 25 & \\ 
\hline
26  &  470 2942 8436 14377 15944 11889 5955 1939 379 37 & \\ 
\hline
27  &  460 2856 8109 13656 14929 10944 5374 1712 328 32 & \\ 
\hline
28  &  449 2741 7634 12594 13487 9702 4695 1486 287 29 & \\ 
\hline
29  &  427 2592 7181 11778 12523 8926 4270 1334 255 26 & \\ 
\hline
30  &  425 2573 7108 11626 12333 8779 4201 1316 253 26 & FFLV \\ 
\hline
31  &  443 2708 7557 12495 13411 9667 4686 1485 287 29 & \\ 
\hline
32  &  397 2363 6416 10313 10755 7536 3561 1109 215 23 &  S22\\ 
\hline
33  &  425 2553 6988 11317 11888 8388 3987 1245 240 25 & \\ 
\hline
34  &  419 2522 6922 11243 11842 8373 3985 1245 240 25 & \\ 
\hline
35  &  405 2407 6518 10442 10851 7578 3571 1110 215 23 & \\ 
\hline
36  &  401 2387 6477 10398 10825 7570 3570 1110 215 23 & \\ 
\hline
37  &  368 2154 5755 9111 9373 6497 3052 953 188 21 & S21\\ 
\hline
38  &  379 2214 5892 9280 9494 6547 3063 954 188 21 & S27, S28\\ 
\hline
39  &  393 2313 6200 9833 10125 7021 3297 1027 201 22 & \\ 
\hline
40  &  358 2069 5453 8516 8653 5941 2778 870 174 20 & S1, S18, S26, S29 (Gelfand-Tsetlin)\\ 
\hline
41  &  459 2851 8111 13720 15118 11223 5614 1834 362 36 & \\ 
\hline
42  &  467 2913 8322 14133 15629 11636 5831 1905 375 37 & \\ 
\hline
43  &  423 2562 7083 11596 12313 8772 4200 1316 253 26 & \\ 
\hline
44  &  425 2573 7108 11626 12333 8779 4201 1316 253 26 & S24 \\ 
\hline
45  &  397 2363 6416 10313 10755 7536 3561 1109 215 23 & S23 \\ 
\hline
46  &  461 2876 8225 13993 15509 11575 5814 1903 375 37 & \\ 
\hline
47  &  400 2366 6377 10175 10546 7363 3480 1089 213 23 & \\ 
\hline
48  &  393 2313 6200 9833 10125 7021 3297 1027 201 22 & \\ 
\hline
49  &  393 2313 6200 9833 10125 7021 3297 1027 201 22 & \\ 
\hline
50  &  379 2214 5892 9280 9494 6547 3063 954 188 21 & S2, S19\\ 
\hline
51  &  426 2599 7257 12034 12981 9420 4602 1470 286 29 & \\ 
\hline
52  &  428 2594 7176 11761 12514 8947 4307 1359 263 27 & \\ 
\hline
53  &  419 2522 6922 11243 11842 8373 3985 1245 240 25 & \\ 
\hline
54  &  466 2917 8371 14288 15879 11870 5960 1944 380 37 & \\ 
\hline
55  &  443 2729 7692 12867 13982 10197 4987 1585 304 30 & \\ 
\hline
56  &  453 2787 7826 13011 14021 10122 4895 1539 293 29 & \\ 
\hline
57  &  469 2926 8358 14188 15679 11663 5839 1906 375 37 & \\ 
\hline
58  &  458 2825 7958 13286 14398 10472 5113 1626 313 31 & \\ 
\hline
59  &  472 2949 8435 14335 15854 11796 5902 1923 377 37 & \\ 
\hline
60  &  440 2704 7602 12684 13752 10014 4897 1560 301 30 & \\ 
\hline
61  &  472 2967 8561 14720 16525 12526 6410 2144 432 43 & \\ 
\hline
62  &  457 2842 8099 13726 15153 11266 5640 1842 363 36 & \\ 
\hline
63  &  465 2902 8296 14096 15588 11594 5795 1884 368 36 & \\ 
\hline
64  &  459 2851 8111 13720 15118 11223 5614 1834 362 36 & \\ 
\hline
65  &  428 2608 7269 12028 12946 9377 4576 1462 285 29 & \\ 
\hline
66  &  441 2681 7438 12228 13056 9369 4525 1430 276 28 & \\ 
\hline
67  &  418 2510 6876 11157 11753 8321 3969 1243 240 25 & \\ 
\hline
68  &  406 2442 6713 10943 11587 8245 3950 1241 240 25 & \\ 
\hline
69  &  373 2199 5926 9474 9849 6897 3267 1024 201 22 & \\ 
\hline
70  &  427 2586 7144 11681 12383 8806 4209 1317 253 26 & \\ 
\hline
71  &  451 2781 7840 13111 14243 10390 5089 1623 313 31 & \\ 
\hline
72  &  440 2704 7602 12684 13752 10014 4897 1560 301 30 & \\ 
\hline
73  &  406 2442 6713 10943 11587 8245 3950 1241 240 25 & \\ 
\hline
74  &  448 2764 7800 13061 14208 10377 5087 1623 313 31 & \\ 
\hline
75  &  462 2873 8181 13846 15258 11321 5656 1844 363 36 & \\ 
\hline
76  &  457 2842 8099 13726 15153 11266 5640 1842 363 36 & \\ 
\hline
77  &  469 2927 8364 14203 15699 11678 5845 1907 375 37 & \\ 
\hline
78  &  454 2802 7903 13216 14348 10453 5110 1626 313 31 & \\ 
\hline
79  &  451 2787 7879 13221 14419 10565 5200 1667 323 32 & \\ 
\hline
80  &  441 2705 7584 12611 13622 9885 4823 1537 298 30 & \\ 
\hline
81  &  454 2803 7914 13263 14455 10598 5231 1687 330 33 & \\ 
\hline
82  &  441 2697 7532 12465 13391 9660 4685 1485 287 29 & \\ 
\hline
83  &  445 2721 7593 12550 13461 9694 4694 1486 287 29 & \\ 
\hline
84  &  441 2697 7532 12465 13391 9660 4685 1485 287 29 & \\ 
\hline
85  &  445 2725 7617 12611 13546 9764 4728 1495 288 29 & \\ 
\hline
86  &  397 2363 6416 10313 10755 7536 3561 1109 215 23 & \\ 
\hline
87  &  368 2154 5755 9111 9373 6497 3052 953 188 21 & S5, S31 \\ 
\hline
88  &  452 2801 7946 13385 14654 10771 5309 1699 327 32 & \\ 
\hline
89  &  430 2624 7318 12097 12974 9329 4497 1411 269 27 & \\ 
\hline
90  &  456 2834 8071 13670 15083 11210 5612 1834 362 36 & \\ 
\hline
91  &  432 2633 7332 12104 12975 9341 4521 1430 276 28 & \\ 
\hline
92  &  467 2919 8359 14230 15769 11756 5892 1922 377 37 & \\ 
\hline
93  &  456 2834 8071 13670 15083 11210 5612 1834 362 36 & \\ 
\hline
94  &  426 2597 7244 11998 12926 9370 4575 1462 285 29 & \\ 
\hline
95  &  440 2708 7630 12769 13898 10169 5001 1603 311 31 & \\ 
\hline
96  &  432 2633 7332 12104 12975 9341 4521 1430 276 28 & \\ 
\hline
97  &  412 2479 6810 11083 11707 8306 3967 1243 240 25 & \\ 
\hline
98  &  415 2511 6945 11391 12133 8679 4174 1313 253 26 & \\ 
\hline
99  &  458 2845 8092 13676 15042 11132 5543 1800 353 35 & \\ 
\hline
100  &  437 2669 7447 12319 13236 9556 4642 1475 286 29 & \\ 
\hline
101  &  441 2703 7569 12562 13531 9780 4746 1502 289 29 & \\ 
\hline
102  &  427 2586 7144 11681 12383 8806 4209 1317 253 26 & \\ 
\hline
103  &  419 2522 6922 11243 11842 8373 3985 1245 240 25 & \\ 
\hline
104  &  437 2669 7447 12319 13236 9556 4642 1475 286 29 & \\ 
\hline
105  &  411 2470 6776 11012 11617 8235 3933 1234 239 25 & \\ 
\hline
106  &  413 2483 6808 11043 11606 8177 3871 1201 230 24 & \\ 
\hline
107  &  425 2553 6988 11317 11888 8388 3987 1245 240 25 & \\ 
\hline
108  &  405 2407 6518 10442 10851 7578 3571 1110 215 23 & \\ 
\hline
109  &  405 2427 6638 10751 11296 7969 3785 1181 228 24 & S30 \\ 
\hline
110  &  465 2904 8312 14152 15700 11734 5907 1940 384 38 & \\ 
\hline
111  &  464 2902 8323 14204 15795 11828 5960 1956 386 38 & \\ 
\hline
112  &  438 2690 7559 12608 13667 9952 4868 1552 300 30 & \\ 
\hline
113  &  445 2725 7617 12611 13546 9764 4728 1495 288 29 & \\ 
\hline
114  &  437 2669 7447 12319 13236 9556 4642 1475 286 29 & \\ 
\hline
115  &  411 2470 6776 11012 11617 8235 3933 1234 239 25 & \\ 
\hline
116  &  424 2574 7139 11737 12529 8983 4332 1367 264 27 & \\ 
\hline
117  &  419 2522 6922 11243 11842 8373 3985 1245 240 25 & \\ 
\hline
118  &  401 2387 6477 10398 10825 7570 3570 1110 215 23 & \\ 
\hline
119  &  405 2427 6638 10751 11296 7969 3785 1181 228 24 & S6 \\ 
\hline
120  &  464 2893 8261 14019 15483 11503 5746 1869 366 36 & \\ 
\hline
121  &  454 2806 7928 13283 14448 10543 5159 1641 315 31 & \\ 
\hline
122  &  451 2794 7928 13370 14676 10840 5387 1746 342 34 & \\ 
\hline
123  &  444 2736 7715 12915 14053 10273 5044 1613 312 31 & \\ 
\hline
124  &  466 2909 8318 14138 15644 11650 5837 1906 375 37 & \\ 
\hline
125  &  456 2815 7939 13271 14398 10480 5118 1627 313 31 & \\ 
\hline
126  &  423 2561 7078 11586 12303 8767 4199 1316 253 26 & \\ 
\hline
127  &  429 2580 7064 11429 11972 8402 3959 1221 232 24 & \\ 
\hline
128  &  431 2626 7309 12058 12915 9290 4494 1422 275 28 & \\ 
\hline
129  &  428 2602 7224 11883 12684 9087 4375 1377 265 27 & \\ 
\hline
130  &  443 2727 7679 12831 13927 10147 4960 1577 303 30 & \\ 
\hline
131  &  432 2637 7354 12152 13024 9356 4505 1412 269 27 & \\ 
\hline
132  &  451 2793 7920 13342 14620 10770 5331 1718 334 33 & \\ 
\hline
133  &  434 2632 7273 11879 12557 8883 4210 1301 246 25 & \\ 
\hline
134  &  452 2781 7813 13004 14042 10171 4944 1566 301 30 & \\ 
\hline
135  &  453 2808 7969 13433 14725 10847 5366 1727 335 33 & \\ 
\hline
136  &  451 2794 7928 13370 14676 10840 5387 1746 342 34 & \\ 
\hline
137  &  433 2646 7390 12236 13150 9482 4589 1448 278 28 & \\ 
\hline
138  &  442 2715 7629 12727 13808 10076 4948 1587 309 31 & \\ 
\hline
139  &  432 2633 7332 12104 12975 9341 4521 1430 276 28 & \\ 
\hline
140  &  423 2564 7096 11632 12368 8822 4227 1324 254 26 & \\ 
\hline
141  &  413 2483 6808 11043 11606 8177 3871 1201 230 24 & \\ 
\hline
142  &  427 2594 7196 11827 12614 9031 4347 1369 264 27 & \\ 
\hline
143  &  431 2622 7281 11973 12769 9135 4390 1379 265 27 & \\ 
\hline
144  &  431 2626 7309 12058 12915 9290 4494 1422 275 28 & \\ 
\hline
145  &  410 2459 6725 10881 11411 8029 3802 1183 228 24 & \\ 
\hline
146  &  428 2594 7176 11761 12514 8947 4307 1359 263 27 & \\ 
\hline
147  &  419 2522 6922 11243 11842 8373 3985 1245 240 25 & \\ 
\hline
148  &  451 2781 7840 13111 14243 10390 5089 1623 313 31 & \\ 
\hline
149  &  464 2900 8310 14168 15740 11778 5933 1948 385 38 & \\ 
\hline
150  &  446 2750 7757 12985 14123 10315 5058 1615 312 31 & \\ 
\hline
151  &  420 2541 7021 11496 12218 8719 4184 1314 253 26 & \\ 
\hline
152  &  441 2705 7584 12611 13622 9885 4823 1537 298 30 & \\ 
\hline
153  &  425 2575 7119 11651 12363 8799 4208 1317 253 26 & \\ 
\hline
154  &  448 2764 7801 13067 14223 10397 5102 1629 314 31 & \\ 
\hline
155  &  444 2737 7724 12949 14124 10363 5115 1647 321 32 & \\ 
\hline
156  &  452 2772 7753 12830 13755 9876 4750 1486 282 28 & \\ 
\hline
157  &  442 2706 7565 12529 13460 9696 4684 1473 281 28 & \\ 
\hline
158  &  441 2708 7602 12655 13676 9915 4821 1525 292 29 & \\ 
\hline
159  &  427 2596 7207 11850 12633 9026 4324 1350 257 26 & \\ 
\hline
160  &  452 2781 7813 13004 14042 10171 4944 1566 301 30 & \\ 
\hline
161  &  427 2586 7144 11681 12383 8806 4209 1317 253 26 & \\ 
\hline
162  &  400 2382 6467 10388 10820 7569 3570 1110 215 23 & \\ 
\hline
163  &  448 2764 7800 13061 14208 10377 5087 1623 313 31 & \\ 
\hline
164  &  470 2943 8444 14405 16000 11959 6011 1967 387 38 & \\ 
\hline
165  &  460 2857 8117 13684 14985 11014 5430 1740 336 33 & \\ 
\hline
166  &  418 2530 6996 11466 12198 8712 4183 1314 253 26 & \\ 
\hline
167  &  434 2640 7325 12025 12788 9108 4348 1353 257 26 & \\ 
\hline
168  &  425 2577 7132 11687 12418 8849 4235 1325 254 26 & \\ 
\hline
169  &  425 2581 7160 11772 12564 9004 4339 1368 264 27 & \\ 
\hline
170  &  430 2614 7255 11928 12724 9109 4382 1378 265 27 & \\ 
\hline
171  &  422 2557 7075 11597 12333 8801 4220 1323 254 26 & \\ 
\hline
172  &  411 2470 6772 10988 11556 8150 3863 1200 230 24 &  S7 \\ 
\hline
173  &  427 2586 7144 11681 12383 8806 4209 1317 253 26 & \\ 
\hline
174  &  400 2382 6467 10388 10820 7569 3570 1110 215 23 & \\ 
\hline
175  &  464 2898 8295 14119 15649 11673 5856 1913 376 37 & \\ 
\hline
176  &  442 2718 7644 12754 13822 10056 4911 1562 301 30 & \\ 
\hline
177  &  440 2698 7563 12576 13587 9864 4816 1536 298 30 & \\ 
\hline
178  &  423 2562 7083 11596 12313 8772 4200 1316 253 26 & \\ 
\hline
179  &  452 2781 7813 13004 14042 10171 4944 1566 301 30 & \\ 
\hline
\caption{Orbits of maximal prime cones for $\Flag_5$, the F-vectors of the corresponding polytopes, and combinatorially equivalent string polytopes resp. FFLV polytope.}\label{tab:f-vector5}
\end{longtable}

\end{center}
\normalsize

\section{Algebraic invariants of the \texorpdfstring{$\Flag_5$}{} string polytopes}\label{app:string5}

The  table below contains information on the $\Flag_5$ string polytopes and the FFLV polytope for $\rho$. It shows the reduced expressions underlying the string polytopes, whether the polytopes satisfy the weak Minkowski property, the weight vectors constructed in \S\ref{string:weight}, and whether the corresponding initial ideal is prime. The last column contains information on unimodular equivalences among these polytopes. If there is no information in this column this means that there is no unimodular equivalence between this polytope and any other polytope in the table.

\footnotesize
\begin{center}
\begin{longtable}{| l|  l| l| l|  l| l| l| }
\hline
Class&$\underline w_0$ & Normal & MP & Weight vector $-\bw_{\w_0}$& Prime & Uni. Eq. 
 \endhead\hline
S1& 1213214321 & yes &yes &  \shortstack{$(0, 512, 384, 112, 0, 256, 96, 768, 608, 480,$\\$ 0, 64, 320, 832, 15, 14,
 526, 398, 126, 12, $\\$ 268, 108, 780, 620, 492, 0, 8, 72, 328, 840)$} & yes & \shortstack{S18, S26,\\ S29} \\ \hline
S2& 1213243212 & yes &yes &  \shortstack{$(0,512,384,98, 0,256,96, 768,608, 480,$\\$
0, 64,320, 832, 30,28,540, 412, 123, 24,$\\$
280,120,792, 632, 504, 0,16,80,336, 848)$} & yes & - \\ \hline
S3 & 1213432312 & yes &no &  \shortstack{$(0, 512, 384, 74, 0, 256, 72, 768, 584, 456,$\\$
0, 64, 320, 832, 58, 56, 568,440, 111, 48, $\\$
304, 108, 816, 620, 492, 0, 32, 96, 352, 864)$} & no & - \\ \hline
S4& 1214321432 & yes &no &  \shortstack{$(0, 512, 384, 56, 0, 256, 48, 768, 560, 432, $\\$
0, 32, 288, 800, 120, 112, 624, 496, 63, 96, $\\$
352, 54, 864, 566, 438, 0, 64, 36, 292, 804)$} & no & - \\ \hline
S5& 1232124321 & yes &yes &  \shortstack{$(0, 512, 288, 224, 0, 256, 192, 768, 704, 432, $\\$
0, 128, 384, 896, 15, 14, 526, 302, 238, 12,$\\$
268, 204, 780, 716, 444, 0, 8, 136, 392, 904)$} & yes & - \\ \hline
S6& 1232143213 & yes &yes &  \shortstack{$(0, 512, 288, 224, 0, 256, 192, 768, 704, 420, $\\$
0,128, 384, 896, 30, 28, 540, 316, 252, 24,$\\$
280, 216, 792, 728, 437, 0, 16, 144, 400, 912)$} & yes & - \\ \hline
S7& 1232432123 & yes &yes &  \shortstack{$(0, 512, 260, 196, 0, 256, 192, 768, 704, 390,  $\\$
0, 128, 384, 896, 60, 56, 568, 310, 246, 48,$\\$
304, 240, 816, 752, 423, 0, 32, 160, 416, 928)$} & yes & - \\ \hline       
S8& 1234321232 & yes &no &  \shortstack{$(0, 512, 264, 152, 0, 256, 144, 768, 656, 396,  $\\$
0, 128, 384, 896, 120, 112, 624, 364, 219, 96,$\\$
352, 210, 864, 722, 462, 0, 64, 192, 448, 960)$} & no & - \\ \hline
S9& 1234321323 & yes &no &  \shortstack{$(0, 512, 264, 152, 0, 256, 144, 768, 656, 394,  $\\$
0, 128, 384, 896, 120, 112, 624,362, 222, 96,$\\$
352, 212, 864, 724, 459, 0, 64, 192, 448, 960)$}& no & - \\	 \hline
S10& 1243212432 & yes &no &  \shortstack{$(0, 512, 272, 112, 0, 256, 96, 768, 608, 344,  $\\$
0, 64, 320, 832, 240, 224, 736, 472, 119, 192,$\\$
448, 102, 960, 614, 350, 0, 128, 68, 324, 836)$} & no & - \\ \hline           
      
S11& 1243214323 & yes &no &  \shortstack{$(0, 512, 272, 112, 0, 256, 96, 768, 608, 338,  $\\$
0, 64, 320, 832, 240, 224, 736, 466, 126, 192,$\\$
448, 108, 960, 620, 347, 0, 128, 72, 328, 840)$}  & no & - \\ \hline
S12& 1321324321 &yes & no &  \shortstack{$(0, 512, 192, 448, 0, 128, 384, 640, 896, 240,  $\\$
0, 256, 160, 672, 15, 14, 526, 206, 462, 12,$\\$
140, 396, 652, 908, 252, 0, 8, 264, 168, 680)$} & no & - \\ \hline    
S13& 1321343231 & yes &no &  \shortstack{$(0, 512, 192, 448, 0, 128, 384, 640, 896, 228, $\\$
0, 256, 160, 672, 29, 28, 540, 220, 476, 24,$\\$
152, 408, 664, 920, 246, 0, 16, 272, 176, 688)$} & no & - \\ \hline
S14& 1321432143 & yes &no &  \shortstack{$(0, 512, 192, 448, 0, 128, 384, 640, 896, 216,  $\\$
0, 256, 144, 656, 60, 56, 568, 248, 504, 48,$\\$
176, 432, 688, 944, 219, 0, 32, 288, 146, 658)$} & no & - \\ \hline
S15& 1323432123 & yes &no &  \shortstack{$(0, 512, 132, 388, 0, 128, 384, 640, 896, 198,  $\\$
0, 256, 192, 704, 60, 56, 568, 182, 438, 48,$\\$
176, 432, 688, 944, 231, 0, 32, 288, 224, 736)$} & no & - \\ \hline
S16& 1324321243 &yes & no &  \shortstack{$(0, 512, 136, 392, 0, 128, 384, 640, 896, 172,  $\\$
0, 256, 160, 672, 120, 112, 624, 236, 492, 96,$\\$
224, 480, 736, 992, 175, 0, 64, 320, 162, 674)$} & no & - \\ \hline
S17& 1343231243 & yes &no &  \shortstack{$(0, 512, 48, 304, 0, 32, 288, 544, 800, 60,  $\\$
0, 256, 40, 552, 240, 224, 736, 188, 444, 192,$\\$
168, 424, 680, 936, 63, 0, 128, 384, 42, 554)$} & no & - \\ \hline
S18& 2123214321 & yes &yes &  \shortstack{$(0, 256, 768, 112, 0, 512, 96, 384, 352, 864, $\\$
0, 64, 576, 448, 15, 14, 270, 782,126, 12, 524,$\\$
108, 396, 364, 876, 0, 8, 72, 584, 456)$} & yes & \shortstack{S1, S26, \\ S29}  \\ \hline
S19& 2123243212 & yes &yes &  \shortstack{$(0, 256, 768, 98, 0, 512, 96, 384, 352, 864, $\\$
0, 64, 576, 448, 30, 28, 284, 796, 123, 24, 536,$\\$
      120, 408, 376, 888, 0, 16, 80, 592, 464)$} & yes & - \\ \hline
S20& 2123432132 & yes &no &  \shortstack{$(0, 256, 768, 76, 0, 512, 72, 384, 328, 840, $\\$
0, 64, 576, 448, 60, 56, 312, 824,111, 48, 560,$\\$
      106, 432, 362, 874, 0, 32, 96, 608, 480)$} & no & - \\ \hline
S21& 2132134321 &yes & yes &  \shortstack{$(0, 256, 768, 224, 0, 512, 192, 320, 448, 960, $\\$
0, 128, 640, 336, 15, 14, 270,782, 238, 12,$\\$
      524, 204, 332, 460, 972, 0, 8, 136, 648, 344)$} & yes & - \\ \hline
S22& 2132143214 &yes & yes &  \shortstack{$(0, 256, 768, 224, 0, 512, 192, 320, 448, 960, $\\$
0, 128, 640, 328, 30, 28, 284,796, 252, 24,$\\$
      536, 216, 344, 472, 984, 0, 16, 144, 656, 329)$} &yes & - \\ \hline
S23& 2132343212 & yes &yes &  \shortstack{$(0, 256, 768, 194, 0, 512, 192, 320, 448, 960, $\\$
0, 128, 640, 352, 30, 28, 284,796, 219, 24,$\\$
      536, 216, 344, 472, 984, 0, 16, 144, 656, 368)$} & yes & - \\ \hline
S24& 2132432124 & yes &yes &  \shortstack{$(0, 256, 768, 196, 0, 512, 192, 320, 448, 960, $\\$
0, 128, 640, 336, 60, 56, 312, 824, 246, 48,$\\$
      560, 240, 368, 496, 1008, 0, 32, 160, 672, 337)$} & yes & - \\ \hline
S25& 2134321324 &yes & no &  \shortstack{$(0, 256, 768, 152, 0, 512, 144, 272, 400, 912, $\\$
0, 128, 640, 276, 120, 112, 368, 880, 222, 96,$\\$
      608, 212, 340, 468, 980, 0, 64, 192, 704, 277)$} & no & - \\ \hline    
S26& 2321234321 &yes & yes &  \shortstack{$(0, 64, 576, 448, 0, 512, 384, 96, 352, 864, $\\$
0, 256, 768, 112, 15, 14, 78,590, 462, 12, 524,$\\$
      396, 108, 364, 876, 0, 8, 264, 776, 120)$} & yes & \shortstack{S1, S18, \\ S29} \\  \hline
S27& 2321243214 & yes &yes &  \shortstack{$(0, 64, 576, 448, 0, 512, 384, 96, 352, 864, $\\$
0, 256, 768, 104, 30, 28, 92,604, 476, 24, 536,$\\$
      408, 120, 376, 888, 0, 16, 272, 784, 105)$} & yes & - \\ \hline
S28& 2321432134 & yes &yes &  \shortstack{$(0, 64, 576, 448, 0, 512, 384, 72, 328, 840, $\\$
0, 256, 768, 74, 60, 56, 120, 632, 504, 48, 560,$\\$
      432, 106, 362, 874, 0, 32, 288, 800, 75)$} & yes & - \\ \hline
S29& 2324321234 & yes &yes &  \shortstack{$(0, 8, 520, 392, 0, 512, 384, 12, 268, 780, $\\$
0, 256, 768, 14, 120, 112, 108,620, 492, 96, 608,$\\$
      480, 78, 334, 846, 0, 64, 320, 832, 15)$} & yes &  \shortstack{S1, S18,\\ S26}\\ \hline
S30& 2343212324 &yes & yes &  \shortstack{$(0, 16, 528, 304, 0, 512, 288, 24, 280, 792, $\\$
0, 256, 768, 28, 240, 224, 216, 728, 438, 192,$\\$
      704, 420, 156, 412, 924, 0, 128, 384, 896, 29)$} & yes & - \\ \hline
S31& 2343213234 &yes & yes &  \shortstack{$(0, 16, 528, 304, 0, 512, 288, 20, 276, 788, $\\$ 
0, 256, 768, 22, 240, 224, 212, 724, 444, 192,$\\$
      704, 424, 150, 406, 918, 0, 128, 384, 896, 23)$} & yes & - \\ \hline
FFLV & - & yes &yes &  \shortstack{$w^{\text{reg}}=(0,4,6,6,0,3,4,6,6,9,0,2,4,6,4,$\\$ 3,4,7,8,2,3,5,4,6,8,0,1,2,3,4) $\\$ 
w^{\text{min}}=(0,3,4,3,0,2,2,4,3,5,0,1,2,3,1,$\\$ 1,1,3,3,1,1,2,1,2,3,0,1,1,1,1)$}  &\shortstack{yes\\ {}\\ yes} &-\\
\hline
\caption{String polytopes depending on $\underline w_0$ and the FFLV polytope for $\Flag_5$ and $\rho$, their normality, the weak Minkowski property, the weight vectors constructed in \S\protect{\ref{string:weight}}, primeness of the binomial initial ideals, and unimodular equivalences among the polytopes.}\label{tab:stringweight5}
\end{longtable}
\end{center}
\normalsize
\end{appendices}

\newpage

\thispagestyle{empty}

Ich   versichere,   dass   ich   die   von   mir   
vorgelegte   Dissertation   selbständig   
angefertigt,  die  benutzten  Quellen  und  
Hilfsmittel  vollständig  angegeben  und  die  
Stellen  der  Arbeit  
- einschließlich  Tabellen,  Karten  und  Abbildungen  -,  die  
anderen  Werken  im  Wortlaut  oder  dem  Sinn  nach  entnommen  sind,  in  jedem  
Einzelfall  als  Entlehnung  kenntlich  gemacht  habe;  dass  diese  Dissertation  noch  
keiner  anderen  Fakultät  oder  Universität  zur  Prüfung  vorgelegen  hat;  dass  sie  - abgesehen  von  unten  angegebenen  Teilpublikationen  -
  noch  nicht  veröffentlicht  
worden   ist   sowie,   dass   ich   eine   solche  Veröffentlichung  vor  Abschluss  des  
Promotionsverfahrens     nicht     vornehmen     werde.     Die     Bestimmungen     der     
Promotionsordnung  sind  mir  bekannt.  Die  
von  mir  vorgelegte  Dissertation  ist  von  
Prof. Dr. Peter Littelmann betreut worden.

\newpage

\end{document}